\documentclass[a4paper, reqno]{amsart}
\usepackage[numbers, sort&compress]{natbib}

\usepackage{microtype}

\usepackage[utf8]{inputenc}

\usepackage{thmtools,thm-restate}

\usepackage[colorlinks,citecolor=blue,urlcolor=blue, linkcolor=blue]{hyperref}

\usepackage{mathrsfs} 
\usepackage[normalem]{ulem}
\usepackage{amssymb}
\usepackage{amsfonts}
\usepackage{quiver}
\usepackage{amssymb}
\usepackage{amsmath}
\usepackage{amsthm}
\usepackage{ulem}

\usepackage[shortlabels]{enumitem}

\usepackage{csquotes}

\usepackage{amsmath}
\usepackage{amssymb}
\usetikzlibrary{arrows.meta}
\usepackage{tabularx}
\usepackage{centernot}
\usepackage{mathtools}
\usepackage{stmaryrd}
\usepackage{amsthm,amssymb}
\usepackage{etoolbox,color}
\usepackage{tikz}
\usepackage{amssymb}
\usetikzlibrary{matrix}
\usepackage{tikz-cd}
\usepackage{tikz}
\usepackage{marginnote}
\definecolor{mygray}{gray}{0.85}
\usepackage[linecolor=black, backgroundcolor=mygray,colorinlistoftodos,prependcaption,textsize=small]{todonotes}
\usepackage{xargs}

\renewcommand{\leq}{\leqslant}
\renewcommand{\geq}{\geqslant}

\newcommand{\mrm}[1]{\mathrm{#1}}

\usepackage[backgroundcolor=mygray,colorinlistoftodos,prependcaption,textsize=small]{todonotes}
\usepackage{xcolor}
\newcommand{\todor}[1]{\todo[inline, color=red!73]{#1}}
\newcommand{\todob}[1]{\todo[inline, color=cyan]{#1}}
\newcommand{\todogr}[1]{\todo[inline, color=green]{#1}}
\newcommand{\todoy}[1]{\todo[inline, color=yellow]{#1}}
\newcommand{\todog}[1]{\todo[inline, color=mygray]{#1}}
\newcommand{\todov}[1]{\todo[inline, color=violet!60]{#1}}
\newcommand{\todoo}[1]{\todo[inline, color=orange]{#1}}

\makeatletter
\def\subsection{\@startsection{subsection}{3}%
  \z@{.5\linespacing\@plus.7\linespacing}{.3\linespacing}%
  {\bfseries\centering}}
\makeatother

\makeatletter
\def\subsubsection{\@startsection{subsubsection}{3}%
  \z@{.5\linespacing\@plus.7\linespacing}{.3\linespacing}%
  {\centering}}
\makeatother

\theoremstyle{definition}
\newtheorem{theorem}{Theorem}[section]

\newtheorem{corollary}[theorem]{Corollary}
\newtheorem{proposition}[theorem]{Proposition}
\newtheorem{example}[theorem]{Example}

\newtheorem{hypothesis}[theorem]{Hypothesis}
\newtheorem{question}[theorem]{Question}

\newtheorem{context}[theorem]{Context}

\theoremstyle{plain}

\theoremstyle{definition}

\newtheorem{lemma}[theorem]{Lemma}
\newtheorem{fact}[theorem]{Fact}
\newtheorem{definition}[theorem]{Definition}

\newtheorem{remark}[theorem]{Remark}

\newtheorem{notation}[theorem]{Notation}

\newcommand{\cf}{\ensuremath{\mathrm{cf}}}
\newcommand{\id}{\ensuremath{\mathrm{id}}}
\newcommand{\at}{\ensuremath{\mathrm{at}}}
\newcommand{\tp}{\ensuremath{\mathrm{tp}}}
\newcommand{\qf}{\ensuremath{\mathrm{qf}}}
\newcommand{\Hom}{\ensuremath{\mathrm{Hom}}}
\newcommand{\Ext}{\ensuremath{\mathrm{Ext}}}

\newcommand{\categ}{\ensuremath{\mathcal{C}}}
\newcommand{\categn}{\ensuremath{\mathcal{D}}}
\newcommand{\collmods}{\ensuremath{\mathbf{C}}}
\newcommand{\morphs}{\ensuremath{\mathfrak{m}}}
\newcommand{\kaec}{\ensuremath{\mathbf{K}}}
\newcommand{\leqk}{\ensuremath{\preccurlyeq}}
\newcommand{\gtp}{\ensuremath{\mathrm{ortp}}}
\newcommand{\gS}{\ensuremath{\mathcal{S}}}
\newcommand{\kaecgoth}{\ensuremath{\mathcal{K}}}
\newcommand{\kaecg}{\ensuremath{\mathbf{K}_{\mathcal{K}}}}
\newcommand{\leqkg}{\ensuremath{\preccurlyeq_{\mathcal{K}}}}
\newcommand{\hul}{\ensuremath{H}}

\newcommand{\ordnl}{\ensuremath{\alpha}}
\newcommand{\tchar}{\ensuremath{\mathbf{t}}}
\newcommand{\tcharset}{\ensuremath{\mathcal{T}}}

\newcommand{\leqp}{\ensuremath{\leq_{\mrm{pp}}}}
\newcommand{\leqo}{\ensuremath{\leq_\oplus}}
\newcommand{\bal}{\ensuremath{\mrm{b}}}
\newcommand{\pp}{\ensuremath{\mrm{pp}}}
\newcommand{\emb}{\ensuremath{\mrm{em}}}
\newcommand{\leqb}{\ensuremath{\leq_\bal}}

\newcommand{\oleq}{\ensuremath{\leq_{\oplus}}}

\newcommand{\frku}{\ensuremath{{\mathfrak{u}}}}
\newcommand{\frkw}{\ensuremath{{\mathfrak{w}}}}
\newcommand{\wlarge}{\ensuremath{\mathcal{W}^{\mrm{large}}}}
\newcommand{\wsmall}{\ensuremath{\mathcal{W}^{\mrm{small}}}}
\newcommand{\nbo}{\ensuremath{\mrm{nb}}}

\newcommand{\rmod}{\ensuremath{{R\mh\mrm{Mod}}}}
\newcommand{\flatmod}{\ensuremath{{\mh\mrm{Flat}}}}
\newcommand{\pinj}{\ensuremath{{\mrm{PI}}}}
\newcommand{\inj}{\ensuremath{{\mrm{Inj}}}}
\newcommand{\rdinj}{\ensuremath{{\mrm{RD}\mh\mrm{Inj}}}}
\newcommand{\sigmainj}{\ensuremath{{\Sigma\mh\mrm{Inj}}}}
\newcommand{\sigmapinj}{\ensuremath{{\Sigma\mh\mrm{PI}}}}
\newcommand{\tfab}{\ensuremath{{\mrm{TF}_\mathbb{Z}}}}
\newcommand{\redtfab}{\ensuremath{{\mrm{RTF}_\mathbb{Z}}}}
\newcommand{\rgrp}{\ensuremath{{\mrm{Red}_\mathbb{Z}}}}
\newcommand{\rpgrp}{\ensuremath{{p\mh\mrm{Red}_\mathbb{Z}}}}
\newcommand{\cotfree}{\ensuremath{{\mrm{CotFree}_\mathbb{Z}}}}
\newcommand{\slendergrp}{\ensuremath{{\mrm{Slen}_\mathbb{Z}}}}
\newcommand{\cotorsion}{\ensuremath{{\mrm{Cot}}}}
\newcommand{\almostfree}{\ensuremath{{\mh\mrm{Free}_\mathbb{Z}}}}
\newcommand{\Superstable}{\ensuremath{{\mrm{Superstable}}}}
\newcommand{\pgrp}{\ensuremath{{p\mh\mrm{Grp}_\mathbb{Z}}}}
\newcommand{\abgrps}{\ensuremath{{\mathbb{Z}\mh\mrm{Mod}}}}

\input xy
\xyoption{all}

\newbox\noforkbox \newdimen\forklinewidth
\setbox0\hbox{$\textstyle\smile$}
\setbox1\hbox to \wd0{\hfil\vrule width \forklinewidth depth-2pt
 height 10pt \hfil}
\setbox\noforkbox\hbox{\lower 2pt\box1\lower 2pt\box0\relax}
\def\unionstick{\mathop{\copy\noforkbox}\limits}

\def\nonfork_#1{\unionstick_{\textstyle #1}}

\setbox0\hbox{$\textstyle\smile$}
\setbox1\hbox to \wd0{\hfil{\sl /\/}\hfil}
\setbox2\hbox to \wd0{\hfil\vrule height 10pt depth -2pt width
              \forklinewidth\hfil}
\newbox\doesforkbox
\setbox\doesforkbox\hbox{\lower 2pt\box1 \lower 2pt\box2\lower2pt\box0\relax}
\def\nunionstick{\mathop{\copy\doesforkbox}\limits}

\def\fork_#1{\nunionstick_{\textstyle #1}}

\newcommand{\dnf}{\ensuremath{\unionstick}}

\newcommand{\nf}{\ensuremath{\unionstick}}
\newcommand{\dnfb}[4]{#2 \overset{#4}{\underset{#1}{\overline{\nf}}} #3}

\newcommand{\inflang}{\ensuremath{\mathfrak{L}}}

\DeclareMathSymbol{\mh}{\mathord}{operators}{`\-}

\let\tdef\textit

\usepackage{xcolor}

\newif\ifshowtodos
\showtodostrue
\ifshowtodos
\else
  \renewcommand{\todog}[1]{}
  \renewcommand{\todob}[1]{}
  \renewcommand{\todogr}[1]{}
  \renewcommand{\todov}[1]{}
  \renewcommand{\todoy}[1]{}
  \renewcommand{\todoo}[1]{}
  \renewcommand{\todor}[1]{}
\fi

\usepackage[framemethod=TikZ]{mdframed}
\usepackage{environ}
\usepackage{xcolor}

\newcommand{\ssk}{\smallskip}
\newcommand{\nin}{\noindent}
\let\ab\allowbreak

\numberwithin{equation}{section}

\newcounter{introcount}

\newtheorem{introthm}[introcount]{Theorem}
\newtheorem{introq}[introcount]{Question}

\newtheorem{introcor}[introcount]{Corollary}

\begin{document}

\begin{abstract}
We prove several new results in the theory of $\mu\mh\mrm{AEC}$s, focusing mainly on (almost) stability, with the primary objective of undertaking a systematic study of $\mu\mh\mrm{AEC}$s of $R$-modules. Our main results are the following.

1.  We show that, under suitable syntactic assumptions, all tame $\mu\mh\mrm{AEC}$s 
of $R$-modules (where $R$ is a ring) are almost stable, and are stable if they additionally 
satisfy a strong amalgamation property. This result extends the work of the second author and Shelah~\cite{paolinishelah} 
to the setting of $\mu\mh\mrm{AEC}$s. 

2.  We then turn to applications to concrete 
$\mu\mh\mrm{AEC}$s of $R$-modules. Our main result in this direction is that $(\rmod, \leqp^\mu)$ has a stable independence relation and 
is a stable and tame $\mu\mh\mrm{AEC}$, where $\leqp^\mu$ denotes the 
$\mu$-pure submodule relation. We also prove similar stability results for various classes of 
abelian groups, including the $\aleph_1\mh\mrm{AEC}$ of torsion-free abelian groups with the balanced subgroup relation. Moreover, we prove the almost stability of all $\mu\mh\mrm{AEC}$s of modules of the form
$(\rmod, \leqk)$, where $\leqk$ refines the direct summand relation and satisfies a strong form of
coherence.

3.  Finally, we study $\mu\mh\mrm{AEC}$s of the form $(\kaec, \leqo)$, where $\kaec$ is a class of pure-injective $R$-modules (note that this is, in general, not an $\mrm{AEC}$), and use our results to show that, for many natural choices of $\kaec$, the class $(\kaec, \leqo)$ has a stable independence relation and is therefore stable and tame. We use these results to give a sufficient condition for abstract classes of modules of the form $(\kaec, \leqp)$ to be stable when $\kaec$ is closed under pure-injective envelopes. This generalizes, by a substantially different proof, results of Mazari-Armida~\cite{somestablenonelementary}.
\end{abstract}

\title{$\mu$-abstract elementary classes of modules}

\author{Roberto Carnevale}
\address{Department of Mathematics ``Giuseppe Peano'', University of Torino, Via Carlo Alberto 10, 10123 {Torino,} Italy.}

\author{Gianluca Paolini}
\address{Department of Mathematics ``Giuseppe Peano'', University of Torino, Via Carlo Alberto 10, 10123 {Torino,} Italy.}
\thanks{Research of Gianluca Paolini was {supported} by project PRIN 2022 ``Models, sets and classifications'', prot. 2022TECZJA, and by INdAM Project 2024 (Consolidator grant) ``Groups, Crystals and Classifications''.}
\date{\today}
\maketitle

\newpage
\thispagestyle{empty}        
\vspace*{\fill}
{\centering
  \tableofcontents
  \par}
\vspace*{\fill}
\newpage

\section{Introduction}

In order to provide an effective common generalization of the methods of first-order model theory to various extensions of first-order logic, and to develop a {\em classification theory} for these logical contexts, in \cite{shelahaecoriginal} Shelah introduced a {model-theoretic} framework that goes under the name of {\em abstract elementary classes} ($\mrm{AEC}$s). In another direction, Makkai and Par{\'e} \cite{makkaipare}, relying on previous work of Grothendieck \cite{grothendieck} and of Gabriel and Ulmer \cite{gabrielulmer}, developed the theory of {\em accessible categories}, from the point of view of categorical model theory. These two lines of research remained {separate} until recent years, when these two {model-theoretic} frameworks were reconciled in the notion of $\mu\mh\mrm{AEC}$s, a natural generalization of the theory of $\mrm{AEC}$s which relaxes some conditions on the existence of limits. This reconciliation was, in fact, complete: up to categorical equivalence, accessible categories with all morphisms being monomorphisms and $\mu\mh\mrm{AEC}$s are exactly the same thing.
This led to {a fruitful transfer of results between the categorical and model-theoretic settings}. Perhaps one of the most striking examples of this is the {precise correspondence} between the notion of cofibrant generation and {the} admissibility of a stable independence relation, one of the core concepts of modern model theory.

\smallskip	One of the criticisms of the theory of $\mrm{AEC}$s has been the lack of genuinely new examples, i.e., examples not subsumed by first-order or infinitary model theory. This state of affairs changed with the work of Zilber
on pseudo-exponential fields \cite{zilber1,zilber2} and, more systematically, with the pioneering works of Mazari-Armida on the study of $\mrm{AEC}$s of $R$-modules \cite{onuniversalmodules, mazarimodeltheoretic, mazarisuperstableflat, somestablenonelementary, mazarinoetherian}. The aim of this paper is to extend this analysis of the model theory of $R$-modules to $\mu\mh\mrm{AEC}$s. Far from producing isolated new examples, we show that a rich and genuinely new model-theoretic theory of modules emerges in this generalized setting, {with substantial structural content and consequences that reach back into the classical context}. Indeed, this broader framework {provides room} for many new examples and case studies, and, far from being a mere generalization for its own sake, it feeds back into the original context: phenomena that are invisible or inaccessible at the level of $\mrm{AEC}$s become natural once $\mu\mh\mrm{AEC}$s are allowed, and yield genuinely new information about $\mrm{AEC}$s of modules themselves. A canonical instance of this phenomenon is the notion of 
``\emph{stability transfer}'' that we introduce and explore in this paper: a mechanism by which the stability of a given class is deduced from that of a ``stronger'' one, whose stability is often easier to establish. The guiding example is the relation of $\mu$-pureness, the main object of study of the present paper, which in a precise sense is cofinal in the poset of relations on $R\text{-}\mrm{Mod}$ giving rise to $\mu\mh\mrm{AEC}$s (cf. Section~\ref{sect:lpure}), and whose stability therefore controls that of all the \mbox{submodel relations below it.}

\subsection{Stability, almost stability and stability transfers for \texorpdfstring{$\mu\mh\mrm{AEC}$s}{μ-AECs}}

	One of the characteristics of Shelah's approach to model theory and classification theory is the identification of ``dividing lines'', i.e., {model-theoretic} properties that exhibit a structure/non-structure dichotomy of some sort. The central notion of this theory, both in first-order and non-elementary model theory, is the notion of {\em stability}, a property which has to do with the number of types. Most of our paper is devoted to the exploration of questions of stability in $\mu\mh\mrm{AEC}$s of modules.
	
\smallskip	Baur~\cite{baur} and, independently, Fisher~\cite{fisher} proved in 1975 that, for every ring $R$, every complete first-order theory of $R$-modules is stable. This result, together with several subsequent results of Mazari-Armida on the stability of $\mrm{AEC}$s of the form $(\mathbf{K}, \leq_\mrm{pp})$, where $\leq_\mrm{pp}$ denotes the pure submodule relation, led him to pose the following central question:

\begin{introq}[\protect{\cite[Question~2.12]{mazarimodeltheoretic}}]
\label{the_question} Let $R$ be a ring and let $\leq_{\mrm{pp}}$ denote the pure submodule relation. If $(\mathbf{K}, \leq_{\mrm{pp}})$ is an abstract elementary class with $\mathbf{K} \subseteq R$-$\mrm{Mod}$, is $(\mathbf{K}, \leq_{\mrm{pp}})$ stable? Is this true when $R = \mathbb{Z}$? Under what conditions on $R$ does this hold?
\end{introq}

Question~\ref{the_question} was resolved by the second named author and S.~Shelah in \cite{paolinishelah}, where they showed that there exist unstable $\mrm{AEC}$s of torsion-free abelian groups with the pure subgroup relation (see also \cite{mazaripaolini}). Despite this failure of stability, it was also shown in \cite{paolinishelah} that, under the assumption of tameness, every $\mrm{AEC}$ of $R$-modules satisfies a newly introduced weaker form of stability; {tameness is a locality property for Galois types that is often considered in the context of $\mu\mh\mrm{AEC}$s}, and this weaker notion is called {\em almost stability}. Interestingly, under the amalgamation property ($\mrm{AP}$), almost stability collapses to the usual notion of stability. {The distinction between almost stability and stability is therefore a genuinely non-elementary phenomenon, since the class of models of a complete first-order theory, ordered by elementary substructure, always satisfies $\mrm{AP}$.} One of the motivations for the first part of this paper is to contribute to the general theory of $\mu\mh\mrm{AEC}$s by exploring almost stability in this setting, a topic that has not previously been investigated, and by adapting the methods of \cite{paolinishelah} to $\mu\mh\mrm{AEC}$s of $R$-modules.

\ssk In Section~\ref{sect:almoststability} we introduce almost stability for $\mu\mh\mrm{AEC}$s and explore under which conditions almost stability implies stability; we will see, e.g., that this is the case under the assumption of what we call {$\infty\mh\mrm{AP}$}. Furthermore, in this section we prove the following {aforementioned} ``stability transfer''  (see also Lemma~\ref{lem:alstabtransfer} for a more general technical result):

\begin{introcor}[Corollary~\ref{cor:alstabtransfermu}]
\label{cor5}
Let $\kaecgoth_1 = (\kaec, \leqk_1)$ and $\kaecgoth_2 = (\kaec, \leqk_2)$ be $\mu\mh\mrm{AEC}$s such that $\leqk_1$ refines $\leqk_2$, and $\kaecgoth_2$ is almost stable in $\lambda = \lambda^{<\mu}+\mrm{LS}_\mu(\kaecgoth_2)$. Then $\kaecgoth_1$ is almost stable in $\lambda$. In particular, if in addition $\kaecgoth_1$ has $\infty\mh\mrm{AP}$, then it is $\lambda$-stable.
\end{introcor}

We then move to the question of almost stability for $\mu\mh\mrm{AEC}$s of $R$-modules. We isolate here a syntactic property of a $\mu\mh\mrm{AEC}$ $\kaecgoth = (\kaec, \leqk)$, which we refer to as $\leqk$ being syntactic on $\kaec$ (cf. Definition~\ref{def:syntacticAC}). This property is precisely what is needed to carry the arguments of \cite{paolinishelah} over to the present setting; the adaptation, however, is far from routine: the original proof relies at several points on features specific to $\mrm{AEC}$s that simply do not survive the passage to $\mu\mh\mrm{AEC}$s, and recovering them requires genuine care. Already in \cite{paolinishelah} the use of \cite{shelahvillaveces} was essential in deducing almost stability; along the way we found a new way of bringing it to bear, which streamlines the argument considerably and yields a noticeably easier proof than the one available there. The end result is the following substantial generalization of the main result of \cite{paolinishelah}. We stress that although \cite{shelahvillaveces} ensures that every abstract elementary class is syntactic in the sense of Definition~\ref{def:syntacticAC} (see also \cite[Fact 2.12]{paolinishelah}), the same is not true of $\mu\mh\mrm{AEC}$s: not every $\mu\mh\mrm{AEC}$ is syntactic. Indeed, Boney and Walker \cite{boneypreparation} recently found an $\aleph_1\mh\mrm{AEC}$ $(\kaec, \leqk)$ {such that $\leqk$ is not syntactic on $\kaec$ (cf. Definition~\ref{def:syntacticAC})}, and it remains open whether such an example can be found with $\kaec\subseteq\rmod$.
	
\begin{introthm}[Theorem~\ref{thm:syn2}]
\label{introthm:syn2}
Let $\kaecgoth = (\kaec, \leqk)$ be a $\mu\mh\mrm{AEC}$ of $R$-modules and $\kappa = \kappa^{<\mu}+ \mrm{LS}_\mu(\kaecgoth)$ an infinite cardinal. If $\kaecgoth$ is $\kappa$-tame and $\leqk$ is $\kappa$-syntactic on $\kaec$, then there is a cardinal $\xi\geq \kappa$ such that, for every $\lambda = \lambda^\xi$, {$\kaecgoth$ is almost $\lambda$-stable}. In particular, $\kaecgoth$ is almost stable, and if $\kaecgoth$ has $\infty\mh\mrm{AP}$, then it is stable.
\end{introthm}

\subsection{A {model-theoretic} analysis of \texorpdfstring{$\mu$}{μ}-pureness}

We now move to the part of our work which deals with concrete applications of $\mu\mh\mrm{AEC}$s to $R$-modules. One of the most important notions in the theory of modules is that of {\em purity}; see, e.g., the book \cite{purityspectra}, which is organized entirely around this notion. We recall that, given $A \leq B$ in $\rmod$, the submodule $A$ is {\em pure} in $B$ if every finite system of linear equations with parameters in $A$ that has a solution in $B$ already has a solution in $A$. Although purity is a purely algebraic notion, it has been studied extensively by model theorists, for the following basic reason: given elementarily equivalent $A \leq B$ in $\rmod$, the module $A$ is an elementary substructure of $B$ if and only if $A$ is pure in $B$. This in turn follows from the well-known elimination of quantifiers for $R$-modules down to {Boolean combinations of $\mrm{pp}$-formulas}.

Generalizing the above definition, we say that a module $A$ is {\em $\mu$-pure} in $B$ if every system of fewer than $\mu$ linear equations with parameters in $A$ that has a solution in $B$ already has a solution in $A$. Clearly, $\aleph_0$-purity coincides with purity. In the context of abelian groups the definition of $\mu$-purity is credited to Gacs{\'a}lyi \cite{originalmupure}, although Fuchs \cite{paperfuchsdsums} was the first to study it thoroughly and to give it its name.

The notion has proved useful in the study of almost free modules (cf.\ \cite{almost}), where it interacts closely with set-theoretic methods and questions of independence from $\mrm{ZFC}$, and it has also been generalized to the setting of accessible categories. Indeed, in any accessible category one can define what it means for a morphism to be {\em $\mu$-pure}, and this notion coincides with the one given above when one considers the category of $R$-modules with homomorphisms (cf.\ \cite{adamekbook}). This is no isolated coincidence: $\mu$-purity is one of a family of categorical notions whose roots lie in module theory but whose natural home turns out to be the theory of accessible categories, which is precisely the bridge our paper sets out to exploit. Therefore, in Section~\ref{sect:lpure} we study  $\mu\mh\mrm{AEC}$s of modules of the form $(\kaec, \leqp^\mu)$, and we show the following generalization of \cite[Theorem 4.17]{somestablenonelementary}, \mbox{which dealt with the case $\mu = \aleph_0$.}

\begin{introthm}[Theorem~\ref{thm:lpuresumup}]
\label{theorem2}
Let $\kaecgoth_\mu = (\kaec, \leqp^\mu)$ be a $\mu\mh\mrm{AEC}$ of $R$-modules such that $\kaec$ is closed under finite direct sums, $\mu$-pure submodules, and $\mu$-pure quotients. Then $\kaecgoth_\mu$ admits a stable independence relation and {is therefore} stable and tame.
\end{introthm}

	We show in Section~\ref{sect:lpureexample} that many classes $\kaec$ of interest satisfy the assumptions of Theorem~\ref{theorem2}. Before moving on to other applications of our methods, we mention here some interesting corollaries of Theorem~\ref{theorem2}. {First, observe that if $A$ is a summand of $B$,} then $A \leqp^\mu B$ for every cardinal $\mu$; in fact, it is well known that the relation $\oleq$ of being a direct summand is the limit of all the relations $\leqp^\mu$ for $\mu \in \mrm{Card}$. 
More generally, following \cite{deconstructibleaec}, we say that a transitive relation $\leqk$ on $\rmod$ {\em refines direct summands} if $A \oleq B$ implies $A \leqk B$.
In \cite{paolinipreparation} the second named author of this paper, together with Hyttinen and Mazari-Armida, began a general {model-theoretic} study of the poset of strong submodel relations on $\rmod$ giving rise to $\mrm{AEC}$s. 
This was inspired by works of Trlifaj et al.\ \cite{categoricityfor, deconstructibleaec, baldwinaec} on $\mrm{AEC}$s of modules refining direct summands, using methods from the theory of deconstructible classes of modules (see also the recent paper \cite{mazaridecon} on this). Several questions are left open in \cite{paolinipreparation}, in particular whether a certain unbounded chain of strong submodel relations identified there is cofinal in the aforementioned poset. We will now see that moving from the context of $\mrm{AEC}$s of $R$-modules to that of $\mu\mh\mrm{AEC}$s of $R$-modules may shed some light on these questions. In perfect analogy with what is done in \cite{paolinipreparation}, one can consider the poset of strong submodel relations on $\rmod$ giving rise to $\mu\mh\mrm{AEC}$s and refining direct summands. One might wonder what general results can be proved about this poset, and we now give evidence that there is some structure here. Recall that any $\mu\mh\mrm{AEC}$ $\kaecgoth = (\kaec, \leqk)$ satisfies the following axiom, known as the {\em coherence axiom}, that is, for every $A, B, C\in \kaec$:
\[
A \leq B \leqk C \text{ and } A\leqk C \; \Rightarrow \; A \leqk B.
\]
We consider a stronger version of this axiom, which we call {\em strong coherence}, that is, for every $A,B,C\in \kaec$ one has:
\[
A \leq B \leq C \text{ and } A\leqk C \; \Rightarrow \; A \leqk B.
\]
The main structural result on the poset of strong submodel relations refining direct summands which give rise to $\mu\mh\mrm{AEC}$s is the following, which, in combination with the ``stability transfer'' of Corollary~\ref{cor5}, allows us to deduce a general (almost) stability result. This represents a potential strategy for bypassing the tameness assumption in Theorem~\ref{introthm:syn2}.
We note that, crucially, all the  examples of $\mu\mh\mrm{AEC}$s of $R$-modules explored in this paper satisfy the axiom of strong coherence, as do all examples refining direct summands known to the authors.
Therefore, the following theorem could potentially be used to establish an almost stability result for all strong submodel relations refining direct summands; we leave this open.

\begin{introthm}[Theorem~\ref{thm:astablecoherent}] 
The relations $(\leqp^\mu \mid \mu \in \mrm{Card})$ are cofinal in the poset of strong submodel relations refining direct summands {that give rise to $\mu\mh\mrm{AEC}$s} (for some regular cardinal $\mu$) and which satisfy strong coherence. In particular, if $\leqk$ is a strong submodel relation as above, then $(\rmod, \leqk)$ is almost stable, and if in addition $(\rmod, \leqk)$ has $\infty\mh\mrm{AP}$, then $(\rmod, \leqk)$ is stable.
\end{introthm}

	Coming back to applications of our methods to $R$-modules, we now focus on the case of $\mathbb{Z}$-modules, i.e., abelian groups. The point of the following theorem is to convince the reader that {many natural strengthenings of the submodule relation considered in module theory} fail to give rise to $\mrm{AEC}$s, yet {\em do} give rise to $\mu\mh\mrm{AEC}$s. This {makes it possible to undertake} a systematic {model-theoretic} analysis of the strong submodel relations studied in module and abelian group theory, a programme we believe can be carried much further. Indeed, we hope that this paper will inspire further {model-theoretic} case studies, which in turn may have concrete algebraic consequences, and in certain cases might lead to the solution of open problems  (see e.g. \cite{mazarimodeltheoretic, armidaherzog}).

The relation we focus on here is the {\em balanced subgroup} relation restricted {to} torsion-free abelian groups, which occupies a central place in Fuchs's book \cite{fuchs}. Its importance is hard to overstate: balanced subgroups provide the natural notion of exactness for torsion-free groups, the one whose projective objects are precisely the completely decomposable groups, and balanced-projective resolutions are the engine behind much of the structure theory of the subject. 
In the more general context of infinite-rank abelian groups, there are two natural generalizations of the definition of Butler groups, {namely, $B_1$-groups and $B_2$-groups}. Although these two definitions agree on finite-rank abelian groups, whether they {coincide} for infinite-rank groups {depends heavily} on the set-theoretic universe. In particular, the two notions agree if $V=L$, but may not agree if the continuum hypothesis does not hold (see~\cite[Section~6]{abgroupssettheory}). 
It is therefore natural to ask how the relation of being a balanced subgroup behaves from the {model-theoretic} point of view. Fuchs implicitly observes in \cite[Chapter~12, Section~2]{fuchs} that the balanced subgroup relation between torsion-free abelian groups gives rise to an $\aleph_1\mh\mrm{AEC}$. We prove, in addition, that this {$\aleph_1\mh\mrm{AEC}$} is remarkably well-behaved. In the body of the paper we state only stability, which follows swiftly from our ``stability transfer''; in Appendix~\ref{sect:appstabletfab} we go further and show that the balanced subgroup relation in fact admits a stable independence relation.

\begin{introcor}[Corollary~\ref{cor:balstableeaxmple}]\label{theorem4}
The following abstract classes are stable $\aleph_1\mh\mrm{AEC}$s:
\begin{enumerate}[(1), leftmargin=*]
\item $(\tfab, \leqb)$, where $\tfab$ denotes the class of torsion-free abelian groups and $\leqb$ denotes the balanced subgroup relation.
\item $(\redtfab, \leqb)$, where $\redtfab$ denotes the class of reduced torsion-free abelian groups.
\end{enumerate}
\end{introcor}

\subsection{Pure-injective envelopes}

Let us now explain the contents of Section~\ref{sect:envelopes}, which is centered on the class of pure-injective modules. Recall that an $R$-module $A$ is {\em pure-injective} if and only if it is {\em algebraically compact}, i.e., every finitely solvable system of equations with parameters in $A$ has a solution in $A$. Pure-injective modules are among the most important objects in the model theory of modules: every module {admits an essentially unique minimal pure embedding into a pure-injective module, namely its {\em pure-injective envelope}}; the indecomposable pure-injectives are precisely the points of the Ziegler spectrum; and the {model-theoretic} behaviour of a module is to a large extent governed by its pure-injective envelope. Despite this centrality, the class of pure-injective modules does not, in general, form an $\mrm{AEC}$: it is not closed under unions of chains, and so falls outside Shelah's original framework. It is precisely here that the additional flexibility of $\mu\mh\mrm{AEC}$s proves decisive. As we now explain, once one passes to the $\mu\mh\mrm{AEC}$ setting, pure-injective modules can be organized into very well-behaved classes for which strong structural results can be established.

In Section~\ref{sect:envelopes} we provide several examples of $\mu\mh\mrm{AEC}$s of pure-injective modules with the relation of being a direct summand, and we give, as an application of Theorem~\ref{theorem2}, a sufficient condition for such classes to be stable and tame. 

\begin{introthm}[Theorem~\ref{thm:dsumsumup}, Examples~\ref{ex:dsums}-\ref{ex:2dsums}]
\label{theorem6}
Let $\kaecgoth_\oplus = (\kaec, \leqo)$ be a $\mu\mh\mrm{AEC}$ of pure-injective $R$-modules such that $\kaec$ is closed under finite direct sums and direct summands. Then $\kaecgoth_\oplus$ has a stable independence relation, and it is therefore stable and tame. This applies in particular to the following $\mu\mh\mrm{AEC}$s (for appropriate values of $\mu$ depending only on $|R|$), where $\kaec$ is:
\begin{enumerate}[(1), leftmargin=*]
\item the class of all pure-injective {$R$-modules};
\item the class of all injective $R$-modules;
\item the class of all $\Sigma$-pure-injective $R$-modules;
\item the class of all $\Sigma$-injective $R$-modules.
\end{enumerate}
\end{introthm}

Finally, in Subsection~\ref{sect:envclasses} we introduce the technology of {\em enveloping classes}, a flexible mechanism for transferring stability from a subclass to a larger class. This is a ``transfer principle'' of a rather different nature from the ``stability transfer'' discussed above: there the underlying class $\mathbf{K}$ was kept fixed and stability was moved along a chain of strengthenings of the submodel relation, whereas here the relation stays fixed and it is the class $\mathbf{K}$ itself that varies, stability passing from a smaller class to a larger one. Given abstract classes $\kaecgoth_1= (\kaec_1, \leqk)$ and $\kaecgoth_2 = (\kaec_2, \leqk)$ with $\kaec_2\subseteq \kaec_1$ and such that every structure in $\kaec_1$ can be \enquote{enveloped} inside a structure in $\kaec_2$, the stability of $\kaecgoth_2$ implies the stability of $\kaecgoth_1$. Despite its simplicity, this principle is remarkably effective: in combination with Theorem~\ref{theorem6} it allows us to deduce the following general result concerning abstract classes of $R$-modules, which in particular lets us recover and generalize, with a {substantially different proof}, results of Mazari-Armida from \cite{somestablenonelementary}. Our theorem is the following.

\begin{introthm}[Theorem~\ref{thm:pinjenvstable}, Examples~\ref{ex:envpinj}-\ref{ex:cotpinj}]
Let ${\kaecgoth_\pp = (\kaec, \leqp)}$ be an abstract class such that $\kaec$ is closed under pure-injective envelopes and $\kaecgoth_\oplus^\pinj = (\kaec\mh\pinj, \leqo)$ satisfies the hypotheses of Theorem~\ref{theorem6} for some regular $\mu$, where $\kaec\mh\pinj$ is the class of pure-injective modules in $\kaec$. Then $\kaecgoth_\pp$ is stable, and this applies in particular to the following abstract classes of modules:
\begin{enumerate}[(1), leftmargin=*]
\item $(R\mh\Superstable, \leqp)$, where $R\mh\Superstable$ is the class of superstable $R$-modules, in the sense of first-order model theory;
\item $(R\mh\cotorsion, \leqp)$, where $R\mh\cotorsion$ is the class of cotorsion $R$-modules;
\item $(\kaec, \leqp)$ when $\kaec$ contains all pure-injective $R$-modules (such as $\kaec=R\mh\cotorsion$).
\end{enumerate}
\end{introthm}

The same methods yield the following further result.

\begin{introthm}[Corollary~\ref{cor:injenvstable}, Examples~\ref{ex:injenvfirst}-\ref{ex:injenvsecond}]
Let $\kaecgoth_\emb = (\kaec, \leq)$ be an abstract class of $R$-modules such that $\kaec$ contains all injective $R$-modules. Then $\kaecgoth_\emb$ is stable. This applies in particular to the following abstract classes, all of which contain every injective $R$-module and thus satisfy the hypotheses of Corollary~\ref{cor:injenvstable}:
\begin{enumerate}[(1), leftmargin=*, series=mainthm:injenvall]
\item $(R\mh\cotorsion, \leq)$, where $R\mh\cotorsion$ is the class of cotorsion $R$-modules;
\item $(R\mh\rdinj, \leq)$, where $R\mh\rdinj$ denotes the class of $\mrm{RD}$-injective $R$-modules.
\end{enumerate}
Moreover, if the ring $R$ is left Noetherian, then the theorem can also be applied to the following abstract classes:
\begin{enumerate}[resume*=mainthm:injenvall]
\item $(R\mh\sigmapinj, \leq)$, where $R\mh\sigmapinj$ denotes the class of $\Sigma$-pure-injective $R$-modules;
\item $(R\mh\Superstable, \leq)$, where $R\mh\Superstable$ is the class of superstable $R$-modules.
\end{enumerate}
\end{introthm}

\section{Preliminaries}


\subsection{Abstract classes}

In this subsection we fix the notation and we introduce most of the notions regarding abstract classes and $\mu$-abstract elementary classes which will be needed in the paper. The following definition is credited to Grossberg. 
\begin{definition}[\protect{\cite[Definition~2.1]{muaecvasey}}]
\label{def:abstractclass}
We say that $\kaecgoth = (\kaecg, \leqkg) = (\kaec, \leqk)$ is an \tdef{abstract class} ($\mrm{AC}$ for short) provided the following conditions are met:
\begin{enumerate}[(1), leftmargin=*]
\item $\kaec$ is a non-empty class of $\tau$-structures for a fixed language $\tau_\kaecgoth = \tau$;
\item $\leqk$ is a reflexive and transitive relation on $\kaec$, and if $M \leqk N$, then $M \leq N$ (where $M \leq N$ denotes the submodel relation);
\item if $f: N\to N'$ is an isomorphism and $N\in \kaec$, then $N'\in \kaec$, and if $M\in\kaec$ and $M\leqk N$, then $f(M)\in\kaec$ and $f(M) \leqk N'$.
\end{enumerate}
\end{definition}

Abstract classes were studied by the second author in \cite{pqhcp1, pqhcp2}, and more recently by Mazari-Armida and Trlifaj in \cite{mazaridecon} in a module-theoretic context.

\begin{notation}
We will often write $\kaecgoth =(\kaec, \leqk)$, but when $\kaec$ or $\leqk$ is not clear from the context we will also use the notation $\kaecgoth = (\kaecg, \leqkg)$.
\end{notation}

\begin{definition}
Let $\kaecgoth = (\kaec, \leqk)$ be an $\mrm{AC}$. We say that an embedding $f : M \to N$ is a \tdef{$\kaecgoth$-embedding} if $M,N\in \kaec$ and $f(M) \leqk N$.
\end{definition}

\begin{definition}
\label{def:mudirected}
Let $\kaecgoth = (\kaec, \leqk)$ be an $\mrm{AC}$, $\mu$ a regular cardinal, and $\alpha$ an ordinal.
\begin{enumerate}[(1), leftmargin=*]
\item \label{def:mudirectposet} We say that a poset $P = (P, \leq)$ is \tdef{$\mu$-directed} provided that every subset of $P$ of cardinality smaller than $\mu$ has an upper bound in $P$.
\item \label{def:mudirectsyst} We say that a family $(M_i\mid i\in I)$ of $\tau_\kaecgoth$-structures is a \tdef{$\mu$-directed system in $\kaecgoth$} if $I$ is a $\mu$-directed poset, $M_i\in \kaec$ for every $i\in I$, and $i< j$ implies $M_i \leqk M_j$. When $I$ is a chain, we shall say that $(M_i\mid i\in I)$ is a \tdef{chain in $\kaecgoth$}.
\item \label{def:smoothchain} We say that a chain $(M_i\mid i< \ordnl)$ in $\kaecgoth$ is \tdef{smooth} (or \tdef{continuous}) if for every limit ordinal $i < \ordnl$ one has $M_i= \bigcup_{j < i} M_j$.
\end{enumerate}
We shall omit the reference to $\mu$ when $\mu= \aleph_0$. Moreover, when the $\mrm{AC}$ $\kaecgoth$ taken into consideration is clear from the context, we shall refer to $(M_i\mid i\in I)$ simply as a $\mu$-directed system (resp. chain). 
\end{definition}

In Definition~\ref{def:mudirected} we required $\mu$ to be a regular  cardinal. Of course, one can give the same definitions with $\mu$ a singular cardinal, but it is easily seen that in such a case a poset is $\mu$-directed if and only if it is $\mu^+$-directed. Therefore, there is no loss of generality in restricting to $\mu$ a regular cardinal.

\begin{definition}
\label{def:propertiesac}
Let $\kaecgoth =  (\kaec, \leqk)$ be an $\mrm{AC}$, and $\kappa,\mu$ be infinite cardinals with $\kappa \geq |\tau_\kaecgoth|$ and $\mu$ regular. 
\begin{enumerate}[(1), leftmargin=*]
\item \label{def:coherence}We say that $\kaecgoth$ satisfies \tdef{coherence} if $M_0, M_1, M_2\in \kaec$ with $M_0\leqk M_2$, $M_0\leq M_1\leqk M_2$, implies $M_0\leqk M_1$.
\item \label{def:LS} We say that $\kaecgoth$ \tdef{has $\mrm{LS}$ at $\kappa$} if for every $M\in \kaec$ and $A\subseteq M$ with $|A| \leq \kappa$ there is  $M_0\in\kaec$ with $A\subseteq M_0\leqk M$ and $|M_0| \leq \kappa$. 
\item \label{def:muclosedchains} We say that $\kaecgoth$ \tdef{is closed under chains} if for every chain indexed by an ordinal $(M_i\mid i<\ordnl)$ in $\kaecgoth$ one has $\bigcup_{i<\ordnl} M_i \in \kaec$. We say that $\kaecgoth$ \tdef{is closed under smooth chains} if the previous holds for every smooth chain.
\item \label{def:muclosed} We say that $\kaecgoth$ \tdef{is closed under $\mu$-directed systems}, or \tdef{$\mu$-closed}, if for every $\mu$-directed system $(M_i\mid i\in I)$ in $\kaecgoth$ one has $\bigcup_{i\in I} M_i \in \kaec$. We omit the reference to $\mu$ when $\mu= \aleph_0$.
\item \label{def:mucontinuity} We say that $\kaecgoth$ \tdef{has $\mu$-continuity} if for every $\mu$-directed system $(M_i\mid i\in I)$ in $\kaecgoth$ one has $M=\bigcup_{i\in I}M_i\in \kaec$ and $M_i \leqk M$ for every $i\in I$. We say that $\kaecgoth$ has \tdef{continuity} if it has $\aleph_0$-continuity.
\item \label{def:musmoothness} We say that $\kaecgoth$ \tdef{has $\mu$-smoothness} if for every $N\in \kaec$ and every $\mu$-directed system $(M_i\mid i\in I)$ in $\kaecgoth$ with $M_i\leqk N$ and $M=\bigcup_{i\in I}M_i\in \kaec$, we have $M \leqk N$. We say that $\kaecgoth$ has \tdef{smoothness} if it has $\aleph_0$-smoothness.
\end{enumerate}
\end{definition}

\begin{remark}
\label{rem:cofmucont}
If $\alpha$ is an ordinal with $\cf(\alpha) \geq \mu$, then every cofinal subset of $\alpha$ must have cardinality $\geq \mu$. Therefore, $\alpha$ is a $\mu$-directed poset. In particular, every $\kaecgoth$-chain of the form $(M_i\mid i<\alpha)$ is $\mu$-directed.
\end{remark}

\begin{fact}[\protect{\cite[Corollary 1.7 and the following Remark]{adamekbook}}]
\label{fact:continuitysmooth}
Let $\kaecgoth = (\kaec, \leqk)$ be an $\mrm{AC}$. Then the following are equivalent:
\begin{enumerate}[(1), leftmargin=*]
\item \label{fact:continuitysmooth1}$\kaecgoth$ has continuity (cf.~\ref{def:propertiesac}\ref{def:mucontinuity}).
\item \label{fact:continuitysmooth2}If $(M_i)_{i<\ordnl}$ is a smooth $\kaecgoth$-chain, then $M =\bigcup_{i<\ordnl}M_i\in\kaec$ and $M_i \leqk M$ for every $i<\ordnl$.
\end{enumerate}
\end{fact}

\begin{remark}
\label{rem:continuityetc}
\begin{enumerate}[(1), leftmargin=*]
\item \label{rem:continuityetc1} Just like in Definition~\ref{def:propertiesac}\ref{def:mucontinuity}, one could have defined continuity with respect to chains indexed by ordinals of cofinality $\geq \mu$. This condition is equivalent to $\mu$-continuity if $\mu = \aleph_0$ by Fact~\ref{fact:continuitysmooth}. But in general, when $\mu > \aleph_0$, the conditions need not be equivalent (see \cite[Exercise 1.c]{adamekbook}).
\item \label{rem:continuityetc2} In the following we shall  work with abstract classes which are not necessarily closed under arbitrary directed systems (cf.~\ref{def:propertiesac}\ref{def:muclosed}). Therefore, when we consider a chain $(M_i\mid i< \ordnl)$ for $\ordnl$ an ordinal, it could very well be the case that such a chain is not smooth (cf.~\ref{def:mudirected}\ref{def:smoothchain}). 
\end{enumerate}
\end{remark}

The following definition was originally introduced in \cite[Definition 2.2]{muaecvasey}.

\begin{definition}
\label{def:muaec}
Let $\kaecgoth =(\kaec, \leqk)$ be an $\mrm{AC}$ and $\mu$ a regular cardinal. We say that $\kaecgoth$ is a \tdef{$\mu$-Abstract Elementary Class} ($\mu\mh\mrm{AEC}$ for short) when we have:
\begin{enumerate}[(1), leftmargin=*] 
\item $\kaecgoth$ satisfies coherence (cf.~\ref{def:propertiesac}\ref{def:coherence});
\item \label{def:muaec_tvaxioms}(\tdef{Tarski-Vaught axioms}) $\kaecgoth$ has $\mu$-continuity and $\mu$-smoothness (cf.~\ref{def:propertiesac}\ref{def:mucontinuity}-\ref{def:musmoothness});
\item (\tdef{L\"owenheim-Skolem axiom}) \label{def:muaec_lstaxiom} There exists a cardinal $\lambda = \lambda^{<\mu}\geq |\tau_\kaecgoth|+\mu$ such that for any $M\in \kaec$ and $A\subseteq M$ there is $M_0\in \kaec$ with $A\subseteq M_0\leqk M$ and $|M_0| \leq |A|^{<\mu} + \lambda$. The least such $\lambda$ is called the \tdef{L\"owenheim-Skolem number} of $\kaecgoth$ (\enquote{$\mrm{LS}$ number} for short), and is denoted by $\mrm{LS}_\mu(\kaecgoth)$.
\end{enumerate}
The L\"owenheim-Skolem axiom will also be called the $\mrm{LS}$ axiom. 
\end{definition}

\begin{remark}
\label{rem:defmuaec}
\begin{enumerate}[(1), leftmargin=*]
\item When $\mu= \aleph_0$ in Definition~\ref{def:muaec} we simply get the definition of \tdef{abstract elementary class} ($\mrm{AEC}$ for short). Standard references for abstract elementary classes are the classical \cite{shelahbook1,shelahbook2}, and the book \cite{categoricity}.
\item Usually, abstract elementary classes are defined as in Definition~\ref{def:muaec} but requiring continuity and smoothness to only hold for smooth chains indexed by limit ordinals (see also Remark~\ref{rem:continuityetc}). This definition is known to be equivalent to the definition we have given \cite[Example (1) following Remark~2.3]{muaecvasey}.
\item Notice the dependence on $\mu$ in $\mrm{LS}_\mu(\kaecgoth)$ from~\ref{def:muaec}\ref{def:muaec_lstaxiom}. When $\mu$ is clear from the context we shall omit it and simply write $\mrm{LS}(\kaecgoth)$.
\item \label{rem:defmuaec3}If $\kaecgoth$ is a $\mu\mh\mrm{AEC}$ and $\lambda = \lambda^{<\mu} + \mrm{LS}_\mu(\kaecgoth)$, then $\kaecgoth$ satisfies $\mrm{LS}$ at $\lambda$ (cf.~\ref{def:propertiesac}\ref{def:LS}).
\end{enumerate}
\end{remark}
\begin{definition}
\label{def:amprop}
Let $\kaecgoth =(\kaec, \leqk)$ be an $\mrm{AC}$. 
\begin{enumerate}[(1), leftmargin=*]
\item \label{def:ampropap} We say that $\kaecgoth$ satisfies the \tdef{amalgamation property} ($\mrm{AP}$ for short) if for any pair of $\kaecgoth$-embeddings $(f_1: M_0 \to M_1, f_2: M_0\to M_2)$ there is a pair of $\kaecgoth$-embeddings $(g_1: M_1\to N,g_2: M_2\to N)$ with $g_1f_1 = g_2f_2$.
\item \label{def:ampropjep} We say that $\kaecgoth$ satisfies the \tdef{joint embedding property} ($\mrm{JEP}$ for short) if for any $M_\ell\in \kaec$, with $\ell\in\{1,2\}$, there are $N\in \kaec$ and $\kaecgoth$-embeddings $f_\ell :M_\ell \to N$.
\end{enumerate}
\end{definition}

\begin{remark}
\label{rem:amalgcutpaste}
A standard cut-and-paste argument actually shows the following:
\begin{enumerate}[(1), leftmargin=*]
\item \label{rem:amalgcutpaste1} If $\kaecgoth$ satisfies the amalgamation property, then for any $M_0,M_1,M_2\in\kaec$ and any pair of $\kaecgoth$-embeddings $(f_1: M_0 \to M_1, f_2: M_0\to M_2)$ there is $N\in\kaec$ with $M_1\leqk N$ and a $\kaecgoth$-embedding $g:M_2\to N$ such that $f_1 = gf_2$. 
\item \label{rem:amalgcutpaste2} If $\kaecgoth$ satisfies the joint embedding property, then for any $M_\ell\in \kaec$, with $\ell \in \{1, 2\}$, there are $M_1\leqk N\in \kaec$ and a $\kaecgoth$-embedding $f: M_2 \to N$ .
\end{enumerate}
\end{remark}

\begin{notation}
$\bar a,\, \bar b,\, \bar c$ always denote sequences of arbitrary length, possibly infinite. We write $\bar a\in M^{<\infty}$ when we mean that $\bar a$ is a sequence in $M$.
\end{notation}

We will now introduce orbital types, originally introduced by Shelah in \cite{universalclasses}. They are, in the context of abstract classes, the semantic analogue of first-order types. The following definitions are standard, see also \cite{vaseyinfst} for another reference.

\begin{definition}
\label{def:eqrelorbtypes}
Let $\kaecgoth = (\kaec, \leqk)$ be an $\mrm{AC}$. 
\begin{enumerate}[(1), leftmargin=*]
\item \label{def:eqrelorbtypes1} For triples  $(\bar b_\ell, A_\ell, N_\ell)$, $A_\ell\subseteq N_\ell\in\kaec$, $\bar b_\ell$ a sequence in $N_\ell^{<\infty}$, with $\ell\in \{1, 2\}$, we define a binary relation $(\bar b_1, A_1, N_1) E_\kaecgoth^\at(\bar b_2, A_2, N_2)$ if $A:= A_1 = A_2$ and there are $\kaecgoth$-embeddings $f_\ell: N_\ell \to N$, with $\ell\in \{1,2\}$, such that $f_1(\bar b_1) = f_2(\bar b_2)$ and $f_1\restriction A = f_2\restriction A = \id_A$.
\item \label{def:eqrelorbtypes2} With $E_\kaecgoth$ we denote the transitive closure of $E_\kaecgoth^\at$.
\end{enumerate}
\end{definition}

\begin{remark}
\label{rem:equivrel}
\begin{enumerate}[(1), leftmargin=*]
\item \label{rem:equivrel1} It is easily seen that $E_\kaecgoth^\at$ is reflexive and symmetric,  thus $E_\kaecgoth$ is an equivalence relation.
\item \label{rem:equivrel2} If $B\subseteq A\subseteq N_\ell$, with $\ell\in \{1, 2\}$, then:
\[
(\bar b_1, A, N_1)E_\kaecgoth^\at (\bar b_2, A, N_2)\;\Rightarrow\;(\bar b_1, B, N_1) E_\kaecgoth^\at (\bar b_2, B, N_2).
\]
In particular, $(\bar b_1, A, N_1)E_\kaecgoth (\bar b_2, A, N_2)$ implies $(\bar b_1, B, N_1) E_\kaecgoth(\bar b_2, B, N_2)$. 
\item \label{rem:equivrel3} \cite[Fact 2.17]{vaseyinfst} If $\kaecgoth$ satisfies the amalgamation property (cf.~\ref{def:amprop}\ref{def:ampropap}), then $E_\kaecgoth^\at$ coincides with $E_\kaecgoth$.
\end{enumerate}
\end{remark}

The following definition  is standard, item \ref{def:nottypes2} is from \cite[Definition 7.8]{vaseynotes}.

\begin{definition}\label{def:nottypes} 
Let $\kaecgoth =(\kaec, \leqk)$ be an $\mrm{AC}$.
\begin{enumerate}[(1), leftmargin=*]
\item \label{def:nottypes1} Given $(\bar{b}, A, N)$, where $N \in \kaec$, $A \subseteq N$, and $\bar{b}$ is a sequence in $N$, the \tdef{orbital type} (also called the \tdef{Galois type}) of $\bar{b}$ over $A$ in $N$ within $\kaecgoth$, denoted by $\gtp_{\kaecgoth}(\bar{b}/A; N)$, is the equivalence class of $(\bar{b}, A, N)$ modulo $E_{\kaecgoth}$ (cf.~\ref{def:eqrelorbtypes}).
\item \label{def:nottypes2} If $M,N\in\kaec$ with $M\leqk N$, $\pi:M\to M'$ is an isomorphism and $p=\gtp_\kaecgoth(\bar b/M; N)$ an orbital type, we define $\pi(p)$ as $\gtp_\kaecgoth(\pi'(\bar b)/M'; N')$ where $\pi':N\to N'$ is some (any) isomorphism extending $\pi$. $\pi(p)$ is called the \tdef{image of $p$ via $\pi$}.
\item \label{def:nottypes3} Given $p= \gtp_\kaecgoth(\bar b/A; N)$ an orbital type and $B\subseteq A$, we let $p\restriction B = \gtp_\kaecgoth(\bar b/B; N)$. This is well-defined because of~\ref{rem:equivrel}\ref{rem:equivrel2}, and it is called the \tdef{restriction of $p$ to $B$}. In such a case, we say that $p$ \tdef{extends} $p\restriction B$.
\item \label{def:nottypes4} Given $p =\gtp_\kaecgoth(\bar b/A; N)$ an orbital type and $A\subseteq M\in \kaec$, we say that $p$ \tdef{is realized in} $M$ if there is $\bar c\in M^{<\infty}$ such that $\gtp_\kaecgoth(\bar b/A; N) = \gtp_\kaecgoth(\bar c/A; M)$.
\end{enumerate}
In the following, we may simply write $\gtp(\bar b/A; N)$ when the reference to $\kaecgoth$ is clear.
\end{definition}

\begin{definition}
\label{def:typespace}
Let $\kaecgoth =(\kaec, \leqk)$ be an $\mrm{AC}$ and $\gamma\in\mrm{Ord}$.
\begin{enumerate}[(1), leftmargin=*]
\item \label{def:typespace1} If $A\subseteq N\in \kaec$, let $\gS_\kaecgoth^\gamma(A; N) = \{\gtp_\kaecgoth(\bar c/A; N) \mid \bar c\in N^\gamma\}$.
\item \label{def:typespace2} If $M \in {\kaec}$, let $\gS^\gamma_{\kaecgoth}(M) = \{\gtp_\kaecgoth(\bar{b}/M; N) \mid M \leqk N \in \kaec \text{ and } \bar{b} \in N^\gamma\}$.
\end{enumerate}
In a similar manner we can define $\gS_\kaecgoth^{<\gamma}(A;N)$ and $\gS_\kaecgoth^{<\infty}(A;N) = \bigcup_{\gamma\in\mrm{Ord}}\gS^\gamma_\kaecgoth(A;N)$ (resp.\ $\gS_\kaecgoth^{<\gamma}(M)$ and $\gS_\kaecgoth^{<\infty}(M)$). When $\gamma = 1$, we may write $\gS_\kaecgoth(A; N)$ instead of $\gS_\kaecgoth^1(A;N)$ (resp.\ $\gS_{\kaecgoth}(M)$ instead of $\gS^1_{\kaecgoth}(M)$).
\end{definition}
The following fact ensures that in $\mu\mh\mrm{AEC}$s the set of orbital types of a fixed length over a model is a set and not a proper class, and we can actually give an explicit bound on its size.

\begin{fact}[\protect{\cite[Remark 7.7]{vaseynotes}}]
\label{fact:cardorbtypes}
Let $\kaecgoth=(\kaec, \leqk)$ be a $\mu\mh\mrm{AEC}$, $\gamma$ an ordinal, and $M\in\kaec$. Then $|\gS_\kaecgoth^\gamma(M)|\leq 2^{(|M|+\gamma +\mrm{LS}_\mu(\kaecgoth))^{<\mu}}$.
\end{fact}

In the following definition we introduce a crucial locality property of types, known as {\em tameness}.

\begin{notation} 
\label{not:kaeclambda} Let $\kaecgoth = (\kaec, \leqk)$ be an $\mrm{AC}$. For $\lambda$ an infinite cardinal, we let  $\kaec_\lambda = \{M \in \kaec \mid |M| = \lambda\}$. Analogously, one defines $\kaec_{\leq\lambda}$ and $\kaec_{<\lambda}$.
\end{notation}

\begin{definition}\label{def:tame}
Let $\kaecgoth = (\kaec, \leqk)$ be an $\mrm{AC}$, $\kappa$ an infinite cardinal, $\gamma\in\mrm{Ord}$.
\begin{enumerate}[(1), leftmargin=*]
\item \label{def:tame1}We say that $\kaecgoth$ is \tdef{$(\kappa, \gamma)$-tame} if for any $M \in \kaec$ and $p,q \in \gS_{\kaecgoth}^\gamma(M)$, if $p\neq q$ then there is $M_0\in\kaec_{\leq\kappa}$ with $M_0 \leqk M$ and $p \restriction M_0 \neq q \restriction M_0$ (cf.~\ref{def:nottypes}\ref{def:nottypes3}).
\item \label{def:tame3} When we say that $\kaecgoth$ is \tdef{$\kappa$-tame} we mean that it is $(\kappa,1)$-tame.
\item \label{def:tame4} When we say that $\kaecgoth$ is \tdef{fully $\kappa$-tame} we mean that it is $(\kappa,\gamma)$-tame for every $\gamma\in\mrm{Ord}$.
\item When we say that $\kaecgoth$ is \tdef{tame} we mean it is $\kappa$-tame for some $\kappa$. Similarly we define being \tdef{fully tame}.
\end{enumerate}
\end{definition}

Tameness was first isolated  and studied in the context of $\mrm{AEC}$s satisfying the amalgamation property in \cite{grvanupwardstability} to prove an upward stability transfer, and in \cite{grvanupwardcategoricity2, grvanupwardcategoricity1} to prove upward categoricity transfers. In fact, $\kappa$-tameness had already been considered in \cite[Claim 2.3]{shelahlocality1} as \enquote{having character $\leq\kappa$}. In \cite{shelahlocality2} tameness, there called \enquote{locality}, was studied for $\mrm{AEC}$s without $\mrm{AP}$.  
In the context of $\mrm{AC}$s $\kappa$-tameness had first been defined in \cite[Definition 2.2]{vaseyinfst}. Their definition slightly differs from ours because they instead require the following: For every $M\in\kaec$ and $p,q\in \gS_\kaecgoth^\gamma(M)$, if $p\neq q$ then there is an $A\subseteq M$ with $|A| \leq \kappa$ such that $p\restriction A \neq q\restriction A$. 
It is easily seen that this condition is equivalent to our definition when $\kaecgoth$ is a $\mu\mh\mrm{AEC}$ and $\kappa=\kappa^{<\mu}$.

The following result was first proved for $\mrm{AEC}$s in \cite{boneytamefromlarge}, and then generalized for $\mu\mh\mrm{AEC}$s in \cite{muaecvasey}. We state it here in a weaker form which suffices for our purposes.

\begin{fact}[\protect{\cite[Theorem 5.5]{muaecvasey}}]
\label{fact:strongtameness}
Let $\kaecgoth = (\kaec, \leqk)$ be a $\mu\mh\mrm{AEC}$ and $\kappa > \mrm{LS}_\mu(\kaecgoth)$ a strongly compact cardinal. Then $\kaecgoth$ is fully $\kappa$-tame (cf.~\ref{def:tame}\ref{def:tame4}).
\end{fact}

\subsection{Module theory}
\label{sect:moduleprelims}

In this subsection we fix the module-theoretic notation which will be used throughout the article, and we state the facts which will be needed in the subsequent parts of the paper.
\begin{notation}
\begin{enumerate}[(1), leftmargin=*]
\item By $R$ we denote an associative ring with identity.
\item When we say \enquote{$R$-module} we always mean  \enquote{unital left $R$-module}. 
\item We may omit the reference to the ring $R$ if it is clear from the context.
\item When we say \enquote{group} we always mean \enquote{abelian group}.
\end{enumerate}
\end{notation}

\begin{notation}
When we say that $\kaecgoth = (\kaec, \leqk)$ is an abstract class (cf.~\ref{def:abstractclass}) of $R$-modules, we mean that $\tau_\kaecgoth$ is the language of $R$-modules\footnote{The language of $R$-modules is $\tau_R$, which consists of a unary function symbol for every $r\in R$, a constant symbol $0$, and two binary function symbols $+$ and $-$.} and $\kaec$ is a class of $R$-modules.
\end{notation}

Recall that an $R$-module $A$ is a \tdef{direct summand} of $B$ (written $A\leqo B$) if $A\leq B$ and there is a module $C$ such that $B = A \oplus C$. We say that a homomorphism $f: A \to B$ is \tdef{split} if it is an embedding and $f(A)\leqo B$.

Let $B$ be an $R$-module. We say that a submodule $A\leq B$ is a \tdef{retract} of $B$ if there is an $R$-module homomorphism $g: B \to A$ such that $g(a) = a$ for every $a\in A$. Such a $g$ is called a \tdef{retraction} from $B$ to $A$ \cite[pg.~50]{rotman}.

\begin{fact}
\label{fact:dsumsolveeq}
Let $B$ be an $R$-module and $A\leq B$. Then the following are equivalent:
\begin{enumerate}[(1), leftmargin=*]
\item \label{fact:dsumsolveeq1}$A$ is a direct summand of $B$;
\item \label{fact:dsumsolveeq2}$A$ is a retract of $B$;
\item \label{fact:dsumsolveeq2.5} there is $q: B/A \to B$ such that $\pi q = \id_{B/A}$, where $\pi: B \to B/A$ is the natural projection;
\item \label{fact:dsumsolveeq3}for every system of $R$-linear equations with parameters in $A$, the system has a solution in $A$ provided it has a solution in $B$.
\end{enumerate}
In particular, if $A\leq B \leq C$ and $A\leqo C$, then $A\leqo B$.
\end{fact}
\begin{proof}
That \ref{fact:dsumsolveeq1},  \ref{fact:dsumsolveeq2}, and \ref{fact:dsumsolveeq2.5} are equivalent follows from \cite[Corollary~2.23, Exercise~2.8(i)]{rotman}. The equivalence of \ref{fact:dsumsolveeq1} and \ref{fact:dsumsolveeq3} follows from the equivalence of $(i)$ and $(ii)$ in \cite[Proposition 7.16]{jensen}. Finally, the last sentence of the statement follows from~\ref{fact:dsumsolveeq2}.
\end{proof}

We quickly recall some basic facts about injective modules. All of these facts are standard, and good  references on the subject are \cite{rotman, lecturesonmodules, ringscatmodules}.

\begin{definition}
\begin{enumerate}[(1), leftmargin=*]
\item \label{def:injenv1} \cite[Section 3.2]{rotman} An $R$-module $A$ is \tdef{injective} if for every $R$-module $B$ we have that $A\leq B$ implies $A \leqo B$.
\item  \label{def:injenv2} \cite[pg.~127]{rotman} Let $A$ be an $R$-module, we say that $A\leq B$ is an \tdef{injective envelope} of $A$ if $B$ is injective and $A\leq C\leqo B$ implies $B=C$.
\end{enumerate} 
\end{definition}

\begin{fact}
\label{fact:triviainjenv}
\begin{enumerate}[(1), leftmargin=*]
\item \label{fact:triviainjenv1} \cite[Proposition~3.28]{rotman} The class of injective modules is closed under finite direct sums and direct summands.
\item \label{fact:triviainjenvbaer} \cite[Theorem~3.30]{rotman} An $R$-module $A$ is injective if and only if every $R$-module homomorphism $f: I \to A$, where $I$ is a left ideal of $R$, can be extended to a homomorphism $f':R\to A$.
\item \label{fact:triviainjenv2} \cite[Theorem~3.45]{rotman} Every $R$-module $A$ has an injective envelope $\mrm{E}(A)$, which is unique up to isomorphism over $A$. That is, if $A\leq B$ and $A\leq C$ are injective envelopes of $A$, there is an isomorphism $f: B \to C$ with $f\restriction A = \id_A$.
\item \label{fact:triviainjenv3}\cite[Lemma~18.9]{ringscatmodules}  Let $A$ and $B$ be $R$-modules and $f: A \to B$ an embedding. Then there is a split embedding $g: \mrm{E}(A) \to \mrm{E}(B)$ with $g\restriction A = f$.
\end{enumerate}
\end{fact}
 
We now introduce the pure-injective modules, which will be studied extensively in Section~\ref{sect:envelopes}. Pure-injective modules and envelopes are the direct analogue of injective modules and envelopes within the context of pure embeddings.

\begin{definition}
\label{def:purity}
\begin{enumerate}[(1), leftmargin=*]
\item \label{def:purity1}\cite[Section~2.1.1]{purityspectra}  Let $B$ be an $R$-module. We say that a submodule $A\leq B$ is \tdef{pure in $B$} (written $A\leqp B$) if every finite system of $R$-linear equations with parameters in $A$ has a solution in $A$ provided it has a solution in $B$.
\item \label{def:purity2}\cite[Section~4.3.1]{purityspectra} An $R$-module $A$ is \tdef{pure-injective} if for every $R$-module $B$ we have that $A\leqp B$ implies $A \leqo B$.
\item \label{def:purity3}\cite[pg.~145]{purityspectra} Let $A$ be an $R$-module. We say that $A\leqp B$ is a \tdef{pure-injective envelope} of $A$ if $B$ is pure-injective and $A\leqp C\leqo B$ implies $B=C$.
\end{enumerate}
\end{definition}

\begin{fact}
\label{fact:triviapinjective}
\begin{enumerate}[(1), leftmargin=*]
\item \label{fact:triviapinjectivesums}\cite[Lemma~4.3.2]{purityspectra} The class of pure-injective modules is closed under finite direct sums and direct summands.
\item \label{fact:triviapinjectivealgcomp}\cite[Theorem~4.3.11]{purityspectra} An $R$-module $A$ is pure-injective if and only if it is algebraically compact, i.e., every finitely solvable system of equations with constants in $A$ is solvable in $A$.
\end{enumerate}
\end{fact}

\begin{fact}
\label{fact:existpenv}
\begin{enumerate}[(1), leftmargin=*, series=fact:existpenv]
\item \cite[Corollary 3.11]{ziegler} Every $R$-module $A$ has a pure-injective envelope $A\leqp B$ with ${|B| \leq |A|^{|R| + \aleph_0}}$.
\item \cite[Theorem 4.3.18]{purityspectra} The pure-injective envelope is unique up to isomorphism, i.e., if $A\leqp B$ and $A\leqp C$ are pure-injective envelopes of $A$, then there is an isomorphism $f: B \to C$ with $f\restriction A = \id_A$.
\end{enumerate}
\nin Therefore, we denote by $\mrm{PE}(A)$ the unique (up to isomorphism) pure-injective envelope of $A$. We also have:
\begin{enumerate}[resume*=fact:existpenv]
\item\cite[Theorem 4.3.17]{purityspectra} Let $A$ and $B$ be $R$-modules and $f: A \to B$ a pure embedding. Then there is a split embedding $g: \mrm{PE}(A) \to \mrm{PE}(B)$ with $g\restriction A = f$.
\end{enumerate}
\end{fact} 

Finally, we will need in Section~\ref{sect:lpure} the following definition:
\begin{definition}[\protect{\cite[Section 7.5]{jensen}}]
\label{def:lambdapresented}
Let $\lambda$ be an infinite cardinal. We say that an $R$-module $A$ is \tdef{$\lambda$-presented} if there is a short exact sequence
\begin{equation}
0\to K \overset{}{\to} F \overset{}{\to}A \to 0.
\end{equation}
where $K$ and $F$ are $({<}\lambda)$-generated and $F$ is free.
\end{definition}

We will often use the following:
\begin{remark}
\label{rem:lambdapresented}
Let $\lambda$ be an infinite cardinal, we have the following:
\begin{enumerate}[(1), leftmargin=*]
\item \label{rem:lambdapresented1}If $|R| < \lambda$, then the following are equivalent for an $R$-module $A$:
\begin{enumerate}[(a),leftmargin=*]
\item $A$ is $\lambda$-presented;
\item $A$ is $({<}\lambda)$-generated;
\item $A$ has cardinality $<\lambda$.
\end{enumerate}
\item \label{rem:lambdapresented2}If $R$ is a countable left Noetherian ring, e.g. $R=\mathbb Z$, then an $R$-module is $\lambda$-presented if and only if it is $({<}\lambda)$-generated, this follows by item \ref{rem:lambdapresented1}  and~\cite[Corollary~3.19]{rotman}. 
\end{enumerate}
\end{remark}

Now let us recall some definitions concerning torsion-free abelian groups. These will be used again in Subsection~\ref{sect:lpurebal}. All of the definitions that follow are standard, and can be found in \cite[Chapter~12, Section~1]{fuchs}.

\begin{definition}
\label{def:charac}
Let $A\in\tfab$ and $a\in A$. We define the \tdef{characteristic of $a$ in $A$} as $\chi_A(a) = (h_{p_0}^A(a), h_{p_1}^A(a), \ldots)$, where $(p_i)_{i<\omega}$ is the increasing enumeration of the prime numbers, and $h_{p_i}^A(a)$ is the supremum of the $k<\omega$ such that $p_i^k$ divides $a$ (so that $h_{p_i}^A(a)\in \omega\cup\{\infty\}$). 
\end{definition}

When the ambient group $A$ in which the characteristic is computed is clear, we may omit the reference to it and simply write $\chi(a) = (h_{p_0}(a), h_{p_1}(a), \ldots)$. Moreover, we have the following:

\begin{fact}[\protect{\cite[Chapter~12, Section~1(f)]{fuchs}}]
\label{fact:pureiffchar}
Let $A\leq B\in\tfab$. Then $A\leqp B$ if and only if $\chi_A(a) = \chi_B(a)$ for every $a\in A$.
\end{fact}

\begin{definition}
\label{def:abtypes}
Let $A$ be a torsion-free abelian group.
\begin{enumerate}[(1), leftmargin=*]
\item \label{def:abtypes1} We say that two characteristics $(k_1, k_2, \ldots)$ and $(\ell_1, \ell_2, \ldots)$ are  \tdef{equivalent} if ${k_n = \ell_n}$ for cofinitely many $n < \omega$, and both $k_n$ and $\ell_n$ are finite whenever $k_n\neq \ell_n$. The equivalence class of a characteristic is called a \tdef{type}, and is usually denoted $\tchar$. We say that $a\in A$ \tdef{has type} $\tchar$ if $\chi(a)\in\mathbf t$, and in such a case write $\tchar(a) = \tchar_A(a) =\tchar$.
\item \label{def:abtypes3} We let $\mrm{Type}(A) = \{\tchar_A(a) \mid a\in A\setminus \{0\}\}$ and we refer to it as the \tdef{typeset of $A$}. 
\item \label{def:abtypes4} If $\tchar$ is a type, we say that $A$ is \tdef{$\tchar$-homogeneous} if $\mrm{Type}(A) = \{\tchar\}$. 
\end{enumerate}
\end{definition}

\section{Almost stability for \texorpdfstring{$\mu\mh\mrm{AEC}$}{μ-AEC}s and \texorpdfstring{$\infty\mh\mrm{AP}$}{∞-AP}}
\label{sect:almoststability}

In this section we define almost stability in the very  general context of abstract classes. Then, we discuss higher forms of amalgamations under which the equivalence between stability and almost stability holds (Lemma~\ref{lemma:equivasamalg}). In particular, this holds for abstract classes with chain bounds and the amalgamation property (Corollary~\ref{cor:coralmostmuap}). Finally, we end the section with a general lemma regarding the transfer of almost stability between different abstract classes (Lemma~\ref{lem:alstabtransfer}).

A central and heavily studied notion in the model theory of $\mu\mh\mrm{AEC}$s, or abstract classes in general, is stability, whose definition we now recall:

\begin{definition}\label{def:stabilityproper}
Let $\kaecgoth = (\kaec, \leqk)$ be an $\mrm{AC}$, $\lambda$ an infinite cardinal, and $\gamma \in \mrm{Ord}$. 
\begin{enumerate}[(1),leftmargin=*]
\item \label{def:stabilityproper1} We say that $\kaecgoth$ is \tdef{$(\lambda, \gamma)$-stable} if for any $M \in \kaec_{\lambda}$ (cf.~\ref{not:kaeclambda}) we have $|\gS^\gamma_{\kaecgoth}(M)| \leq \lambda$ (cf.~\ref{def:typespace}\ref{def:typespace2}).  If $\gamma = 1$, then we simply say that $\kaecgoth$ is \tdef{$\lambda$-stable}.
\item \label{def:stabilityproper2} $\kaecgoth$ is \tdef{stable} if it is $\lambda$-stable for unboundedly many infinite cardinals $\lambda$.
\end{enumerate} 
\end{definition}

Very recently, the second author and Shelah introduced, in the context of abstract elementary classes, a weaker notion of stability called almost stability  \cite[Definition~2.14]{paolinishelah}. They arrived at this definition by finding the first example of an unstable abstract elementary class of modules \cite[Theorem~1.2]{paolinishelah}. In this example the failure of stability is due to a lack of the amalgamation property, as there can be many models with a common substructure which cannot be amalgamated, even if all of them have few types over this common substructure. Indeed, in $\mrm{AEC}$s with the amalgamation property almost stability and stability coincide.


The following definition is the straightforward generalization of almost stability to the context of abstract classes:

\begin{definition}\label{def:stability}
Let $\kaecgoth = (\kaec, \leqk)$ be an $\mrm{AC}$, $\lambda$ an infinite cardinal, and $\gamma$ an ordinal. 
\begin{enumerate}[(1),leftmargin=*]
\item \label{def:stable2} We say that $\kaecgoth$ is \tdef{almost $(\lambda, \gamma)$-stable} if for any $M, N \in \kaec$ with $M \leqk N$ and $M \in \kaec_{\lambda}$ (cf.~\ref{not:kaeclambda}) we have $|\gS^\gamma_{\kaecgoth}(M; N)| \leq \lambda$ (cf.~\ref{def:typespace}\ref{def:typespace1}). If $\gamma = 1$, then we simply say \tdef{almost $\lambda$-stable}.
\item \label{def:stable4} $\kaecgoth$ is \tdef{almost stable} if it is almost $\lambda$-stable for unboundedly many cardinals $\lambda$.
\end{enumerate} 
\end{definition}

The reader should notice that, for $M\in\kaec_\lambda$, in \ref{def:stabilityproper}\ref{def:stabilityproper1} we require $|\gS^\gamma_{\kaecgoth}(M)|\leq \lambda$, but in \ref{def:stability}\ref{def:stable2} we require $|\gS^\gamma_{\kaecgoth}(M;N)|\leq \lambda$ for every $M\leqk N$, and this is why we consider almost stability as a local version of stability.

\begin{remark}
A definition very similar to~\ref{def:stability}\ref{def:stable2} was first given in \cite[Definition 2.23]{vaseyinfst} in the context of $\mrm{AC}$s, where it is referred to as \enquote{stability}. There, an abstract class $\kaecgoth$ is said to be $\lambda$-stable if for any $N\in \kaec$ and $A\subseteq N$ with $|A| \leq \lambda$ one has $|\gS_\kaecgoth(A; N)|\leq \lambda$ (cf.~\ref{def:typespace}\ref{def:typespace1}). This last definition is easily seen to be equivalent to~\ref{def:stability}\ref{def:stable2} when $\kaecgoth$ is a $\mu\mh\mrm{AEC}$ and $\lambda = \lambda^{<\mu}$.
\end{remark}

By \cite[Observation~2.6]{paolinishelah}, in the context of abstract elementary classes, $\lambda$-stability (cf.~\ref{def:stabilityproper}\ref{def:stabilityproper1}) is equivalent to almost $\lambda$-stability when the $\mrm{AEC}$ has the  amalgamation property (cf.~\ref{def:amprop}\ref{def:ampropap}). We now introduce the following stronger forms of amalgamation, which will be needed to recover, in the context of abstract classes, the equivalence between almost stability and stability.

\begin{definition}
\label{def:amalgam}
Let $\kaecgoth = (\kaec, \leqk)$ be an $\mrm{AC}$, and $(f_i: M \to M_i\mid i\in I)$ a collection of $\kaecgoth$-embeddings. 
\begin{enumerate}[(1), leftmargin=*]
\item We say that a collection of $\kaecgoth$-embeddings $(g_i: M_i \to N)$ is an \tdef{amalgam} for $(f_i: M \to M_i\mid i\in I)$ if $g_if_i = g_jf_j$ for every $i,j\in I$. 
\item \label{def:amalgam2}In case $f_i: M \to M_i$ is a set-theoretic inclusion, so that $M\leqk M_i$, for every $i\in I$, then we say that an amalgam of $(f_i)_{i\in I}$ is an \tdef{amalgam of $(M_i)_{i\in I}$ over $M$}.
\end{enumerate}
\end{definition}

\begin{definition}
\label{def:amalgac}
Let $\kaecgoth = (\kaec, \leqk)$ be an $\mrm{AC}$ and $\lambda$ a cardinal. 
\begin{enumerate}[(1), leftmargin=*]
\item \label{def:amalgac1} We say that $\kaecgoth$ satisfies the  \tdef{$({<}\lambda)$-amalgamation property} (${<}\lambda\mh\mrm{AP}$ for short) if any collection of $\kaecgoth$-embeddings $(f_i: M\to M_i\mid i\in I)$ with $|I|<\lambda$ has an amalgam.
\item \label{def:amalgac3} We say that $\kaecgoth$ satisfies the \tdef{$\infty$-amalgamation property} ($\infty\mh\mrm{AP}$ for short)  if it has $({<}\lambda)\mh\mrm{AP}$ for every cardinal $\lambda$.
\end{enumerate}
\end{definition}

\begin{remark}
\label{rem:lambdaapcutpaste}
Let $\kaecgoth = (\kaec, \leqk)$ be an $\mrm{AC}$ and $\lambda$ an infinite cardinal. A standard cut-and-paste argument shows that $\kaecgoth$ has $({<}\lambda)\mh\mrm{AP}$ if and only if for every  $\ordnl<\lambda$, $M\in \kaec$ and $(M_i)_{i<\ordnl}$ in $\kaec$ , there is an amalgam of $(M_i)_{i<\ordnl}$ over $M$ (cf.~\ref{def:amalgam}\ref{def:amalgam2}).
\end{remark}

Notice that $({<}\aleph_0)\mh\mrm{AP}$ is equivalent to $({<}3)\mh\mrm{AP}$, and the latter is exactly $\mrm{AP}$ (cf.~\ref{def:amprop}\ref{def:ampropap}). The following lemma gives the promised equivalence between almost $\lambda$-stability and $\lambda$-stability when the abstract class satisfies a strong form of the amalgamation property.

\begin{lemma}
\label{lemma:equivasamalg}
Let $\kaecgoth  = (\kaec, \leqk)$ be an $\mrm{AC}$, $\lambda$ an infinite cardinal, $\gamma\in \mrm{Ord}$. If $\kaecgoth$ has $(<\lambda^{++})\mh\mrm{AP}$, then the following are equivalent:
\begin{enumerate}[(1), leftmargin=*]
\item \label{lemma:equivasamalg1}$\kaecgoth$ is almost $(\lambda, \gamma)$-stable,
\item \label{lemma:equivasamalg2} $\kaecgoth$ is $(\lambda,\gamma)$-stable.
\end{enumerate}
In particular, this holds if $\kaecgoth$ satisfies $\infty\mh\mrm{AP}$.
\end{lemma}
\begin{proof}
Clearly~\ref{lemma:equivasamalg2} implies~\ref{lemma:equivasamalg1}. Let us prove the contrapositive of the other direction. Assume $\kaecgoth$ is not $(\lambda, \gamma)$-stable. Then there is $M\in \kaec_\lambda$ with $|\gS^\gamma_\kaecgoth(M)| \geq \lambda^+$. Therefore, there is a family $\{\gtp(\bar c_i/M; N_i) \mid i<\lambda^+ \}$ of distinct orbital types over $M$, with $M\leqk N_i$ for every $i<\lambda^+$. Using $({<}\lambda^{++})\mh\mrm{AP}$ there is a collection $(f_i: N_i\to N\mid i<\lambda^+)$ of $\kaecgoth$-embeddings with $f_i \restriction M =\id_M$ for every $i<\lambda^+$. We thus have $\gtp(\bar c_i/M; N_i) = \gtp(f_i(\bar c_i)/M; N)$, and so $\lambda^+\leq|\gS^\gamma_\kaecgoth(M; N)|$.
\end{proof}

As we are interested in stability, i.e., $\lambda$-stability over unboundedly many cardinals $\lambda$, then by the previous lemma we want the classes we are considering to also satisfy $\infty\mh\mrm{AP}$. We now show that for some abstract classes the amalgamation property already implies $\infty\mh\mrm{AP}$. For this, we are going to introduce a new property of an abstract class.

\begin{definition}[\protect{\cite[Definition~7.6]{forkingcategorical}}]
\label{def:chainbounds}
Let $\kaecgoth = (\kaec, \leqk)$ be an $\mrm{AC}$. We say that $\kaecgoth$  \tdef{has chain bounds} if whenever $\ordnl$ is an ordinal and $(M_i)_{i<\ordnl}$ is a chain in $\kaecgoth$, there exists $M\in \kaec$ such that $M_i \leqk M$ for every $i<\ordnl$.
\end{definition} 

\begin{remark}
\begin{enumerate}[(1), leftmargin=*]
\item Clearly, if $\kaecgoth$ has continuity (cf.~\ref{def:propertiesac}\ref{def:mucontinuity}), then it has chain bounds. Therefore, every $\mrm{AEC}$ has chain bounds.
\item It is a remarkable fact that a $\mu\mh\mrm{AEC}$ has $\mrm{AP}$, $\mrm{JEP}$ and chain bounds if and only if it has monster models (\cite[Theorem~7.10]{forkingcategorical}), that is, sufficiently homogeneous and universal models. For the precise definitions, the reader should consult \cite[Section~7]{forkingcategorical}. 
\end{enumerate}
\end{remark}

The following lemma is a clear  generalization of the fact that in $\mrm{AEC}$s with $\mrm{AP}$ we have $\infty\mh\mrm{AP}$, which in turn is the reason why in such a case we have that almost $\lambda$-stability coincides with $\lambda$-stability.

\begin{lemma}
\label{lem:chainboundsinftyap}
Let $\kaecgoth = (\kaec, \leqk)$ be an $\mrm{AC}$ with chain bounds and $\mrm{AP}$. Then $\kaecgoth$ has $\infty\mh\mrm{AP}$.
\end{lemma}
\begin{proof}
Assume $\ordnl$ is a limit ordinal, $M\in \kaec$, and $(M_i)_{i<\ordnl}$ is a collection in $\kaec$ with $M\leqk M_i$ for every $i<\ordnl$. We are going to build an amalgam of $(M_i)_{i<\ordnl}$ over $M$. This will give us $\infty\mh\mrm{AP}$ by Remark~\ref{rem:lambdaapcutpaste}.

\nin We wish to define a $\kaecgoth$-chain $(N_i)_{i\leq\ordnl}$ and a 
collection of $\kaecgoth$-embeddings $(g_i: M_i \to N_{i+1}\mid i<\ordnl)$ with $g_i\restriction M = \id_M$ for every $i<\ordnl$. For $i<\ordnl$, let  $g_i': M_i\to N_\ordnl$ denote the map which extends the codomain of $g_i$ to $N_\ordnl$. 
Clearly, $(g_i': M_i \to N_\ordnl\mid i<\ordnl)$ is the required amalgam of $(M_i)_{i<\ordnl}$ over $M$. 
We now show that we can do the construction.
\newline \uline{Case 1}. $i = 0$.
\newline In this case we let $N_0 = M$.
\newline \uline{Case 2}. $i = k+1<\ordnl$ is successor.
\newline Using $\mrm{AP}$, by Remark~\ref{rem:amalgcutpaste}\ref{rem:amalgcutpaste1} there is $N_i\in \kaec$ with $N_k\leqk N_i$, and there is a $\kaecgoth$-embedding $g_k: M_k \to N_i$ with $g_k\restriction M = \id_M$.
\newline \uline{Case 3}. $i\leq \ordnl$ is limit. 
\newline Using chain bounds, find $N_i\in \kaec$ with $N_j \leqk N_i$ for every $j<i$.
\end{proof}

Combining Lemmas~\ref{lemma:equivasamalg} and~\ref{lem:chainboundsinftyap} we get:
\begin{corollary}
\label{cor:coralmostmuap}
Let $\kaecgoth = (\kaec, \leqk)$ be an abstract class with chain bounds (cf.~\ref{def:chainbounds}) and $\mrm{AP}$. Then for every infinite cardinal $\lambda$ and ordinal $\gamma$ the following are equivalent:
\begin{enumerate}[(1), leftmargin=*]
\item $\kaecgoth$ is almost $(\lambda, \gamma)$-stable.
\item $\kaecgoth$ is $(\lambda, \gamma)$-stable.
\end{enumerate}
\end{corollary}

Finally, we end this section with a very useful lemma which highlights the flexibility of the concept of almost stability. Notice that in Lemma~\ref{lem:alstabtransfer} we only assume that $\kaecgoth_1$ is an $\mrm{AC}$ and $\kaecgoth_2$ has $\mrm{LS}$ at the almost stability cardinal.

\begin{lemma}
\label{lem:alstabtransfer}
Let $\kaecgoth_1 = (\kaec, \leqk_1)$, $\kaecgoth_2 = (\kaec, \leqk_2)$ be $\mrm{AC}$s and $\lambda$ an infinite cardinal. Assume $\kaecgoth_2$ has $\mrm{LS}$ at $\lambda$ and $\leqk_1$ refines $\leqk_2$, i.e., $M\leqk_2 N$ implies $M\leqk_1 N$ for every $M,N\in\kaec$. 
If ${\kaecgoth_2 = (\kaec, \leqk_2)}$ is almost $\lambda$-stable, then $\kaecgoth_1 = (\kaec, \leqk_1)$ is almost $\lambda$-stable.
In particular, if $\kaecgoth_1$ also satisfies $({<}\lambda^{++})\mh\mrm{AP}$ (e.g., it has $\infty\mh\mrm{AP}$), then $\kaecgoth_1$ is $\lambda$-stable.
\end{lemma}
\begin{proof}
Throughout, for $\ell\in\{1,2\}$, we let $E_\ell$ denote $E_{\kaecgoth_\ell}$, and similarly for $E_\ell^\at$.
\nin\begin{enumerate}[$(*_1)$, leftmargin=*, series=prf:astablerels]
\item \label{prf:astablerels1} For any $A\subseteq N_\ell\in \kaec$ and $\bar b_\ell \in N_\ell$, with $\ell\in\{1, 2\}$, we have that $\newline{(\bar b_1, A, N_1)E_2^\at(\bar b_2, A, N_2)}$ implies $(\bar b_1, A, N_1)E_1^\at(\bar b_2, A, N_2)$.
\end{enumerate}
\nin Why \ref{prf:astablerels1}? By definition there is $N\in\kaec$ and $\kaecgoth_2$-embeddings $f_\ell: N_\ell \to N$, with $\ell\in\{1, 2\}$, such that $f_1(\bar b_1) = f_2(\bar b_2)$. Because $\leqk_1$ refines $\leqk_2$, $f_1$ and $f_2$ are $\kaecgoth_1$-embeddings, and this gives us $(\bar b_1, A, N_1)E_1^\at(\bar b_2, A, N_2)$.
\nin\begin{enumerate}[resume*=prf:astablerels]
\item \label{prf:astablerels2} For any $A\subseteq N_\ell\in \kaec$ and $\bar b_\ell \in N_\ell$, with $\ell\in\{1, 2\}$, we have that $\newline{(\bar b_1, A, N_1)E_2(\bar b_2, A, N_2)}$ implies $(\bar b_1, A, N_1)E_1(\bar b_2, A, N_2)$.
\end{enumerate}
\nin Clearly \ref{prf:astablerels2} follows from \ref{prf:astablerels1} by induction, because $E_\ell$ is the transitive closure of $E_\ell^\at$ (cf.~\ref{def:eqrelorbtypes}), with $\ell\in\{1,2\}$.
\nin\begin{enumerate}[resume*=prf:astablerels]
\item \label{prf:astablerels3} $\kaecgoth_1 =(\kaec, \leqk_1)$ is almost $\lambda$-stable.
\end{enumerate}
\nin Why \ref{prf:astablerels3}? Let $M, N\in \kaec$ with $M \leqk_1 N$ and $|M| \leq \lambda$. By $\mrm{LS}$ at $\lambda$ for $\kaecgoth_2 = (\kaec, \leqk_2)$, there is $M\subseteq M_0 \leqk_2 N$ with $|M_0|\leq \lambda$.
 By almost $\lambda$-stability of $\kaecgoth_2$ there is a set $(c_i\mid i<\lambda)$ in $N$ such that for every $d\in N$ there is $i<\lambda$ such that $(d, M_0, N)E_2(c_i, M_0, N)$. For $d\in N$, take $i<\lambda$ such that  $(d, M_0, N)E_2(c_i, M_0, N)$ holds. By \ref{prf:astablerels2} we have $(d, M_0, N) E_1(c_i, M_0, N)$, and by Remark~\ref{rem:equivrel}\ref{rem:equivrel2} one has $(d, M, N) E_1(c_i, M, N)$. Therefore $\kaecgoth_1$ is almost $\lambda$-stable.
\nin\begin{enumerate}[resume*=prf:astablerels]
\item \label{prf:astablerels4} If $\kaecgoth_1$ has $({<}\lambda^{++})\mh\mrm{AP}$, then it is $\lambda$-stable.
\end{enumerate}
Why \ref{prf:astablerels4}? Follows from \ref{prf:astablerels3} and Lemma \ref{lemma:equivasamalg}.
\end{proof}

Lemma~\ref{lem:alstabtransfer} easily yields the following:
\begin{corollary}
\label{cor:alstabtransfermu}
Let $\kaecgoth_1 = (\kaec, \leqk_1)$ be an $\mrm{AC}$, $\kaecgoth_2 = (\kaec, \leqk_2)$ a $\mu\mh\mrm{AEC}$, and $\lambda =\lambda^{<\mu}+\mrm{LS}_\mu(\kaecgoth_2)$ an infinite cardinal. Assume $\leqk_1$ refines $\leqk_2$,  i.e., $M\leqk_2 N$ implies $M\leqk_1 N$ for every $M,N\in\kaec$.  
If ${\kaecgoth_2 = (\kaec, \leqk_2)}$ is almost $\lambda$-stable, then $\kaecgoth_1 = (\kaec, \leqk_1)$ is almost $\lambda$-stable.
In particular, if $\kaecgoth_1$ also satisfies $\infty\mh\mrm{AP}$, then $\kaecgoth_1$ is $\lambda$-stable.
\end{corollary}
\begin{proof}
Follows from Lemma~\ref{lem:alstabtransfer} using the fact that $\kaecgoth_2$ has $\mrm{LS}$ at $\lambda$ by Remark~\ref{rem:defmuaec}\ref{rem:defmuaec3}.
\end{proof}

\section{Syntactic abstract classes}
In this section we introduce syntactic abstract classes (Definition~\ref{def:syntacticAC}) which are, to a good approximation, the abstract classes $(\kaec, \leqk)$ for which $\kaec$ is the class of models of a sentence of the infinitary logic $\inflang_{\infty,\infty}$, and the relation $\leqk$ can be checked on the models in $\kaec$ by an infinitary formula in $\inflang_{\infty,\infty}$. Then, we show that the abstract classes coming from infinitary logics are syntactic (Proposition~\ref{prop:infsynt}). The definition and results of this section will then be used in Section~\ref{sect:syntalstable} to show that tame syntactic $\mu\mh\mrm{AEC}$s of modules are almost stable.

Throughout, we assume basic familiarity with the definitions of infinitary logics, as presented for example in \cite{largerinfinitary, kuekerinfinitary}. For ease of reference, we introduce the notation we are going to use in the following.
\begin{notation}
\label{not:inflogic}
Let $\tau$ be a fixed vocabulary.
\begin{enumerate}[(1), leftmargin=*]
\item \label{not:inflogic1}Let $\nu_0\geq\nu_1$ be infinite regular cardinals. We let $\mathfrak L_{\nu_0, \nu_1}(\tau)$ denote the infinitary logic which contains all atomic formulas, is closed under negation, admits conjunctions and disjunctions of $<\nu_0$ formulas, admits existential and universal quantifiers bounding $<\nu_1$ variables, and such that every formula has $<\nu_1$ free variables. It is straightforward to extend the definition of satisfiability in this context.
\item \label{not:inflogic2}Let $\nu_1$ be an infinite regular cardinal. We let $\inflang_{\infty,\nu_1}(\tau)$ denote the union of all the logics $\inflang_{\nu_0,\nu_1}$ for $\nu_0\geq \nu_1$ regular.
\item \label{not:inflogic3}We let $\inflang_{\infty,\infty}(\tau)$ denote the union of all the logics $\inflang_{\infty,\nu_1}(\tau)$ for $\nu_1$ regular.
\item \label{not:inflogic4}If $\inflang(\tau)$ is one of the logics just defined, we let $\inflang^{\mrm{qf}}(\tau)$ denote the collection of quantifier-free formulas in $\inflang(\tau)$.
\item \label{not:inflogicqd}If $\nu$ is an infinite regular cardinal and $\varphi\in \inflang_{\infty,\nu}(\tau)$, we let $\mrm{qd}_\nu(\varphi)$ denote the \tdef{quantifier depth} of $\varphi$, which is defined in the following way:
\begin{enumerate}[(a), leftmargin=*]
\item $\mrm{qd}_\nu(\varphi) = 0$ if $\varphi$ is atomic,
\item $\mrm{qd}_\nu(\neg\varphi) = \mrm{qd}_\nu(\varphi)$,
\item $\mrm{qd}_\nu(\bigwedge_{i\in I}\varphi_i)=\mrm{qd}_\nu(\bigvee_{i\in I}\varphi_i) = \sup\{\mrm{qd}_\nu(\varphi_i)\mid i\in I\}$,
\item $\mrm{qd}_\nu(\forall \bar x \varphi_0)=\mrm{qd}_\nu(\exists \bar x \varphi_0) = \mrm{qd}_\nu(\varphi_0) +1$.
\end{enumerate}
\end{enumerate}
When clear from the context, we shall omit the reference to the vocabulary $\tau$.
\end{notation}
\begin{notation} 
\label{not:formulas}
\begin{enumerate}[(1), leftmargin=*]
\item All sequences of variables considered are indexed by ordinals.
\item If $\bar x$ is a sequence of variables, we let $|\bar x|$ denote the length of the sequence.  
\item If $\bar x$ and $\bar y$ are sequences of variables, we let $\bar x^\frown \bar y$ be the concatenation of the two sequences.
\item \label{not:formulas1} Given a formula $\varphi\in\inflang_{\infty,\infty}$ and ordinals $\alpha, \beta$, when we write $\varphi(\bar{x}_\alpha, \bar{y}_\beta)$ we mean that $\bar{x}_\alpha = (x_i \mid i < \alpha)$ and  $\bar{y}_\beta = (y_i \mid i < \beta)$, that all the free variables which appear in $\varphi$ are in $\bar x_\alpha^\frown \bar y_\beta$, and that the variables are partitioned, i.e., the sets $\{x_i\mid i<\alpha\}$ and $\{y_i\mid i<\beta\}$ are disjoint. In such a case we say that $\varphi$  \tdef{is in the partitioned variables} $(\bar x_\alpha,\bar y_\beta)$. Similar notational remarks hold when the variables appearing in the formula $\varphi$ are partitioned in more than two pieces.
\item \label{not:formulas5}Given a set of formulas $\Delta \subseteq \inflang_{\infty,\infty}$ and ordinals $\alpha,\beta$, when we write \newline$\Delta = \Delta(\bar x_\alpha, \bar y_\beta)$, we mean that $\Delta$ is a set of formulas in the partitioned variables $(\bar x_\alpha, \bar y_\beta)$. We may also write $\Delta = \Delta(\bar x_\alpha, \bar y)$ when we are not interested in the length of $\bar y$.
\end{enumerate}
\end{notation}

The following definition isolates those abstract classes which satisfy the analogue of \cite[Fact~2.12]{paolinishelah}, which says that every abstract elementary class is syntactic, in the following sense:

\begin{definition}
\label{def:syntacticAC}
Let $\kaecgoth = (\kaec, \leqk)$ be an $\mrm{AC}$.
\begin{enumerate}[(1), leftmargin=*]
\item \label{def:syntacticAC1}  Let $\nu_0, \nu_1,\kappa$ be infinite cardinals with $\nu_0,\nu_1$ regular and $\nu_0\geq \nu_1\geq \kappa^+ > |\tau_\kaecgoth|$. We say that $\leqk$ is \tdef{$(\inflang_{\nu_0,\nu_1}, \kappa)$-syntactic on} $\kaec$ if there is a formula $\varphi^\star_\kappa(\bar x_\kappa)\in\inflang_{\nu_0, \nu_1}$ with $|\bar x| = \kappa$, such that if  $N\in\kaec$, $\bar a\in N^\kappa$ and $N\restriction \bar a$ is a substructure of $N$, then:
\[
(N\restriction \bar a\in \kaec \text{ and } N \restriction \bar{a} \leqk N )\;\; \Leftrightarrow  \;\; N \vDash \varphi^\star_{\kappa}(\bar{a}).
\]
\item \label{def:syntacticAC2} Let $\kappa\geq |\tau_\kaecgoth|$ be an infinite cardinal. We say that $\leqk$ is \tdef{$\kappa$-syntactic on} $\kaec$ if it is $(\inflang_{\nu_0, \nu_1}, \kappa)$-syntactic for some $\nu_0, \nu_1$ (so in particular $\nu_0\geq \nu_1\geq \kappa^+$).
\item \label{def:syntacticAC3} We say that $\leqk$ is \tdef{syntactic on} $\kaec$ if it is $\kappa$-syntactic for unboundedly many infinite cardinals $\kappa$.
\item \label{def:syntacticAC4} We say that $\kaec$ is \tdef{syntactic} if there is a sentence $\psi^\star\in \inflang_{\infty, \infty}$ such that for every $\tau_\kaecgoth$-structure $N$ we have that:
\[
N\in\kaec \;\; \Leftrightarrow  \;\; N\vDash \psi^\star.
\]
\item \label{def:syntacticAC5} We say that $\kaecgoth$ is \tdef{syntactic} if $\kaec$ is syntactic and $\leqk$ is syntactic on $\kaec$.
\end{enumerate}
\end{definition}

Notice that in~\ref{def:syntacticAC}\ref{def:syntacticAC4} we only require the class $\kaec$ to be syntactic, while in~\ref{def:syntacticAC}\ref{def:syntacticAC5} we require both the class of structures $\kaec$ and the substructure relation $\leqk$ to be syntactic on $\kaec$. 

\begin{remark}
\label{rem:ksyncloseddown}
Let $\kaecgoth = (\kaec, \leqk)$ be an $\mrm{AC}$, $\kappa_1\geq \kappa_2\geq |\tau_\kaecgoth|$ be infinite cardinals. If $\leqk$ is $\kappa_1$-syntactic on $\kaec$, then $\leqk$ is $\kappa_2$-syntactic on $\kaec$.
\end{remark}

Using the terminology we have just introduced, all $\mrm{AEC}$s are syntactic. This follows from the following important result of Shelah and Villaveces:

\begin{fact}[\cite{shelahvillaveces}]\label{fact:villaveces} Let $\kaecgoth  = (\kaec, \leqk)$ be an $\mrm{AEC}$. The following hold:
\begin{enumerate}[(1), leftmargin=*]
\item Let $\kappa \geq \mrm{LS}(\kaecgoth)$ and $\lambda_\kappa =\beth_2(\kappa)^{++}$. There is $\varphi^\star_{\kappa}(\bar{x}_{\kappa}) \in \inflang_{\lambda^+_\kappa, \kappa^+}$ such that  if $N \in \mathbf{K}$, $\bar{a} \in N^\kappa$ and $N \restriction \bar{a}$ is a substructure of $N$, then:
\[
(N\restriction \bar a\in \kaec \text{ and } N \restriction \bar{a} \leqk N )\;\; \Leftrightarrow  \;\; N \vDash \varphi^\star_{\kappa}(\bar{a}).
\]
\item Let $\kappa_0 = \mrm{LS}(\kaecgoth)$ and $\lambda_0 = \beth_2(\mrm{LS}(\kaecgoth))^{++}$. There is $\psi^\star\in \inflang_{\lambda_0^+, \kappa_0^+}$ such that if $N$ is a $\tau_\kaecgoth$-structure, then $N\in \kaec$ if and only if $N\vDash \psi^\star$.
\end{enumerate}
\end{fact}

The preceding fact is a consequence of the proof of the \enquote{Tarski-Vaught} criterion for $\mrm{AEC}$s proved in \cite{shelahvillaveces}. Although it is not stated in the journal version of \cite{shelahvillaveces}, it is explicitly stated in the latest arXiv version \cite[Theorem~3.1]{shelahvillavecesarxiv}, and also stated
in \cite[Fact~2.12]{paolinishelah}.

\begin{remark}
\label{rem:aecsyntactic}
Using the terminology of Definition~\ref{def:syntacticAC}, Fact~\ref{fact:villaveces} says that if $\kaecgoth$ is an abstract elementary class, then $\kaecgoth$ is syntactic. In particular, for $\kappa \geq \mrm{LS}(\kaecgoth)$ and $\lambda_{\kappa} = \beth_2(\kappa)^{++}$, we have that $\leqkg$ is $(\inflang_{\lambda_\kappa^+, \kappa^+},\kappa)$-syntactic on $\kaec$.
\end{remark}

Although the previous remark ensures that every abstract elementary class is syntactic in the sense of Definition~\ref{def:syntacticAC}, it is not true that also every $\mu\mh\mrm{AEC}$ is syntactic. Indeed, Boney and Walker \cite{boneypreparation}  recently found an $\aleph_1\mh\mrm{AEC}$ $(\kaec, \leqk)$ with neither $\kaec$ syntactic, nor $\leqk$ syntactic on $\kaec$ (cf.~\ref{def:syntacticAC}).

We now present classical examples of $\mu\mh\mrm{AEC}$s arising from infinitary logic and verify that they are indeed syntactic.

\begin{definition}
\label{def:infinitaryAC}
Let $\Delta$ be a set of formulas in $\inflang_{\infty,\infty}$ and {let} $T\subseteq \Delta$ {be} a set of sentences. We define $\kaecgoth_{(T, \Delta)} = (\kaec_T, \leqk_\Delta)$, where $\kaec_T = \mrm{Mod}(T)$, and for $M,N\in \kaec_T$ we let $M\leqk_\Delta N$ if and only if for every $\varphi(\bar x)\in \Delta$ and $\bar a\in M^{|\bar x|}$ one has 
\[
M\vDash \varphi(\bar a) \;\; \Leftrightarrow \;\; N\vDash \varphi(\bar a).
\]
\end{definition}

Clearly, $\kaecgoth_{(T, \Delta)}$ as defined in~\ref{def:infinitaryAC} is an abstract class. To have that $\kaecgoth_{(T, \Delta)}$ is a $\mu\mh\mrm{AEC}$ we need to assume that the set $\Delta$ is syntactically well-behaved.

\begin{definition}
\label{def:fragment}
A \tdef{fragment}  of $\inflang_{\infty,\infty}$ is a set $\Delta$ of formulas in $\inflang_{\infty,\infty}$ closed under subformulas. We say that a theory $T$ is in the fragment $\Delta$ if $T\subseteq \Delta$.
\end{definition}
\begin{fact}[\protect{\cite[pg.~3052]{muaecvasey}}]
\label{fact:fragmentmuaec}
Let $\mu$ be a regular cardinal, $\Delta$ a fragment of $\inflang_{\infty,\infty}$, $T$ a theory in $\Delta$. Then $\kaecgoth_{(T, \Delta)} = (\kaec_T, \leqk_\Delta)$ is a $\mu\mh\mrm{AEC}$ with $\mrm{LS}_\mu(\kaecgoth_{(T, \Delta)}) \leq (|\Delta| + |\tau_{\kaecgoth_{(T, \Delta)}}|)^{<\mu}$ (cf.~\ref{def:muaec}).
\end{fact}

\begin{notation}
\label{not:acsyntactic}
In the following, if $\alpha$ is an ordinal and $\bar x$ is a tuple of variables, by $\bar x'\in [\bar x, \alpha]$ we mean that $\bar x' = (x_{j(i)}\mid i<\alpha)$ for some $j: \alpha \to |\bar x|$. 
\end{notation}

\begin{lemma}
\label{lemma:acsyntactic1}
Let $\mu_0, \mu_1,\kappa$ be infinite cardinals with $\mu_0,\mu_1$ regular and $\mu_0\geq \mu_1$. If $\nu_0 = \mu_0 + (\kappa^{<\mu_1})^+$ and $\nu_1 = \kappa^+ + \mu_1$, then for every $\varphi(\bar y)\in \inflang_{\mu_0,\mu_1}$ there is a formula $\varphi^\square(\bar x, \bar y)\in \inflang_{\nu_0, \nu_1}^\qf$ (cf.~\ref{not:inflogic}\ref{not:inflogic4}) with $|\bar x| = \kappa$, such that:
\begin{enumerate}[$(*)$, leftmargin=*]
\item for every $M \leq N$, $\bar a\in M^\kappa$, $\bar b\in M^{|\bar y|}$ with $N\restriction \bar a =M$, then
\[
N\vDash\varphi^\square(\bar a, \bar b) \;\; \Leftrightarrow  \;\; M \vDash \varphi(\bar b).
\]
\end{enumerate}
\end{lemma}
\begin{proof} We proceed by induction on the complexity of the formulas in $\inflang_{\mu_0, \mu_1}$.
\begin{enumerate}[$(\cdot_1)$,leftmargin=*]
\item (Atomic formula) If $\varphi(\bar y)$ is atomic, let $\varphi^\square(\bar x, \bar y) = \varphi(\bar y)$.
\item (Negation) If $\varphi(\bar y) = \neg \varphi_0(\bar y)$, define $\varphi^\square(\bar x, \bar y) = \neg \varphi_0^\square (\bar x, \bar y)$.
\item (Conjunction) If $\varphi(\bar y) = \bigwedge_{i\in I} \varphi_i(\bar y)$, let $\varphi^\square (\bar x, \bar y) = \bigwedge_{i\in I} \varphi_i^\square (\bar x, \bar y)$.
\item (Disjunction) If $\varphi(\bar y) = \bigvee_{i\in I} \varphi_i(\bar y)$, let $\varphi^\square (\bar x, \bar y) = \bigvee_{i\in I} \varphi_i^\square (\bar x, \bar y)$.
\item (Existential quantifier) If $\varphi(\bar y) = \exists \bar z \varphi_0(\bar y, \bar z)$, let 
\[
\varphi^\square(\bar x, \bar y) = \bigvee \{\varphi_0^\square (\bar x, \bar y, \bar y')\mid \bar y' \in [\bar y, |\bar z|]\},
\]
where $[\bar y, |\bar z|]$ is as in Notation~\ref{not:acsyntactic}.
\item (Universal quantifier) If $\varphi(\bar y) = \forall \bar z \varphi_0(\bar y, \bar z)$, let 
\[
\varphi^\square(\bar x, \bar y) = \bigwedge\{\varphi_0^\square (\bar x, \bar y, \bar y')\mid \bar y' \in [\bar y, |\bar z|]\}.
\]
\end{enumerate}
It is easily verified that the formulas defined are as wanted.
\end{proof}

\begin{lemma}
\label{lemma:acsyntactic2}
Let $\mu_0,\mu_1,\kappa$ be infinite cardinals with $\mu_0,\mu_1$ regular and $\mu_0 \geq \mu_1$, $\Delta\subseteq\inflang_{\mu_0, \mu_1}$. If $\nu_0 = \mu_0 + (\kappa^{<\mu_1})^+ + |\Delta|^+$, $\nu_1 = \kappa^+ + \mu_1$, then there is a formula $\varphi^\Delta_\kappa(\bar x)\in \inflang_{\nu_0, \nu_1}$ such that  the following condition holds.
\begin{enumerate}[$(*)$, leftmargin=*]
\item For every $M \leq N$, $\bar a\in M^\kappa$ with $N \restriction \bar a = M$ (so in particular $\bar{a}$ is the domain of a $\tau_{\kaecgoth}$-substructure), we have that:
\[
N \vDash \varphi^\Delta_\kappa(\bar a) \;\; \Leftrightarrow \;\; M \leqk_\Delta N.
\]
\end{enumerate}
\end{lemma}
\begin{proof}
It is easily verified that the formula
\[
\varphi^\Delta_\kappa(\bar x) = \bigwedge\{\varphi^\square(\bar x, \bar x') \leftrightarrow \varphi(\bar x')\mid \varphi(\bar y)\in \Delta\text{ and }\bar x'\in [\bar x, |\bar y|]\}
\]
is as wanted.
\end{proof}

\begin{proposition}
\label{prop:infsynt}
Let $\Delta\subseteq\inflang_{\infty,\infty}$, $T\subseteq \Delta$ a set of sentences, and $\kaecgoth_{(T, \Delta)} = (\kaec_T, \leqk_\Delta)$. Then $\kaec_T$ is syntactic and $\leqk_\Delta$ is $\kappa$-syntactic on $\kaec_T$ for every $\kappa \geq |\tau_\kaecgoth|$. In particular, $\kaecgoth_{(T, \Delta)}$ is syntactic.
\end{proposition}
\begin{proof}
That $\kaec_T$ is syntactic is clear, and the second assertion follows readily from Lemma~\ref{lemma:acsyntactic2}.
\end{proof}

%
%

In Section~\ref{sect:lpure} we are going to study the relation $\leqp^\lambda$ of $\lambda$-purity between $R$-modules (Definition~\ref{def:lpure}), where $\lambda$ is an infinite cardinal. By Proposition~\ref{prop:equivlpurelppform}, the relation of being a $\lambda$-pure submodule is syntactic on $\rmod$. Analogously, in Subsection~\ref{sect:lpurebal} and Appendix~\ref{sect:appstabletfab} we are going to study the relation $\leqb$ of being a balanced subgroup (Definition~\ref{def:proper}) between torsion-free abelian groups. By  Proposition~\ref{prop:balsyntaxequiv} the relation $\leqb$ is syntactic on $\tfab$, the class of all torsion-free abelian groups. 

\section{Almost stability for syntactic \texorpdfstring{$\mu\mh\mrm{AEC}$}{μ-AEC}s of \texorpdfstring{$R$}{R}-modules}
\label{sect:syntalstable}
\subsection{The results}
A central question in the study of abstract elementary classes of modules, asked by Mazari-Armida \cite[Question~2.12]{mazarimodeltheoretic}, is whether every $\mrm{AEC}$ of $R$-modules with the pure submodule relation is stable. The second author and Shelah answered this negatively, exhibiting an unstable $\mrm{AEC}$ $(\kaec, \leqp)$ in which $\kaec$ is a class of torsion-free abelian groups \cite[Theorem~1.2]{paolinishelah}. Nonetheless, in the same paper the authors establish a positive result on the (almost) stability of $\mrm{AEC}$s of $R$-modules:

\begin{fact}[\protect{\cite[Theorem~1.3]{paolinishelah}}]
\label{fact:paolinishelahthm}
Let $\kaecgoth = (\kaec, \leqk)$ be an $\mrm{AEC}$ of $R$-modules and $\kappa \geq \mrm{LS}(\kaecgoth)$ an infinite cardinal such that $\kaecgoth$ is $\kappa$-tame. Then there is a cardinal $\xi \geq \kappa$ such that for every $\lambda = \lambda^{<\xi}$ we have that $\kaecgoth$ is almost $\lambda$-stable. In particular, $\kaecgoth$ is almost stable, and if $\kaecgoth$ has $\mrm{AP}$ then it is stable.
\end{fact} 

In this section we aim to show that every tame syntactic $\mu\mh\mrm{AEC}$ of $R$-modules is almost stable, and  stable if it also satisfies $\infty\mh\mrm{AP}$. The main theorem is the following very general result (see also Theorem~\ref{thm:syn2}):

\begin{theorem}
\label{thm:syn1}
Let $\kaecgoth = (\kaec, \leqk)$ be an $\mrm{AC}$ of $R$-modules and $\kappa\geq |\tau_\kaecgoth|$ an infinite cardinal. Assume $\kaecgoth$ has coherence, $\mrm{LS}$ at $\kappa$, is $\kappa$-tame, and $\leqk$ is $\kappa$-syntactic on $\kaec$ (cf.~\ref{def:syntacticAC}). Then there is a cardinal $\xi\geq \kappa$ such that for every $\lambda = \lambda^\xi$ we have that $\kaecgoth$ is almost $\lambda$-stable (cf.~\ref{def:stability}). In particular, $\kaecgoth$ is almost stable, and if $\kaecgoth$ has $\infty\mh\mrm{AP}$ (cf.~\ref{def:amalgac}) then it is stable.
\end{theorem}

Now, we show how Theorem~\ref{thm:syn1} can be used to obtain the following important results regarding the (almost) stability of $\mu\mh\mrm{AEC}$s of modules.  Recall that, by Remark~\ref{rem:aecsyntactic},  every $\mrm{AEC}$ is syntactic, so that the following theorem recovers Fact~\ref{fact:paolinishelahthm}.

\begin{theorem}
\label{thm:syn2}
Let $\kaecgoth = (\kaec, \leqk)$ be a $\mu\mh\mrm{AEC}$ of $R$-modules and $\kappa = \kappa^{<\mu}+ \mrm{LS}_\mu(\kaecgoth)$ an infinite cardinal. If $\kaecgoth$ is $\kappa$-tame and $\leqk$ is $\kappa$-syntactic on $\kaec$, then there is a cardinal $\xi\geq \kappa$ such that for every $\lambda = \lambda^\xi$ we have that $\kaecgoth$ is almost $\lambda$-stable. In particular, $\kaecgoth$ is almost stable, and if $\kaecgoth$ has $\infty\mh\mrm{AP}$ then it is stable.
\end{theorem}
\begin{proof}
Follows from Theorem~\ref{thm:syn1}, because $\kaecgoth$ has $\mrm{LS}$ at $\kappa$ by Remark~\ref{rem:defmuaec}\ref{rem:defmuaec3}.
\end{proof}

 Moreover, we have the following:

\begin{corollary}
\label{cor:syn3}
Assume there are unboundedly many strongly compact cardinals. If $\kaecgoth = (\kaec, \leqk)$ is a $\mu\mh\mrm{AEC}$ of $R$-modules with $\leqk$ syntactic on $\kaec$, then $\kaecgoth$ is almost stable.
\end{corollary}
\begin{proof}
By Fact~\ref{fact:strongtameness}, if $\kappa > \mrm{LS}_\mu(\kaecgoth)$ is a strongly compact cardinal, then $\kaecgoth$ is $\kappa$-tame, so that the result follows from Theorem~\ref{thm:syn2}.
\end{proof}

The next subsection is devoted to proving Theorem~\ref{thm:syn1}.
The proof closely follows the one given in \cite{paolinishelah} for the almost stability of $\mrm{AEC}$s of modules; indeed, that argument can readily be adapted to our setting. Nonetheless, we have been able to shorten it considerably by avoiding the technically involved \cite[Claim~3.5]{paolinishelah}, on which the main theorem of that paper relied. Instead, our proof (Theorem~\ref{thm:synfund}) uses directly the elimination of quantifiers for the infinitary logic of $R$-modules (Fact~\ref{fact:elqfmod}), together with an argument closely inspired by the standard proof of stability for first-order theories of modules.

We recall that, by the results of \cite{boneypreparation}, there exists an $\aleph_1\mh\mrm{AEC}$ $(\kaec, \leqk)$ in which $\leqk$ is not syntactic on $\kaec$. It is therefore natural to ask whether this can still occur when $\kaec\subseteq \rmod$:

\begin{question}
Is there a $\mu\mh\mrm{AEC}$ $(\kaec, \leqk)$ with $\kaec\subseteq \rmod$ and $\leqk$ not syntactic on $\kaec$? With $\kaec = \rmod$? Can we find an example with $\leqk$ refining direct summands?
\end{question}

Moreover, motivated by the almost stability transfer (cf. Corollary~\ref{cor:alstabtransfermu}), we are also led to ask the following question:
\begin{question}
Let $\kaecgoth_1 = (\kaec, \leqk_1)$ {be} a $\mu_1\mh\mrm{AEC}$ (not necessarily of modules), is there a $\mu_2\mh\mrm{AEC}$ $\kaecgoth_2= (\kaec, \leqk_2)$ such that $\leqk_2$ is syntactic on $\kaec$ and $\leqk_1$ refines $\leqk_2$? Can we find an example with $\kaec\subseteq \rmod$? With $\leqk_1$ refining direct summands?
\end{question}

Finally, we close this subsection by giving a definition which will be used crucially in the next subsection to prove Theorem~\ref{thm:syn1}:

\begin{definition}\label{def:syntacticstable} 
Let $\kaecgoth =(\kaec, \leqk)$ be an $\mrm{AC}$, $\lambda$ an infinite cardinal, $\gamma\in \mrm{Ord}$,  $\Delta = \Delta(\bar x_\gamma, \bar y) \subseteq \inflang_{\infty, \infty}(\tau_\kaecgoth)$ a set of formulas (cf.~\ref{not:formulas}\ref{not:formulas5}).
\begin{enumerate}[(1),leftmargin=*]
\item \label{formulas2} Given $M\in \kaec$, $\bar c\in M^{\gamma}$, and $A\subseteq M$, we let
\[
\tp_\Delta(\bar c/A; M)= \{\varphi(\bar x, \bar a) \mid M \vDash \varphi(\bar c, \bar a) \text{ and } \bar a\in A^{|\bar y|}\}.
\]
\item \label{def:syntacticstable2} We say that $\kaecgoth$ is \tdef{syntactically almost $(\lambda, \gamma, \Delta)$-stable} when for every $M, N \in\kaec$ with $M \leqk N$ and $M \in \kaec_{\lambda}$ we have that $|\mathbf{S}^\gamma_\Delta(M; N)| \leq \lambda$, where:
\[
\mathbf{S}^\gamma_\Delta(M; N) = \{\tp_\Delta(\bar{c}/M; N) \mid \bar{c} \in N^\gamma \}.
\]
\end{enumerate}
\end{definition}

\subsection{The proof of Theorem~\ref{thm:syn1}}

We start by giving an idea of the proof. First, in Lemma~\ref{lem:2syn} we show that in the context of Theorem~\ref{thm:syn1} we have that almost stability is equivalent to syntactic almost stability with respect to a certain set of formulas (cf.~\ref{def:syntacticstable}). Then, in Theorem~\ref{thm:synfund} we show that, by the elimination of quantifiers for infinitary formulas in the language of $R$-modules (Fact~\ref{fact:elqfmod}), we always have syntactic almost stability.

We remark that  Lemmas~\ref{lem:firstsyn} and \ref{lem:2syn} are the straightforward adaptations of \cite[Major~Claim~3.3]{paolinishelah} and \cite[Claim~3.4]{paolinishelah} respectively. Although the proofs are very similar, we repeat them for completeness.

We now fix the context of the present subsection.

\begin{context}\label{cont:syntactic}
Let $\kaecgoth = (\kaec, \leqk)$, $\kappa$, $\nu$, $\gamma$ be as follows:
\begin{enumerate}[(1), leftmargin=*]
\item $\kaecgoth= (\kaec, \leqk)$ is an abstract class with coherence (cf.~\ref{def:propertiesac}\ref{def:coherence});
\item $\kappa$ is an infinite cardinal and $\kaecgoth$ has $\mrm{LS}$ at $\kappa$ (cf.~\ref{def:propertiesac}\ref{def:LS});
\item\label{cont:syntactice} $\nu$ is a regular infinite cardinal and $\nu\geq (2^\kappa)^+$;
\item $\leqk$ is $(\inflang_{\nu, \nu},\kappa)$-syntactic on $\kaec$ (cf.~\ref{def:syntacticAC});
\item $\gamma < \kappa^+$ is an ordinal.
\end{enumerate}
\end{context}

\begin{remark}
\label{rem:prfthmhyp}
If $\kaecgoth = (\kaec, \leqk)$ and $\kappa$ are as in the hypotheses of Theorem~\ref{thm:syn1}, then there is  a regular infinite cardinal $\nu$ such that $\leqk$ is $(\inflang_{\nu,\nu}, \kappa)$-syntactic on $\kaec$, and clearly we can assume that $\nu$ satisfies Context~\ref{cont:syntactic}\ref{cont:syntactice}.
\end{remark}

We introduce the following notation, which will be useful in the following:

\begin{notation}
Let $\kaecgoth = (\kaec, \leqk)$ be an abstract class, $\kappa$ an infinite cardinal, $\gamma<\kappa^+$ an ordinal.
\begin{enumerate}[(1), leftmargin=*]
\item\label{not:synw1} Let $\wsmall_{(\kaecgoth, \kappa, \gamma)}$ be the class of quintuples of the form
\[
\frku = (M_\frku, N_\frku, \bar{a}_\frku, \bar{b}_\frku, \bar{c}_\frku) = (M, N, \bar a, \bar b, \bar c)
\]
with:
\begin{enumerate}[(a), leftmargin=*]
\item $M, N\in \kaec$, $M\leqk N$, $|N|\leq \kappa$;
\item $\bar a\in M^\kappa$ enumerates $M$ (possibly with repetitions);
\item $\bar b\in N^\kappa$ enumerates $N$ (possibly with repetitions);
\item $\bar c\in N^\gamma$.
\end{enumerate}
\item\label{not:synw2} Let $\wlarge_{(\kaecgoth, \kappa, \gamma)}$ be the class of quadruples of the form
\[
\frkw = (M_\frkw, N_\frkw, \bar{a}_\frkw, \bar{c}_\frkw) = (M, N, \bar a, \bar c)
\]
with:
\begin{enumerate}[(a), leftmargin=*]
\item $M, N\in \kaec$, $M\leqk N$, $|M|\leq \kappa$;
\item $\bar a\in M^\kappa$ enumerates $M$ (possibly with repetitions);
\item $\bar c\in N^\gamma$.
\end{enumerate}
\item Let $\frku_\ell = (M_\ell, N_\ell, \bar{a}_\ell, \bar{b}_\ell, \bar{c}_\ell)\in\wsmall_{(\kaecgoth, \kappa, \gamma)}$, with $\ell\in\{1,2\}$. We say that $\frku_1, \frku_2$ are isomorphic (written $\frku_1\cong \frku_2$) if there is an isomorphism $f: N_1 \to N_2$ such that $f(\bar a_1) = \bar a_2$, $f(\bar b_1) = \bar b_2$, $f(\bar c_1) = \bar c_2$. Similarly, when $\frkw_1, \frkw_2\in\wlarge$ we define $\frkw_1\cong \frkw_2$.
\end{enumerate}
Notice that in \ref{not:synw2} we only require $|M|\leq \kappa$, while in \ref{not:synw1} we require $|M|\leq |N| \leq \kappa$.
\end{notation}

\begin{lemma}
\label{lem:firstsyn} In Context~\ref{cont:syntactic}, we have:
\begin{enumerate}[(1), leftmargin=*]
\item \label{lem:firstsyn1}For $\frku = (M_\frku, N_\frku, \bar a_\frku, \bar b_\frku, \bar c_\frku)\in \wsmall_{(\kaecgoth, \kappa, \gamma)}$ there is
\[
\theta_\frku(\bar z_\gamma, \bar y_\kappa, \bar x_\kappa)\in \inflang_{\nu,\nu}
\]
such that:
\begin{enumerate}[(a), leftmargin=*]
\item for every two isomorphic $\frku_1, \frku_2\in\wsmall_{(\kaecgoth, \kappa, \gamma)}$, then $\theta_{\frku_1} = \theta_{\frku_2}$;
\item if $\frku\in \wsmall_{(\kaecgoth, \kappa,\gamma)}$, then $N_\frku \vDash \theta_\frku(\bar c_\frku, \bar b_\frku, \bar a_\frku)$;
\item \label{lem:firstsyn1c}if $\frku_1, \frku_2\in \wsmall_{(\kaecgoth, \kappa, \gamma)}$ and $\bar a_{\frku_1} = \bar a_{\frku_2}$ (so that $M_{\frku_1} = M_{\frku_2}$), then 
\[
\gtp(\bar c_{\frku_1}/M_{\frku_1}; N_{\frku_1}) = \gtp(\bar c_{\frku_2}/M_{\frku_2}; N_{\frku_2}) \Leftrightarrow \theta_{\frku_1} = \theta_{\frku_2};
\]
\item \label{lem:firstsyn1d}if $\frku\in\wsmall_{(\kaecgoth, \kappa, \gamma)}$, $N\in\kaec$, and $N\vDash \theta_\frku(\bar c, \bar b, \bar a)$, then 
\[
{\frku' =(\bar c, N\restriction \bar a, N\restriction  \bar b, \bar a, \bar b)\in\wsmall_{(\kaecgoth, \kappa, \gamma)}}
\]
and $\theta_\frku = \theta_{\frku'}$.
\end{enumerate}
\item \label{lem:firstsyn2}For $\frkw = (M_\frkw, N_\frkw, \bar a_\frkw, \bar c_\frkw)\in \wlarge_{(\kaecgoth, \kappa, \gamma)}$ there is
\[
\psi_\frkw(\bar z_\gamma, \bar x_\kappa)\in \inflang_{\nu, \nu}
\]
such that:
\begin{enumerate}[(a), leftmargin=*]
\item for every two isomorphic $\frkw_1, \frkw_2\in\wlarge_{(\kaecgoth, \kappa, \gamma)}$, then $\psi_{\frkw_1} = \psi_{\frkw_2}$;
\item \label{lem:firstsyn2b}if $\frkw\in \wlarge_{(\kaecgoth, \kappa,\gamma)}$, then $N_\frkw \vDash \psi_\frkw(\bar c_\frkw, \bar a_\frkw)$;
\item \label{lem:firstsyn2c}if $\frkw_1, \frkw_2\in \wlarge_{(\kaecgoth, \kappa, \gamma)}$ and $\bar a_{\frkw_1} = \bar a_{\frkw_2}$ (so that $M_{\frkw_1} = M_{\frkw_2}$), then 
\[
\gtp(\bar c_{\frkw_1}/M_{\frkw_1}; N_{\frkw_1}) = \gtp(\bar c_{\frkw_2}/M_{\frkw_2}; N_{\frkw_2}) \Leftrightarrow \psi_{\frkw_1} = \psi_{\frkw_2}.
\]
\item \label{lem:firstsyn2d}if $\frkw\in\wlarge_{(\kaecgoth, \kappa, \gamma)}$, $N\in\kaec$, and  $N\vDash \psi_\frkw(\bar c, \bar a)$, then 
\[
{\frkw' =(\bar c, N\restriction \bar a, N, \bar a)\in\wlarge_{(\kaecgoth, \kappa, \gamma)}}
\]
and $\psi_\frkw = \psi_{\frkw'}$.
\end{enumerate}
\end{enumerate}
\end{lemma}

Before starting with the proof, we remark that the statement of Lemma~\ref{lem:firstsyn} slightly differs from the analogous one of \cite[Major Claim~3.3]{paolinishelah}, because we also show \ref{lem:firstsyn1}\ref{lem:firstsyn1d} and \ref{lem:firstsyn2}\ref{lem:firstsyn2d}, as they will be used in the following  to make the arguments of Lemma~\ref{lem:2syn} smoother.
\begin{proof}
First, we show \ref{lem:firstsyn1}. 
\begin{enumerate}[$(*_{\arabic*})$, leftmargin=*, align=left, series=prf:firstsyn]
\item Given $\frku = (M_\frku, N_\frku, \bar a_\frku, \bar b_\frku, \bar c_\frku)\in \wsmall_{(\kaecgoth, \kappa, \gamma)}$, we let $\theta_\frku^0$ be the conjunction of all atomic or negated atomic formulas ${\varphi(\bar z_\gamma\restriction u_3, \bar y_\kappa\restriction u_2, \bar x_\kappa\restriction u_1)}$, where $u_1, u_2$ are finite subsets of $\kappa$ and $u_3$ is a finite subset of $\gamma$, and such that 
\[
{N_\frku \vDash \varphi(\bar c_\frku\restriction u_3, \bar b_\frku\restriction u_2, \bar a_\frku\restriction u_1)}.
\]
\end{enumerate}
\begin{enumerate}[resume*=prf:firstsyn] 
\item Define an equivalence relation $E_{(\kaecgoth, \kappa, \gamma)}^{\mrm{small}}$ on $\wsmall_{(\kaecgoth, \kappa, \gamma)}$ by requiring $\frku_1 E^{\mrm{small}}_{(\kaecgoth, \kappa, \gamma)} \frku_2$ if and only if there is a map $\pi: M_{\frku_1} \to M_{\frku_2}$ with the following properties:
\begin{enumerate}[$(\cdot_1)$, leftmargin=*]
\item $\pi(\bar a_{\frku_1})= \bar a_{\frku_2}$,
\item $\pi$ is an isomorphism from $M_{\frku_1}$ to $M_{\frku_2}$,
\item $\gtp(\bar c_{\frku_2}/M_{\frku_2}; N_{\frku_2}) = \pi(\gtp(\bar c_{\frku_1}/M_{\frku_1}; N_{\frku_1}))$ (cf.~\ref{def:nottypes}\ref{def:nottypes2}).
\end{enumerate}
\end{enumerate}
\begin{enumerate}[resume*=prf:firstsyn] 
\item \label{prf:firstsyn3}If $\frku\in \wsmall_{(\kaecgoth, \kappa, \gamma)}$, then $\Phi_\frku= \{\theta_{\frku_1}^0\mid \frku_1 E^{\mrm{small}}_{(\kaecgoth, \kappa, \gamma)}\frku\}$ is a set of cardinality $\leq 2^\kappa$.
\end{enumerate} Why \ref{prf:firstsyn3}? Obvious.
\begin{enumerate}[resume*=prf:firstsyn] 
\item \label{prf:firstsyn4}$\theta_\frku  = \bigvee\Phi_\frku$ is the required formula. 
\end{enumerate} 
Why \ref{prf:firstsyn4}? We only  elaborate on the proof of \ref{lem:firstsyn1}\ref{lem:firstsyn1d}, so assume that $N, \frku, \frku'$ are as in the statement of \ref{lem:firstsyn1}\ref{lem:firstsyn1d}. Because $\frku_1 E^{\mrm{small}}_{(\kaecgoth, \kappa, \gamma)}\frku$ implies $\theta_{\frku_1} = \theta_\frku$, then without loss of generality we can assume that $N\vDash \theta_{\frku}^0(\bar c_\frku, \bar b_\frku, \bar a_\frku)$. By definition of $\theta_\frku^0$, the map $\pi:N\restriction \bar b_{\frku} \to N\restriction \bar b$ with $\pi(\bar b_\frku) = \bar b$ is an isomorphism, and $\pi(\bar a_\frku) = \bar a$, $\pi(\bar c_\frku) = \bar c$, so that $\frku'\in\wsmall_{(\kaecgoth, \kappa, \gamma)}$ by closure of $\kaecgoth$ under isomorphism. By definition of image of a type (cf.~\ref{def:nottypes}\ref{def:nottypes2}) we have
\[
\gtp(\bar c/M\restriction \bar a; N\restriction\bar b) = \pi(\gtp(\bar c_\frku/M_\frku; N_\frku)), 
\]
so that $\frku' E^{\mrm{small}}_{(\kaecgoth, \kappa, \gamma)} \frku$, and thus $\theta_{\frku} = \theta_{\frku'}$. This finishes the proof of \ref{lem:firstsyn1}.

\ssk\nin Now we show \ref{lem:firstsyn2}.
\begin{enumerate}[resume*=prf:firstsyn]
\item If $\frkw\in \wlarge_{(\kaecgoth, \kappa, \gamma)}$, define $\nbo(\frkw)\subseteq \wlarge_{(\kaecgoth,\kappa, \gamma)}$ as follows:
\begin{enumerate}[$(\cdot)$,leftmargin=*]
\item $\frku=(M_\frku, N_\frku, \bar a_\frku, \bar b_\frku, \bar c_\frku)\in\nbo(\frkw)$ if and only if $\frku\in \wsmall_{(\kaecgoth, \kappa,\gamma)}$, $M_\frku = M_\frkw$, $\bar a_\frku = \bar a_\frkw$, $\bar c_\frku= \bar c_\frkw$, and $N_\frku \leqk N_\frkw$.
\end{enumerate}
\end{enumerate}
\begin{enumerate}[resume*=prf:firstsyn]
\item \label{prf:firstsyn6}$\mrm{nb}(\frkw) \neq \emptyset$ 
\end{enumerate}
Why \ref{prf:firstsyn6}? By coherence and $\mrm{LS}$ at $\kappa$ of $\kaecgoth$.  
\begin{enumerate}[resume*=prf:firstsyn]
\item \label{prf:firstsyn7}If $\frku_1, \frku_2\in \nbo(\frkw)$, then $\theta_{\frku_1} = \theta_{\frku_2}$. 
\end{enumerate}
Why \ref{prf:firstsyn7}? Let $M= M_{\frku_1} = M_{\frku_2}$ and $\bar c = \bar c_{\frku_1} = \bar c_{\frku_2}$. Because $N_{\frku_\ell}\leqk N_{\frkw}$, with $\ell\in\{1,2\}$, we have 
\[
\gtp(\bar c/M; N_{\frku_1}) = \gtp(\bar c/M; N_{\frkw}) = \gtp(\bar c/M; N_{\frku_2}),
\]
and thus, because $\bar a_{\frku_1} =\bar a_{\frkw}= \bar a_{\frku_2}$, from \ref{lem:firstsyn1}\ref{lem:firstsyn1c} it follows $\theta_{\frku_1} = \theta_{\frku_2}$. 
\begin{enumerate}[resume*=prf:firstsyn]
\item Let $\varphi^\star_\kappa(\bar x_\kappa)\in\inflang_{\nu, \nu}$ be the formula which witnesses that $\leqk$ is $(\inflang_{\nu, \nu}, \kappa)\ab$-syntactic on $\kaec$ (cf.~\ref{def:syntacticAC}). 
\end{enumerate}
\begin{enumerate}[resume*=prf:firstsyn]
\item \label{prf:firstsyn9}If $\frkw\in\wlarge_{(\kaecgoth, \kappa, \gamma)}$, let $\frku\in\nbo(\frkw)$, and define $\psi_\frkw(\bar z_\gamma, \bar x_\kappa)$ as the formula:
\[
\exists \bar y_\kappa(\varphi_\kappa^\star(\bar y_\kappa)\land \theta_\frku(\bar z_\gamma, \bar y_\kappa, \bar x_\kappa)),
\]
where $\theta_\frku(\bar z_\gamma, \bar y_\kappa, \bar x_\kappa)$ is as in \ref{lem:firstsyn1}. This formula is well-defined, i.e., it is independent of the choice of $\frku\in\nbo(\frkw)$.
\end{enumerate}
Why \ref{prf:firstsyn9}?  Follows from \ref{prf:firstsyn7}.
\begin{enumerate}[resume*=prf:firstsyn]
\item \label{prf:firstsyn10}$\psi_\frkw(\bar z_\gamma, \bar x_\kappa)$ is the required formula.
\end{enumerate}
Why \ref{prf:firstsyn10}? We only elaborate on \ref{lem:firstsyn2}\ref{lem:firstsyn2d}, so assume that $N, \frkw, \frkw'$ are as in the statement. By \ref{prf:firstsyn9}, $N\vDash \psi(\bar c, \bar a)$ implies that there is $\bar b\in N^\kappa$ such that, for all $\frku\in \nbo(\frkw)$, ${N \vDash \varphi_\kappa^\star(\bar b) \land \theta_\frku(\bar c, \bar b, \bar a)}$.
By \ref{lem:firstsyn1}\ref{lem:firstsyn1d} we have
\[
\frku' = (\bar c, N\restriction \bar a, N\restriction \bar b, \bar a, \bar b)\in \wsmall_{(\kaecgoth, \kappa, \gamma)}.
\]
By definition of $\varphi^\star_\kappa(\bar x_\kappa)$ we have $N\restriction \bar b\leqk N$, so that $\frkw'\in \wlarge_{(\kaecgoth, \kappa, \gamma)}$ and $\frku'\in \nbo(\frkw')$. If $\frku\in \nbo(\frkw)$, then $\theta_{\frku} = \theta_{\frku'}$ by \ref{lem:firstsyn1}\ref{lem:firstsyn1d}, and thus $\psi_{\frkw} = \psi_{\frkw'}$. This finishes the proof of \ref{lem:firstsyn2}.
\end{proof}

\begin{lemma}
\label{lem:2syn}
In Context~\ref{cont:syntactic}, assume $\kaecgoth$ is $(\kappa, \gamma)$-tame and let
\[
\Delta = \Delta(\bar z_\gamma, \bar x_\kappa) = \{\psi_\frkw(\bar z_\gamma, \bar x_\kappa) \mid \frkw\in \wlarge_{(\kaecgoth, \kappa, \gamma)}\},
\]
with $\psi_\frkw(\bar z_\gamma, \bar x_\kappa)$ as in \ref{lem:firstsyn}\ref{lem:firstsyn2}. Then for every infinite cardinal $\lambda\geq \kappa$ the following are equivalent:
\begin{enumerate}[(1), leftmargin=*]
\item $\kaecgoth$ is almost $(\lambda, \gamma)$-stable,
\item $\kaecgoth$ is syntactically almost $(\lambda, \gamma, \Delta)$-stable (cf.~\ref{def:syntacticstable}\ref{def:syntacticstable2}).
\end{enumerate} 
\end{lemma}
\begin{proof}
Let $M,N\in \kaec$ with $|M| = \lambda$ and $M\leqk N$. It is enough to show:
\begin{enumerate}[$(*_1)$, leftmargin=*, series=2syn]
\item \label{prf:2syn1}For every $\bar c_1, \bar c_2\in N^\gamma$, the following are equivalent:
\begin{enumerate}[(a), leftmargin=*]
\item $\gtp(\bar c_1/M; N) = \gtp(\bar c_2/M; N)$;
\item $\tp_\Delta(\bar c_1/M; N) = \tp_\Delta(\bar c_2/M; N)$.
\end{enumerate}
\end{enumerate}
Fix $\bar c_1, \bar c_2\in N^\gamma$. The rest of the proof is concerned with showing \ref{prf:2syn1}.
\begin{enumerate}[resume*=2syn]
\item \label{prf:2syn2}If $\gtp(\bar c_1/M; N) = \gtp(\bar c_2/M; N)$, then $\tp_\Delta(\bar c_1/M; N) = \tp_\Delta(\bar c_2/M; N)$.
\end{enumerate}
Why \ref{prf:2syn2}? Assume $\frkw\in \wlarge_{(\kaecgoth, \kappa, \gamma)}$, $\bar a\in M^\kappa$, and $N \vDash \psi_\frkw(\bar c_1, \bar a)$. We wish to show $N\vDash \psi_\frkw(\bar c_2, \bar a)$. Define, $\frkw_\ell = (\bar c_\ell, N\restriction \bar a, N, \bar a)\in \wlarge_{(\kaecgoth, \kappa, \gamma)}$, for $\ell\in\{1,2\}$. By Lemma~\ref{lem:firstsyn}\ref{lem:firstsyn2}\ref{lem:firstsyn2d} we can assume $\frkw = \frkw_1$. From the hypothesis of \ref{prf:2syn2} and  Lemma~\ref{lem:firstsyn}\ref{lem:firstsyn2}\ref{lem:firstsyn2c} we have $\psi_{\frkw}=\psi_{\frkw_1} = \psi_{\frkw_2}$, and thus $N \vDash \psi_\frkw(\bar c_2, \bar a)$ by Lemma~\ref{lem:firstsyn}\ref{lem:firstsyn2}\ref{lem:firstsyn2b}.
\begin{enumerate}[resume*=2syn]
\item \label{prf:2syn3}If $\tp_\Delta(\bar c_1/M; N) = \tp_\Delta(\bar c_2/M; N)$, then $\gtp(\bar c_1/M; N) = \gtp(\bar c_2/M; N)$.
\end{enumerate}
Why \ref{prf:2syn3}? By $(\kappa, \gamma)$-tameness of $\kaecgoth$ it is enough to show the following:
\begin{enumerate}[resume*=2syn]
\item \label{prf:2syn4}If $\tp_\Delta(\bar c_1/M; N) = \tp_\Delta(\bar c_2/M; N)$, then $\gtp(\bar c_1/M'; N) = \gtp(\bar c_2/M'; N)$ for every $M'\leqk N$ with $|M'| \leq \kappa$.
\end{enumerate}
Why \ref{prf:2syn4}? Let $M'$ be as an in the statement, and  $\bar a\in (M')^\kappa$ an enumeration of $M'$ (possibly with repetitions). Define ${\frkw_\ell = (\bar c_\ell, M', N, \bar a)\in \wlarge_{(\kaecgoth, \kappa, \gamma)}}$, with $\ell\in\{1,2\}$. By Lemma~\ref{lem:firstsyn}\ref{lem:firstsyn2}\ref{lem:firstsyn2b} we have $N\vDash \psi_{\frkw_2}(\bar c_2, a)$. By the hypothesis of  \ref{prf:2syn4} we have {$N \vDash \psi_{\frkw_2}(\bar c_1,\bar a)$}, so that by Lemma~\ref{lem:firstsyn}\ref{lem:firstsyn2}\ref{lem:firstsyn2d} we have $\psi_{\frkw_1} = \psi_{\frkw_2}$. Finally,  Lemma~\ref{lem:firstsyn}\ref{lem:firstsyn2}\ref{lem:firstsyn2c} gives us  \ref{prf:2syn4}.

\ssk\nin Therefore, \ref{prf:2syn2} and \ref{prf:2syn3} give us \ref{prf:2syn1}, which concludes the proof.
\end{proof}

To finish the proof of Theorem~\ref{thm:syn1}, we will crucially use the  result on the elimination of quantifiers for infinitary formulas in the language of $R$-modules proved by Shelah in \cite{shelah}. We state it in the following weaker form, which will suffice for our purposes.
\begin{fact}
\label{fact:elqfmod}
Let $\nu$ be an infinite regular cardinal, $\bar x$ a tuple of free variables of length less than $\nu$, and $M$ an $R$-module. Then there is $\mathbf{I}\subseteq M^{|\bar x|}$  and a set of formulas $\Lambda\subseteq \inflang_{\infty,\nu}$ with:
\begin{enumerate}[(a), leftmargin=*]
\item \label{fact:elqfmoda}$\Lambda= \Lambda(\bar x)$;
\item \label{fact:elqfmodb}if $\varphi(\bar x)\in \Lambda$, then $\varphi(M^{|\bar x|})=\{\bar a\in M^{|\bar x|} \mid M \vDash \varphi(\bar a)\}$ is a subgroup of $M^{|\bar x|}$;
\item \label{fact:elqfmodc}$|\Lambda| \leq \beth_\nu(|R|^{<\nu})$;
\item \label{fact:elqfmodd} $|\mathbf{I} |\leq \beth_\nu(|R|^{<\nu})$;
\item \label{fact:elqfmode}every $\psi(\bar x) \in \inflang_{\infty, \nu}$ with $\mrm{qd}_\nu(\psi(\bar x))< \nu$ (cf.~\ref{not:inflogic}\ref{not:inflogicqd}) is equivalent in $M$ to an infinitary boolean combination of formulas of the form $\varphi(\bar x - \bar a)$ with $\varphi(\bar x) \in \Lambda$ and $\bar a\in \mathbf{I}$.
\end{enumerate}
In particular, \ref{fact:elqfmode} holds for every $\psi(\bar x)\in \inflang_{\nu,\nu}$.
\end{fact}
\begin{proof}
The existence of $\Lambda$ and $\mathbf{I}$ which satisfy all of the conditions, with possibly the exception of \ref{fact:elqfmodb}-\ref{fact:elqfmodc}, is by  \cite[Theorem~2.4]{shelah}, and that they satisfy \ref{fact:elqfmodb}-\ref{fact:elqfmodc} follows from  \cite[Claim~2.3]{shelah}. The last sentence is true because, being $\nu$ regular, it is easily seen that every formula in $\inflang_{\nu, \nu}$ has $\mrm{qd}_\nu(\psi)<\nu$.
\end{proof}

The argument we use in the proof of the next theorem is an adaptation to our context of the proof of the stability of first-order theories of modules \cite[Theorem~2.1(1)]{ziegler}. This method of proof has been adapted to show the syntactic stability of types of infinitary pp-formulas in  \cite[Theorem~3.3(1)]{shelah}. This latter result, alongside the elimination of quantifiers of the previous fact, could be used to make the proof of the following theorem slightly shorter. Nonetheless, we provide all the details for completeness.

\begin{theorem}
\label{thm:synfund}
In Context~\ref{cont:syntactic}, assume $\kaecgoth = (\kaec, \leqk)$ is an $\mrm{AC}$ of $R$-modules, let $\xi_\nu = \beth_\nu(|R|^{<\nu})$ and $\lambda = \lambda^{\xi_\nu}$. If $\kaecgoth$ is $(\kappa,\gamma)$-tame, then it is almost $(\lambda, \gamma)$-stable.
\end{theorem}
\begin{proof}
Let $\Delta=\Delta(\bar z_\gamma, \bar x_\kappa)$ be as in Lemma~\ref{lem:2syn}. It  is enough to show that for every $M, N\in \kaec$ with $|M| = \lambda$ and $M\leqk N$, then 
\begin{equation}
\label{eq:syn1}
|\{\tp_\Delta(\bar c/M; N)\mid \bar c\in N^\gamma\}| \leq \lambda.
\end{equation}  
From now on, fix $M$ and $N$ as before. We wish to show  that \eqref{eq:syn1} holds.
\begin{enumerate}[$(*_1)$, leftmargin=*, series=prf:thmsyn]
\item \label{prf:thmsyn1}There are  $M\subseteq A\subseteq N$, an ordinal $\eta$, and a set of formulas $\Gamma$, such that the following hold:
\begin{enumerate}[(a), leftmargin=*]
\item $|A| +|\Gamma|\leq \xi_\nu = \beth_\nu(|R|^{<\nu})$;
\item $\eta = \kappa + \gamma + \kappa$, so that in particular $\eta<\kappa^+$;
\item $\Gamma = \Gamma(\bar z_\gamma, \bar w_\eta)$;
\item \label{prf:thmsyn1d}if $\varphi(\bar z_\gamma, \bar w_\eta)\in \Gamma$, then $\varphi(N^{\gamma + \eta})=\{\bar a\in N^{\gamma + \eta} \mid N \vDash \varphi(\bar a)\}$ is a subgroup of $N^{\gamma + \eta}$;
\item \label{prf:thmsyn1e}if $\psi(\bar z_\gamma, \bar x_\kappa)\in \Delta$ and $\bar a_\kappa\in A^\kappa$, then $\psi(\bar z_\gamma, \bar a_\kappa)$ is equivalent in $N$ to an infinitary boolean combination of formulas of the form $\varphi(\bar z_\gamma, \bar d_\eta)$ with $\varphi(\bar z_\gamma, \bar w_\eta)\in \Gamma$, $\bar d_\eta\in A^\eta$.
\end{enumerate}
\end{enumerate}
Why \ref{prf:thmsyn1}? Apply Fact~\ref{fact:elqfmod} to $N$ and find $\mathbf{I}$ and $\Lambda$ as in the statement. Let $A$ be the union of $M$ with the elements appearing in the sequences of $\mathbf{I}$. Let
\[
\Gamma = \Gamma(\bar z_\gamma, \bar w_\eta) = \Gamma(\bar z_\gamma,\bar x_\kappa^\frown \bar u_\gamma^\frown \bar v_\kappa)  = \{\varphi(\bar z_\gamma - \bar u_\gamma, \bar x_\kappa - \bar v_\kappa)\mid \varphi(\bar z_\gamma, \bar x_\kappa)\in \Lambda \},
\]
then it is easily verified that $A$ and $\Gamma$ are as required.
\begin{enumerate}[resume*=prf:thmsyn]
\item \label{prf:thmsyn2}$|\{\tp_\Gamma(\bar c/A; N)\mid \bar c\in N^\gamma \}|\leq \lambda$.
\end{enumerate}
Why \ref{prf:thmsyn2}? For every $\bar c\in N^\gamma$, we wish to define a partial function $F_{\bar c}: \Gamma \rightharpoonup A^\eta$. Let $\varphi\in \Gamma$, we define $F_{\bar c}(\varphi)$ in the following way:
\begin{enumerate}[$(\cdot_1)$, leftmargin=*]
\item $F_{\bar c}(\varphi)$ is not defined if there is no $\bar d \in A^\eta$ such that $N \vDash \varphi(\bar c, \bar d)$,
\item otherwise, define $F_{\bar c}(\varphi)$ by choosing $\bar d \in A^\eta$ such that $N \vDash \varphi(\bar c, \bar d)$.
\end{enumerate}
\begin{enumerate}[$(*_{{2}.1})$, leftmargin=*]
\item \label{prf:thmsyn2.1}If $\bar c_\ell\in N^\gamma$, with $\ell\in\{1,2\}$, and $F_{\bar c_1} = F_{\bar c_2}$, then $\tp_\Gamma(\bar c_1/A; N) = \tp_\Gamma(\bar c_2/A; N)$.
\end{enumerate}
Why \ref{prf:thmsyn2.1}? Assume $\varphi\in \Gamma$, $\bar d\in A^\eta$, and $N \vDash \varphi(\bar c_1, \bar d)$. Thus $F_{\bar c_1}(\varphi)$ is defined, and because $F_{\bar c_1} =F_{\bar c_2}$, then $F_{\bar c_2}(\varphi)$ is defined and we let $\bar e \in A^\eta$ be such that $F_{\bar c_1}(\varphi) = F_{\bar c_2}(\varphi) = \bar e$. Therefore, by  \ref{prf:thmsyn1}\ref{prf:thmsyn1d} we have 
\[
\bar c_2^\frown\bar d =\bar c_1^\frown\bar d - \bar c_1^\frown\bar e  + \bar c_2^\frown\bar e\in \varphi(N^{\gamma + \eta}).
\]
We have just shown $\tp_\Gamma(\bar c_1/A; N) \subseteq \tp_\Gamma(\bar c_2/A; N)$, and by symmetry we get \ref{prf:thmsyn2.1}.

\ssk\nin Coming back to the proof of \ref{prf:thmsyn2}, we have that, for $\bar c\in N^\gamma$, then $\tp_\Gamma(\bar c/A;N)$ is uniquely determined by the values of the partial function $F_{\bar c}: \Gamma \rightharpoonup A^\eta$. Therefore, we get:
\[
|\{\tp_\Gamma(\bar c/A; N)\mid \bar c\in N^\gamma \}| \leq (|A^\eta|)^{|\Gamma|}\leq \lambda^{\kappa + |\Gamma|} = \lambda^{\xi_\nu} = \lambda.
\]
\begin{enumerate}[resume*=prf:thmsyn]
\item\label{prf:thmsyn3} If $\bar c_\ell\in N^\gamma$, with $\ell\in\{1,2\}$, and $\tp_\Gamma(\bar c_1/A; N) = \tp_\Gamma(\bar c_2/A; N)$, then 
\[
\tp_\Delta(\bar c_1/M; N) = \tp_\Delta(\bar c_2/M; N).
\]
\end{enumerate}
Why \ref{prf:thmsyn3}? Assume $N\vDash \psi(\bar c_1, \bar a_\kappa)$ with $\psi(\bar z_\gamma, \bar x_\kappa)\in \Delta$ and $\bar a_\kappa\in M^\kappa$. By \ref{prf:thmsyn1}\ref{prf:thmsyn1e} there are $\alpha(\bar z_\gamma, \bar w')$ and $\bar d'\in A^{|\bar w'|}$ such that:
\begin{enumerate}[$(\cdot_{\text{\alph*}})$, leftmargin=*]
\item\label{prf:thmsyn3.1} $\alpha(\bar z_\gamma, \bar d')$ is a boolean combination of formulas of the form $\varphi(\bar z_\gamma, \bar d_\eta)$ with  $\varphi(\bar z_\gamma, \bar w_\eta)\in\Gamma$, $\bar{d}_\eta\in A^\eta$;
\item\label{prf:thmsyn3.2} $ N \vDash \forall \bar z_\gamma (\psi(\bar z_\gamma, \bar a_\kappa)\leftrightarrow \alpha(\bar z_\gamma, \bar d'))$.
\end{enumerate}
By \ref{prf:thmsyn3.2}, we have $N \vDash \alpha(\bar c_1, \bar d')$. Therefore, by the hypothesis of \ref{prf:thmsyn3} and \ref{prf:thmsyn3.1} we have $N \vDash \alpha(\bar c_2, \bar d')$, and thus $N \vDash \psi(\bar c_2, \bar a_\kappa)$, again by \ref{prf:thmsyn3.2}. Therefore, $\tp_\Delta(\bar c_1/M; N) \subseteq \tp_\Delta(\bar c_2/M; N)$, and by symmetry we get \ref{prf:thmsyn3}.
\begin{enumerate}[resume*=prf:thmsyn]
\item\label{prf:thmsyn4} \eqref{eq:syn1} holds.
\end{enumerate}
Why \ref{prf:thmsyn4}? Follows from \ref{prf:thmsyn2} and \ref{prf:thmsyn3}. 
\end{proof}

We can finally give the following:
\begin{proof}[Proof of Theorem~\ref{thm:syn1}]
Follows immediately from Remark~\ref{rem:prfthmhyp} and Theorem~\ref{thm:synfund}.
\end{proof}

\section{Independence relations and pre-cellular categories}
\label{sect:precellular}

To show that the $\mu\mh\mrm{AEC}$s which we will consider in Section~\ref{sect:lpure} are stable, we will show that they carry a categorical stable independence relation. Such relations were introduced in a category-theoretic setting in \cite{forkingcategorical} as a generalization of model-theoretic forking independence, which Shelah had introduced for the study of first-order theories \cite{shelah1990classification}. A categorical independence relation $\dnf$ is formally a collection of commutative squares but, for $\mu\mh\mrm{AEC}$s, this allows us to define a closely related relation $\dnfb{}{}{}{}$ resembling the forking independence relation from first-order model theory (Definition~\ref{def:indrelmuaec}\ref{def:indrelmuaec2}), as well as a notion of non-forking for orbital types (Definition~\ref{def:dnftype}).

Let us now describe the contents of this section. To equip the $\mu\mh\mrm{AEC}$s of Section~\ref{sect:lpure} with an independence relation, we show that they are coherent pre-cellular categories, which we introduce in this section. Cellular categories were introduced in \cite{cellularcategories}: they are pairs $(\categ, \morphs)$, where $\categ$ is a cocomplete category and $\morphs$ is a collection of morphisms in $\categ$ that contains all isomorphisms and is closed under transfinite composition and pushouts. Independence relations on cellular categories were studied in \cite{cellularstable}, where it was shown that every coherent cellular category carries a natural weakly stable independence relation.
In this paper we introduce pre-cellular categories, a weakening of cellular categories that is nonetheless strong enough for some of the arguments of \cite{cellularstable} to go through. Although Definition~\ref{def:precellularcat} is new, it had already been observed (see, e.g., \cite[Fact~3.12]{somestablenonelementary}) that a similar set of assumptions suffices to yield an independence relation.

We fix the notation used in this section and recall some standard definitions from category theory.

\begin{notation}
Let $\categn$ be a category.
\begin{enumerate}[(1), leftmargin=*]
\item We say that a pair of morphisms $(f_1, f_2)$ in $\categn$ is a \tdef{span} if $\mrm{dom}(f_1) = \mrm{dom}(f_2)$. Dually, a \tdef{cospan} in $\categn$ is a pair of morphisms $(g_1, g_2)$ in $\categn$ with $\mrm{ran}(g_1) = \mrm{ran}(g_2)$.
\item We say that a quadruple of morphisms $(f_1, f_2, g_1, g_2)$ in $\categn$ is a \tdef{commutative square} if the following diagram is commutative: 
\[\begin{tikzcd}
	B & D \\
	A & C
	\arrow["{g_1}", from=1-1, to=1-2]
	\arrow["{f_1}", from=2-1, to=1-1]
	\arrow["{f_2}"', from=2-1, to=2-2]
	\arrow["{g_2}"', from=2-2, to=1-2]
\end{tikzcd}\]
\end{enumerate}
\end{notation}

\begin{definition}[\protect{\cite[Chapter~11]{joyofcats}}]
Let $\categn$ be a category.
We say that a cospan $(g_1: B \to P, g_2: C \to P)$ in $\categn$ is a \tdef{pushout} of the span $(f_1: A \to B, f_2: A \to C)$ if $g_1 f_1 = g_2f_2$ and for every cospan $(h_1: B\to D, h_2: C \to D)$ with $h_1f_1 = h_2f_2$ there is a unique morphism $t: P \to D$ with $h_\ell = tg_\ell$. In such a case we will say that $(f_1, f_2, g_1, g_2)$ is a \tdef{pushout square}, and use the notation:
\[\begin{tikzcd}
	B & P \\
	A & C
	\arrow["{g_1}", from=1-1, to=1-2]
	\arrow["\ulcorner"{anchor=center, pos=0.125, rotate=-90}, draw=none, from=1-2, to=2-1]
	\arrow["{f_1}", from=2-1, to=1-1]
	\arrow["{f_2}"', from=2-1, to=2-2]
	\arrow["{g_2}"', from=2-2, to=1-2]
\end{tikzcd}\]
Sometimes, when the maps $(g_1, g_2)$ are clear by the context, we will refer to $P$ simply as the pushout of $(f_1, f_2)$. As the pushout is unique up to isomorphism, we will also say that $(g_1, g_2)$ is \tdef{the pushout of } $(f_1, f_2)$.
\end{definition}

\begin{definition}[\protect{\cite[Definition 3.1]{mazarinoetherian}}]
\label{def:indrelation}
Let $\categn$ be a category. An \tdef{independence relation} on $\categn$ is a collection $\dnf$ of commutative squares such that for any commutative diagram:
\[\begin{tikzcd}
	&& E \\
	B & D \\
	A & C
	\arrow["{h_1}", curve={height=-12pt}, from=2-1, to=1-3]
	\arrow["{g_1}", from=2-1, to=2-2]
	\arrow["t", from=2-2, to=1-3]
	\arrow["{f_1}", from=3-1, to=2-1]
	\arrow["{f_2}"', from=3-1, to=3-2]
	\arrow["{h_2}"', curve={height=12pt}, from=3-2, to=1-3]
	\arrow["{g_2}"', from=3-2, to=2-2]
\end{tikzcd}\]
we have $(f_1, f_2, g_1, g_2)\in \dnf$ if and only if $(f_1, f_2, h_1, h_2)\in \dnf$.
\end{definition}
\begin{notation}
If $(f_1, f_2, g_1, g_2)\in\dnf$, we will use the notation:
\[\begin{tikzcd}
	B & D \\
	A & C
	\arrow["{g_1}", from=1-1, to=1-2]
	\arrow["{f_1}", from=2-1, to=1-1]
	\arrow["\dnf"{description}, draw=none, from=2-1, to=1-2]
	\arrow["{f_2}"', from=2-1, to=2-2]
	\arrow["{g_2}"', from=2-2, to=1-2]
\end{tikzcd}\]
\end{notation}
\begin{definition}
\label{def:eqcommsq}
Let $\categn$ be a category and $(f_1, f_2)$ a span.
\begin{enumerate}[(1), leftmargin=*]
\item \label{def:eqcommsq1} If $(f_1, f_2, g_1^a, g_2^a)$ and $(f_1, f_2, g_1^b, g_2^b)$ are commutative squares, we say that they are \tdef{equivalent}, written $(f_1, f_2, g_1^a, g_2^a)\sim^*(f_1, f_2, g_1^b, g_2^b)$, if there are morphisms $g^a$ and $g^b$ such that the following diagram is commutative:
\[\begin{tikzcd}
	& {D^b} & D \\
	B && {D^a} \\
	A & C
	\arrow["{g^b}", dashed, from=1-2, to=1-3]
	\arrow["{g_1^b}", from=2-1, to=1-2]
	\arrow["{g_1^a}"{description, pos=0.8}, from=2-1, to=2-3]
	\arrow["{g^a}"', dashed, from=2-3, to=1-3]
	\arrow["{f_1}", from=3-1, to=2-1]
	\arrow["{f_2}"', from=3-1, to=3-2]
	\arrow["{g_2^b}"{description, pos=0.8}, from=3-2, to=1-2]
	\arrow["{g_2^a}"', from=3-2, to=2-3]
\end{tikzcd}\]
\item \label{def:eqcommsq2} We let $\sim$ denote the transitive closure of $\sim^*$.
\end{enumerate}
\end{definition}

\begin{remark}
Originally, independence relations were defined in \cite[Definition 3.4]{forkingcategorical} as a collection of commutative squares closed under $\sim$. By {\cite[Remark~2.6(1)]{cellularstable}} this is equivalent to  Definition~\ref{def:indrelation}.
\end{remark}

\begin{definition}
\label{def:propindep}
Let $\categn$ be a category and $\dnf$ an independence relation on $\categn$.
\begin{enumerate}[(1), leftmargin=*]
\item \label{def:propindep1}\cite[Definition 3.9]{forkingcategorical} We say that $\dnf$ is \tdef{symmetric} if $(f_1, f_2, g_1, g_2)\in\dnf$ if and only if $(f_2, f_1, g_2, g_1)\in\dnf$.
\item \label{def:propindep2}\cite[Definition 3.15]{forkingcategorical} We say that $\dnf$ is \tdef{right transitive} if whenever we have a commutative diagram of the following form:
\[\begin{tikzcd}
	B & D & F \\
	A & C & E
	\arrow["{g_1}", from=1-1, to=1-2]
	\arrow["{h_1}", from=1-2, to=1-3]
	\arrow["{f_1}", from=2-1, to=1-1]
	\arrow["{f_2}"', from=2-1, to=2-2]
	\arrow["k"', from=2-2, to=1-2]
	\arrow["{g_2}"', from=2-2, to=2-3]
	\arrow["{h_2}"', from=2-3, to=1-3]
\end{tikzcd}\]
with $(f_1, f_2, g_1, k)\in\dnf$ and $(k, g_2, h_1, h_2)\in\dnf$, then $(f_1, g_2\circ f_2, h_1\circ g_1, h_2)\in\dnf$.
\item \label{def:propindep3}\cite[Definition 3.10]{forkingcategorical}  We say that $\dnf$ has the \tdef{existence property} if for any span $(f_1, f_2)$ there is $(f_1, f_2, g_1, g_2)\in\dnf$. 
\item \label{def:propindep4}\cite[Definition 3.13]{forkingcategorical} We say that $\dnf$ has the \tdef{uniqueness property} if whenever $(f_1, f_2, g_1^a, g_2^a)\in \dnf$ and $(f_1, f_2, g_1^b, g_2^b)\in\dnf$, then $(f_1, f_2, g_1^a, g_2^a)\sim(f_1, f_2, g_1^b, g_2^b)$.
\end{enumerate}
\end{definition}

\begin{definition}[\protect{\cite[Remark 2.6]{cellularstable}}]
\label{def:weaklystableind}
We say that $\dnf$ is  \tdef{weakly stable} if it is  symmetric, right transitive, and has the existence and the uniqueness property.
\end{definition}

As already remarked in the introduction of the section, the following definition generalizes that of cellular categories  \cite{cellularcategories}. In particular, one important feature is that we will not require completeness and closure under arbitrary pushouts. We will explain in more detail the difference between cellular and pre-cellular categories in the remark following Definition~\ref{def:cellularsquare}.  

\begin{definition}
\label{def:precellularcat}
Let $\categ$ be a category and $\morphs$ a collection of morphisms in $\categ$. We say that $(\categ, \morphs)$ is a \tdef{pre-cellular category} if the following holds:
\begin{enumerate}[(1), leftmargin=*]
\item \label{def:precellularcat1} If $(f_1, f_2)$ is a span in $\categ$ with $f_1,f_2\in \morphs$, then $(f_1, f_2)$ has a pushout in $\categ$.
\item \label{def:precellularcat2} $\morphs$ contains all isomorphisms.
\item \label{def:precellularcat3}
$\morphs$ is closed under pushouts, i.e., if $(f_1, f_2, g_1, g_2)$ is a pushout square with $f_1, f_2\in\morphs$, then $g_1, g_2\in\morphs$.
\item \label{def:precellularcat4} $\morphs$ is closed under finite compositions, i.e., whenever $f,g\in\morphs$ and $gf$ is defined in $\categ$, then $gf\in \morphs$.
\end{enumerate}
We write $\categ_\morphs$ for the subcategory of $\categ$ with the same objects as $\categ$ and morphisms those in $\morphs$.
\end{definition}

\begin{definition}
\label{def:coherentprecell}
We say that a pre-cellular category $(\categ, \morphs)$ is \tdef{coherent} if for any two composable morphisms $f,g$ in $\categ$  with $gf\in\morphs$ and $g\in\morphs$, we have that $f\in\morphs$.
\end{definition}

We now introduce the pre-cellular squares. Cellular squares were introduced in \cite[Definition 2.2]{cellularstable} for cellular categories, and we adapt their definition to our context. In case the category $\categ$ is cocomplete, what we here call pre-cellular squares are also called \tdef{$\morphs$-effective squares} \cite[Definition~5.2]{cellulargeneration}.

\begin{definition}
\label{def:cellularsquare}
Let $(\categ, \morphs)$ be a pre-cellular category. We say that a commutative square $(f_1, f_2, g_1, g_2)$ in $\categ_\morphs$ (cf.~\ref{def:precellularcat}) is \tdef{pre-cellular} if the unique morphism $t: P \to D$ from the pushout of $(f_1, f_2)$ belongs to $\morphs$.
\[\begin{tikzcd}
	&& D \\
	B & P \\
	A & C
	\arrow["{g_1}", curve={height=-12pt}, from=2-1, to=1-3]
	\arrow[from=2-1, to=2-2]
	\arrow["r", from=2-2, to=1-3]
	\arrow["\ulcorner"{anchor=center, pos=0.125, rotate=-90}, draw=none, from=2-2, to=3-1]
	\arrow["{f_1}", from=3-1, to=2-1]
	\arrow["{f_2}"', from=3-1, to=3-2]
	\arrow["{g_2}"', curve={height=12pt}, from=3-2, to=1-3]
	\arrow[from=3-2, to=2-2]
\end{tikzcd}\]
Notice that the pushout of $(f_1, f_2)$ exists by \ref{def:precellularcat}\ref{def:precellularcat1}, and the definition we have given does not depend on the particular choice of pushout by~\ref{def:precellularcat}\ref{def:precellularcat2} and \ref{def:precellularcat4}.
\end{definition}

We now explain the main differences between cellular categories and the pre-cellular categories studied here.
\tdef{Cellular categories} are pairs $(\categ, \morphs)$ as in \ref{def:precellularcat}, but requiring that $\categ$ is cocomplete, and $\morphs$ is closed under transfinite compositions and pushouts in the stronger sense that if $(f_1, f_2, g_1, g_2)$ is a pushout square in $\categ$ with $f_1\in\morphs$, then $g_2\in\morphs$. 
 In \cite[Definition 2.2]{cellularstable} \tdef{cellular squares} are defined over cellular categories exactly as in Definition~\ref{def:cellularsquare}, but without requiring that any of the morphisms $f_1, f_2,g_1, g_2$ are in $\morphs$. We have to make this assumption because we have to ensure that the pushout of $(f_1, f_2)$ exists.

It is shown in \cite[Theorem 2.7]{cellularstable} that cellular squares over a coherent cellular category $(\categ, \morphs)$ form a weakly stable independence relation on $\categ_\morphs$ (cf.~\ref{def:precellularcat}). The proof goes through to show that pre-cellular squares also form a weakly stable independence relation (cf.~\ref{def:weaklystableind}) on $\categ_\morphs$ when $(\categ, \morphs)$ is a coherent pre-cellular category. We provide the details of the proof for completeness.

\begin{proposition}
\label{prop:indepcellular}
If $(\categ, \morphs)$ is a coherent pre-cellular category, then the pre-cellular squares form a weakly stable independence relation on the category $\categ_\morphs$ (cf.~\ref{def:precellularcat}).
\end{proposition}
\begin{proof}
\nin We now verify all of the axioms of a weakly stable independence relation. Notice that, by the definition of pre-cellular squares, every such square is in $\categ_\morphs$.
\nin\begin{enumerate}[$(*_1)$, leftmargin=*, series=prf:indrel]
\item \label{prf:indrel1} $\dnf$ is an independence relation on $\categ_\morphs$ (cf.~\ref{def:indrelation}).
\end{enumerate}
Why \ref{prf:indrel1}? Assume we have in $\categ_\morphs$ a commutative diagram of the form:
\[\begin{tikzcd}
	&& E \\
	B & D \\
	A & C
	\arrow["{h_1}", curve={height=-12pt}, from=2-1, to=1-3]
	\arrow["{g_1}", from=2-1, to=2-2]
	\arrow["t", from=2-2, to=1-3]
	\arrow["{f_1}", from=3-1, to=2-1]
	\arrow["{f_2}"', from=3-1, to=3-2]
	\arrow["{h_2}"', curve={height=12pt}, from=3-2, to=1-3]
	\arrow["{g_2}"', from=3-2, to=2-2]
\end{tikzcd}\]
Let $P$ be the pushout of $(f_1, f_2)$ in $\categ$; this exists because $f_1, f_2\in \morphs$ and by \ref{def:precellularcat}\ref{def:precellularcat1}. Let $r: P \to D$, $s: P \to E$ be the maps induced from the pushout. If $(f_1, f_2, g_1, g_2)\in \dnf$, then by definition $r\in\morphs$, so that $tr\in \morphs$ by closure under finite compositions. But then $s = tr\in\morphs$ by uniqueness of the map from the pushout, and thus $(f_1, f_2, h_1, h_2)\in\dnf$. Assume $(f_1, f_2, h_1, h_2)\in \dnf$, then $tr = s\in\morphs$ by hypothesis, and because $t\in\morphs$ we get $r\in\morphs$ by coherence.
\nin\begin{enumerate}[resume*=prf:indrel]
\item \label{prf:indrel2} $\dnf$ is symmetric.
\end{enumerate}
\ref{prf:indrel2} follows easily from the definition of pre-cellular squares (cf.~\ref{def:cellularsquare}).
\nin\begin{enumerate}[resume*=prf:indrel]
\item \label{prf:indrel3} $\dnf$ is right transitive.
\end{enumerate}
Why \ref{prf:indrel3}? Assume we have in $\categ_\morphs$ a commutative diagram of the form
\[\begin{tikzcd}
	B & D & F \\
	A & C & E
	\arrow["{g_1}", from=1-1, to=1-2]
	\arrow["{h_1}", from=1-2, to=1-3]
	\arrow["{f_1}", from=2-1, to=1-1]
	\arrow["{f_2}"', from=2-1, to=2-2]
	\arrow["k"', from=2-2, to=1-2]
	\arrow["{g_2}"', from=2-2, to=2-3]
	\arrow["{h_2}"', from=2-3, to=1-3]
\end{tikzcd}\]
with $(f_1, f_2, g_1, k)\in\dnf$ and $(k, g_2, h_1, h_2)\in\dnf$. We want to show:
\[
{(f_1, g_2 f_2, h_1g_1, h_2)\in\dnf}.
\] 
By closure under finite compositions of $\morphs$, we have that ${f_1, g_2f_2, h_1g_1, h_2\in\morphs}$. Therefore, we are left to show that the induced map $r: P \to F$ is in $\morphs$, where $P$ is the pushout of $(f_1, g_2f_2)$, which exists because $f_1, g_2f_2\in\morphs$.
\[\begin{tikzcd}
	&& F \\
	B & P \\
	A & E
	\arrow["{h_1g_1}", curve={height=-12pt}, from=2-1, to=1-3]
	\arrow[from=2-1, to=2-2]
	\arrow["r", from=2-2, to=1-3]
	\arrow["\ulcorner"{anchor=center, pos=0.125, rotate=-90}, draw=none, from=2-2, to=3-1]
	\arrow["{f_1}", from=3-1, to=2-1]
	\arrow["{g_2f_2}"', from=3-1, to=3-2]
	\arrow["{h_2}"', curve={height=12pt}, from=3-2, to=1-3]
	\arrow[from=3-2, to=2-2]
\end{tikzcd}\]
We have $f_1, f_2\in\morphs$, so that we can consider the pushout square $(f_1, f_2, u, v)$, and by \ref{def:precellularcat}\ref{def:precellularcat3} we have in particular $v\in \morphs$. By \ref{def:precellularcat}\ref{def:precellularcat1}-\ref{def:precellularcat3} there is a pushout square $(v, g_2, u', v')$ with $u',v'\in\morphs$. Therefore, the pushout in the above diagram is the composition of the following pushout squares
\[\begin{tikzcd}
	B & Q & P \\
	A & C & E
	\arrow["u", from=1-1, to=1-2]
	\arrow["{u'}", from=1-2, to=1-3]
	\arrow["\ulcorner"{anchor=center, pos=0.125, rotate=-90}, draw=none, from=1-2, to=2-1]
	\arrow["\ulcorner"{anchor=center, pos=0.125, rotate=-90}, draw=none, from=1-3, to=2-2]
	\arrow["{f_1}", from=2-1, to=1-1]
	\arrow["{f_2}"', from=2-1, to=2-2]
	\arrow["v"', from=2-2, to=1-2]
	\arrow["{g_2}"', from=2-2, to=2-3]
	\arrow["{v'}"', from=2-3, to=1-3]
\end{tikzcd}\]
where all the maps are in $\morphs$. Because $(f_1, f_2, g_1, k)\in\dnf$, then the induced map $s: Q \to D$ is in $\morphs$. 
Because $u'\in \morphs$, by \ref{def:precellularcat}\ref{def:precellularcat1} we can consider the pushout square
\[\begin{tikzcd}
	D & {P'} \\
	Q & P
	\arrow["{\bar u}", from=1-1, to=1-2]
	\arrow["\ulcorner"{anchor=center, pos=0.125, rotate=-90}, draw=none, from=1-2, to=2-1]
	\arrow["s", from=2-1, to=1-1]
	\arrow["{u'}"', from=2-1, to=2-2]
	\arrow["{\bar s}"', from=2-2, to=1-2]
\end{tikzcd}\]
By \ref{def:precellularcat}\ref{def:precellularcat3} we have in particular $\bar s\in \morphs$. Composing this pushout with the pushout square at the right-hand side of the diagram above it, we get the following pushout diagram:
\[\begin{tikzcd}
	D & {P'} \\
	Q & P \\
	C & E
	\arrow["{\bar u}", from=1-1, to=1-2]
	\arrow["\ulcorner"{anchor=center, pos=0.125, rotate=-90}, draw=none, from=1-2, to=3-1]
	\arrow["s", from=2-1, to=1-1]
	\arrow["{\bar s}"', from=2-2, to=1-2]
	\arrow["v", from=3-1, to=2-1]
	\arrow["{g_2}"', from=3-1, to=3-2]
	\arrow["{v'}"', from=3-2, to=2-2]
\end{tikzcd}\]
Notice that it follows easily from the definition of $s$ that one has $k = sv$, so that using $(k, g_2, h_1, h_2)\in\dnf$ we have that the induced morphism $r': P' \to F$ is in $\morphs$. Since $r = r' \bar s$, by closure under finite compositions of $\morphs$ we get $r\in \morphs$, which was what we wanted to show.
\nin\begin{enumerate}[resume*=prf:indrel]
\item \label{prf:indrel4} $\dnf$ has the existence property.
\end{enumerate}
Why \ref{prf:indrel4}? Let $(f_1, f_2)$ be a span in $\categ_\morphs$. By \ref{def:precellularcat}\ref{def:precellularcat1} the span has a pushout  in $\categ$, say that $(f_1, f_2, g_1, g_2)$ is a pushout square. We have $g_1, g_2\in \morphs$ by \ref{def:precellularcat}\ref{def:precellularcat3}, and $\id_P\in \morphs$ by \ref{def:precellularcat}\ref{def:precellularcat2}. Therefore $(f_1, f_2, g_1, g_2)\in\dnf$.
\[\begin{tikzcd}
	&& P \\
	B & P \\
	A & C
	\arrow["{g_1}", curve={height=-12pt}, from=2-1, to=1-3]
	\arrow["{g_1}", from=2-1, to=2-2]
	\arrow["{\id_P}", from=2-2, to=1-3]
	\arrow["\ulcorner"{anchor=center, pos=0.125, rotate=-90}, draw=none, from=2-2, to=3-1]
	\arrow["{f_1}", from=3-1, to=2-1]
	\arrow["{f_2}"', from=3-1, to=3-2]
	\arrow["{g_2}"', curve={height=12pt}, from=3-2, to=1-3]
	\arrow["{g_2}"', from=3-2, to=2-2]
\end{tikzcd}\]
\nin\begin{enumerate}[resume*=prf:indrel]
\item \label{prf:indrel5} $\dnf$ has the uniqueness property.
\end{enumerate}
Why \ref{prf:indrel5}? Assume $(f_1, f_2, g_1^a, g_2^a)\in \dnf$ and $(f_1, f_2, g_1^b, g_2^b)\in\dnf$. 
\[\begin{tikzcd}
	& {D^b} & \\
	B && {D^a} \\
	A & C
	\arrow["{g_1^b}", from=2-1, to=1-2]
	\arrow["{g_1^a}"{description, pos=0.7}, from=2-1, to=2-3]
	\arrow["{f_1}", from=3-1, to=2-1]
	\arrow["{f_2}"', from=3-1, to=3-2]
	\arrow["{g_2^b}"{description, pos=0.8}, from=3-2, to=1-2]
	\arrow["{g_2^a}"', from=3-2, to=2-3]
\end{tikzcd}\]
By \ref{def:precellularcat}\ref{def:precellularcat1} we can find the pushout $P$ of $(f_1, f_2)$, then the induced maps $r: P \to D^a$ and $s: P \to D^b$ are in $\morphs$ by hypothesis. Thus by \ref{def:precellularcat}\ref{def:precellularcat1} we can take the pushout 
\[\begin{tikzcd}
	{D^a} & D \\
	P & {D^b}
	\arrow["{g^a}", from=1-1, to=1-2]
	\arrow["r", from=2-1, to=1-1]
	\arrow["s"', from=2-1, to=2-2]
	\arrow["{g^b}"', from=2-2, to=1-2]
\end{tikzcd}\]
where all the maps are in $\morphs$ by \ref{def:precellularcat}\ref{def:precellularcat3}. It is easily seen that this gives a commutative diagram as in \ref{def:eqcommsq}\ref{def:eqcommsq1}.
\end{proof}

Finally, we introduce independence relations on $\mu\mh\mrm{AEC}$s and recall the definition of a stable independence relation on a $\mu\mh\mrm{AEC}$ as done in \cite[Section~8]{forkingcategorical}, and then we  define a notion of non-forking for orbital types (Definition~\ref{def:dnftype}). These notions will be used in the next sections to establish the  stability of the $\mu\mh\mrm{AEC}$s we will consider.

\begin{remark}
If $\kaecgoth = (\kaec, \leqk)$ is an abstract class, we can consider the category $\categ$ with objects those in $\kaec$ and with morphisms the $\tau_\kaecgoth$-homomorphisms\footnote{Recall that a map $f: M \to N$ between $\tau$-structures is a $\tau$-homomorphism if it preserves the truth of atomic formulas, i.e.,  for every atomic formula $\varphi(\bar x)$ and $\bar a\in M^{|\bar x|}$ one has that $M \vDash \varphi(\bar a)$ implies $N \vDash \varphi(f(\bar a))$.}. Let $\morphs$ be the collection of all $\kaecgoth$-embeddings. Then we can consider, following  Definition~\ref{def:precellularcat}, the category $\categ_\morphs$, which is the category with objects the models in $\kaec$ and with morphisms the $\kaecgoth$-embeddings. For this reason, when we consider $\kaecgoth$ as a category, we are identifying it with $\categ_\morphs$. Therefore, it makes sense to consider independence relations on $\kaecgoth$.
\end{remark}

\begin{definition}
\label{def:indrelmuaec}
Let $\kaecgoth = (\kaec, \leqk)$ be a $\mu\mh\mrm{AEC}$ and $\dnf$ an independence relation on $\kaecgoth$.
\begin{enumerate}[(1), leftmargin=*]
\item \label{def:indrelmuaec1}\cite[Notation~3.7]{forkingcategorical} We write $M_1 \dnf^{M_3}_{M_0} M_2$ if $M_0\leqk M_\ell\leqk M_3$, with $\ell\in\{1, 2\}$, and $(i_{0,1}, i_{0,2}, i_{1,3}, i_{2,3})\in \dnf$, where $i_{j,k}$ is the inclusion map from $M_j$ to $M_k$.
\[\begin{tikzcd}
	{M_1} & {M_3} \\
	{M_0} & {M_2}
	\arrow["{i_{1,3}}", from=1-1, to=1-2]
	\arrow["{i_{0,1}}", from=2-1, to=1-1]
	\arrow["{i_{0,2}}"', from=2-1, to=2-2]
	\arrow["{i_{2,3}}"', from=2-2, to=1-2]
\end{tikzcd}\]
\item \label{def:indrelmuaec2}\cite[Definition~8.2]{forkingcategorical} We write $\dnfb{M_0}{A}{B}{M_3}$ if $M_0\leqk M_3, M_0\subseteq A\cup B\subseteq M_3$, and there are $M_1',M_2',M_3'\in\kaec$ with $A\subseteq M_1'$, $B\subseteq M_2'$, $M_3\leqk M_3'$, and $M_1'\dnf_{M_0}^{M_3'}M_2'$.
\end{enumerate}
\end{definition}

We now recall the definition of stable independence relations on $\mu\mh\mrm{AEC}$s. These are weakly stable independence relation which moreover satisfy two crucial locality properties, i.e., local character (Definition~\ref{def:indeplocalchar}) and the witness property (Definition~\ref{def:indepwitnessprop}).
Their usefulness comes from Fact~\ref{fact:stableindep}, which says that every $\mu\mh\mrm{AEC}$ with a stable independence relation is stable and tame.

\begin{definition}[\protect{\cite[Definition 8.6]{forkingcategorical}}]
\label{def:indeplocalchar}
Let $\kaecgoth  = (\kaec, \leqk)$  be a $\mu\mh\mrm{AEC}$, $\dnf$ an independence relation on $\kaecgoth$. We say that $\dnf$ has \tdef{right local character} if for each cardinal $\nu$ there is a cardinal $\lambda_\nu$ (depending on $\nu$) such that for any $M_1\leqk M_3$ and $M_2 \leqk M_3$ with $|M_1|\leq\nu$, then there are $M_0,M_1',M_3'\in\kaec$ with $|M_0|\leq\lambda_\nu$, $M_1\leqk M_1'$, $M_3\leqk M_3'$, and $M_1' \dnf_{M_0}^{M_3'}M_2$.
That is, we have the following commutative diagram of inclusions where all the arrows are $\kaecgoth$-embeddings and the outer square is in $\dnf$.
\[\begin{tikzcd}
	{M_1'} & {M_1'} & {M_3'} \\
	& {M_1} & {M_3} \\
	{M_0} && {M_2}
	\arrow[equals, from=1-1, to=1-2]
	\arrow[from=1-2, to=1-3]
	\arrow[from=2-2, to=1-2]
	\arrow[from=2-2, to=2-3]
	\arrow[from=2-3, to=1-3]
	\arrow[from=3-1, to=1-1]
	\arrow[from=3-1, to=3-3]
	\arrow[from=3-3, to=2-3]
\end{tikzcd}\]
\end{definition}

\begin{definition}[\protect{\cite[Definition 8.7]{forkingcategorical}}]
\label{def:indepwitnessprop}
Let $\kaecgoth =(\kaec, \leqk)$ be a $\mu\mh\mrm{AEC}$, $\dnf$ an independence relation on $\kaecgoth$, $\theta$ an infinite cardinal. 
\begin{enumerate}[(1), leftmargin=*]
\item \label{def:indepwitnessprop1} We say that $\dnf$ has  the \tdef{right $({<}\theta)$-witness property} if $M_1\dnf_{M_0}^{M_3}M_2$ holds whenever $M_0\leqk M_\ell \leqk M_3$, with $\ell\in \{1,2\}$, and $\dnfb{M_0}{M_1}{A}{M_3}$ for all $A\subseteq M_2$ with $|A|<\theta$.
\item \label{def:indepwitnessprop2} We say that $\dnf$ has the \tdef{right witness property} if it has the $({<}\theta)$-witness property for some $\theta$.
\end{enumerate}
\end{definition}

\begin{definition}
\label{def:stableindep}
Let $\kaecgoth = (\kaec, \leqk)$ be a $\mu\mh\mrm{AEC}$ and $\dnf$ an independence relation on $\kaecgoth$. We say that $\dnf$ is \tdef{stable} if it is weakly stable, satisfies  right local character, and the right witness property.
\end{definition}

\begin{remark}
In \cite[Definition 3.24]{forkingcategorical} stable independence relations are defined for an arbitrary category. Nonetheless, by \cite[Theorem 8.14]{forkingcategorical} the definition we have given here is equivalent to that in \cite{forkingcategorical} when $\dnf$ is a weakly stable independence relation on a $\mu\mh\mrm{AEC}$ $\kaecgoth$.
\end{remark}

\begin{fact}[\protect{\cite[Corollary 8.16]{forkingcategorical}}]
\label{fact:stableindep}
Let $\kaecgoth =(\kaec, \leqk)$ be a $\mu\mh\mrm{AEC}$ and $\dnf$ a stable independence relation on $\kaecgoth$. Then:
\begin{enumerate}[(1), leftmargin=*]
\item \label{fact:stableindep1} \tdef{(Stability)} For any ordinal $\gamma$ there is a proper class $C_\gamma$ of cardinals such that for any $\lambda\in C_\gamma$ we have that $\kaecgoth$ is $(\lambda, \gamma)$-stable (cf.~\ref{def:stabilityproper}\ref{def:stabilityproper1}). In particular, $\kaecgoth$ is stable (cf.~\ref{def:stabilityproper}\ref{def:stabilityproper2}).
\item \label{fact:stableindep2} \tdef{(Tameness)} For any ordinal $\gamma$ there is a cardinal $\kappa$ such that $\kaecgoth$ is $(\kappa, \gamma)$-tame (cf.~\ref{def:tame}). In particular, $\kaecgoth$ is tame.
\end{enumerate}
\end{fact}

We conclude this subsection with the definition of non-forking for orbital types, which is closely related to the notion of non-forking defined in first-order model theory for (syntactic) types. It will also be useful when we come to compute the precise stability spectrum of the $\mu\mh\mrm{AEC}$s under consideration.

\begin{definition}[\protect{\cite[Definition 8.2]{forkingcategorical}}]
\label{def:dnftype}
Let $\kaecgoth = (\kaec, \leqk)$ be a $\mu\mh\mrm{AEC}$ and $\dnf$ a stable independence relation on $\kaecgoth$. We say that $\gtp(\bar a/B; M_3)$ \tdef{does not fork over} $M_0$ if $\dnfb{M_0}{\mrm{ran}(\bar a)}{B}{M_3}$\footnote{It can be verified (see \cite[Definition~8.2, Fact~8.4]{forkingcategorical}) that this definition does not depend on the choice of representative of the orbital type $\gtp(\bar a/B; M_3)$.}. 
\end{definition}

\begin{fact}[\protect{\cite[Fact 8.4, Theorem 8.5]{forkingcategorical}}]
\label{fact:indeptypes}
Let $\kaecgoth = (\kaec, \leqk)$ be a $\mu\mh\mrm{AEC}$ and $\dnf$ a stable independence relation on $\kaecgoth$. Then:
\begin{enumerate}[(1), leftmargin=*]
\item \label{fact:indeptypesuniq} \tdef{(Uniqueness)} Given $p,q\in \gS_\kaecgoth^{<\infty}(B; N)$ with $M\leqk N$ and $M\subseteq B\subseteq N$, if $p\restriction M = q\restriction M$ and both $p$ and $q$ do not fork over $M$, then $p=q$.
\item \tdef{(Extension)} If $M\leqk N$ and $p\in \gS^{<\infty}_\kaecgoth(M)$, then there is $q\in \gS_\kaecgoth^{<\infty}(N)$ extending $p$ (cf.~\ref{def:nottypes}\ref{def:nottypes3}) such that $q$ does not fork over $M$.
\end{enumerate}
\end{fact}

\section{\texorpdfstring{$\lambda$}{λ}-pure embeddings}
\label{sect:lpure}

In this section we will consider abstract classes of the form $(\kaec, \leqp^\lambda)$, where $\leqp^\lambda$ is the relation of $\lambda$-purity. We will give sufficient conditions for these classes to be stable and tame. We will see, for example, that $(\rmod, \leqp^\lambda)$ is a $\lambda\mh\mrm{AEC}$. Following the current notation on $\lambda$-purity, we will use the letter $\lambda$ instead of $\mu$ throughout, so that we will talk about $\lambda\mh\mrm{AEC}$s (instead of $\mu\mh\mrm{AEC}$s). 

\subsection{Definition}

\begin{context}
\label{con:lambdainf}
Throughout this section $\lambda$ denotes an infinite cardinal.
\end{context}

\begin{definition}[\protect{\cite[Definition 7.17]{jensen}}]
\label{def:lpure}
Let $A\leq B$ be $R$-modules. We say that $A$ is \tdef{$\lambda$-pure in $B$} (written $A\leqp^\lambda B$) if every system of $<\lambda$ $R$-linear equations with parameters in $A$ has a solution in $A$ provided it has a solution in $B$.
\end{definition}

Clearly, $\aleph_0$-purity coincides with purity (cf.~\ref{def:purity}\ref{def:purity1}).
\begin{remark}
\label{rem:lpuredsumequiv}
By Fact~\ref{fact:dsumsolveeq}, for any pair of $R$-modules  $A\leq B$ one has $A\leqo B$ if and only if $A \leqp^\lambda B$ for every infinite cardinal $\lambda$.
\end{remark}

We recall here a characterization of $\lambda$-purity which will be useful in the following:

\begin{fact}[\protect{\cite[Proposition 7.16]{jensen}}]
\label{fact:lpuresplit}
Let $A\leq B$ be $R$-modules, then the following are equivalent:
\begin{enumerate}[(1), leftmargin=*, series=fact:lpreslpure]
\item \label{fact:lpuresplit1}$A$ is $\lambda$-pure in $B$,
\item \label{fact:lpuresplit2}For every $\lambda$-presented module $C$ (cf.~\ref{def:lambdapresented}) and every $\gamma: C \to B/A$, there  is $\alpha: C \to B$ with $\gamma = \pi\alpha$, where $\pi: B\to B/A$ is the natural projection.
\[\begin{tikzcd}
	&& C && \\
	0 & A & B & {B/A} & 0
	\arrow["\alpha"', dashed, from=1-3, to=2-3]
	\arrow["\gamma", from=1-3, to=2-4]
	\arrow[from=2-1, to=2-2]
	\arrow[from=2-2, to=2-3]
	\arrow["\pi", from=2-3, to=2-4]
	\arrow[from=2-4, to=2-5]
\end{tikzcd}\]
Equivalently, for every $\lambda$-presented module $C$ one has that the induced sequence $0\to \mrm{Hom}(C,A)\to \mrm{Hom} (C,B) \to \mrm{Hom}(C,B/A) \to 0$ is exact.
\end{enumerate}
In particular, if $A\leqp^\lambda B$, then for every $A\leq C\leq B$ with $C/A$ $\lambda$-presented we have that $A\leqo C$.
\end{fact}
The following remark is folklore, and it had already been noticed by Fuchs in \cite{paperfuchsdsums} for the case of abelian groups. Indeed, the definition of $\lambda$-purity in Fuchs's book on abelian groups \cite[pg.~153]{fuchs} is equivalent to ours by the following:

\begin{remark}
\label{rem:lpureab}
Let $R$ be a ring and $|R| < \lambda$. The following are equivalent:
\begin{enumerate}[(1), leftmargin=*,series=remlpureab]
\item \label{rem:lpureab1}$A \leqp^\lambda B$;
\item \label{rem:lpureab2}for every $A\leq C\leq B$ with $C/A$ $({<}\lambda)$-generated we have $A \leqo C$.
\end{enumerate}
Moreover, if $R$ is a countable left Noetherian ring (e.g., $\mathbb{Z}$), then the previous equivalence also holds when $\lambda=\aleph_0$. In fact more is true when $R= \mathbb{Z}$, in such a case $A\leqp B$ if and only if for every $A\leq C\leq B$ with $|C/A| < \aleph_0$ we have $A \leqo C$.
\end{remark}
\begin{proof}
To show the equivalence between \ref{rem:lpureab1} and \ref{rem:lpureab2} we use the fact that the $\lambda$-presented $R$-modules are exactly the $({<}\lambda)$-generated ones because $|R|<\lambda$ by  Remark~\ref{rem:lambdapresented}\ref{rem:lambdapresented1}.  \ref{rem:lpureab1} implies \ref{rem:lpureab2} by Fact~\ref{fact:lpuresplit}, and the other direction follows from Fact~\ref{fact:lpuresplit}\ref{fact:lpuresplit2} and Fact~\ref{fact:dsumsolveeq}\ref{fact:dsumsolveeq2.5}. If $R$ is a countable left Noetherian ring, then we use Remark~\ref{rem:lambdapresented}\ref{rem:lambdapresented2} and argue as before. Finally, the last sentence, if $R=\mathbb{Z}$, is shown in \cite[2.H]{paperfuchsdsums}.
\end{proof}

We now introduce some model-theoretic terminology which will elucidate Definition~\ref{def:lpure}.
\begin{notation} 
\label{not:zeroae}
If $A$ is a left $R$-module and $(a_i)_{i\in I}\in A^I$, we say that $(a_i)_{i\in I}$ is  \tdef{zero almost everywhere} (\enquote{zero a.e.} for short) if the set $\{i\in I\mid a_i\neq 0\}$ is finite. In this case, it makes sense to consider $\sum_{i\in I} a_i$. We implicitly assume $(a_i)_{i\in I}$ is zero a.e. when we write $\sum_{i\in I}a_i$.
\end{notation}

\begin{definition}
\label{def:ppform}
\begin{enumerate}[(1), leftmargin=*]
\item \label{def:ppform1} An \tdef{equation} is a formula $\varphi(\bar x, \bar y)$ of the form:
\[
r_1 x_{j_1} + \ldots + r_m x_{j_m} = s_1 y_{k_1} + \ldots + s_n y_{k_n}.
\]
\item A \tdef{system of equations} is a conjunction of equations.
\item \label{def:ppform2} A \tdef{$\lambda$-positive primitive formula} (also called \tdef{$\lambda\mh\pp$-formula}) is a formula $\varphi(\bar x)\in \mathfrak{L}_{\lambda\lambda}$ of the form
\[
\exists \bar y \,\bigwedge_{k\in K} \varphi_k(\bar x, \bar y),
\]
with $|\bar x| +|\bar y|+ |K| <\lambda$ and each $\varphi_k(\bar x, \bar y)$ an equation. If $\lambda =\aleph_0$ we just say that $\varphi(\bar x)$ is a \tdef{pp-formula}.
\item \label{def:ppform3} A \tdef{basic $\lambda$-positive primitive formula} (also \tdef{basic $\lambda\mh\pp$-formula}) is a formula $\varphi(\bar x)\in \mathfrak{L}_{\lambda\lambda}$ of the form
\[
\exists (y_j)_{j\in J} \,\bigwedge_{k\in K} \sum_{j\in J} r_{jk} y_j = x_k,
\]
with $|J| + |K|<\lambda$ and $(r_{jk})_{j\in J}$ zero a.e. (cf.~\ref{not:zeroae}) for every $k\in K$. If $\lambda =\aleph_0$ we just say that $\varphi(\bar x)$ is a \tdef{basic pp-formula}.
\end{enumerate}
\end{definition}

\begin{remark}
A $\lambda\mh\pp$-formula can be equivalently written in the following way:
\[
\exists y(H \bar y = K\bar x),
\]
where $H$ and $K$ are infinite matrices with the same number of $<\lambda$ rows and which are zero a.e. on every row, and $\bar x$ and $\bar y$ are vectors of variables of length $<\lambda$ such that the multiplications $H\bar y$ and $K \bar x$ can be defined in the obvious way. Clearly, a basic $\lambda\mh\pp$-formula is written as above but with $K$ being the identity matrix.
\end{remark}

\begin{notation}
\label{not:basicppform}
When we have a basic $\lambda\mh\pp$-formula $\varphi(\bar x)$ of the form 
\[
\exists (y_j)_{j\in J} \,\bigwedge_{k\in K} \sum_{j\in J} r_{jk} y_j = x_k,
\]
we use the notation $\bar x = (x_k \mid k\in K)$ and $\bar y = (y_j\mid j\in J)$. Similarly for $\lambda\mh\pp$-formulas.
\end{notation}

We are now able to give a useful  equivalent definition of $\lambda$-purity.

\begin{proposition}
\label{prop:equivlpurelppform}
Let $A\leq B$ be $R$-modules. Then the following are equivalent:
\begin{enumerate}[(1), leftmargin=*]
\item \label{prop:equivlpurelppform1} $A\leqp^\lambda B$.
\item \label{prop:equivlpurelppform2}For every $\lambda\mh\pp$-formula  $\varphi(\bar x)$ and $\bar a\in A^{|\bar x|}$, we have that $A \vDash \varphi(\bar a)$ if and only if $B\vDash \varphi(\bar a)$.
\item \label{prop:equivlpurelppform3}For every basic $\lambda\mh\pp$-formula  $\varphi(\bar x)$ and $\bar a\in A^{|\bar x|}$, we have that $A \vDash \varphi(\bar a)$ if and only if $B\vDash \varphi(\bar a)$. 
\end{enumerate}
\end{proposition}
\begin{proof}
Clearly~\ref{prop:equivlpurelppform2} is just a restatement of~\ref{prop:equivlpurelppform1}, and obviously~\ref{prop:equivlpurelppform2} implies~\ref{prop:equivlpurelppform3}. We show that~\ref{prop:equivlpurelppform3} implies~\ref{prop:equivlpurelppform2}. Assume $\varphi(\bar x)$ is a $\lambda\mh\pp$-formula and $\bar a\in A^{|\bar x|}$. We are left to show that $B\vDash \varphi(\bar a)$ implies $A\vDash \varphi(\bar a)$.  By definition of $\lambda\mh\pp$-formula $\varphi(\bar x)$ is of the form
\[
\exists (y_j)_{j\in J}\bigwedge_{k\in K} \sum_{j\in J} r_{jk}y_j = \sum_{i\in I} s_{ik}x_i.
\]
Let $\psi(\bar z)$ be the formula
\[
\exists (y_j)_{j\in J}\bigwedge_{k\in K} \sum_{j\in J} r_{jk}y_j = z_k.
\]
Setting $b_k = \sum_{i\in I} s_{ik}a_i$, with $k\in K$, we have that $B\vDash \psi(\bar b)$. Since $\psi(\bar z)$ is a basic $\lambda\mh\pp$-formula, we have by hypothesis $A\vDash \psi(\bar b)$, so that $A\vDash \varphi(\bar a)$.
\end{proof}

We will see that Proposition~\ref{prop:equivlpurelppform}\ref{prop:equivlpurelppform3} provides the easiest and most direct way to check $\lambda$-purity of a submodule. In what follows, we are going to use the equivalences of Proposition~\ref{prop:equivlpurelppform} without further mention. We now state some facts that follow easily from the definition of $\lambda\mh\pp$-formulas.

\begin{remark}
\label{rem:lpuretrivia}
\begin{enumerate}[(1), leftmargin=*]
\item\label{rem:lpuretrivia1} The set of $\lambda\mh\pp$-formulas has cardinality $|R|^{<\lambda}$.
\item\label{rem:lpuretrivia2} If $\varphi(\bar x)$ is a $\lambda\mh\pp$-formula and $A$ is an $R$-module, then the set $\varphi(A)=\{\bar a\in A^{|\bar x|} \mid A \vDash \varphi(\bar a)\}$ is a subgroup of $A^{|\bar x|}$.
\item \label{rem:lpuretrivia3}If $A\leq B$, $\varphi(\bar x)$ is a $\lambda\mh\pp$-formula, and $\bar a\in A^{|\bar x|}$, then $A\vDash \varphi(\bar a)$ implies $B\vDash \varphi(\bar a)$. In particular, if $A\leq B \leq C$ and $A\leqp^\lambda C$, then $A\leqp^\lambda B$.
\item If $(\varphi_i(\bar x): i\in I)$ with $|I|<\cf(\lambda)$ is a collection of $\lambda\mh\pp$-formulas, then it is easily seen that $\bigwedge_{i\in I} \varphi_i(\bar x)$ is equivalent to a $\lambda\mh\pp$-formula.
\end{enumerate}
\end{remark}

\begin{proposition}
\label{prop:lpurelaec}
Let $\lambda$ be a regular cardinal. Then $\kaecgoth_\lambda^{\rmod} =  (\rmod, \leqp^\lambda)$ is a $\lambda\mh\mrm{AEC}$ with ${\mrm{LS}_\lambda(\kaecgoth_\lambda^{\rmod}) \leq |R|^{<\lambda}}$.
\end{proposition}
\begin{proof}
We split the proof by verifying separately all of the axioms of a $\lambda\mh\mrm{AEC}$.
\begin{enumerate}[$(*_1)$, leftmargin=*, series=prf:lpurelaec]
\item \label{prf:lpurelaec1} $\kaecgoth_\lambda^\rmod$ has coherence (cf.~\ref{def:propertiesac}\ref{def:coherence}).
\end{enumerate}
Why \ref{prf:lpurelaec1}? Follows from Remark~\ref{rem:lpuretrivia}\ref{rem:lpuretrivia3}.
\begin{enumerate}[resume*=prf:lpurelaec]
\item \label{prf:lpurelaec2} $\kaecgoth_\lambda^\rmod$ satisfies the Tarski-Vaught axioms (cf.~\ref{def:muaec}\ref{def:muaec_tvaxioms}).
\end{enumerate}
Why \ref{prf:lpurelaec2}? We only show $\lambda$-continuity, the proof of $\lambda$-smoothness being similar.  Assume $(A_i)_{i\in I}$ is a $\kaecgoth^\rmod_\lambda$-directed system and let $A= \bigcup_{i\in I}A_i$. Fix $i\in I$. We show $A_i\leqp^\lambda A$. Let $\varphi(\bar x)$ be a basic $\lambda\mh\pp$-formula, $\bar a\in A_i^{|\bar x|}$, and $A\vDash \varphi(\bar a)$. By definition, $\varphi(\bar x)$ is of the form $\exists (y_j)_{j\in J} \bigwedge_{k\in K}\sum_{j\in J}r_{jk}y_j = x_k$ with $|J| + |K| < \lambda$. Let $\bar b\in A^{|\bar y|}$ with $A\vDash \bigwedge_{k\in K}\sum_{j\in J}r_{jk}b_j = a_k$. Because $|\bar b|<\lambda$ and $(A_i)_{i\in I}$ is $\lambda$-directed, there is $i\leq \ell\in I$ such that $\bar b\in A_\ell^{|\bar y|}$. Thus $A_\ell \vDash \varphi(\bar a)$, and by $A_i \leqp^\lambda A_\ell$ we get $A_i \vDash \varphi(\bar a)$. 
\begin{enumerate}[resume*=prf:lpurelaec]
\item \label{prf:lpurelaec3}$\kaecgoth^\rmod_\lambda$ satisfies the $\mrm{LS}$-axiom (cf.~\ref{def:muaec}\ref{def:muaec_lstaxiom}).
\end{enumerate}
Why \ref{prf:lpurelaec3}? This is shown in \cite[Lemma 4.1]{coto}.
\end{proof}
\subsection{Wide pushouts in \texorpdfstring{$\rmod$}{R-Mod}}

It is well-known that pushouts are closed under pure embeddings, as shown for example in \cite[Proposition 2.1.13]{purityspectra} using the fact that a pure embedding is a direct system of split embeddings. The proof in \cite{purityspectra}, using that a $\lambda$-pure embedding is a $\lambda$-directed limit of split embeddings \cite[Proposition~2.30]{adamekbook}, easily generalizes to show that  pushouts are closed under $\lambda$-pure embeddings. In this subsection we introduce wide pushouts in the category $\rmod$, and we show that wide pushouts are also closed under $\lambda$-pure embeddings. In turn, this will show that $(\rmod, \leqp^\lambda)$ has $\infty\mh\mrm{AP}$ (cf.~\ref{def:amalgac}\ref{def:amalgac3}). 
Wide pushouts are the dual notion of the wide pullbacks originally introduced in \cite[Definition~1.8]{widepullbacks}, and the ordinary (binary) pushout is a special case of a wide pushout.

\begin{definition}
\label{def:widepushout}
Let $(f_i: A \to B_i\mid i\in I)$ be a collection of $R$-module homomorphisms, we say that a collection $(g_i: B_i \to P\mid i\in I)$ of $R$-module homomorphisms is a \tdef{wide pushout} for $(f_i: A\to B_i\mid i\in I)$ if:
\begin{enumerate}[(1), leftmargin=*]
\item $g_if_i = g_jf_j$ for every $i,j\in I$; 
\item for every other collection $(h_i:B_i \to Q\mid i\in I)$ of $R$-module homomorphisms with $h_if_i = h_jf_j$, for $i,j\in I$, there is a unique $r: P \to Q$ such that $h_i = rg_i$ for every $i\in I$.
\end{enumerate}
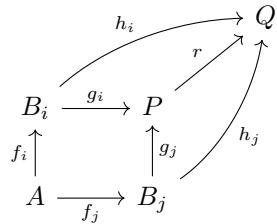
\[\begin{tikzcd}
	&& Q \\
	{B_i} & P \\
	A & {B_j}
	\arrow["{h_i}", curve={height=-12pt}, from=2-1, to=1-3]
	\arrow["{g_i}", from=2-1, to=2-2]
	\arrow["r", from=2-2, to=1-3]
	\arrow["{f_i}", from=3-1, to=2-1]
	\arrow["{f_j}"', from=3-1, to=3-2]
	\arrow["{h_j}"', curve={height=12pt}, from=3-2, to=1-3]
	\arrow["{g_j}"', from=3-2, to=2-2]
\end{tikzcd}\]
\end{definition}

\begin{remark}
If $I = \{1, 2\}$, then the wide pushout is simply the pushout.
\end{remark}
Clearly a wide pushout can be viewed as a colimit in the category $\rmod$ with arrows the $R$-module homomorphisms. By cocompleteness of this category \cite[Chapter~XVI, Proposition~10.3]{grillet} we know that the wide pushout always exists and is unique up to isomorphism. We now show how to concretely build wide pushouts.

\begin{proposition}
\label{prop:widepushoutexist}
Let $(f_i \colon A \to B_i \mid i \in I)$ be a family of $R$-module homomorphisms.
Set $B = \bigoplus_{i \in I} B_i$ and define the $R$-submodule
\[
N = \Bigl\{\, \sum_{i \in I} \rho_i f_i(a^i) \;\Big|\;
(a^i)_{i \in I} \in A^I\text{ zero a.e. and } \sum_{i \in I} a^i = 0 \,\Bigr\}
\leq \bigoplus_{i\in I}B_i,
\]
\noindent
where $\rho_i: B_i \to B$ is the canonical inclusion of $B_i$ into $B$. Let $P = B/N$, let
$\pi \colon B \to P$ be the canonical projection, and set
$g_i = \pi\rho_i \colon B_i \to P$. Then $(g_i \colon B_i \to P \mid i \in I)$
is the wide pushout of $(f_i \colon A \to B_i \mid i \in I)$.
\end{proposition}

\begin{proof} Let $(f_i)_{i\in I}$ and $(g_i)_{i\in I}$ be as in the statement.
\begin{enumerate}[$(*_1)$, leftmargin=*, series=prf:propwidepuout]
\item \label{prf:propwidepuout1} $g_if_i = g_jf_j$ for every $i,j\in I$.
\end{enumerate}
\nin Why \ref{prf:propwidepuout1}? Take $i,j\in I$ and $a\in A$, then 
\[
g_i f_i(a) -g_jf_j(a) = \pi(\rho_if_i(a) + \rho_jf_j(-a)) = 0
\]
by definition of $N$.
\nin\begin{enumerate}[resume*=prf:propwidepuout]
\item \label{prf:propwidepuout2}  $N$ is generated by the set $X=\{\rho_if_i(a) - \rho_jf_j(a)\mid i,j\in I,\, a\in A\}$.
\end{enumerate}
\nin Why \ref{prf:propwidepuout2}? Assume $(a^i)_{i\in I}\in A^I$ is zero a.e. and $\sum_{i\in I} a^i =0$. Call $\{i_1, \ldots, i_n\} = \{i\in I\mid a^i\neq 0\}$. We show by induction on $n$ that $b = \sum_{j=1}^n \rho_{i_j}f_{i_j}(a^{i_j}) \in \langle X \rangle$. If $n=1$, then $b= 0$. Assume $n=m+1$ and the result holds for $m$. By definition of $X$ we have
\begin{equation}
\label{eq:wpoutexist1}
\rho_{i_m}f_{i_m}(-a^{i_{m+1}}) + \rho_{i_{m+1}}f_{i_{m+1}}(a^{i_{m+1}}) \in X,
\end{equation}
and by the induction hypothesis
\begin{equation}
\label{eq:wpoutexist2}
\rho_{i_1}f_{i_1}(a^{i_1}) + \ldots +\rho_{i_{m-1}}f_{i_{m-1}}(a^{i_{m-1}}) + \rho_{i_m}f_{i_m}(a^{i_m} + a^{i_{m+1}})\in \langle X \rangle.
\end{equation}
Summing \eqref{eq:wpoutexist1} and \eqref{eq:wpoutexist2} we get $b$, and thus $b\in \langle X\rangle$.
\begin{enumerate}[resume*=prf:propwidepuout]
\item \label{prf:propwidepuout3} If $(h_i: B_i \to C\mid i\in I)$ is such that $h_if_i = h_jf_j$ for every $i,j\in I$, then there is a unique $R$-module homomorphism $r: P \to C$ which satisfies $rg_i = h_i$ for every $i\in I$. 
\end{enumerate}
Why \ref{prf:propwidepuout3}? By the universal property of direct sums there is a unique $s: \bigoplus_{i\in I}B_i \to C$ such that $s\rho_i = h_i$ for every $i\in I$. By~\ref{prf:propwidepuout2} we have $N \leq \mrm{ker}(s)$. By the universal property of quotients there is a unique $r: P \to C$ with $r\pi = s$, and thus $rg_i = r\pi \rho_i =s\rho_i= h_i$. Finally, uniqueness of $r: P \to C$ is by construction.
\end{proof}

\begin{proposition}
\label{prop:widepushoutembedding}
Assume $(g_i: B_i \to P\mid i\in I)$ is the wide pushout of $(f_i: A \to B_i\mid i\in I)$, $\ell \in I$, and $f_i$ is an embedding for every $i\in I\setminus \{\ell\}$. Then $g_{\ell}$ is an embedding.
\end{proposition} 
\begin{proof}
Without loss of generality we can assume that $(g_i: B_i \to P\mid i\in I)$ is as in Proposition~\ref{prop:widepushoutexist}, so that $g_i = \pi\rho_i$, where $\pi: \bigoplus_{i\in I}B_i\to (\bigoplus_{i\in I}B_i)/N$ is the natural projection. Fix $b^{\ell} \in B_{\ell}$ with $g_{\ell}(b^{\ell}) = 0$, so that $\rho_{\ell}(b^{\ell})\in N$. Therefore, by the definition of $N$, there is $(a^i)_{i\in I}\in A^I$ zero a.e. such that $\sum_{i\in I} a^i = 0$ and
\begin{equation}
\label{eq:embwpsout1}
\rho_{\ell}(b^{\ell}) = \sum_{i\in I} \rho_if_i(a^i).
\end{equation}
Because $f_i$ is an embedding for every $i\neq \ell$, then \eqref{eq:embwpsout1} ensures $a^i = 0$ for every $i\neq \ell$. Therefore $a^{\ell} = -\sum_{i\neq  \ell}a^i= 0$, and \eqref{eq:embwpsout1} finally gives $b^{\ell} = 0$. 
\end{proof}

\begin{corollary}
\label{cor:wpushoutsclosedemb}
Wide pushouts are closed under embeddings, i.e., for any collection of  embeddings $(f_i: A\to B_i\mid i\in I)$ with wide pushout $(g_i: B_i \to P\mid i\in I)$, we have that $g_i:B_i\to P$ is an embedding for every $i\in I$.
\end{corollary} 
\begin{proof}
Follows from Proposition~\ref{prop:widepushoutembedding}.
\end{proof}

We now show that wide pushouts are closed under $\lambda$-pure embeddings, which will give us that $(\rmod, \leqp^\lambda)$ has $\infty\mh\mrm{AP}$. To show this, we will prove the following technical proposition. Although the second item of the following is not directly related to what we aim to prove now, it will be needed in the next subsection (in Lemma~\ref{lem:hyppushouts}).
\begin{proposition}
\label{prop:wpushoutclosedlambda}
Assume $(g_i \colon B_i \to P \mid i \in I)$ is the wide pushout of the $R$-module
homomorphisms $(f_i \colon A \to B_i \mid i \in I)$, $\ell \in I$, and $f_i$ is a
$\lambda$-pure embedding for every $i \in I \setminus \{\ell\}$. Then we have the following:

\begin{enumerate}[(1), leftmargin=*]
\item \label{prop:wpushoutclosedlambda1} $g_\ell$ is a $\lambda$-pure embedding.
\item \label{prop:wpushoutclosedlambda2} Assume $(g_i \colon B_i \to P \mid i \in I)$
is the wide pushout of $(f_i)_{i \in I}$ as in Proposition~\ref{prop:widepushoutexist},
so that $P = \bigl(\bigoplus_{i \in I} B_i\bigr)/N$ with
\[
N = \Bigl\{\, \sum_{i \in I} \rho_i f_i(a^i) \;\Big|\;
(a^i)_{i \in I} \in A^I\text{ zero a.e. and } \sum_{i \in I} a^i = 0 \,\Bigr\}
\leq \bigoplus_{i\in I}B_i.
\]
Then $N \leqp^\lambda \bigoplus_{i \in I} B_i$.
\end{enumerate}
\end{proposition}
\begin{proof}
We start by proving~\ref{prop:wpushoutclosedlambda1}.
Without loss of generality we can assume that $(g_i: B_i \to P\mid i\in I)$ is as in Proposition~\ref{prop:widepushoutexist}.
\noindent
Proposition~\ref{prop:widepushoutembedding} ensures that $g_{\ell}$ is an embedding. To show that $g_{\ell}(B_{\ell})\leqp^\lambda P$ we use Proposition~\ref{prop:equivlpurelppform}\ref{prop:equivlpurelppform3}, in particular we verify that basic $\lambda\mh\pp$-formulas are preserved. Assume that $|K| + |J| < \lambda$ and 
\[
P\vDash \bigwedge_{k\in K}\sum_{j\in J} r_{jk} [(b_j^i)_{i\in I}]_N = g_{\ell}(\tilde b_k),
\]
where $(b_j^i)_{i\in I}\in \bigoplus_{i\in I}B_i$ for $j\in J$, and $\tilde b_k\in B_{\ell}$ for $k\in K$. By definition of $P = (\bigoplus_{i\in I} B_i)/N$, for every $k\in K$ there is $(a_{ik})_{i\in I}\in A^I$ zero a.e. with $\sum_{i\in I} a_{ik} = 0$, such that:
\begin{equation}
\label{eq:lpurewpush-1}
\bigoplus_{i\in I}B_i\vDash \bigwedge_{k\in K}\sum_{j\in J} r_{jk} (b_j^i)_{i\in I} = \rho_{\ell}(\tilde b_k) + \sum_{i\in I} \rho_i f_i (a_k^i).
\end{equation}
Projecting \eqref{eq:lpurewpush-1} on $B_\ell$ we get:
\begin{equation}
\label{eq:lpurewpush-1.1}
B_{\ell}\vDash \bigwedge_{k\in K}\sum_{j\in J} r_{jk} b_j^{\ell} = \tilde b_k + f_{\ell} (a_k^\ell).
\end{equation}
For $i\in I\setminus \{\ell\}$, projecting \eqref{eq:lpurewpush-1} on $i$ we get: 
\begin{equation}
\label{eq:lpurewpush0}
B_i \vDash \varphi_i( (f_i(a_k^i))_{k\in K}),
\end{equation}
where $\varphi_i( (y_k)_{k\in K})$ is the formula:
\begin{equation}
\label{eq:lpurewpush0.0}
\exists (x_j)_{j\in J}\left[\bigwedge_{\substack{j\in J \\ b_j^i= 0}}x_j= 0\land \bigwedge_{k\in K}\sum_{j\in J}r_{jk}x_j = y_k\right].
\end{equation}
Notice that $\varphi_i((y_k)_{k\in K})$ is a $\lambda\mh\pp$-formula, as $|J|+|K| < \lambda$. For $i\in I\setminus \{\ell\}$ we have $f_i(A)\leqp^\lambda B_i$ and $f_i$ is an embedding, so by \eqref{eq:lpurewpush0} there are $(c_j^i)_{j\in J}\in A^J$ such that:
\begin{equation}
\label{eq:lpurewpush0.1}
A\vDash \bigwedge_{\substack{j\in J \\ b_j^i= 0}}c_j^i= 0\land \bigwedge_{k\in K}\sum_{j\in J}r_{jk}c_j^i = a_k^i.
\end{equation}
Notice that for every $j\in J$ the set $\{i\in I\mid i\neq \ell\, \land c_j^i \neq 0\}$ is finite. If it were not the case, then by the choice of the $c_j^i$'s we would have that for some $j$ the set $\{i\in I\mid b_j^i\neq 0\}$ is infinite, but this is a contradiction because $(b_j^i)_{i\in I}\in \bigoplus_{i\in I}B_i$ for every $j\in J$. Therefore, for every $j\in J$ the sum $\sum_{\substack{i\neq \ell\\ i\in I}} c_j^i$
is well defined. 

\smallskip \noindent Fixing $k\in K$ and computing everything in $B_{\ell}$ we have
\begin{equation}
\label{eq:lpurewpush1}
\sum_{j\in J}r_{jk} (b_j^{\ell} + f_{\ell}(\sum_{i\neq \ell}c_j^i)) = \tilde b_k +f_{\ell}(a^{\ell}_k) + f_{\ell}(\sum_{i\neq \ell}a_k^i) = \tilde b_k,
\end{equation}
where we have used \eqref{eq:lpurewpush-1.1} and \eqref{eq:lpurewpush0.1} for the first equality, and $\sum_{i\in I}a_k^i = 0$ in the second. Finally, applying $g_{\ell}$ to \eqref{eq:lpurewpush1} we get our result.

\smallskip \noindent  Now we prove~\ref{prop:wpushoutclosedlambda2}. The proof is very similar to that of~\ref{prop:wpushoutclosedlambda1}. To show $N\leqp^\lambda \bigoplus_{i\in I}B_i$ we again use Proposition~\ref{prop:equivlpurelppform}\ref{prop:equivlpurelppform3}.
 Assume that $|K| + |J| < \lambda$ and 
\[
\bigoplus_{i\in I}B_i\vDash \bigwedge_{k\in K}\sum_{j\in J} r_{jk} (b_j^i)_{i\in I} =  \sum_{i\in I} \rho_i f_i (a_k^i),
\]
where $(b_j^i)_{i\in I}\in \bigoplus_{i\in I}B_i$ for $j\in J$, and $\sum_{i\in I} a_k^i = 0$ with each $a_k^i\in A$.
For $i\in I\setminus \{\ell\}$ we have 
\[
B_i \vDash \varphi_i((f_i(a_k^i))_{k\in K}),
\]
where $\varphi_i((y_k)_{k\in K})$ is the $\lambda\mh\pp$-formula as in \eqref{eq:lpurewpush0.0}.

\smallskip \noindent  For $i\in I\setminus \{\ell\}$ we have $f_i(A)\leqp^\lambda B_i$ and $f_i$ is an embedding, thus there are $(c_j^i)_{j\in J}$ in $A$ such that:
\[
A\vDash \bigwedge_{\substack{j\in J \\ b_j^i= 0}}c_j^i= 0\land \bigwedge_{k\in K}\sum_{j\in J}r_{jk}c_j^i = a_k^i.
\]
Notice that for each $j\in J$ the sum $\sum_{i\neq \ell} c_j^i$ is well defined, as $\{i\in I\mid i\neq \ell\, \land\, c_j^i \neq 0\}$ is finite, just as in the proof of~\ref{prop:wpushoutclosedlambda1}. 

\smallskip \noindent  We define $c_j^{\ell} = -\sum_{i\neq\ell}c_j^i\in A$. Notice that this gives us $\sum_{i\in I}c_j^i = 0$, and thus $\sum_{i\in I}\rho_if_i(c_j^i)\in N$ for every $j\in J$. Therefore, for every $i\in I\setminus \{\ell\}$ we immediately get 
\begin{equation}
\label{eq:lpurewpush2}
B_i \vDash \bigwedge_{k\in K}\sum_{j\in J}r_{jk} f_i(c_j^i) = f_i(a_k^i).
\end{equation}
by the choice of $c_j^i$.
Computing everything inside $B_{\ell}$ and fixing $k\in K$, we have:
\begin{equation}
\label{eq:lpurewpush3}
\sum_{j\in J}r_{jk} f_{\ell}(c_j^{\ell}) =- \sum_{j\in J}r_{jk} f_{\ell}(\sum_{i\neq \ell}c_j^i) = -\sum_{i\neq \ell} f_{\ell}(\sum_{j\in J}r_{jk}c_j^i) = -\sum_{i\neq \ell} f_{\ell}(a_k^i) = f_{\ell}(a^{\ell}_k),
\end{equation}
where we have used the definition of $c_j^i$ in the third equality, and the fact that $\sum_{i\in I}a_k^i = 0$ in the last equality.
Clearly $\sum_{i\in I}\rho_if_i(c_j^i)\in N$ for every $j\in J$, and combining \eqref{eq:lpurewpush2} and \eqref{eq:lpurewpush3} we finally get
\[
\bigoplus_{i\in I} B_i\vDash\bigwedge_{k\in K}\sum_{j\in J} r_{jk}\sum_{i\in I}\rho_i f_i(c_j^i) = \sum_{i\in I}\rho_i f_i(a_k^i),
\]
which is the desired equation.
\end{proof}

\begin{corollary}
\label{cor:widepushclolambda}
Wide pushouts are closed under $\lambda$-pure embeddings, i.e., for any collection of $\lambda$-pure embeddings $(f_i: A\to B_i\mid i\in I)$ with wide pushout $(g_i: B_i \to P\mid i\in I)$, we have that $g_i:B_i\to P$ is a $\lambda$-pure embedding for every $i\in I$.
\end{corollary}
\begin{proof}
Follows from Proposition~\ref{prop:wpushoutclosedlambda}\ref{prop:wpushoutclosedlambda1}.
\end{proof}

\begin{corollary}
\label{cor:inftyaprmodpure}
$\kaecgoth^\rmod_\lambda = (\rmod, \leqp^\lambda)$ has $\infty\mh\mrm{AP}$ and $\mrm{JEP}$.
\end{corollary}
\begin{proof}
That $\kaecgoth_\lambda^\rmod$ has $\mrm{JEP}$ is clear as we can take direct sums, and $\infty\mh\mrm{AP}$ follows from Corollary~\ref{cor:widepushclolambda}. 
\end{proof}

\subsection{The framework}

In this subsection we show how to construct a stable independence relation on $\kaecgoth_\lambda^\rmod=(\rmod, \leqp^\lambda)$, or more generally on classes satisfying Hypothesis~\ref{hyp:lambdapure}. To do so, we generalize the methods introduced in  \cite{somestablenonelementary, mazarinoetherian}, which deal with  the case $\lambda = \aleph_0$. In particular, we will introduce on $\kaecgoth = (\kaec, \leqk)$, when $\kaecgoth$ satisfies Hypothesis~\ref{hyp:lambdapure}, a stable independence relation (cf.~\ref{def:stableindep}), and to do so we will in particular show that $(\kaec, \leqk)$ can be viewed as a pre-cellular category (cf.~\ref{def:precellularcat}).

\begin{context}
\label{con:lambdaregular}
From now on $\lambda$ denotes an infinite \textbf{regular} cardinal.
\end{context}

\begin{hypothesis}
\label{hyp:lambdapure}
Let $\lambda$ be an infinite regular cardinal, and $\kaecgoth_\lambda = (\kaec, \leqp^\lambda)$ an abstract class (cf.~\ref{def:abstractclass}) such that:
\begin{enumerate}[(1), leftmargin=*]
\item $\kaec$ is a class of $R$-modules;
\item $\kaecgoth_\lambda$ is closed under $\lambda$-directed systems (cf.~\ref{def:propertiesac}\ref{def:muclosed});
\item $\kaec$ is closed under:
\begin{enumerate}[(a), leftmargin=*]
\item finite direct sums;
\item $\lambda$-pure submodules, i.e., if $B\in \kaec$ and $A\leqp^\lambda B$, then $A\in\kaec$;
\item $\lambda$-pure quotients, i.e., if $A, B\in \kaec$ and $A\leqp^\lambda B$, then $B/A\in \kaec$.
\end{enumerate}
\end{enumerate}
\end{hypothesis}

\begin{notation}
\label{not:hyplpure}
When we write \enquote{$\kaecgoth_\lambda=(\kaec, \leqp^\lambda)$ satisfies Hypothesis~\ref{hyp:lambdapure}}, we mean that $\lambda$ is the cardinal for which Hypothesis~\ref{hyp:lambdapure} holds. This notation should not be confused with $\kaec_\lambda$ (cf. Notation~\ref{not:kaeclambda}).
\end{notation}

\begin{remark}
\label{rem:hyplpure}
\begin{enumerate}[(1), leftmargin=*]
\item \label{rem:intersecthyp}Assume $\kaecgoth^\ell_\lambda = (\kaec^\ell, \leqp^\lambda)$ satisfies Hypothesis~\ref{hyp:lambdapure}, with $\ell\in \{1,2\}$. Then $\kaecgoth^3_\lambda = (\kaec^1\cap \kaec^2, \leqp^\lambda)$ satisfies Hypothesis~\ref{hyp:lambdapure}.
\item \label{rem:hyplpurenu}If $\nu \geq \lambda$ and $\kaecgoth_\lambda = (\kaec, \leqp^\lambda)$ satisfies Hypothesis~\ref{hyp:lambdapure}, then $\kaecgoth_\nu = (\kaec, \leqp^\nu)$ satisfies Hypothesis~\ref{hyp:lambdapure}.
\end{enumerate}
\end{remark}

\begin{lemma}
\label{lem:hyppushouts}
Assume $\kaecgoth_\lambda = (\kaec, \leqp^\lambda)$ satisfies Hypothesis~\ref{hyp:lambdapure}. Then: 
\begin{enumerate}[(1), leftmargin=*]
\item \label{lem:hyppushoutslaec} $\kaecgoth_\lambda$ is a $\lambda\mh\mrm{AEC}$ with $\mrm{LS}_\lambda(\kaecgoth_\lambda) \leq |R|^{<\lambda}$.
\item \label{lem:hyppushoutsjep} $\kaecgoth_\lambda$ has $\mrm{JEP}$.
\item \label{lem:hyppushouts1} $\kaecgoth_\lambda$ has $\mrm{AP}$ and is closed under pushouts, i.e., if $(f_1: A \to B_1, f_2: A \to B_2)$ is a pair of $\kaecgoth_\lambda$-embeddings with pushout $(g_1: B_1\to P, g_2: B_2 \to P)$ in $\rmod$, then $P\in\kaec$ and $g_1, g_2$ are $\kaecgoth_\lambda$-embeddings.
\item \label{lem:hyppushouts2} If $\kaec$ is closed under arbitrary direct sums, then $\kaecgoth_\lambda$ has $\infty\mh\mrm{AP}$ and is closed under wide pushouts.
\end{enumerate}
\end{lemma}
\begin{proof}
We verify~\ref{lem:hyppushoutslaec}.  By Proposition~\ref{prop:lpurelaec} and closure under $\lambda$-pure submodules,  the coherence and L\"owenheim--Skolem axioms follow immediately, and $\mrm{LS}_\lambda(\kaecgoth_\lambda) \leq |R|^{<\lambda}$. By Proposition~\ref{prop:lpurelaec} and because $\kaecgoth_\lambda$ is closed under $\lambda$-directed systems, then the Tarski-Vaught axioms follow. Therefore $\kaecgoth_\lambda$ is a $\lambda\mh\mrm{AEC}$. 

\ssk\nin Item \ref{lem:hyppushoutsjep} is by closure under finite direct sums.

\ssk\nin We prove~\ref{lem:hyppushouts1}. It is enough to show closure under pushouts.  Assume $(f_1:A\to B_1, f_2: A \to B_2)$ are $\kaecgoth_\lambda$-embeddings with pushout $(g_1: B_1 \to P, g_2:B_2 \to P)$ in $\rmod$. We can assume that the pushout is of the form of Proposition~\ref{prop:widepushoutexist}. By Corollary~\ref{cor:widepushclolambda} $g_1$ and $g_2$ are $\lambda$-pure embeddings. We are only left to show that $P = B_1\oplus B_2/N\in\kaec$.
We have $B_1\oplus B_2\in \kaec$ by closure under finite direct sums, and $N\leqp^\lambda B_1\oplus B_2$ by Proposition~\ref{prop:wpushoutclosedlambda}\ref{prop:wpushoutclosedlambda2}. By closure under $\lambda$-pure submodules we have $N\in\kaec$, and by closure under $\lambda$-pure quotients we finally get $P\in \kaec$.

\ssk\nin Finally, the proof of~\ref{lem:hyppushouts2} is similar to that of~\ref{lem:hyppushouts1}.
\end{proof}

We are now going to define an independence relation on $\kaecgoth_\lambda$, when $\kaecgoth_\lambda$ satisfies Hypothesis~\ref{hyp:lambdapure}. As we said at the beginning of the subsection, we aim to use the machinery of pre-cellular categories introduced in Section~\ref{sect:precellular}. Notice that the independent squares defined in the following are the pre-cellular squares (cf.~\ref{def:cellularsquare}) defined in Section~\ref{sect:precellular}.

\begin{definition}
\label{def:dnflambdapure}
Assume $\kaecgoth_\lambda = (\kaec, \leqp^\lambda)$ satisfies Hypothesis~\ref{hyp:lambdapure}. For a commutative square $(f_1, f_2, g_1, g_2)$ in $\kaecgoth_\lambda$, we let $(f_1, f_2, g_1, g_2)\in \dnf$ if and only if the unique $R$-module homomorphism $r: P \to C$ from the pushout $P$ of $(f_1, f_2)$ is a $\lambda$-pure embedding, where the maps are as in the following commutative diagram:
\[\begin{tikzcd}
	&& C \\
	{B_1} & P \\
	A & {B_2}
	\arrow["{g_1}", curve={height=-12pt}, from=2-1, to=1-3]
	\arrow["{h_1}", from=2-1, to=2-2]
	\arrow["r", from=2-2, to=1-3]
	\arrow["\ulcorner"{anchor=center, pos=0.125, rotate=-90}, draw=none, from=2-2, to=3-1]
	\arrow["{f_1}", from=3-1, to=2-1]
	\arrow["{f_2}"', from=3-1, to=3-2]
	\arrow["{g_2}"', curve={height=12pt}, from=3-2, to=1-3]
	\arrow["{h_2}"', from=3-2, to=2-2]
\end{tikzcd}\] 
\end{definition}

\begin{lemma}
\label{lem:lpureindepweaklystable}
Assume $\kaecgoth_\lambda=(\kaec, \leqp^\lambda)$ satisfies Hypothesis~\ref{hyp:lambdapure}. Then $\dnf$ as defined in~\ref{def:dnflambdapure} is a weakly stable independence relation (cf. Definition~\ref{def:weaklystableind}).
\end{lemma}
\begin{proof}
Let $\categ$ be the category with objects the modules in $\kaec$ and with arrows the $R$-module homomorphisms. Let $\morphs$ be the collection of all $\lambda$-pure embeddings. It is enough to verify that $(\categ, \morphs)$ is a coherent pre-cellular category (cf. Definitions~\ref{def:precellularcat}-\ref{def:coherentprecell}). Indeed, if $(\categ, \morphs)$ is pre-cellular, then $\categ_\morphs= \kaecgoth_\lambda$ and the independence relation $\dnf$ from Definition~\ref{def:dnflambdapure} is exactly the collection of pre-cellular squares (cf. Definition~\ref{def:cellularsquare}), and then, by Proposition~\ref{prop:indepcellular}, $\dnf$ is weakly stable.  We are thus left to show that $(\categ, \morphs)$ is a pre-cellular category. In particular, we have to verify that  the conditions~\ref{def:precellularcat1}-\ref{def:precellularcat4} of Definition~\ref{def:precellularcat} are satisfied, and that $(\categ, \morphs)$ is coherent (cf.~\ref{def:coherentprecell}). 
\begin{enumerate}[(1), leftmargin=*]
\item If $(f_1, f_2)$ is a span in $\categ$ and both $f_1$ and $f_2$ are $\lambda$-pure embeddings, then the pushout exists in $\categ$ by Lemma~\ref{lem:hyppushouts}\ref{lem:hyppushouts1}.
\item Clearly every isomorphism is a $\lambda$-pure embedding.
\item Closure of $\morphs$ under pushouts follows from Lemma~\ref{lem:hyppushouts}\ref{lem:hyppushouts1}.
\item Closure under finite compositions is clear.
\item We verify that $(\categ, \morphs)$ is coherent. Assume $f: A\to B$ and $g: B\to C$ are $R$-module homomorphisms with $g$ and $gf$ $\lambda$-pure embeddings. Clearly $f$ is an embedding. We have $gf(A) \leqp^\lambda C$ and $g(B)\leqp^\lambda C$, so by coherence of $\kaecgoth_\lambda$ we have $gf(A) \leqp^\lambda g(B)$, and by closure under isomorphisms we get $f(A)\leqp^\lambda B$.
\end{enumerate}
\end{proof}

We now wish to show that $\dnf$ as defined in \ref{def:dnflambdapure}  is a stable independence relation. In particular, we are left to show that $\dnf$ satisfies in addition the following properties: local character (cf.~\ref{def:indeplocalchar}) and the witness property (cf.~\ref{def:indepwitnessprop}).
In the following lemma we will prove a slightly stronger form of right local character. Notice that this proof is the straightforward generalization of the proof in \cite[Theorem~4.14]{somestablenonelementary}, which dealt with the case $\lambda = \aleph_0$. The only difference in our proof is that, in our setting, we are not able to take countable unions, as $\kaecgoth_\lambda$ only has $\lambda$-continuity (cf.~\ref{def:propertiesac}\ref{def:mucontinuity}). To circumvent the problem we will use the $\mrm{LS}$ axiom to build chains of length $\lambda$, for which we can then take the union by $\lambda$-continuity.

\begin{lemma}
\label{lem:lpureindeplocal}
Assume $\kaecgoth_\lambda=(\kaec, \leqp^\lambda)$ satisfies Hypothesis~\ref{hyp:lambdapure}, then $\dnf$ has right local character (cf.~\ref{def:indeplocalchar}). In particular, if $B_1, B_2, C\in \kaec$ with $B_1, B_2 \leqp^\lambda C$, then there are $B_1', A\in \kaec$ such that $A \leqp^\lambda B_1'$, $B_1\leqp^\lambda B_1'$, $|A| +|B_1'|\leq (|B_1| + |R|) ^{<\lambda}$, and $B_1' \dnf_A^C B_2$.
\end{lemma}

Before showing Lemma~\ref{lem:lpureindeplocal}, we will prove a rather technical lemma which will basically give us  $A$ and $B_1'$ of the statement.

\begin{lemma}
\label{lem:lpureloccharfund}
Let $B_1, B_2\leq C$ be $R$-modules and $\lambda$ an infinite regular cardinal. Then there are $A\leq B_1'$ with the following properties: 
\begin{enumerate}[(a), leftmargin=*]
\item \label{lem:lpureloccharfund1}$A\leqp^\lambda B_2$;
\item \label{lem:lpureloccharfund2}$A\cup B_1\subseteq B_1' \leqp^\lambda C$;
\item \label{lem:lpureloccharfund3}$|A| + |B_1'| \leq (|B_1| + |R|)^{<\lambda}$;
\item \label{lem:lpureloccharfundbox}if $\varphi(\bar x, \bar y)$ is a $\lambda\mh\pp$-formula, $\bar b_1'\in (B_1')^{|\bar x|}$, and there is $\bar b_2 \in B_2^{|\bar y|}$ such that $C \vDash \varphi(\bar b_1', \bar b_2)$, then there is $\bar a\in A^{|\bar y|}$ such that $C \vDash \varphi(\bar b_1', \bar a)$.
\end{enumerate}
\end{lemma}
\begin{proof}
Recall that by Proposition~\ref{prop:lpurelaec} we have that $\kaecgoth^\rmod_\lambda = (\rmod, \leqp^\lambda)$ is a $\lambda\mh\mrm{AEC}$. By induction on $i < \lambda$ we are going to construct chains  of $R$-modules $(A_i: i < \lambda)$, $ (B_{1,i}' : i<\lambda)$ such that the following conditions are satisfied:
\begin{enumerate}[$(\cdot_1)$,leftmargin=*]
\item \label{cond:lchar1}$(A_i)_{i<\lambda}$ and $(B_{1,i}')_{i<\lambda}$ are $\kaecgoth_\lambda^\rmod$-chains (cf.~\ref{def:mudirected}\ref{def:mudirectsyst});
\item \label{cond:lchar1.5}$B_1 \leq B_{1,0}'$;
\item \label{cond:lchar2} $A_i\leqp^\lambda B_2$ and $B_{1, i}' \leqp^\lambda C$, for $i < \lambda$;
\item \label{cond:lchar2.5} $A_j \leq B_{1,i}'$ for $j< i < \lambda$;
\item \label{cond:lchar3} $|A_i| + |B_{1, i}'| \leq (|B_1| + |R|)^{<\lambda}$, for $i < \lambda$;
\item \label{cond:lchar4} for every $i < \lambda$, if $\varphi(\bar x, \bar y)$ is a $\lambda\mh\pp$-formula (cf.~\ref{def:ppform}\ref{def:ppform1}), $\bar b_1'\in (B_{1, i}')^{|\bar{x}|}$, and there is $\bar b_2 \in B_2^{|\bar{y}|}$ such that $C \vDash \varphi(\bar b_1', \bar b_2)$, then there is $\bar a\in A_i^{|\bar{y}|}$ such that $C \vDash \varphi(\bar b_1', \bar a)$;
\end{enumerate}
\uline{The construction can be done.} 
\begin{enumerate}[$(*_1)$,leftmargin=*, series=prf:lpureindeplocalf]
\item We define $B_{1, 0}'$.
\end{enumerate}
Apply the $\mrm{LS}$ axiom of $\kaecgoth^\rmod_\lambda$ inside $C$ to find $B_1 \leq B_{1, 0}'\leqp^\lambda C$ with $|B_{1, 0}'|\leq (|B_1| + |R|)^{<\lambda}$.
\begin{enumerate}[resume*=prf:lpureindeplocalf]
\item \label{prf:lpureindeplocalf2}We now show how to construct $A_i$, for $ i< \lambda$.
\end{enumerate}
For $i<\lambda$, assume $B_{1, i}'$ and  $(A_j)_{j<i}$ have already been defined. For every $\lambda\mh\pp$-formula $\varphi(\bar x, \bar y)$ and for every $\bar b_1'\in (B_{1, i}')^{|\bar{x}|}$, define $\bar b_2^\varphi(\bar b_1') \in B_2^{|\bar{y}|}$ in the following manner:
\begin{enumerate}[$(\cdot_{\text{\alph*}})$, leftmargin=*]
\item $\bar b_2^\varphi(\bar b_1')$ is the sequence of all $0$'s if there is no $\bar b_2\in B_2^{|\bar{y}|}$ such that $C \vDash \varphi(\bar b_1', \bar b_2)$;
\item otherwise, choose $\bar b_2^\varphi(\bar b_1')\in B_2^{|\bar y|}$ with $C \vDash \varphi(\bar b_1', \bar b_2^\varphi(\bar b_1'))$.
\end{enumerate}
By Lemma~\ref{lem:hyppushouts}\ref{lem:hyppushoutslaec} we can use the $\mrm{LS}$ axiom of $\kaecgoth_\lambda^\rmod$ inside $B_2$ to obtain ${A_i\leqp^\lambda B_2}$ such that 
\[
\bigcup\{\bar b_2^\varphi(\bar b_1')\mid \varphi(\bar x, \bar y)\text{ a }\lambda\mh\pp\text{-formula and }\bar b_1'\in (B_{1, i}')^{|\bar x|}\}\cup \bigcup_{j<i} A_j \subseteq A_i
\]
\nin and $|A_i| \leq (|B_1| + |R|)^{<\lambda}$. Fix $j<i$, we have $A_j\leq A_i$. Notice that $A_i \leqp^\lambda B_2$, and $A_j \leqp^\lambda B_2$ by hypothesis. Therefore, using coherence we get $A_j \leqp^\lambda A_i$ for $j<i$.

\ssk\nin \begin{enumerate}[resume*=prf:lpureindeplocalf]
\item We now show how to construct  $B_{1, i}'$, for $0 < i< \lambda$.
\end{enumerate} 
For $i<\lambda$, assume $(B_{1, j}')_{j<i}$ and $(A_j)_{j<i}$ have already been defined.  We apply the $\mrm{LS}$ axiom of $\kaecgoth_\lambda^\rmod$ inside $C$ to obtain $B_{1, i}'\leqp^\lambda C$ such that $\bigcup_{j<i} (A_j \cup B_{1, j}') \subseteq B_{1, i}'$ and $|B_{1,i}'|  \leq (|B_1| + |R|)^{<\lambda}$. As in \ref{prf:lpureindeplocalf2}, by coherence $B_{1, j}' \leqp^\lambda B_{1, i}'$ for $j<i$.

\smallskip \noindent \begin{enumerate}[resume*=prf:lpureindeplocalf]
\item \label{prf:lpureindeplocalf4}The chains $(A_i)_{i<\lambda}$ and $(B_{1, i}')_{i<\lambda}$ are as desired.
\end{enumerate} 
Why \ref{prf:lpureindeplocalf4}?
By construction, \ref{cond:lchar1}-\ref{cond:lchar4} are satisfied. 

\ssk\nin\uline{It is enough.}
\begin{enumerate}[resume*=prf:lpureindeplocalf]
\item \label{prf:lpureindeplocalf5}$A = \bigcup_{i<\lambda} A_i$ and $B_1' = \bigcup_{i<\lambda}B_{1,i}'$ are as desired in the statement of the lemma. 
\end{enumerate} 
The rest of the proof will deal with showing \ref{prf:lpureindeplocalf5}.
 \begin{enumerate}[resume*=prf:lpureindeplocalf]
\item \label{prf:lpureindeplocalf6.1} Items \ref{lem:lpureloccharfund1}-\ref{lem:lpureloccharfund2} of the statement hold.
\end{enumerate}
Why \ref{prf:lpureindeplocalf6.1}? We have $B_1\leq B_1'$ by condition \ref{cond:lchar1.5} of the construction. Moreover,  by $\lambda$-smoothness of $\kaecgoth_\lambda^\rmod$ and by conditions \ref{cond:lchar1} and \ref{cond:lchar2} of the construction, we have $A\leqp^\lambda B_2$ and $B_1'\leqp^\lambda C$.
 \begin{enumerate}[resume*=prf:lpureindeplocalf]
\item \label{prf:lpureindeplocalf6.2} Item \ref{lem:lpureloccharfund3} holds.
\end{enumerate}
Why \ref{prf:lpureindeplocalf6.2}? Immediate from condition~\ref{cond:lchar3} of the construction.
 \begin{enumerate}[resume*=prf:lpureindeplocalf]
\item \label{prf:lpureindeplocalf7}Item \ref{lem:lpureloccharfundbox} holds.
\end{enumerate} 
Why \ref{prf:lpureindeplocalf7}? Assume $\varphi(\bar x, \bar y)$ and $\bar b_1', \bar b_2$ are as in \ref{lem:lpureloccharfundbox}. Because $|\bar b_1'|< \lambda$, there is $i<\lambda$ such that $\bar b_1'\in (B_{1,i}')^{|\bar x|}$, and using \ref{cond:lchar4} we are done. 

\ssk\nin This finishes the proof of \ref{prf:lpureindeplocalf5}.
\end{proof}

We now show:
\begin{proof}[Proof of Lemma~\ref{lem:lpureindeplocal}]
Let $B_1, B_2, C\in \kaec$ with $B_1, B_2 \leqp^\lambda C$, we want to find $B_1', A\in \kaec$ such that $A \leqp^\lambda B_1'$, $B_1\leqp^\lambda B_1'$, $|A| +|B_1'|\leq (|B_1| + |R|) ^{<\lambda}$, and $B_1' \dnf_A^C B_2$. 
\begin{enumerate}[$(*_1)$, leftmargin=*, series=prf:lpureindeplocal]
\item \label{prf:lpureindeplocal4} We can find $R$-modules $A, B_1'$ such that the following conditions are satisfied:
\begin{enumerate}[$(\cdot_1)$, leftmargin=*] 
\item\label{prf:lpureindeplocal4.2} $A\leqp^\lambda B_2$;
\item \label{prf:lpureindeplocal4.3}$B_1'\leqp^\lambda C$;
\item \label{prf:lpureindeplocal4.1}$|A| + |B_1'| \leq (|B_1| + |R|)^{<\lambda}$;
\item \label{prf:lpureindeplocal4.6}if $\varphi(\bar x, \bar y)$ is a $\lambda\mh\pp$-formula, $\bar b_1'\in (B_1')^{|\bar x|}$, and there is $\bar b_2 \in B_2^{|\bar y|}$ such that $C \vDash \varphi(\bar b_1', \bar b_2)$, then there is $\bar a\in A^{|\bar y|}$ such that $C \vDash \varphi(\bar b_1', \bar a)$;
\item \label{prf:lpureindeplocal4.4}$A, B_1'\in \kaec$;
\item \label{prf:lpureindeplocal4.45} $B_1\leqp^\lambda B_1'$;
\item \label{prf:lpureindeplocal4.5}$A\leqp^\lambda B_1'$.
\end{enumerate}
\end{enumerate}
Why \ref{prf:lpureindeplocal4}? Apply Lemma~\ref{lem:lpureloccharfund} to $B_1, B_2\leq C$ to have $A\leq B_1'\leq C$ which satisfy conditions \ref{lem:lpureloccharfund1}-\ref{lem:lpureloccharfundbox} of Lemma~\ref{lem:lpureloccharfund}, so that conditions \ref{prf:lpureindeplocal4.2}-\ref{prf:lpureindeplocal4.6}  hold automatically. Because $B_2, C\in \kaec$, we have \ref{prf:lpureindeplocal4.4} by \ref{prf:lpureindeplocal4.2}-\ref{prf:lpureindeplocal4.3} and closure of $\kaec$ under $\lambda$-pure submodules. Item \ref{prf:lpureindeplocal4.45} follows immediately from coherence, because $B_1\leqp^\lambda C$ and $B_1 \leq B_1'\leqp^\lambda C$.
Finally, we show  \ref{prf:lpureindeplocal4.5}.
Recall that we have $A \leq B_1'$, and by hypothesis $B_2\leqp^\lambda C$. Thus, also using  \ref{prf:lpureindeplocal4.2}-\ref{prf:lpureindeplocal4.3},  we have $A\leqp^\lambda B_2\leqp^\lambda C$ and $A\leq B_1'\leqp^\lambda C$, which implies $A\leqp^\lambda B_1'$ by coherence. This finishes the proof of \ref{prf:lpureindeplocal4}.

\ssk\nin We are only left to show:
\begin{enumerate}[resume*=prf:lpureindeplocal]
\item \label{prf:lpureindeplocal5}We have $B_1' \dnf^{C}_{A} B_2$.
\end{enumerate} 
The rest of the proof will deal with showing \ref{prf:lpureindeplocal5}. Recall that by Proposition~\ref{prop:widepushoutexist} the pushout of $B_1'$ and $B_2$ over $A$ is given by $P = (B_1' \oplus B_2)/N$, where $N = \ab\{(a, -a)\mid a\in A\}$. We have to show that the unique $R$-module homomorphism  $r: P \to C$ given by $r([(b_1', b_2)]_N) = b_1' + b_2$ is a $\lambda$-pure embedding. 
\begin{enumerate}[resume*=prf:lpureindeplocal]
\item \label{prf:lpureindeplocal6}$r: P \rightarrow C$ is an embedding. 
\end{enumerate} 
We verify \ref{prf:lpureindeplocal6}. Take $b_1'\in B_1'$ and $b_2\in B_2$ and assume $r([(b_1', b_2)]_N) = b_1' + b_2 = 0$, so that 
\begin{equation}
\label{eq:lchar1}
C \vDash b_1' + b_2 = 0.
\end{equation}
Applying~\ref{prf:lpureindeplocal4}\ref{prf:lpureindeplocal4.6} to \eqref{eq:lchar1}, there is $a\in A$ such that 
\[
C\vDash b_1' + a = 0
\]
Therefore $b_1' = -b_2 = -a$, and by definition of $N$ it follows $[(b_1', b_2)]_N = 0$.

\nin
\begin{enumerate}[resume*=prf:lpureindeplocal]
\item \label{prf:lpureindeplocal7}$r: P \rightarrow C$ is a $\lambda$-pure embedding. 
\end{enumerate}
We verify \ref{prf:lpureindeplocal7}. We use the syntactic characterization of $\lambda$-purity of  Proposition~\ref{prop:equivlpurelppform}\ref{prop:equivlpurelppform2}.
Let $\varphi(\bar x)$ a $\lambda\mh\pp$-formula and assume
\[
C \vDash \varphi(\bar b_1' + \bar b_2)
\]
with $\bar b_1'\in (B_1')^{|\bar x|}$ and $\bar b_2\in B_2^{|\bar x|}$. We have to show that
\begin{equation}
\label{eq:lcharfinal}
P \vDash \varphi([(\bar b_1', \bar b_2)]_N).
\end{equation}
Let $\psi(\bar z, \bar z')$ be the formula:
\[
\varphi(\bar z + \bar z').
\]
Then $\psi(\bar z, \bar z')$ is a $\lambda\mh\pp$-formula, and obviously we have that
\begin{equation}
\label{eq:lchar2}
C \vDash \psi(\bar b_1', \bar b_2).
\end{equation}
Now, applying~\ref{prf:lpureindeplocal4}\ref{prf:lpureindeplocal4.6} to \eqref{eq:lchar2}, there is $\bar a\in A^{|\bar z'|}$ such that $C \vDash \psi(\bar b_1', \bar a)$, so we have:
\begin{equation}
\label{eq:lchar2.5}
C \vDash \varphi(\bar b_1' + \bar a).
\end{equation} 
Because $\varphi(\bar x)$ is a $\lambda\mh\pp$-formula, it is of the form $\exists\bar y\, \theta(\bar x, \bar y)$ with $\theta(\bar x, \bar y)$ a system of $<\lambda$ equations.
Notice that $\bar b_1'+ \bar a\in (B_1')^{|\bar x|}$, and by $B_1' \leqp^\lambda C$ applied to \eqref{eq:lchar2.5} there is  $\bar b^\star \in (B_1')^{|\bar y|}$ such that 
\begin{equation}
\label{eq:lchar3}
C \vDash \theta(\bar b_1' + \bar a, \bar b^\star).
\end{equation}
By Remark~\ref{rem:lpuretrivia}\ref{rem:lpuretrivia2} the set of solutions of a $\lambda\mh\pp$-formula is a subgroup, thus
$C\vDash \varphi(\bar b_1' + \bar b_2)$ and $C \vDash \varphi(\bar b_1' + \bar a)$ implies $C \vDash \varphi(\bar b_2 - \bar a)$.
Notice that $\bar b_2 - \bar a\in B_2^{|\bar x|}$ and, because $B_2\leqp^\lambda C$, there is $\bar c^\star\in B_2^{|\bar y|}$ such that
\begin{equation}
\label{eq:lchar4}
C \vDash \theta(\bar b_2 - \bar a, \bar c^\star).
\end{equation}
As $\theta(\bar x, \bar y)$ is a system of equations, we can sum  \eqref{eq:lchar3} and \eqref{eq:lchar4} to have
\begin{equation}
\label{eq:lchar5}
C \vDash \theta(\bar b_1' + \bar b_2, \bar b^\star + \bar c^\star).
\end{equation}
Because $\bar b^\star \in (B_1')^{|\bar y|}$ and $\bar c^\star\in B_2^{|\bar y|}$, we have that, projecting \eqref{eq:lchar5} on $P = (B_1'\oplus B_2)/N=(B_1'\oplus B_2)/\{(a, -a)\mid a\in A\}$, we get \eqref{eq:lcharfinal}.
\end{proof}

Finally, we are only left to show that $\dnf$ has the right witness property. Just as in the case of local character, this is the straightforward generalization of the result of \cite[Lemma~3.10]{mazarinoetherian}, which dealt with the case $\lambda = \aleph_0$.

\begin{lemma}
\label{lem:lpureindepwitness}
Assume $\kaecgoth_\lambda=(\kaec, \leqp^\lambda)$ satisfies Hypothesis~\ref{hyp:lambdapure}, then $\dnf$ has the right $({<}\lambda)$-witness property (cf.~\ref{def:indepwitnessprop}). 
\end{lemma}
\begin{proof} 
Let $A\leqp^\lambda B_\ell \leqp^\lambda C$, with $\ell\in \{1, 2\}$, and assume that for every $X\subseteq B_2$ with $|X|<\lambda$ one has $\dnfb{A}{B_1}{X}{C}$, we want to show that $B_1 \dnf_A^C B_2$. Recall that by Proposition~\ref{prop:widepushoutexist} the pushout of $B_1$ and $B_2$ over $A$ is given by $P =(B_1\oplus B_2)/N$, where $N = \{(a, -a)\mid a\in A\}$. We have to prove that the unique $R$-module homomorphism $r: P \to C$ given by $r([(b_1, b_2)]_N) = b_1 + b_2$ is a $\lambda$-pure embedding.
\begin{enumerate}[$(*_1)$, leftmargin=*]
\item\label{prf:lpurewitnprop1}  $r: P \to C$ is an embedding.
\end{enumerate} 
We verify \ref{prf:lpurewitnprop1}. Take $b_1\in B_1$ and $b_2\in B_2$ with $r([(b_1, b_2)]_N) = 0$. Let $X = \{b_2\}$, then $\dnfb{A}{B_1}{X}{C}$. Therefore, there are $D_1, D_2\leqp^\lambda E$ in $\kaec$ such that $B_1\subseteq D_1$, $X\subseteq D_2$, $C\leqp^\lambda E$, and $D_1\dnf_{A}^{E}D_2$. If $Q=(D_1\oplus D_2)/\{(a, -a)\mid a\in A\} =(D_1\oplus D_2)/N$ is the pushout of $D_1$ and $D_2$ over $A$, we have the following commutative diagram:
\[\begin{tikzcd}
	&& E \\
	{D_1} & Q \\
	A & E
	\arrow[curve={height=-12pt}, from=2-1, to=1-3]
	\arrow["{h_1}", from=2-1, to=2-2]
	\arrow["s", from=2-2, to=1-3]
	\arrow[from=3-1, to=2-1]
	\arrow[from=3-1, to=3-2]
	\arrow[curve={height=12pt}, from=3-2, to=1-3]
	\arrow["{h_2}"', from=3-2, to=2-2]
\end{tikzcd}\]
with $s: Q \to E$ a $\lambda$-pure embedding. But $s([(b_1, b_2)]_N) = r([(b_1, b_2)]_N) = 0$, and because $s$ is an embedding we have $[(b_1, b_2)]_N = 0$.

\ssk\nin \begin{enumerate}[$(*_2)$, leftmargin=*]
\item\label{prf:lpurewitnprop2} $r: P \to C$ is a $ \lambda$-pure embedding. 
\end{enumerate} 
We verify \ref{prf:lpurewitnprop2}. We use the characterization of Proposition~\ref{prop:equivlpurelppform}\ref{prop:equivlpurelppform3}. Let $\varphi(\bar x)$ be a basic $\lambda\mh\pp$-formula and suppose $C \vDash \varphi(\bar b_1 + \bar b_2)$ for $\bar b_\ell\in B_\ell$, with $\ell\in\{1,2\}$. We have to show:
\begin{equation}
\label{eq:lpwitn0}
P \vDash \varphi([(\bar b_1, \bar b_2)]_N). 
\end{equation}
\nin Let $X$ be the set of entries of $\bar b_2$. Clearly $|X| < \lambda$, so that we can find $D_1, D_2, E, Q$ as before with the unique map $s: Q \to E$ a $\lambda$-pure embedding. Because $C\vDash \varphi(\bar b_1 + \bar b_2)$ and $C\leqp^\lambda E$, then $E\vDash \varphi(\bar b_1 + \bar b_2)$. Since $s: Q\to E$ is a $\lambda$-pure embedding with $s([(d_1, d_2)]_N) = d_1 + d_2$ for $d_\ell \in D_\ell$, $\ell\in\{1,2\}$, then 
\begin{equation}
\label{eq:lpwitn1}
Q \vDash \varphi([(\bar b_1, \bar b_2)]_N).
\end{equation}
Suppose $\varphi(\bar x)$ is the formula:
\[
\exists (y_j)_{j\in J}\bigwedge_{k\in K} \sum_{j\in J} r_{jk} y_j = x_k,
\]
and define $\psi(\bar x, \bar z)$ to be the formula:
\[
\exists (y_j)_{j\in J}\bigwedge_{k\in K} \sum_{j\in J} r_{jk} y_j = x_k + z_k.
\]
We have $Q=(D_1\oplus D_2)/N$, so that \eqref{eq:lpwitn1} implies that there are $(a_k)_{k\in K}$ in $A$ such that, denoting $\bar a = (a_k)_{k\in K}$, then
\begin{equation}
\label{eq:lpwitn2}
D_1\oplus D_2 \vDash \psi((\bar b_1, \bar b_2), (\bar a,-\bar a)).
\end{equation}

\smallskip \noindent We have $B_1\leqp^\lambda C \leqp^\lambda E$ and $B_1\leq D_1\leqp^\lambda E$, so by coherence $B_1\leqp^\lambda D_1$. Projecting \eqref{eq:lpwitn2} on $D_1$ and using $B_1\leqp^\lambda D_1$, we get
\begin{equation}
\label{eq:lpwitn3}
B_1\vDash \psi(\bar b_1, \bar a).
\end{equation}
Projecting \eqref{eq:lpwitn2} on $D_2$, we get $D_2 \vDash \psi(\bar b_2,-\bar a)$. Using $D_2\leqp^\lambda E$ and $B_2\leqp^\lambda C\leqp^\lambda E$, we get 
\begin{equation}
\label{eq:lpwitn4}
B_2 \vDash \psi(\bar b_2, -\bar a).
\end{equation}
Putting \eqref{eq:lpwitn3} and \eqref{eq:lpwitn4} together gives us 
\begin{equation}
\label{eq:lpwitn5}
B_1\oplus B_2 \vDash \psi((\bar b_1, \bar b_2), (\bar a, -\bar a)),
\end{equation}
and projecting \eqref{eq:lpwitn5} on $P=(B_1\oplus B_2)/N=(B_1\oplus B_2)/\{(a,-a)\mid a\in A\}$ we finally get \eqref{eq:lpwitn0}.
\end{proof}

\begin{theorem}
\label{thm:lpurestableind}
Assume $\kaecgoth_\lambda$ satisfies Hypothesis~\ref{hyp:lambdapure}, then $\dnf$ is a stable independence relation (cf.~\ref{def:stableindep}). In particular, $\kaecgoth_\lambda$ is tame (cf.~\ref{def:tame}) and stable (cf.~\ref{def:stabilityproper}\ref{def:stabilityproper2}).
\end{theorem}
\begin{proof}
That $\dnf$ is a stable independence relation follows from Lemmas~\ref{lem:lpureindepweaklystable},~\ref{lem:lpureindeplocal} and \ref{lem:lpureindepwitness}. Tameness and stability follow from Fact~\ref{fact:stableindep}.
\end{proof}

Now we are going to compute the cardinals in which $\kaecgoth_\lambda$ is stable, when $\kaecgoth_\lambda$ satisfies Hypothesis~\ref{hyp:lambdapure}. The proof follows the \enquote{standard proof} of \cite[Theorem 4.17]{somestablenonelementary}.

\begin{theorem}
\label{thm:lpurestabcard}
Assume $\kaecgoth_\lambda=(\kaec, \leqp^\lambda)$ satisfies Hypothesis~\ref{hyp:lambdapure}, let $\lambda_0 = |R|^{<\lambda}$ and $\kappa^{\lambda_0} = \kappa$. Then $\kaecgoth_\lambda$ is $\kappa$-stable.
\end{theorem}
\begin{proof}
Assume for a contradiction that there exist $A\in \kaec$ with $|A| = \kappa$ and a sequence $(p_i\mid i< \kappa^+)$ of pairwise distinct orbital types over $A$. By Lemma~\ref{lem:lpureindeplocal} we can find $B_i\leqp^\lambda A$ in $\kaec$ such that $p_i$ does not fork over $B_i$ (cf.~\ref{def:dnftype}) and $|B_i|\leq |R|^{<\lambda}=\lambda_0$. Because $\kappa^{\lambda_0} =\kappa$, by the pigeon-hole principle we can assume there is $B\leqp^\lambda A$ such that $B = B_i$ for every $i<\kappa^+$. Notice that for every $i<\kappa^+$ one has $p_i\restriction B\in \gS_{\kaecgoth_\lambda}(B)$, and by Fact~\ref{fact:cardorbtypes} we have $|\gS_{\kaecgoth_\lambda}(B)| \leq 2^{\lambda_0}\leq \kappa^{\lambda_0} = \kappa$.
Therefore, there must be $i<j<\kappa^+$ with $p_i\restriction B = p_j\restriction B$. But $p_i$ and $p_j$ do not fork over $B$, and by Fact~\ref{fact:indeptypes}\ref{fact:indeptypesuniq} we have $p_i =p_j$, a contradiction. 
\end{proof}
Summing up the results of this subsection we have the following:
\begin{theorem}
\label{thm:lpuresumup}
Let $\lambda_0 = |R|^{<\lambda}$ and   $\kappa= \kappa^{\lambda_0}$. If $\kaecgoth_\lambda= (\kaec, \leqp^\lambda)$ satisfies Hypothesis~\ref{hyp:lambdapure}, then $\kaecgoth_\lambda$ is a $\lambda\mh\mrm{AEC}$, it is $\kappa$-stable,  tame, has $\mrm{AP}$, and $\mrm{JEP}$. Moreover:
\begin{enumerate}[(1), leftmargin=*]
\item \label{thm:lpuresumup1}If $\kaec$ is closed under arbitrary direct sums, then $\kaecgoth_\lambda$ has $\infty\mh\mrm{AP}$.
\item \label{thm:lpuresumup3} If $\nu\geq \lambda$, then $\kaecgoth_\nu = (\kaec, \leqp^\nu)$ also satisfies Hypothesis~\ref{hyp:lambdapure}.
\end{enumerate} 
\end{theorem}
\begin{proof}
Everything, with the exception of ~\ref{thm:lpuresumup3}, follows from Lemma~\ref{lem:hyppushouts}, Theorems~\ref{thm:lpurestableind},~\ref{thm:lpurestabcard}. Finally, item \ref{thm:lpuresumup3} is  Remark~\ref{rem:hyplpure}.
\end{proof}

\subsection{Examples}
\label{sect:lpureexample}

Throughout this subsection we follow Notation~\ref{not:hyplpure}. In this subsection we are going to give examples of classes satisfying Hypothesis~\ref{hyp:lambdapure}, so that in particular they satisfy the hypotheses of Theorem~\ref{thm:lpuresumup}. By this same theorem, when such classes are closed under arbitrary direct sums, they also satisfy $\infty\mh\mrm{AP}$.

Let us start out with some concrete  examples of classes of $R$-modules which satisfy Hypothesis~\ref{hyp:lambdapure}.

\begin{example}
\label{ex:ringlpure}
The following classes satisfy Hypothesis~\ref{hyp:lambdapure} and are closed under arbitrary direct sums:
\begin{enumerate}[(1), leftmargin=*]
\item \label{ex:ringlpure1}$(\rmod, \leqp^\lambda)$, where $\lambda \geq \aleph_0$.
\item \label{ex:ringlpure1.5} $(R\flatmod, \leqp^\lambda)$, where $R\flatmod$ is the class of flat $R$-modules and $\lambda \geq \aleph_0$. An $R$-module $A$ is flat if $A\otimes_R -$ is an exact functor \cite[Section~3.3]{rotman}. By \cite[Proposition~4.4]{lecturesonmodules} the class of flat modules is closed under direct limits, and by  \cite[Corollary~4.86]{lecturesonmodules} the class of flat modules is closed under pure submodules and pure quotients. Finally, it is closed under arbitrary direct sums by \cite[Proposition~4.2]{lecturesonmodules}.
\end{enumerate}
\end{example}

The proof of the following example uses advanced module-theoretic tools. Namely, we use the fact that the flat cotorsion theory is cogenerated by a set, which was crucially used  by Enochs to solve the flat cover conjecture \cite{flatcoverconj}. Because the proof is rather long, and uses standard homological machinery, we refer the interested reader to  Appendix~\ref{sect:appcotorsion}.

\begin{example}
\label{ex:lpurecotorsion}
$(R\mh\cotorsion, \leqp^\lambda)$ satisfies Hypothesis~\ref{hyp:lambdapure}, where $R\mh\cotorsion$ is the class of cotorsion $R$-modules and $\lambda \geq \ab(|R|+\aleph_0)^+$. An $R$-module $A$ is cotorsion if $\Ext_R^1(F, A) = 0$\footnote{\label{fn:ext}Recall that $\Ext(F, A) =\Ext^1_R(F, A)= 0$ if and only if any short exact sequence of the form ${0 \to A \to B \to F \to 0}$ is split \cite[Theorem~7.31]{rotman}.} for every flat $R$-module $F$ \cite[Section~4.6]{purityspectra}.
\end{example}
\begin{proof}
This follows immediately from Proposition~\ref{prop:cotapphyp}.
\end{proof}

We will see other classes of $R$-modules which satisfy Hypothesis~\ref{hyp:lambdapure} when we will study classes of pure-injective modules in Subsection~\ref{sect:pinjmodules} (see Lemma~\ref{lem:dsumhyplpurehyp}). 
Finally, we end this section by giving many examples of classes of abelian groups which satisfy Hypothesis~\ref{hyp:lambdapure}. 

\begin{example}
\label{ex:hyplpure}
The following classes of abelian groups satisfy Hypothesis~\ref{hyp:lambdapure} and are closed under arbitrary direct sums:
\begin{enumerate}[(1), leftmargin=*]
\item \label{ex:hyplpure2}$(\tfab, \leqp^\lambda)$, where $\tfab$ denotes the class of torsion-free abelian groups and $\lambda \geq \aleph_0$.
\item \label{ex:hyplpure3}$(\rgrp, \leqp^\lambda)$, where $\rgrp$ is the class of reduced abelian groups and $\lambda \geq \aleph_1$. An abelian group is reduced if it contains no divisible subgroup.
\item \label{ex:hyplpure5}$ (\rpgrp, \leqp^\lambda)$, where $\rpgrp$ is the class of reduced $p$-groups and $\lambda \geq \aleph_1$.
\item \label{ex:hyplpure6}$(\nu\almostfree, \leqp^\lambda)$, where $\nu\almostfree$ is the class of $\nu$-free abelian groups and $\lambda \geq \nu$. An abelian group is $\nu$-free if all of its $({<}\nu)$-generated subgroups are free.
\item \label{ex:hyplpure7}$ (\cotfree, \leqp^\lambda)$, where $\cotfree$ is the class of cotorsion-free groups and $\lambda \geq (2^{\aleph_0})^+$. A group is cotorsion-free if it contains no non-zero cotorsion subgroups \cite[Chapter~V, Definition~2.8]{almost}.
\item \label{ex:hyplpure8}$ (\slendergrp, \leqp^\lambda)$, where $\slendergrp$ is the class of slender groups and $\lambda \geq (2^{\aleph_0})^+$. An abelian group $A$ is slender if for any homomorphism $\eta: \mathbb Z^{\omega}\to A$ then $\eta(e_n) = 0$ for cofinitely many $n<\omega$, where $e_n$ is the $n$-th canonical vector \cite[Chapter~12, Section~2]{fuchs}.
\end{enumerate}
\end{example}

We now start showing that the classes of Example~\ref{ex:hyplpure} actually satisfy Hypothesis~\ref{hyp:lambdapure}. First, we prove the following useful lemma, which is interesting in its own right. 

\begin{lemma}
\label{lem:sufficienthyplpure}
Let $\lambda$ be an infinite regular cardinal and $\collmods$ a collection of $\lambda$-presented $R$-modules. Let $\kaec^\collmods$ denote the class of $R$-modules which do not contain an isomorphic copy of a non-zero module in $\collmods$. Then the following hold: 
\begin{enumerate}[(1), leftmargin=*]
\item \label{lem:sufficienthyplpure1} $\kaecgoth^{\collmods}_\lambda =(\kaec^\collmods, \leqp^\lambda)$ satisfies  all of the conditions of  Hypothesis~\ref{hyp:lambdapure} with possibly the exception of closure under finite direct sums.
\item \label{lem:sufficienthyplpure2} If $\collmods$ contains a copy of every epimorphic image of all of its elements, then $\kaecgoth^\collmods_\lambda$ satisfies Hypothesis~\ref{hyp:lambdapure} and is closed under arbitrary direct sums.
\end{enumerate}
\end{lemma}
\begin{proof}
Before starting, notice that:
\begin{enumerate}[$(*_1)$, leftmargin=*, series=suffhyplp]
\item \label{prf:suffhyplp1}Without loss of generality we can assume $\collmods$ to be  closed under isomorphism. 
\end{enumerate}
We show~\ref{lem:sufficienthyplpure1}.
\begin{enumerate}[resume*=suffhyplp]
\item \label{prf:suffhyplp2}$\kaecgoth_\lambda^\collmods$ is closed under $\lambda$-directed systems. 
\end{enumerate}
Why \ref{prf:suffhyplp2}? If $(A_i)_{i\in I}$ is a $\lambda$-directed system in $\kaecgoth_\lambda^\collmods$, let $A= \bigcup_{i\in I} A_i$ and assume for a contradiction that there is a non-zero $C\leq A$ with $C\in\collmods$. Because $C$ is $\lambda$-presented, then it is generated by $<\lambda$ elements. Therefore, there is $i\in I$ such that $C \leq A_i$, which is a contradiction. This shows \ref{prf:suffhyplp2}.
\begin{enumerate}[resume*=suffhyplp]
\item \label{prf:suffhyplp2.5}$\kaec^\collmods$ is closed under arbitrary submodules. 
\end{enumerate}
Why \ref{prf:suffhyplp2.5}? Obvious.
\begin{enumerate}[resume*=suffhyplp]
\item \label{prf:suffhyplp3} $\kaec^\collmods$ is closed under $\lambda$-pure quotients. 
\end{enumerate}
Why \ref{prf:suffhyplp3}? Assume $A\leqp^\lambda B$, and there is $A \leq C \leq B$ such that $C/A$ is non-zero and is in $\collmods$, so that $C/A$ is $\lambda$-presented.  Therefore, by Fact~\ref{fact:lpuresplit} we have $A\leqo C$, so that $A\oplus C/A \cong C\leq B$, but this is a contradiction because $B\in \kaec^\collmods$. This finishes the proof of \ref{prf:suffhyplp3} and \ref{lem:sufficienthyplpure1}.

\ssk\nin Now we show~\ref{lem:sufficienthyplpure2}, by \ref{prf:suffhyplp1} we can assume that $\collmods$ is closed under epimorphic images, i.e., any epimorphic image of an element of $\collmods$ is again in $\collmods$. We are left to show:
\begin{enumerate}[resume*=suffhyplp]
\item \label{prf:suffhyplp4} $\kaec^\collmods$ is closed under arbitrary direct sums. 
\end{enumerate}
Why \ref{prf:suffhyplp4}? Assume $(A_i)_{i\in I}$ is a collection of $R$-modules with $A_i \in \kaec^\collmods$ for every $i\in I$, and there is a non-zero $C\leq \bigoplus_{i\in I} A_i$ with $C\in \collmods$. For every $i\in I$ let $C_i$ be the projection of $C$ onto $A_i$, so that by closure under epimorphic images we have $C_i\in \collmods$. Since $A_i\in \kaec^\collmods$, $C_i = \{0\}$ for every $i\in I$. But $C$ is non-zero, a contradiction. This finishes the proof of \ref{lem:sufficienthyplpure2}.
\end{proof}

\begin{proof}[Proof of Example~\ref{ex:hyplpure}]
We wish to apply  Lemma~\ref{lem:sufficienthyplpure}, so that for each class $\kaecgoth_\lambda = (\kaec, \leqp^\lambda)$ of Example~\ref{ex:hyplpure} we will find a class $\collmods$ of $\lambda$-presented groups such that $\kaec = \kaec^\collmods$.
We will use without mention the fact, mentioned in Remark~\ref{rem:lambdapresented}\ref{rem:lambdapresented2}, that an abelian group is $\lambda$-presented if and only if it is $({<}\lambda)$-generated.
\begin{enumerate}[(1), leftmargin=*]
\item Let $\collmods(\tfab)$ be the closure under isomorphism of $\{\mathbb Z(p)\mid p\text{ prime}\}$, then $\tfab = \kaec^{\collmods(\tfab)}$. Clearly $\collmods(\tfab)$ is also closed under epimorphic images, and we can apply Lemma~\ref{lem:sufficienthyplpure}\ref{lem:sufficienthyplpure2}.
\item Let $\collmods(\rgrp)$ be the class of countable divisible groups, this class is clearly closed under epimorphic images and $\rgrp = \kaec^{\collmods(\rgrp)}$, so that we can use  Lemma~\ref{lem:sufficienthyplpure}\ref{lem:sufficienthyplpure2}.
\item To show that $(\rpgrp, \leqp^\lambda)$ satisfies Hypothesis~\ref{hyp:lambdapure} for $\lambda \geq \aleph_1$, notice that $(\pgrp, \leqp^\lambda)$, where $\pgrp$ denotes the class of $p$-groups, clearly satisfies Hypothesis~\ref{hyp:lambdapure} and is closed under arbitrary direct sums. Finally, the  result follows from \ref{ex:hyplpure3} and Remark~\ref{rem:hyplpure}\ref{rem:intersecthyp}.
\item Let $\collmods(\nu\almostfree)$ be the class of $({<}\nu)$-generated non-free abelian groups. We have $\nu\almostfree = \kaec^{\collmods(\nu\almostfree)}$, so that we can apply Lemma~\ref{lem:sufficienthyplpure}\ref{lem:sufficienthyplpure1}. Finally, the class of $\nu$-free abelian groups is easily verified to be closed under arbitrary direct sums.
\item Let $\collmods(\cotfree)=\{\mathbb Q\}\cup \{\mathbb Z(p), J_p\mid p\text{ prime}\}$. We have that $\cotfree = \kaec^{\collmods(\cotfree)}$ by  \cite[Chapter~V, Theorem~2.9]{almost}, and we can apply Lemma~\ref{lem:sufficienthyplpure}\ref{lem:sufficienthyplpure1}. Moreover, the class of cotorsion-free groups is closed under transfinite extensions by \cite[Corollary 4.3]{unionslenderfuchsgobel}, and thus it is closed under arbitrary direct sums.
\item Let $\collmods(\slendergrp)=\{\mathbb Q, \mathbb Z^\omega\}\cup \{\mathbb Z(p), J_p\mid p\text{ prime}\}$. We have $\slendergrp = \kaec^{\collmods(\slendergrp)}$ by \cite[Chapter~IX, Corollary 2.4]{almost}, so that we can apply Lemma ~\ref{lem:sufficienthyplpure}\ref{lem:sufficienthyplpure1}. Moreover, the class of slender groups is closed under arbitrary direct sums by \cite[Chapter~III, Corollary~1.10]{almost}.
\end{enumerate}
\end{proof}

\subsection{An application: balanced embeddings}
\label{sect:lpurebal}

In this subsection we introduce the relation of being a balanced subgroup between torsion-free abelian groups. 
The aforementioned relation plays a vital role in the theory of torsion-free abelian groups \cite[Chapter~12]{fuchs}; for example, the completely decomposable groups are exactly the balanced-projective torsion-free groups~\cite[Chapter~12, Theorem~3.1]{fuchs}.
We will show that many natural classes of torsion-free abelian groups turn out to be stable $\aleph_1\mh\mrm{AEC}$s when considered with the relation of being a balanced subgroup. To show this we use the results obtained in the previous subsections and the almost stability transfer, i.e.,  Corollary~\ref{cor:alstabtransfermu}.  Nonetheless, a lengthier proof, see Appendix~\ref{sect:appstabletfab},  actually shows that the classes considered in this subsection have a stable independence relation, so that in particular they are stable and tame.

\begin{definition}[\protect{\cite[Chapter~12.2]{fuchs}}]
\label{def:proper}
Let $B$ be a torsion-free abelian group and $A\leqp B$.
\begin{enumerate}[(1), leftmargin=*]
\item\label{def:proper1} An element $b\in B$ is called \tdef{proper with respect to $A$} if $\chi(a+b) \leq \chi(b)$ for all $a\in A$ (cf.~\ref{def:charac}). Equivalently, $b$ is proper with respect to $A$ if we have:
$$\chi_{B/A}([b]_A) = \chi_B(b).$$
\item \label{def:proper2}We say that $A$ is \tdef{balanced} in $B$ (written $A\leqb B$) if every coset of $A$ in $B$ contains an element proper with respect to $A$. Equivalently, $A\leqb B$ if for every $[b]_A\in B/A$ there is $b'\in [b]_A$ such that $\chi_{B/A}([b]_A) = \chi_B(b')$.
\end{enumerate}
\end{definition}

The following characterization will be especially useful in the following:

\begin{fact}[\protect{\cite[Chapter~12, Lemmas~2.2-2.3]{fuchs}}]
\label{fact:balequiv}
Let $A\leq B$ be torsion-free abelian groups. The following are equivalent:
\begin{enumerate}[(1), leftmargin=*]
\item\label{fact:balequiv1} $A$ is balanced in $B$.
\item\label{fact:balequiv2} $A$ is pure in $B$ and for every $b\in B$ and every $(a_n)_{n\in \omega}\in A^\omega$ there is $a\in A$ such that $\chi(b+ a_n) \leq \chi (a + a_n)$ for every $n\in \omega$.
\item\label{fact:balequiv3} $A$ is pure in $B$ and for every rank one torsion-free abelian group\footnote{Recall that the \tdef{rank} of a torsion-free abelian group is the maximal size of a linearly independent subset \cite[Chapter~3, Section~4]{fuchs}. Moreover, the torsion-free abelian groups of rank one are, up to isomorphism, exactly the subgroups of $\mathbb{Q}$ by \cite[Chapter~12, Theorem~1.1]{fuchs}.} $C$ and $\gamma: C \to B/A$, then there is $\alpha: C \to B$ with $\gamma = \pi \alpha$, where $\pi: B \to B/A$ is the natural projection.
\end{enumerate}
\end{fact}

\begin{remark}
\label{rem:baltrivia}
\begin{enumerate}[(1), leftmargin=*]
\item \label{rem:baltrivia1}From Fact~\ref{fact:balequiv}\ref{fact:balequiv2} it follows that if $A,B,C$ are torsion-free groups with $A\leq B\leq C$ and $A\leqb C$, then $A\leqb B$.
\item \label{rem:balaleph1pure}By Fact~\ref{fact:lpuresplit}, Fact~\ref{fact:balequiv}\ref{fact:balequiv3}, and the fact that every rank-one group is countably generated, we have that $\leqb$ refines $\leqp^{\aleph_1}$.
\end{enumerate}
\end{remark}
\begin{notation}
\label{not:tcharsetkaec}
Let $\tcharset$ be a set of types (cf.~\ref{def:abtypes}\ref{def:abtypes1}). $\kaec^\tcharset$ denotes the class of torsion-free abelian groups $A$ with $\mrm{Type}(A)\subseteq \tcharset$ (cf.~\ref{def:abtypes}\ref{def:abtypes3}). If $\tcharset = \{\tchar\}$ we  simply write $\kaec^{\tchar}$, which is the class of $\tchar$-homogeneous groups (cf. \ref{def:abtypes}\ref{def:abtypes4}).
\end{notation}

When $\chi_1, \chi_2$ are characteristics (cf.~\ref{def:charac}), we define $\chi_1\wedge\chi_2$ as the pointwise minimum. It is easily verified that if $\tchar_1,\tchar_2$ are two types, then the operation $\tchar_1\wedge \tchar_2$ is also well-defined \cite[pg.~411]{fuchs}. For ease of notation, we introduce the following:

\begin{definition}
\label{def:tcharsetclosed}
We say that a set of types $\tcharset$ is \tdef{closed} if $\tchar_1, \tchar_2\in \tcharset$ implies $\tchar_1\wedge \tchar_2\in\tcharset$.
\end{definition}

The examples which motivate Definition~\ref{def:tcharsetclosed} and the following discussion are the following:
\begin{example}
\label{ex:tcharset}
It is easily verified that the following sets of types are closed:
\begin{enumerate}[(1), leftmargin=*]
\item $\tcharset_{\tfab}$, the set of all possible types. Here $\kaec^{\tcharset_{\tfab}} = \tfab$.
\item $\tcharset_{\redtfab} = \tcharset_{\tfab}\setminus \{(\infty, \infty,\ldots)\}$. Here $\kaec^{\tcharset_{\redtfab}} = \redtfab$, where $\redtfab$ denotes the class of reduced torsion-free abelian groups (cf.~\ref{ex:hyplpure}\ref{ex:hyplpure3}).
\item $\tcharset_\tchar = \{\tchar\}$, with $\tchar$ a type. Here $\kaec^{\tcharset_\tchar} =\tchar\mh\tfab$, where $\tchar\mh\tfab$ denotes the class of $\tchar$-homogeneous torsion-free abelian groups (cf.~\ref{def:abtypes}\ref{def:abtypes4}).
\end{enumerate}
\end{example}

Coming back to abstract classes, we now verify that $(\kaec^\tcharset, \leqb)$ is an $\aleph_1\mh\mrm{AEC}$ and it is almost stable. First, we are going to need the following:
\begin{remark}
\label{rem:tcharsetbalhyp}
Let $\tcharset$ be a closed set of types, then $\kaec^\tcharset$ is closed under:
\begin{enumerate}[(1), leftmargin=*]
\item \label{rem:tcharsetbalhyp1} Finite direct sums;
\item \label{rem:tcharsetbalhyp2} Pure subgroups;
\item \label{rem:tcharsetbalhyp3} Balanced quotients, i.e., if $A,B\in \kaec^\tcharset$ and $A\leqb B$, then $B/A\in\kaec^\tcharset$. 
\end{enumerate}
\end{remark}
\begin{proof}
We show \ref{rem:tcharsetbalhyp1}. This is the only point where we use  that $\tcharset$ is closed. If $A=B\oplus C$, $B, C\in \kaec^\tcharset$, $b\in B$, and $c\in C$, then by \cite[Chapter~12, Section~1(e)]{fuchs} we have $\tchar(b+c) = \tchar(b) \wedge \tchar(c)\in \tcharset$.

\ssk\nin Item \ref{rem:tcharsetbalhyp2} follows from Fact~\ref{fact:pureiffchar}.

\ssk\nin We show \ref{rem:tcharsetbalhyp3}. If $A\leqb B$ and $B\in \kaec^\tcharset$, then by Definition~\ref{def:proper} for $b\in B$ there is $b'\in [b]_A$ such that $\chi_{B/A}([b']_A) =\chi_B(b')$. Therefore, $\tchar_{B/A}([b']_A) = \tchar_B(b')\in \tcharset$.
\end{proof}

\begin{proposition}
\label{prop:tcharsetaleph1aec}
Let $\tcharset$ be a closed set of types. Then $\kaecgoth_\bal^\tcharset = (\kaec^\tcharset, \leqb)$  is an $\aleph_1\mh\mrm{AEC}$ with $\mrm{LS}_{\aleph_1}(\kaecgoth_\bal^\tcharset) \leq 2^{\aleph_0}$. Moreover, $\kaecgoth_\bal^\tcharset$ has continuity (cf.~\ref{def:propertiesac}\ref{def:mucontinuity}) and is almost $\kappa$-stable (cf.~\ref{def:stability}) in every $\kappa = \kappa^{2^{\aleph_0}}$.
\end{proposition}
\begin{proof}
We divide the proof in different steps.
\begin{enumerate}[$(*_1)$, leftmargin=*, series=prf:typesaec]
\item\label{prf:typesaec1} $\kaecgoth_\bal^\tcharset$ has continuity.
\end{enumerate}
Why \ref{prf:typesaec1}? Let $(A_i)_{i\in I}$ be a directed system in $\kaecgoth_\bal^\tcharset$. Let $a\in A=\bigcup_{i\in I}A_i$, then there is an $i\in I$ such that $a\in A_i\leqp A$. By Fact~\ref{fact:pureiffchar} we have $\tchar_A(a) = \tchar_{A_i}(a)\in \tcharset$, so that $A\in \kaec^\tcharset$. Moreover, $A_i\leqb A$ for every $i\in I$ by the equivalent formulation of balancedness given in Fact~\ref{fact:balequiv}\ref{fact:balequiv2}.
\begin{enumerate}[resume*=prf:typesaec]
\item \label{prf:typesaec2} $\kaecgoth_\bal^\tcharset$ is an $\aleph_1\mh\mrm{AEC}$ with $\mrm{LS}_{\aleph_1}(\kaecgoth_\bal^\tcharset) \leq 2^{\aleph_0}$.
\end{enumerate}
Why \ref{prf:typesaec2}? $\aleph_1$-continuity follows from \ref{prf:typesaec1}. $\aleph_1$-smoothness (cf.~\ref{def:propertiesac}\ref{def:musmoothness}) and coherence (cf.~\ref{def:propertiesac}\ref{def:coherence}) both follow from the equivalent definition in Fact~\ref{fact:balequiv}\ref{fact:balequiv2}. The $\mrm{LS}$ axiom follows from the $\mrm{LS}$ axiom of $(\abgrps, \leqp^{\aleph_1})$ by Remark~\ref{rem:baltrivia}\ref{rem:balaleph1pure}.
\begin{enumerate}[resume*=prf:typesaec]
\item \label{prf:typesaec4} $\kaecgoth_\bal^\tcharset$ has $\mrm{JEP}$.
\end{enumerate}
Why \ref{prf:typesaec4}? Because $\kaec^\tcharset$ is closed under direct sums by Remark~\ref{rem:tcharsetbalhyp}\ref{rem:tcharsetbalhyp1}.
\begin{enumerate}[resume*=prf:typesaec]
\item \label{prf:typesaec5} $\kaecgoth_\bal^\tcharset$ is almost $\kappa$-stable in every $\kappa = \kappa^{2^{\aleph_0}}$.
\end{enumerate}
Why \ref{prf:typesaec5}? Recall from Remark~\ref{rem:baltrivia}\ref{rem:balaleph1pure} that $\leqb$ refines $\aleph_1$-purity. Therefore, by Theorem~\ref{thm:lpuresumup} and the almost stability transfer, i.e.,  Corollary~\ref{cor:alstabtransfermu},  it will be enough to show:
\begin{enumerate}[resume*=prf:typesaec]
\item \label{prf:typesaec6}  $(\kaec^\tcharset, \leqp^{\aleph_1})$ satisfies Hypothesis~\ref{hyp:lambdapure}.
\end{enumerate}
Why \ref{prf:typesaec6}? By \ref{prf:typesaec1} we have that $(\kaec^\tcharset, \leqp^{\aleph_1})$ is closed under $\aleph_1$-directed systems, and the other requirements follow from Remark~\ref{rem:tcharsetbalhyp}.
\end{proof}
We want to improve Proposition~\ref{prop:tcharsetaleph1aec} to show that $(\kaec^\tcharset, \leqb)$ is stable. We know from Proposition~\ref{prop:tcharsetaleph1aec}  that $\kaecgoth^\tcharset_\bal$ is almost $\kappa$-stable for every $\kappa=\kappa^{2^{\aleph_0}}$, so that by  Lemma~\ref{lemma:equivasamalg} it will be enough to prove:

\begin{lemma}
\label{lem:tcharsetinftyap}
Let $\tcharset$ be a closed set of types, then $\kaecgoth_\bal^\tcharset = (\kaec^\tcharset, \leqb)$ has $\infty\mh\mrm{AP}$ and is closed under pushouts, i.e., if $(f_1: A \to B_1, f_2: A \to B_2)$ is a pair of {$\kaecgoth_\bal^\tcharset$-embeddings} with pushout $(g_1: B_1\to P, g_2: B_2 \to P)$ in $\abgrps$, then $P\in\kaec^\tcharset$ and $g_1, g_2$ are $\kaecgoth_\bal^\tcharset$-embeddings.
\end{lemma}

To show Lemma~\ref{lem:tcharsetinftyap} we will have to talk about pushouts of balanced embeddings. First, we start with a rather technical proposition, which the reader  should compare with Proposition~\ref{prop:wpushoutclosedlambda}.

\begin{proposition}
\label{prop:wpushoutclosedbal}
Let $A, B_1, B_2$ be torsion-free abelian groups. Assume that ${(g_1: B_1\to P, g_2: B_2 \to P)}$ is the pushout in $\abgrps$ of the homomorphisms $(f_1: A \to B_1,\ab f_2: A \to B_2)$, and $f_2$ is a balanced embedding. We have the following:
\begin{enumerate}[(1), leftmargin=*]
\item \label{prop:wpushoutclosedbal1}$P$ is torsion-free and $g_1$ is a balanced embedding.
\item \label{prop:wpushoutclosedbal2} Assume $(g_1: B_1 \to P, g_2: B_2 \to P)$ is the pushout of $(f_1,f_2)$ as in Proposition~\ref{prop:widepushoutexist}, so that $P = (B_1\oplus B_2)/N$ with $N = \{(f_1(a), -f_2(a))\mid a\in A\}$. Then $N\leqb B_1\oplus B_2$.
\end{enumerate}
\end{proposition}
\begin{proof}
Before starting with the proof of \ref{prop:wpushoutclosedbal1}, let us start with some remarks which will be used throughout the proof.

\ssk\nin
\begin{enumerate}[$(*_1)$, leftmargin=*, series=prf:wpbal]
\item \label{prf:wpbal2}Without loss of generality we can assume that $(g_1, g_2)$ is as in Proposition~\ref{prop:widepushoutexist}, so that in particular $P=(B_1\oplus B_2)/N$, where $N$ is as in~\ref{prop:wpushoutclosedbal2},  $g_1(b_1) = [(b_1, 0)]_N$ and $g_2(b_2) = [(0, b_2)]_N$ for $b_\ell\in B_\ell$, with $\ell\in\{1,2\}$. 
\end{enumerate}
\begin{enumerate}[resume*=prf:wpbal]
\item\label{prf:wpbal2.5} 
The homomorphism $q_2 : P \to B_2/f_2(A)$ with $q_2([(b_1, b_2)]_N) = [b_2]_{f_2(A)}$ is well defined. 
\end{enumerate} 
Why \ref{prf:wpbal2.5}? By a direct calculation using that $P = (B_1\oplus B_2)/N$, where $N = \{(f_1(a), -f_2(a))\mid a\in A\}$.
\begin{enumerate}[resume*=prf:wpbal]
\item\label{prf:wpbal3} Let $\pi_2: B_2 \to B_2/f_2(A)$ denote the natural projection. Then ${q_2g_2 = \pi_2}$.
\end{enumerate}
Why \ref{prf:wpbal3}? Follows immediately from the definition.
\begin{enumerate}[resume*=prf:wpbal]
\item\label{prf:wpbal4} $0\to B_1 \overset{g_1}{\to} P \overset{q_2}{\to}B_2/f_2(A) \to 0$ is an exact sequence.
\end{enumerate}
Why \ref{prf:wpbal4}? Clearly $q_2g_1 = 0$. If $q_2([(b_1, b_2)]_N) = 0$, there is $a\in A$ with $b_2 = f_2(a)$, so that $g_1(b_1 + f_1(a)) = [(b_1 + f_1(a), 0)]_N = [(b_1 + f_1(a), 0) + (-f_1(a), f_2(a))]_N  = [(b_1, b_2)]_N$. This shows \ref{prf:wpbal4}.

\ssk\nin We start showing~\ref{prop:wpushoutclosedbal1}.
\begin{enumerate}[resume*=prf:wpbal]
\item\label{prf:wpbal5}$N$ is pure in $B_1\oplus B_2$ and $g_1$ is a pure embedding.
\end{enumerate}
Why \ref{prf:wpbal5}? By Proposition~\ref{prop:wpushoutclosedlambda}. 
\begin{enumerate}[resume*=prf:wpbal]
\item\label{prf:wpbal6}$P$ is torsion-free.
\end{enumerate}
Why \ref{prf:wpbal6}? Follows immediately from \ref{prf:wpbal5} because $B_1\oplus B_2$ is torsion-free.
\begin{enumerate}[resume*=prf:wpbal]
\item\label{prf:wpbal7} $g_1$ is a balanced embedding.
\end{enumerate}
To show \ref{prf:wpbal7} we  use Fact~\ref{fact:balequiv}\ref{fact:balequiv3} and the exact sequence of  \ref{prf:wpbal4}. By \ref{prf:wpbal5} $g_1$ is a pure embedding, and we are only left to prove:
\begin{enumerate}[resume*=prf:wpbal] 
\item \label{prf:wpbal8}Let $\gamma: C \to B_2/f_2(A)$ be a homomorphism with $C$ a rank one torsion-free abelian group, there is $\alpha: C \to P$ with $q_2 \alpha = \gamma$. 
\end{enumerate}
Why \ref{prf:wpbal8}? We have that $f_2: A \to B_2$ is a balanced embedding, so by Fact~\ref{fact:balequiv}\ref{fact:balequiv3} there is $\alpha_2: C \to B_2$ such that $\pi_2 \alpha_2 = \gamma$. Let $\alpha=g_2\alpha_2: C \to P$ then
\[
q_2 \alpha = q_2g_2\alpha_2 =\pi_2\alpha_2 = \gamma,
\]
where we have used \ref{prf:wpbal3} in the second equality. This ends the proof of \ref{prf:wpbal8} and \ref{prop:wpushoutclosedbal1}.

\ssk\nin We show~\ref{prop:wpushoutclosedbal2}, again using Fact~\ref{fact:balequiv}\ref{fact:balequiv3}. By \ref{prf:wpbal5} we have $N\leqp B_1\oplus B_2$, and we are only left to show:
\begin{enumerate}[resume*=prf:wpbal]
\item \label{prf:wpbal9}Let $\gamma: C \to (B_1\oplus B_2)/N$ be a homomorphism with $C$ a rank one torsion-free abelian group, there is $\alpha: C \to B_1\oplus B_2$ such that $\pi\alpha = \gamma$, where $\pi: B_1\oplus B_2 \to P$ is the natural projection.
\end{enumerate}
Why \ref{prf:wpbal9}? $f_2: A \to B_2$ is a balanced embedding and $q_2\gamma: C\to B_2/f_2(A)$, so that we can find $\alpha_2: C \to B_2$ such that $\pi_2\alpha_2=q_2\gamma$. Using \ref{prf:wpbal3}, we have:
\begin{equation}
\label{eq:poutbal3} 
q_2(g_2 \alpha_2-\gamma) = \pi_2 \alpha_2 - q_2\gamma = 0.
\end{equation} 
Recalling the exact sequence of \ref{prf:wpbal4}, then  \eqref{eq:poutbal3} implies that there  is $\alpha_1: C \to B_1$ such that $g_1 \alpha_1= \gamma - g_2\alpha_2$. Let $\alpha = (\alpha_1, \alpha_2): C \to B_1\oplus B_2$, we finally have 
\[\pi\alpha = \pi(\alpha_1, \alpha_2) = g_1\alpha_1 + g_2\alpha_2 = (\gamma - g_2 \alpha_2) + g_2\alpha_2 = \gamma,
\]
and thus we are done.
\end{proof}

\begin{corollary}
\label{cor:poutbalanced}
Pushouts in $\abgrps$ are closed under balanced embeddings, i.e., for any pair of balanced embeddings $(f_1: A\to B_1, f_2: A \to B_2)$ with pushout $(g_1: B_1 \to P, g_2: B_2 \to P)$, then $P$ is torsion-free and $g_1,\,g_2$ are balanced embeddings.
\end{corollary}
\begin{proof}
Follows from Proposition~\ref{prop:wpushoutclosedbal}\ref{prop:wpushoutclosedbal1}.
\end{proof}

We are now ready to prove Lemma~\ref{lem:tcharsetinftyap}:

\begin{proof}[Proof of Lemma~\ref{lem:tcharsetinftyap}]
By Proposition~\ref{prop:tcharsetaleph1aec} we have that $\kaecgoth^\tcharset_\bal$ has continuity, so that by  Lemma~\ref{lem:chainboundsinftyap} it is enough to show $\mrm{AP}$. In particular, we will show closure of $\kaecgoth^\tcharset_\bal$ under pushouts.

\ssk\nin Let $A, B_1, B_2\in\kaec^\tcharset$, and consider a pair of balanced embeddings $(f_1: A \to B_1, f_2: A \to B_2)$, with pushout $(g_1: B_1\to P, g_2: B_2\to P)$. By Corollary~\ref{cor:poutbalanced} we know that $g_1,\,g_2$ are balanced, so that it is enough to show $P\in\kaec^\tcharset$. Clearly we can assume that $P = (B_1\oplus B_2)/N$ is just as in Proposition~\ref{prop:wpushoutclosedbal}. By Remark~\ref{rem:tcharsetbalhyp}\ref{rem:tcharsetbalhyp1} we have that $B_1\oplus B_2\in \kaec^\tcharset$. By Proposition~\ref{prop:wpushoutclosedbal}\ref{prop:wpushoutclosedbal2} we have $N\leqb B_1\oplus B_2$, and by Remark~\ref{rem:tcharsetbalhyp}\ref{rem:tcharsetbalhyp3} we have $P=(B_1\oplus B_2)/N\in \kaec^\tcharset$.
\end{proof}

Putting together everything we have proved so far, we have:

\begin{theorem}
\label{thm:balansumup}
Let $\tcharset$ be a closed set of types. Then the $\aleph_1\mh\mrm{AEC}$ $\kaecgoth_\bal^\tcharset = (\kaec^\tcharset, \leqb)$ has $\infty\mh\mrm{AP}$, $\mrm{JEP}$, and is $\kappa$-stable for every $\kappa = \kappa^{2^{\aleph_0}}$.
\end{theorem}
\begin{proof}
With the exception of $\infty\mh\mrm{AP}$ and stability, everything follows from Proposition~\ref{prop:tcharsetaleph1aec}. 
That the class has $\infty\mh\mrm{AP}$ is by Lemma~\ref{lem:tcharsetinftyap}, so we are only left to show $\kappa$-stability for $\kappa= \kappa^{2^{\aleph_0}}$. By Proposition~\ref{prop:tcharsetaleph1aec} $\kaecgoth_\bal^\tcharset$ is almost $\kappa$-stable, and the result follows from Corollary~\ref{cor:alstabtransfermu} because $\kaecgoth_\bal^\tcharset$ has $\infty\mh\mrm{AP}$.
\end{proof}

Recalling Example~\ref{ex:tcharset},  from Theorem~\ref{thm:balansumup} we immediately have:

\begin{corollary}
\label{cor:balstableeaxmple}
The class $(\kaec, \leqb)$ is a stable $\aleph_1\mh\mrm{AEC}$ with $\infty\mh\mrm{AP}$ and $\mrm{JEP}$ for the following choices of $\kaec$:
\begin{enumerate}[(1), leftmargin=*]
\item the class of torsion-free abelian groups;
\item the class of reduced torsion-free abelian groups;
\item the class of $\tchar$-homogeneous torsion-free abelian groups (cf.~\ref{def:abtypes}\ref{def:abtypes4}).
\end{enumerate}
\end{corollary}

\subsection{Relations that refine direct summands}

\label{sect:relrefine}
Consider a class of $R$-modules $\kaec$. It is natural to consider all the possible relations $\leqk$ for which $(\kaec, \leqk)$ forms a $\mu\mh\mrm{AEC}$ for some $\mu$ and $\leqk$ refines direct summands, i.e., for every $A,B\in\kaec$ we have that $A\leqo B$ implies $A\leqk B$. Let us first consider the easier case of $\kaec =\rmod$. 
We know that, for every infinite cardinal $\lambda$, the relation $\leqp^\lambda$ refines direct summands, and that a submodule is a direct summand if and only if it is a $\lambda$-pure submodule for every $\lambda$. It is therefore natural to ask whether the relations $\leqp^\lambda$ are cofinal among the relations $\leqk$ which refine direct summands and for which $(\rmod, \leqk)$ is a $\mu\mh\mrm{AEC}$ for some $\mu$, i.e., if $(\rmod, \leqk)$ is a $\mu\mh\mrm{AEC}$, then there is a $\lambda$ such that $\leqk$ refines $\leqp^\lambda$. In particular, this would show, by the results of Section~\ref{sect:lpure} (cf.~\ref{thm:lpuresumup}) and the almost stability transfer (cf.~\ref{cor:alstabtransfermu}),  that every such class $(\rmod, \leqk)$ is almost stable. In this subsection we make progress towards answering this question by showing that the relations $\leqp^\lambda$ are cofinal among the relations $\leqk$ for which $(\rmod, \leqk)$ is a $\mu\mh\mrm{AEC}$ which also  satisfies a natural stronger form of coherence, which we introduce in this paper.

\begin{definition}
\label{def:strongcoherence}
Let $\kaecgoth = (\kaec, \leqk)$ be an $\mrm{AC}$. We say that $\kaecgoth$ has \tdef{strong coherence} if whenever $A,B, C\in \kaec$ with $A\leqk C$ and $A\leq B\leq C$, then $A\leqk B$.
\end{definition}

Notice that the difference with coherence (cf.~\ref{def:propertiesac}\ref{def:coherence}) is that in \ref{def:strongcoherence} we do not require $B\leqk C$, but only $B\leq C$. All of the relations we have studied so far actually satisfy this stronger form of coherence:

\begin{example}
\label{ex:strcoh}
The following abstract classes satisfy strong coherence:
\begin{enumerate}[(1), leftmargin=*]
\item \label{ex:strcoh1}$(\rmod, \leqo)$, where $\leqo$ is the relation of being a direct summand;
\item \label{ex:strcoh2}$(\rmod, \leqp^\lambda)$ where $\leqp^\lambda$ is the relation of being a $\lambda$-pure submodule;
\item \label{ex:strcoh3}$(\tfab, \leqb)$, where $\leqb$ is the relation of being a balanced subgroup.
\end{enumerate}
\end{example}
\begin{proof}
Items \ref{ex:strcoh1}-\ref{ex:strcoh3} follow from  \ref{fact:dsumsolveeq}, \ref{rem:lpuretrivia}\ref{rem:lpuretrivia3}, and \ref{rem:baltrivia}\ref{rem:baltrivia1} respectively.
\end{proof}

The following easy lemma will be useful in the following.
\begin{lemma}
\label{lem:lsdirected}
Let $\kaecgoth= (\kaec, \leqk)$ be an $\mrm{AC}$ with $\mrm{LS}$ at $\kappa$ (cf.~\ref{def:propertiesac}\ref{def:LS}) and coherence (cf.~\ref{def:propertiesac}\ref{def:coherence}). Let $A\in \kaec$, then there is a $\kappa^+$-directed system $(A_i)_{i\in I}$ in $\kaecgoth$ such that $A = \bigcup_{i\in I} A_i$ and $|A_i| \leq \kappa$ for every $i\in I$. 
\end{lemma}
\begin{proof}
If $|A| \leq \kappa$ this is trivially satisfied, so assume $\kappa <|A|$. Let 
\[
I = \{B \in\kaec\mid B\leqk A\text{ and }|B| \leq \kappa\},
\]
where here $I$ is viewed as a poset with the ordering $\leqk$. We verify that $(B)_{B\in I}$ is the required $\kappa^+$-directed system.
\begin{enumerate}[$(*_1)$, leftmargin=*, series=prf:lsdirected]
\item \label{prf:lsdirected1} $I$ is $\kappa^+$-directed.
\end{enumerate}
Why \ref{prf:lsdirected1}? Take $J\subseteq I$ with $|J| \leq \kappa$. Then $|\bigcup_{B\in J}B| \leq \kappa$, so that by $\mrm{LS}$ at $\kappa$ we can find $C\in\kaec$ with $\bigcup_{B\in J} B \subseteq C\leqk A$ and $|C| \leq \kappa$. For every $B\in J$, we have that $B\leqk A$ and $C\leqk A$. Therefore, by coherence we have $B\leqk C$ for every $B\in J$, and thus we are done.
\begin{enumerate}[resume*=prf:lsdirected]
\item \label{prf:lsdirected2}$(B)_{B\in I}$ is a directed system in $\kaecgoth$.
\end{enumerate}
Why \ref{prf:lsdirected2}? This is obvious by the definition of the ordering on $I$.
\begin{enumerate}[resume*=prf:lsdirected]
\item \label{prf:lsdirected3} $\bigcup_{B\in I} B = A$.
\end{enumerate}
Why \ref{prf:lsdirected3}? If $x\in A$, because $\kaecgoth$ has $\mrm{LS}$ at $\kappa$, then there is $B\leqk A$ with $x\in B$ and $|B| \leq \kappa$.
\end{proof}

\begin{theorem}
\label{thm:lsrefinedsumscoh}
Let  $\kaecgoth = (\kaec, \leqk)$ be an abstract class of $R$-modules such that $\leqk$ refines direct summands and $\kappa\geq |R|$ is an infinite cardinal. Assume $\kaecgoth$ has $\mrm{LS}$ at $\kappa$, $\kappa^+$-smoothness, strong coherence, and $\kaec$ is closed under submodules, i.e., $A\leq B\in \kaec$ implies $A\in\kaec$. Then $\leqk$ refines $\kappa^+$-purity. 
\end{theorem}
\begin{proof}
Assume $A\leqp^{\kappa^+} B$. Let $(A_i)_{i\in I}$ and $(B_j)_{j\in J}$ be as in  Lemma~\ref{lem:lsdirected} for $A$ and $B$ respectively applied with the cardinal $\kappa$. Fix $i\in I$. Because $J$ is $\kappa^+$-directed and $|A_i|\leq \kappa$, there is $j\in J$ such that $A_i \leq B_j$. Notice that we have $A+B_j\in \kaec$ by closure under submodules of $\kaec$. We have $|(A+B_j)/A| \leq |B_j| \leq \kappa$, so that, because $|R| \leq \kappa$, we have that $(A+B_j)/A$ is $\kappa^+$-presented, and by Fact~\ref{fact:lpuresplit} $A\leqo A+B_j$, so that $A\leqk A+B_j$ because $\leqk$ refines direct summands.
Because $B_j\leqk B$, by strong coherence we have $B_j\leqk A+B_j\leq B$. But then ${A_i\leqk A \leqk A+ B_j}$ and $A_i\leq B_j \leqk A+B_j$, so that by  coherence $A_i\leqk B_j\leqk B$. But $(A_i)_{i\in I}$ is a $\kappa^+$-directed system in $\kaecgoth$ with $A = \bigcup_{i\in I}A_i$, so that by $\kappa^+$-smoothness we have $A \leqk B$.
\end{proof}

\begin{corollary}
\label{cor:lpurerefine}
Let $\kaecgoth = (\kaec, \leqk)$ be a $\mu\mh\mrm{AEC}$ of $R$-modules such that $\leqk$ refines direct summands, $\kaecgoth$ has strong coherence, and $\kaec$ is closed under submodules. Then $\leqk$ refines $\mrm{LS}_\mu(\kaecgoth)^+$-purity. 
\end{corollary}
\begin{proof}
By definition of $\mrm{LS}$ number we have $\mrm{LS}_\mu(\kaecgoth) \geq |R|+\mu$, so that we can apply Theorem~\ref{thm:lsrefinedsumscoh} with $\kappa=\mrm{LS}_\mu(\kaecgoth)$.
\end{proof}

Using the results of Section~\ref{sect:lpure} we now show how the result we have obtained can be used to get stability results for various abstract classes.

\begin{hypothesis}
\label{hyp:strongcoherence}
Let $\kaec$ be a class of $R$-modules satisfying the following:
\begin{enumerate}[(1), leftmargin=*]
\item $\kaec$ is closed under submodules and under isomorphisms;
\item there is an infinite regular cardinal $\lambda$ such that $\kaecgoth_\lambda = (\kaec, \leqp^\lambda)$ satisfies Hypothesis~\ref{hyp:lambdapure}.
\end{enumerate}
\end{hypothesis}

\begin{theorem}
\label{thm:astablecoherent}
Let $\kaecgoth = (\kaec, \leqk)$ be a $\mu\mh\mrm{AEC}$ such that $\leqk$ refines direct summands, $\kaecgoth$ has strong coherence, and $\kaec$ satisfies Hypothesis~\ref{hyp:strongcoherence}. Then $\kaecgoth$ is almost stable. Moreover, if $\kaecgoth$ satisfies $\infty\mh\mrm{AP}$, then it is  stable.
\end{theorem}
\begin{proof}
We have that $(\kaec, \leqp^\lambda)$ satisfies Hypothesis~\ref{hyp:lambdapure} for some infinite regular cardinal $\lambda$, and without loss of generality by Remark~\ref{rem:hyplpure}\ref{rem:hyplpurenu} we can assume  $\mrm{LS}_\mu(\kaecgoth)^+\leq \lambda$. Let $\lambda_0 = |R|^{<\lambda}$. By Theorem~\ref{thm:lpuresumup}, for every $\kappa = \kappa^{\lambda_0}$, the $\lambda\mh\mrm{AEC}$ $(\kaec, \leqp^\lambda)$ is $\kappa$-stable. Moreover, $\leqk$ refines $\leqp^\lambda$ by Corollary~\ref{cor:lpurerefine}. Therefore, because $\kappa = \kappa^{<\lambda}$, we have that $(\kaec, \leqk)$ is almost $\kappa$-stable by the almost stability transfer, i.e., Corollary~\ref{cor:alstabtransfermu}. Finally, the last assertion follows from Lemma~\ref{lemma:equivasamalg}.
\end{proof}

We end this subsection with some examples of classes satisfying Hypothesis~\ref{hyp:strongcoherence}.

\begin{example}
The following classes satisfy Hypothesis~\ref{hyp:strongcoherence}:
\begin{enumerate}[(1), leftmargin=*]
\item $\rmod$;
\item $\tfab$, where $\tfab$ is the class of torsion-free abelian groups, 
\item $\pgrp$, where $\pgrp$ is the class of $p$-groups.
\end{enumerate}
\end{example}

\section{Envelopes and classes of pure-injective modules}

\label{sect:envelopes}
\subsection{Injective and pure-injective envelopes}

In this subsection we will cover some general facts regarding injective and pure-injective modules which will be needed to show that the classes $(R\mh\inj, \leqo)$ and $(R\mh\pinj, \leqo)$ are $(|R|+\aleph_0)^+\mh\mrm{AEC}$s, where $R\mh\inj$ and $R\mh\pinj$ denote the injective and pure-injective $R$-modules respectively. In particular, for a ring $R$, we introduce two cardinal invariants $\gamma(R)$ and $\mu(R)$, both less than or equal to $(|R| + \aleph_0)^+$. We will see that a $\gamma(R)$-pure submodule or quotient of an injective $R$-module is again injective, and the class of injective $R$-modules is closed under $\gamma(R)$-directed systems. The cardinal invariant $\gamma(R)$ had already been introduced, with a slightly different definition, in \cite{eklofuniversal}. On the other hand, $\mu(R)$ is a novel cardinal invariant which we introduce in the present paper, and which provides closure properties for the class of pure-injective $R$-modules analogous to those of $\gamma(R)$ for the class of injective $R$-modules. The results of this subsection will then be used in Subsection~\ref{sect:pinjmodules} to obtain results regarding the stability of various classes of pure-injective $R$-modules.

\begin{notation}
We let $R\mh\inj$ denote the class of injective $R$-modules, and we let $R\mh\pinj$ denote the class of pure-injective $R$-modules (see Subsection~\ref{sect:moduleprelims} for the definitions).
\end{notation}

Throughout this subsection, we will use freely the following obvious remark.

\begin{remark}
\label{rem:relationcollapse}
\begin{enumerate}[(1), leftmargin=*]
\item The relations $\leq$ and $\leqo$ coincide on $R\mh\inj$, so that in particular $\leqp^\lambda$ coincides with $\leqo$ on $R\mh\inj$ for every infinite cardinal $\lambda$.
\item The relations $\leqp$ and $\leqo$ coincide on $R\mh\pinj$, so that in particular $\leqp^\lambda$ coincides with $\leqo$ on $R\mh\pinj$ for every infinite cardinal $\lambda$.
\end{enumerate}
\end{remark}

First, we recall what is known in the case of injective $R$-modules.  
The following definition is exactly the one given in \cite[Definition 2.3]{mazariparam} or \cite{eklofuniversal}, with the exception that we also require $\gamma(R)$ to be regular.

%

\begin{definition}
\label{def:gammar}
Let $R$ be a ring. We let $\gamma(R)$ denote the least infinite regular cardinal such that every left ideal of $R$ is $({<}\gamma(R))$-generated as an $R$-module. Clearly $\gamma(R) \leq (|R|+\aleph_0)^+$.
\end{definition}

We will use \ref{fact:triviainj2}-\ref{fact:triviainj3} of the following to show that $(R\mh\inj, \leqo)$ is a $\gamma(R)\mh\mrm{AEC}$. Recall the definition of the injective envelope of $A$, denoted as $E(A)$, from Fact~\ref{fact:triviainjenv}.

\begin{fact}\label{fact:triviainj}
\begin{enumerate}[(1), leftmargin=*]
\item \label{fact:triviainj1}Let $A$ be an $R$-module, then ${|\mrm{E}(A)| \leq (|A|+|R|)^{<\gamma(R)}}$. 
\item \label{fact:triviainj2}The class of injective $R$-modules is closed under $\gamma(R)$-directed systems.
\item \label{fact:triviainj3} The class of injective $R$-modules is closed under $\gamma(R)$-pure submodules.
\end{enumerate}
\end{fact}
\begin{proof}
\ref{fact:triviainj1} follows from \cite[Theorem 1]{eklofuniversal} adapted to our definition of $\gamma(R)$.

\ssk\nin We show \ref{fact:triviainj2}. Let $(A_i)_{i\in I}$ be a $\gamma(R)$-directed system of injective modules with union $A$. By Baer's criterion, that is Fact~\ref{fact:triviainjenv}\ref{fact:triviainjenvbaer}, it is enough to show that for every left ideal $J$ of $R$ and $f: J \to A$ an $R$-module homomorphism, there is a homomorphism $g: R \to A$ extending $f$. By the definition of $\gamma(R)$ there is a set $X\subseteq J$ such that $|X| < \gamma(R)$ and $X$ generates $J$ as an $R$-module. Because $I$ is $\gamma(R)$-directed, there is $i\in I$ with $f(X)\subseteq A_i$, and hence $f(J)\leq A_i$. Because $A_i$ is injective, again using Fact~\ref{fact:triviainjenv}\ref{fact:triviainjenvbaer},  then there is $g: R \to A_i$ which extends $f$. This finishes the proof of \ref{fact:triviainj2}.

\ssk\nin Finally,  \ref{fact:triviainj3} is an easy consequence of \cite[Lemma 2]{eklofuniversal}, adapted to our definition of $\gamma(R)$. 
\end{proof}

\begin{corollary}
\label{cor:injmuaec}
$(R\mh\inj, \leqo)$ is a $\gamma(R)\mh\mrm{AEC}$, where $R\mh\inj$ denotes the class of injective $R$-modules. In particular, it is a $(|R|+\aleph_0)^+\mh\mrm{AEC}$.
\end{corollary}
\begin{proof}
This follows from Fact~\ref{fact:triviainj}\ref{fact:triviainj2}-\ref{fact:triviainj3}, using that $(\rmod, \leqp^{\gamma(R)})$ is a $\gamma(R)\mh\mrm{AEC}$ by Proposition~\ref{prop:lpurelaec}, and that $\leqp^{\gamma(R)}$ and $\leqo$ coincide on $R\mh\inj$. 
\end{proof}

Now we shift our attention to pure-injective modules. We aim to show the analogues of Fact~\ref{fact:triviainj} and Corollary~\ref{cor:injmuaec} for the class of pure-injective $R$-modules.

\begin{definition}
Let $\mu$ be an infinite cardinal. We say that an $R$-module $A$ is \tdef{$\mu$-compact} if every finitely solvable system of $<\mu$ equations with constants in $A$ is solvable in $A$.
\end{definition}

The previous definition is the straightforward adaptation of \cite[Chapter~V, Definition~1.1]{almost} when $\mu$ is not a successor cardinal. The reader should keep in mind that what we here call \enquote{$\mu$-compact modules} are called \enquote{$\mu^{<}$-compact modules} in \cite[Definition 7.22]{jensen}, and that their definition is equivalent to ours by \cite[Proposition 7.28]{jensen}. Finally, what in \cite{jensen} are called \enquote{$\mu$-compact modules} are what we here would call \enquote{$\mu^+$-compact modules}.

Finally, notice that by Fact~\ref{fact:triviapinjective}\ref{fact:triviapinjectivealgcomp} every pure-injective module is algebraically compact, and thus $\mu$-compact for every infinite cardinal $\mu$. In fact, a converse holds if $\mu$ is large enough, as the following fact shows:

\begin{fact}[\protect{\cite[Proposition 7.29]{jensen}}]
\label{fact:lcompact}
An $R$-module is $(|R|+\aleph_0)^+$-compact if and only if it is pure-injective.
\end{fact}

Now we introduce a novel cardinal invariant of $R$, which is the analogue of $\gamma(R)$ (cf.~\ref{def:gammar}) for pure-injectivity:

\begin{definition}
\label{def:mur}
Let $R$ be a ring. We let $\mu(R)$ denote the least infinite regular cardinal such that every $\mu(R)$-compact left $R$-module is pure-injective. (Notice that Fact~\ref{fact:lcompact} ensures $\mu(R) \leq (|R|+\aleph_0)^+$ and so this is well-defined.)
\end{definition}

We recall that a ring $R$ is called \enquote{left pure-semisimple} if every left $R$-module is pure-injective, see \cite[Section~4.5.1]{purityspectra} for more on this class of rings.

\begin{remark}
\label{rem:lcomptrivia}
\begin{enumerate}[(1), leftmargin=*]
\item \label{rem:lcomptrivia1}Clearly every $R$-module is $\aleph_0$-compact, therefore $\mu(R) = \aleph_0$ if and only if every $R$-module is pure-injective, i.e., $R$ is left pure-semisimple.
\item \label{rem:lcomptrivia2} For every infinite cardinal $\mu$, the class of $\mu$-compact modules is closed under $\mu$-pure submodules. 
\end{enumerate}
\end{remark}

\begin{fact}[\protect{\cite[Theorem 7.31]{jensen}}]
\label{fact:pidmucompact}
If $R$ is a principal ideal domain, then $\mu(R)\leq(|\{I \subseteq R\mid I\text{ is a left ideal}\}|+\aleph_0)^+$. In particular, by Remark~\ref{rem:lcomptrivia}\ref{rem:lcomptrivia1} we have that $\mu(\mathbb{Z}) = \aleph_1$, because $\mathbb{Z}$ is not left pure-semisimple\footnote{Clearly, $\mathbb{Z}$ is not algebraically compact.}.
\end{fact}

We now state some closure properties of the class of pure-injective $R$-modules. These will be crucially used to show that $(R\mh\pinj, \leqo)$ forms a $\mu(R)\mh\mrm{AEC}$.

\begin{lemma}
\label{lem:closurepinj}
\begin{enumerate}[(1), leftmargin=*]
\item \label{lem:closurepinj1}The class of pure-injective $R$-modules is closed under $\mu(R)$-pure submodules.
\item \label{lem:closurepinj2}The class of pure-injective $R$-modules is closed under $\mu(R)$-directed systems.
\end{enumerate}
\end{lemma}
\begin{proof}
Item~\ref{lem:closurepinj1} follows from Remark~\ref{rem:lcomptrivia}\ref{rem:lcomptrivia2} and the definition of $\mu(R)$. 

\ssk\nin We show~\ref{lem:closurepinj2}. If $\mu(R) = \aleph_0$, then by Remark~\ref{rem:lcomptrivia}\ref{rem:lcomptrivia1} there is nothing to show. Assume $\aleph_1\leq \mu(R)$. Without loss of generality consider a $\mu(R)$-directed system $(M_i)_{i\in I}$ with $M_i \leq M_j$ for $i\leq j$. For $M=\bigcup_{i\in I}M_i$, we have to show that $M$ is pure-injective. Assume one has a set $\Phi(\bar x, \bar y)$ of $<\mu(R)$ equations, $\bar a\in M^{|\bar y|}$, and every finite subset of $\Phi(\bar x, \bar a)$ has a solution. The set of finite subsets of $\Phi$ has cardinality $\leq |\Phi|^{<\aleph_0} = |\Phi| + \aleph_0 < \mu(R)$. For any finite $\Phi_0\subseteq \Phi$, take $\bar b(\Phi_0)\in M^{|\bar x|}$ which solves $\Phi_0(\bar x, \bar a)$. We have 
\[
|\{\bar b(\Phi_0) \mid \Phi_0\subseteq \Phi\text{ is finite}\}|<\mu(R),
\]
and thus there is an $i\in I$ such that for every finite $\Phi_0\subseteq \Phi$ one has $\bar b(\Phi_0)\in M_i^{|\bar x|}$. Since $M_i$ is pure-injective, and $\Phi(\bar x, \bar a)$ is finitely solvable in $M_i$, by Fact~\ref{fact:triviapinjective}\ref{fact:triviapinjectivealgcomp} we have  that $\Phi(\bar x, \bar a)$ is solvable in $M_i\leq M$, and thus it is solvable in $M$.
\end{proof}

We are now able to show the following fact, which should be compared with the analogous result for injective envelopes mentioned in Fact~\ref{fact:triviainj}\ref{fact:triviainj1} and originally proved in \cite{eklofuniversal}.

\begin{corollary}
\label{cor:cardpinj}
If $A$ is an $R$-module, then $|\mrm{PE}(A)|\leq (|A| +|R|)^{<\mu(R)}$ (cf.~\ref{fact:existpenv}).
\end{corollary}
\begin{proof}
By Proposition~\ref{prop:lpurelaec} $(\rmod, \leqp^{\mu(R)})$ is a $\mu(R)\mh\mrm{AEC}$ with $\mrm{LS}$ number  less than or equal to $|R|^{<\mu(R)}$. Therefore, there is $A\leq B\leqp^{\mu(R)} \mrm{PE}(A)$ with $|B| \leq (|A| + |R|)^{<\mu(R)}$. By Lemma~\ref{lem:closurepinj}\ref{lem:closurepinj2}, we have that $B$ is pure-injective. Clearly $A\leqp B$ because $A\leqp\mrm{PE}(A)$, and by the definition of pure-injective envelope (cf.~\ref{def:purity}\ref{def:purity3}) we have $B=\mrm{PE}(A)$.
\end{proof}

We finally show that $(R\mh\pinj, \leqo)$ forms a $\mu(R)\mh\mrm{AEC}$. In fact, we will also show that the same holds for $(R\mh\pinj, \leq)$. This should be compared with Corollary~\ref{cor:injmuaec}, where it is shown that $(R\mh\inj, \leqo)$ forms a $\gamma(R)\mh\mrm{AEC}$, and clearly $\leq$ coincides with $\leqo$ on $R\mh\inj$.

\begin{corollary}
\label{cor:pinjmuaec}
$(R\mh\pinj, \leq)$ and $(R\mh\pinj, \leqo)$ are $\mu(R)\mh\mrm{AEC}$s, where $R\mh\pinj$ denotes the class of pure-injective $R$-modules. In particular, they are $(|R|+\aleph_0)^+\mh\mrm{AEC}$s.
\end{corollary}
\begin{proof}
All the axioms of a $\mu(R)\mh\mrm{AEC}$, with the exception of the $\mrm{LS}$ axiom, hold for $(R\mh\pinj, \leq)$ by Lemma~\ref{lem:closurepinj}\ref{lem:closurepinj2} and  because $(\rmod, \leq)$ is an $\mrm{AEC}$. Similarly,  $(R\mh\pinj, \leqo)$ is a $\mu(R)\mh\mrm{AEC}$ because $(\rmod, \leqp^{\mu(R)})$ is a $\mu(R)\mh\mrm{AEC}$ by Proposition~\ref{prop:lpurelaec}, and $\leqp^{\mu(R)}$ coincides with $\leqo$ on $R\mh\pinj$. Finally, it is enough to show the $\mrm{LS}$ axiom for $(R\mh\pinj, \leqo)$, but this follows from the $\mrm{LS}$ axiom of $(\rmod, \leqp^{\mu(R)})$ and Lemma~\ref{lem:closurepinj}\ref{lem:closurepinj1}.
\end{proof}

We end this subsection with a question. Notice that the definition of $\gamma(R)$ only mentions the ring $R$, while the definition of $\mu(R)$ is in terms of $R$-modules. If $R$ is a principal ideal domain, then Fact~\ref{fact:pidmucompact} gives us information on $\mu(R)$ from the structure of the ring $R$. Therefore, it is natural to ask the following.
\begin{question}
Is there a ring-theoretic characterization of the invariant $\mu(R)$?
\end{question}

\subsection{Classes of pure-injective modules}
\label{sect:pinjmodules}

In the previous subsection we showed that the classes $(R\mh\inj, \leqo)$ and ${(R\mh\pinj, \leqo)}$ are $(|R|+\aleph_0)^+\mh\mrm{AEC}$s. In this subsection we are going to show that they are stable. In fact, we will give a general criterion for a $\mu\mh\mrm{AEC}$ of the form $(\kaec, \leqo)$, where $\kaec$ is a class of pure-injective modules, to have a stable independence relation, thus being stable and tame.  We will then use the results of this section to obtain in Subsection~\ref{sect:appliedenvclasses} results about the stability of various abstract elementary classes of modules, and more generally abstract classes of modules.

We now sketch the proof strategy for the results of this section. We will show that if $(\kaec, \leqo)$ satisfies Hypothesis~\ref{hyp:dsums}, then for some infinite regular cardinal $\nu$ the class $(\kaec, \leqp^\nu)$ satisfies  Hypothesis~\ref{hyp:lambdapure}, so that the latter class has a stable independence relation by Theorem~\ref{thm:lpuresumup}.
Because $\kaec$ is a class of pure-injective modules, by Remark~\ref{rem:relationcollapse} the relations $\leqp^\nu$ and $\leqo$ coincide on $\kaec$, and this will imply that $(\kaec, \leqo)$ has a stable independence relation.

Throughout this subsection we will freely use the cardinal invariant $\mu(R)$ introduced in Definition~\ref{def:mur} of the previous subsection.

\begin{hypothesis}
\label{hyp:dsums}
Let $\mu$ be an infinite regular cardinal and $\kaecgoth_\oplus = (\kaec, \leqo)$ an abstract class such that:
\begin{enumerate}[(1), leftmargin=*]
\item $\kaec$ is a class of pure-injective $R$-modules;
\item $\kaec$ is closed under finite direct sums;
\item $\kaec$ is closed under direct summands;
\item $\kaecgoth_\oplus$ is closed under $\mu$-directed systems.
\end{enumerate}
\end{hypothesis}

\begin{notation}
\label{not:hypdsums}
\begin{enumerate}[(1), leftmargin=*]
\item When we say \enquote{$\kaecgoth_\oplus=(\kaec, \leqo)$ satisfies Hypothesis~\ref{hyp:dsums} for $\mu$}, we mean that $\mu$ is the infinite regular cardinal for which Hypothesis~\ref{hyp:dsums} holds. 
\item Most of the time, we are going to say \enquote{the $\mu\mh\mrm{AEC}$ $\kaecgoth_\oplus=(\kaec, \leqo)$ satisfies Hypothesis~\ref{hyp:dsums}}. In such cases we mean that $\mu$ is the cardinal for which Hypothesis~\ref{hyp:dsums} holds.
\end{enumerate}
\end{notation}

\begin{remark}
\label{rem:dsumsjoin}
\begin{enumerate}[(1), leftmargin=*]
\item \label{rem:dsumsjoin1}Clearly, if $\kaecgoth_\oplus=(\kaec, \leqo)$ is a $\mu\mh\mrm{AEC}$ and $\kaec$ is a class of pure-injective modules, then it is  closed under $\mu$-directed systems of pure embeddings. In such cases, to verify that $\kaecgoth_\oplus=(\kaec, \leqo)$ satisfies Hypothesis~\ref{hyp:dsums}, we only have to check closure of $\kaec$ under finite direct sums and direct summands. It will be useful to remember that the class of pure-injective modules is closed under finite direct sums and direct summands (cf.~\ref{fact:triviapinjective}\ref{fact:triviapinjectivesums}).
\item \label{rem:dsumsjoin2}Assume $\kaecgoth_\oplus^\ell= (\kaec^\ell, \leqo)$ satisfies  Hypothesis~\ref{hyp:dsums} for $\mu$, with $\ell\in\{1,2\}$. Then $\kaecgoth^3_\oplus = (\kaec^1\cap \kaec^2, \leqo)$ satisfies Hypothesis~\ref{hyp:dsums} for $\mu$.
\end{enumerate}
\end{remark}

\begin{lemma}
\label{lem:dsumhyplpurehyp}
Assume  $\kaecgoth_\oplus=(\kaec, \leqo)$ satisfies Hypothesis~\ref{hyp:dsums} for $\mu$ and $\nu = \mu(R) + \mu$. Then $\kaecgoth_\nu = (\kaec, \leqp^\nu)$ satisfies Hypothesis \ref{hyp:lambdapure}.
\end{lemma}
\begin{proof}
Notice that every module in $\kaec$ is pure-injective, so that $\leqp^\nu$ and $\leqo$ coincide on $\kaec$ (cf. Remark~\ref{rem:relationcollapse}). We verify Hypothesis \ref{hyp:lambdapure} for $\kaecgoth_\nu$ as in the statement. Closure of $\kaec$ under finite direct sums is clear. Because $\mu\leq \nu$, then closure of $\kaecgoth_\nu$ under $\nu$-directed systems is clear. Finally, closure under $\nu$-pure submodules and $\nu$-pure quotients follows from closure of $\kaec$ under direct summands, because $\leqp^\nu$ and $\leqo$ coincide.
\end{proof}

\begin{theorem}
\label{thm:dsumsumup}
Let $\mu$ be a regular infinite cardinal, $\mu_1 = |R|^{<(\mu(R) + \mu)}$, and ${\kappa = \kappa^{\mu_1}}$. If ${\kaecgoth_\oplus= (\kaec, \leqo)}$ satisfies Hypothesis~\ref{hyp:dsums} for $\mu$, then $\kaecgoth_\oplus = (\kaec, \leqo)$ is a $(\mu(R) +\mu)\mh\mrm{AEC}$ with a stable independence relation. Furthermore, it is $\kappa$-stable, tame, has $\mrm{AP}$, and $\mrm{JEP}$.
\end{theorem}
\begin{proof}
Let $\nu = \mu(R) + \mu$. Notice that for every $A, B\in \kaec$ then $A\leqp^\nu B$ if and only if $A\leqo B$.
Therefore, everything follows from Lemma~\ref{lem:dsumhyplpurehyp} and Theorem~\ref{thm:lpuresumup}. 
\end{proof}

\begin{remark}
If $\kaecgoth_\oplus = (\kaec, \leqo)$ satisfies Hypothesis~\ref{hyp:dsums}, the previous theorem ensures the existence of a stable independence relation on it. Recalling  Definition~\ref{def:dnflambdapure} and Theorem~\ref{thm:lpurestableind}, we can give an explicit description of the stable independence relation $\dnf$ on $\kaecgoth_\oplus$. For a commutative square $(f_1, f_2, g_1, g_2)$ in $\kaecgoth_\oplus$, we have that $(f_1, f_2, g_1, g_2)\in \dnf$ if and only if the unique $R$-module homomorphism $r: P \to C$ from the pushout $P$ of $(f_1, f_2)$ is a split embedding, where the maps are as follows:
\[\begin{tikzcd}
	&& C \\
	{B_1} & P \\
	A & {B_2}
	\arrow["{g_1}", curve={height=-12pt}, from=2-1, to=1-3]
	\arrow["{h_1}", from=2-1, to=2-2]
	\arrow["r", from=2-2, to=1-3]
	\arrow["\ulcorner"{anchor=center, pos=0.125, rotate=-90}, draw=none, from=2-2, to=3-1]
	\arrow["{f_1}", from=3-1, to=2-1]
	\arrow["{f_2}"', from=3-1, to=3-2]
	\arrow["{g_2}"', curve={height=12pt}, from=3-2, to=1-3]
	\arrow["{h_2}"', from=3-2, to=2-2]
\end{tikzcd}\] 
\end{remark}

We now turn to applications of the results of this section for classes of pure-injective modules. First, we consider the modules which are injective with respect to various kinds of morphisms.

\begin{example}
\label{ex:dsums}
The following abstract classes of pure-injective modules satisfy Hypothesis~\ref{hyp:dsums}:
\begin{enumerate}[(1), leftmargin=*]
\item \label{ex:dsums0}The $\mu(R)\mh\mrm{AEC}$ $(R\mh\pinj, \leqo)$, where $R\mh\pinj$ denotes the class of pure-injective {$R$-modules} and $\mu(R)$ is as in Definition~\ref{def:mur}. It satisfies the hypotheses by Corollary~\ref{cor:pinjmuaec} and Fact~\ref{fact:triviapinjective}\ref{fact:triviapinjectivesums}.
\item \label{ex:dsums1} The $\gamma(R)\mh\mrm{AEC}$ $(R\mh\inj, \leqo)$, where $R\mh\inj$ denotes the class of injective $R$-modules and $\gamma(R)$ is as in Definition~\ref{def:gammar}. It satisfies the hypotheses by Corollary~\ref{cor:injmuaec} and Fact~\ref{fact:triviainjenv}\ref{fact:triviainjenv1}.
\item \label{ex:dsums5}The $(|R|+\aleph_0)^+\mh\mrm{AEC}$ $(R\mh\rdinj, \leqo)$ satisfies Hypothesis~\ref{hyp:dsums}, where $R\mh\rdinj$ denotes the class of $\mrm{RD}$-injective $R$-modules. An $R$-module $A$ is $\mrm{RD}$-injective if $A\leq_{\mrm{RD}} B$\footnote{Recall that for two $R$-modules $A\leq B$ we say that $A\leq_{\mrm{RD}} B$ if and only if $A\cap rB = rA$ for every $r\in R$ (cf.~\cite{rdflatness}).} implies $A\leqo B$ (see \cite{rdflatness} for more on this class of modules). Clearly every pure embedding is an $\mrm{RD}$-embedding, so that every $\mrm{RD}$-injective module is pure-injective. The class $\rdinj$ is closed under finite direct sums and direct summands \cite[Proposition 4.4]{mazarinoetherian}. Closure under $(|R|+\aleph_0)^+$-directed systems follows, with an argument similar to the one of Fact~\ref{fact:triviainj}\ref{fact:triviainj2}, from the Baer-like criterion for $\mrm{RD}$-injective modules proved in \cite[Corollary~4.10]{mazarinoetherian}.
\end{enumerate}
\end{example}

Moreover, we are able to get stability results for the following very relevant classes of pure-injective modules.

\begin{example}
\label{ex:2dsums}
The following abstract classes of pure-injective modules satisfy Hypothesis~\ref{hyp:dsums}:
\begin{enumerate}[(1), leftmargin=*]
\item \label{ex:dsums2}The $\aleph_1\mh\mrm{AEC}$ $(R\mh\sigmapinj, \leqo)$, where $R\mh\sigmapinj$ denotes the class of $\Sigma$-pure-injective $R$-modules.
An $R$-module is $\Sigma$-pure-injective if arbitrary direct sums of it are pure-injective \cite[Section~4.4.2]{purityspectra}, equivalently if there is no infinite descending chain of $\pp$-definable subgroups\footnote{\label{fn:ppsub}Let $A$ be an $R$-module, a \tdef{$\pp$-definable subgroup} of $A$ is a subgroup of the form ${\varphi(A) =\{a\in A\mid A\vDash \varphi(a)\}}$ where $\varphi(x)$ is a $\pp$-formula (cf.~\ref{def:ppform},~\ref{rem:lpuretrivia}\ref{rem:lpuretrivia2}), also see  \cite[Section~1.1.1]{purityspectra}.} \cite[Theorem 4.4.5]{purityspectra}. Therefore, $R\mh\sigmapinj$ is closed under pure submodules and is an $\aleph_1\mh\mrm{AEC}$ with pure embeddings. Moreover, it is closed under finite direct sums \cite[Proposition 4.4.11]{purityspectra}.
\item \label{ex:dsums3}The $\aleph_1\mh\mrm{AEC}$ $(R\mh\sigmainj, \leqo)$, where $R\mh\sigmainj$ denotes the class of $\Sigma$-injective $R$-modules. An $R$-module is $\Sigma$-injective if arbitrary direct sums of it are injective \cite[Section~4.4.2]{purityspectra}, equivalently if it is injective and $\Sigma$-pure-injective \footnote{In \cite[Lemma 4.4.16]{purityspectra} it is shown that a module is $\Sigma$-injective if and only if it is $\Sigma$-pure-injective and absolutely pure (a module $A$ is absolutely pure if $A\leq B$ implies $A\leqp B$). But clearly the pure-injective absolutely pure modules are exactly the injective ones.}. Therefore, by Examples~\ref{ex:dsums}\ref{ex:dsums1}, \ref{ex:2dsums}\ref{ex:dsums2}, and Remark~\ref{rem:dsumsjoin}\ref{rem:dsumsjoin1}, the hypotheses are satisfied.
\item \label{ex:dsums4}The $\mu(R)\mh\mrm{AEC}$ $(R\mh\mrm{FlatPI}, \leqo)$, where $R\mh\mrm{FlatPI}$ denotes the class of flat pure-injective $R$-modules. This follows from Example~\ref{ex:ringlpure}\ref{ex:ringlpure1.5} and Example~\ref{ex:dsums}\ref{ex:dsums0}.
\end{enumerate}
\end{example}

In conclusion, we give some examples in the context of abelian group theory. Recall that $\mu(\mathbb Z) = \aleph_1$ by Fact~\ref{fact:pidmucompact}. The proof that the classes of the next example satisfy Hypothesis~\ref{hyp:dsums} is straightforward, so we omit it.

\begin{example}
\label{ex:dsumsab}
The abstract class $(\kaec, \leqo)$ is an $\aleph_1\mh\mrm{AEC}$ satisfying Hypothesis~\ref{hyp:dsums} for the following choices of $\kaec$:
\begin{enumerate}[(1), leftmargin=*]
\item the class of torsion pure-injective abelian groups;
\item the class of torsion-free pure-injective abelian groups;
\item the class of reduced pure-injective abelian groups;
\item the class of reduced torsion-free pure-injective abelian groups.
\end{enumerate}
\end{example}

\subsection{Enveloping classes}
\label{sect:envclasses}

In this subsection we introduce the machinery of enveloping classes, which will be used to obtain novel  stability results in the next subsection. To a good approximation, enveloping classes generalize the relationship between the abstract classes $(\rmod, \leqp)$ and $(R\mh\pinj, \leqo)$ (see Example~\ref{ex:enveloping}). In particular, if $(\kaec_1, \leqk)$ is an envelope for $(\kaec_2, \leqk)$, we will be able to infer the stability of the latter class from the stability of $(\kaec_1, \leqk)$ (see Lemma~\ref{lem:transferenvelope}). Many applications of these techniques to concrete abstract classes of modules will be shown in the next subsection.  

\begin{definition}
\label{def:newefunctor}
Let $\kaecgoth_1 = (\kaec_1, \leqk)$ and $\kaecgoth_2 = (\kaec_2, \leqk)$ be $\mrm{AC}$s and $\mu$ an infinite regular cardinal. We say that $\kaecgoth_2$ is a $\mu$-envelope for $\kaecgoth_1$ if the following hold:
\begin{enumerate}[(1), leftmargin=*]
\item \label{def:newefunctor0}$\kaec_2\subseteq \kaec_1$;
\item \label{def:newefunctor1}for every $A\in\kaec_1$ there is a non-empty collection $\Gamma(A) \subseteq \kaec_2$ such that $\hul (A)\in \Gamma(A)$ implies $A\leqk \hul (A)$ and $|\hul (A)| \leq |A|^{<\mu}$, we call such a $\hul (A)$ an \tdef{envelope} for $A$;
\item if $A\in \kaec_2$, then $A\in \Gamma(A)$;
\item \label{def:newefunctor2} for every $\kaecgoth_1$-embedding $f: A \to B$, $\hul (A)\in \Gamma(A)$, $\hul(B)\in\Gamma(B)$, there is a $\kaecgoth_2$-embedding $\hul(f): \hul(A) \to \hul(B)$ such that $f\subseteq \hul(f)$.
\end{enumerate}
When we write $\hul(A)$, for $A\in\kaec_1$, we will implicitly mean that $\hul(A)\in \Gamma(A)$, and when we write $\hul(f)$, for $f$ a $\kaecgoth_1$-embedding, we mean that it is as in \ref{def:newefunctor}\ref{def:newefunctor2}.
\end{definition}

The motivating example for the previous definition is that $(R\mh\pinj, \leqo)$ is a {$\mu(R)$-envelope} for $(\rmod, \leqp)$, and in such a case the reader should keep in mind that, for $A\in \rmod$, $\Gamma (A)$ is the collection of pure-injective envelopes $B$ of $A$ with $A\leqp B$. In this case, when we write $\hul(A)\in \Gamma(A)$, we mean that $\hul(A)$ is a pure-injective envelope for $A$. More generally, the reader should keep in mind the following examples.

\begin{notation}
\label{not:pinjkaec}
Let $\kaec$ be a class of $R$-modules. We let $\kaec\mh\pinj$ denote the class of pure-injective modules in $\kaec$. Similarly, $\kaec\mh\inj$ denotes the class of injective modules in $\kaec$.
\end{notation}

\begin{example}
\label{ex:enveloping}
\begin{enumerate}[(1), leftmargin=*]
\item \label{ex:enveloping1}Let $\kaecgoth_\pp = (\kaec, \leqp)$ be an abstract class of $R$-modules such that $\kaec$ is closed under pure-injective envelopes. Then $\kaecgoth_\oplus^\pinj = (\kaec\mh\pinj, \leqo)$ is a $\mu(R)$-envelope for $\kaecgoth_\pp$. 
\item \label{ex:enveloping2} Let $\kaecgoth_\emb = (\kaec, \leq)$ be an abstract class of $R$-modules such that $\kaec$ is closed under injective envelopes. Then $\kaecgoth_\oplus^\inj = (\kaec\mh\inj, \leqo)$ is a $\gamma(R)$-envelope for $\kaecgoth_\emb$. 
\end{enumerate}
\end{example}
\begin{proof}
Let us show \ref{ex:enveloping1}. Before starting, notice that by Remark~\ref{rem:relationcollapse} the relations $\leqp$ and $\leqo$ coincide on $\kaec\mh\pinj$. For every $A\in\kaec$, let $\Gamma(A)$ be the collection of all pure-injective envelopes $B$ of $A$ with $A\leqp B$. By Fact~\ref{fact:existpenv} and Corollary~\ref{cor:cardpinj} we have that $(\kaec\mh\pinj, \leqo)$ is a $\mu(R)$-envelope for $(\kaec, \leqp)$.

\ssk \nin The proof of item  \ref{ex:enveloping2} is just as that of the previous item, but now taking injective envelopes and  using  Fact~\ref{fact:triviainjenv} and Fact~\ref{fact:triviainj}\ref{fact:triviainj1} instead.
\end{proof}

Clearly, if $\mu_1\leq \mu_2$ are infinite regular cardinals and $\kaecgoth_2$ is a $\mu_1$-envelope for $\kaecgoth_1$, we have that $\kaecgoth_2$ is a $\mu_2$-envelope for $\kaecgoth_1$.

\begin{remark}
\label{rem:muenvcofinal}
Let $\kaecgoth_2$ be a $\mu$-envelope for $\kaecgoth_1$. Then the following hold:
\begin{enumerate}[(1), leftmargin=*]
\item \label{rem:muenvcofinal2}If $A,B\in \kaec_1$ and $f': A\to B$ is a $\kaecgoth_1$-embedding, then the map $f: A \to \hul(B)$  obtained from $f'$ by extending the codomain is a $\kaecgoth_1$-embedding, because we have $B\leqk \hul(B)$. Moreover, if $A\in \kaec_2$, then $f$ is also a $\kaecgoth_2$-embedding.
\item \label{rem:muenvcofinal3}If $A,B\in\kaec_1$ and $f:A\to B$ is a  $\kaecgoth_1$-embedding, then $\hul(f)\restriction A: A \to \hul(B)$ is a $\kaecgoth_1$-embedding, because $f(A) = \hul(f)(A)\leqk \hul(f)(\hul(A))\leqk \hul(B)$.
\end{enumerate}
\end{remark}

The following lemma shows that, if $\kaecgoth_2$ is a $\mu$-envelope for $\kaecgoth_1$, then there is a clear transfer between the model-theoretic properties of $\kaecgoth_1$ and $\kaecgoth_2$.

\begin{lemma}
\label{lem:transferenvelope}
Let $\kaecgoth_1 = (\kaec_1, \leqk)$, $\kaecgoth_2 = (\kaec_2, \leqk)$ be abstract classes, and $\kaecgoth_2$ a $\mu$-envelope for $\kaecgoth_1$.
Then:
\begin{enumerate}[(1), leftmargin=*]
\item \label{lem:transferenvelope0} $\kaecgoth_1$ has $\mrm{JEP}$ if and only if $\kaecgoth_2$ has $\mrm{JEP}$.
\item \label{lem:transferenvelope1}Let $\lambda$ be an infinite cardinal, $\kaecgoth_1$ has $\lambda\mh\mrm{AP}$ if and only if $\kaecgoth_2$ has $\lambda\mh\mrm{AP}$. In particular, $\kaecgoth_1$ has $\infty\mh\mrm{AP}$ if and only $\kaecgoth_2$ has $\infty\mh\mrm{AP}$.
\item \label{lem:transferenvelope1.5}Let $\kappa$ be an infinite cardinal, if $\kaecgoth_1$ is (almost) $\kappa$-stable, then so is $\kaecgoth_2$. In particular, if $\kaecgoth_1$ is (almost) stable, then so is $\kaecgoth_2$.
\item \label{lem:transferenvelope2}
If $\kappa = \kappa^{<\mu}$ and $\kaecgoth_2$ is (almost) $\kappa$-stable, then so is $\kaecgoth_1$.
\end{enumerate}
\end{lemma}
\begin{proof}
We start by showing  \ref{lem:transferenvelope0}.
\begin{enumerate}[$(*_{\arabic*})$, leftmargin=*, align=left,  series=prf:envelopes]
\item \label{prf:envelopes1.1}If $\kaecgoth_1$ has $\mrm{JEP}$, then so does $\kaecgoth_2$.
\end{enumerate}
Why \ref{prf:envelopes1.1}? Take $B_1,B_2\in\kaec_2\subseteq \kaec_1$, because $\kaecgoth_1$ has $\mrm{JEP}$ there is a pair of $\kaecgoth_1$-embeddings $(f_1': B_1 \to C, f_2': B_2 \to C)$ with $C\in\kaec_1$. Take $C\leqk \hul(C)\in \kaec_2$, and $f_\ell$ the map which extends the codomain of $f_\ell'$ to $\hul(C)$ as in Remark~\ref{rem:muenvcofinal}\ref{rem:muenvcofinal2}, with $\ell\in\{1,2\}$. Then $(f_1: B_1 \to \hul(C), f_2: B_2 \to \hul(C))$ is the required pair of $\kaecgoth_2$-embeddings.
\begin{enumerate}[resume*=prf:envelopes]
\item\label{prf:envelopes1.2} If $\kaecgoth_2$ has $\mrm{JEP}$, then so does $\kaecgoth_1$.
\end{enumerate}
Why \ref{prf:envelopes1.2}? Take $B_1, B_2\in\kaec_1$, because $\kaecgoth_2$ has $\mrm{JEP}$ there is a pair of $\kaecgoth_2$-embeddings $(f_1': \hul(B_1) \to C, f_2': \hul(B_2) \to C)$ with $C\in\kaec_2$. Letting $f_\ell = f_\ell' \restriction B_\ell$, with $\ell\in\{1,2\}$, we are done by Remark~\ref{rem:muenvcofinal}\ref{rem:muenvcofinal3} and because $C\in \kaec_1$. This shows \ref{prf:envelopes1.2}.

\ssk\nin We show \ref{lem:transferenvelope1}. Let $\lambda$ be an infinite cardinal.
\begin{enumerate}[resume*=prf:envelopes]
\item\label{prf:envelopes1.3} If $\kaecgoth_1$ has $\lambda\mh\mrm{AP}$, then so does $\kaecgoth_2$.
\end{enumerate}
Why \ref{prf:envelopes1.3}?  Argue as in the proof of \ref{prf:envelopes1.1}, using Remark~\ref{rem:muenvcofinal}\ref{rem:muenvcofinal2}.
\begin{enumerate}[resume*=prf:envelopes]
\item\label{prf:envelopes1.4} If $\kaecgoth_2$ has $\lambda\mh\mrm{AP}$, then so does $\kaecgoth_1$.
\end{enumerate}
Why \ref{prf:envelopes1.4}? Assume $(f_i: A \to B_i\mid i\in I)$ is a collection of $\kaecgoth_1$-embeddings with $|I| <\lambda$. This induces a collection of {$\kaecgoth_2$-embeddings} $(\hul(f_i): \hul(A) \to \hul(B_i)\mid i\in I)$. Because $\kaecgoth_2$ satisfies $\lambda\mh\mrm{AP}$, there is a collection $(g_i': \hul(B_i) \to C\mid i\in I)$ of $\kaecgoth_2$-embeddings with $C\in\kaec_2$ and $g_i' \hul(f_i) = g_j' \hul (f_j)$ for every $i,j\in I$. Letting $g_i = g_i'\restriction B_i$ for every $i\in I$, we are done by Remark~\ref{rem:muenvcofinal}\ref{rem:muenvcofinal3}. This shows \ref{prf:envelopes1.4}.

\ssk \nin We now show \ref{lem:transferenvelope1.5}. From now on, to simplify the notation, for $\ell\in\{1,2\}$, we let $E_\ell$ denote $E_{\kaecgoth_\ell}$, and similarly for $E_\ell^\at$ (cf.~\ref{def:eqrelorbtypes}).
\begin{enumerate}[resume*=prf:envelopes]
\item \label{prf:envelopes1.5}If $A, B_\ell\in \kaec_2$, $c_\ell\in B_\ell$, with $\ell\in \{1, 2\}$, then $(c_1, A, B_1) E_1 (c_2, A, B_2)$ if and only if $(c_1, A, B_1) E_2 (c_2, A, B_2)$.
\end{enumerate}
By induction and the definition of the orbital equivalence relation  (cf.~\ref{def:eqrelorbtypes}), to show \ref{prf:envelopes1.5} it will be enough to show:
\begin{enumerate}[resume*=prf:envelopes]
\item\label{prf:envelopes1.6} If $A, B_\ell\in \kaec_2$, $c_\ell\in B_\ell$, with $\ell\in \{1, 2\}$, then $(c_1, A, B_1) E_1^\at (c_2, A, B_2)$ if and only if $(c_1, A, B_1) E_2^\at (c_2, A, B_2)$.
\end{enumerate}
Why \ref{prf:envelopes1.6}? We have that $\kaec_2\subseteq \kaec_1$, therefore $(c_1, A, B_1) E_2^\at (c_2, A, B_2)$ implies that $(c_1, A, B_1) E_1^\at (c_2, A, B_2)$. For the other direction, argue as in \ref{prf:envelopes1.1} using  Remark~\ref{rem:muenvcofinal}\ref{rem:muenvcofinal2}.
\begin{enumerate}[resume*=prf:envelopes]
\item Item \label{prf:envelopes1.7} \ref{lem:transferenvelope1.5} holds.
\end{enumerate}
Why \ref{prf:envelopes1.7}? Follows from \ref{prf:envelopes1.5}.
\begin{enumerate}[resume*=prf:envelopes]
\item \label{prf:envelopes1.85}Item \ref{lem:transferenvelope2} holds
\end{enumerate}
The rest of the proof will be concerned with showing \ref{prf:envelopes1.85}.
\begin{enumerate}[resume*=prf:envelopes]
\item \label{prf:envelopes1.8}If $\kappa = \kappa^{<\mu}$ and $\kaecgoth_2$ is $\kappa$-stable, then so is $\kaecgoth_1$.
\end{enumerate}
Why \ref{prf:envelopes1.8}? Before starting with the proof, recall that $\mu$ is such that  for every $A\in \kaec_1$ and $\hul (A)\in \Gamma(A)$, we have that $|\hul (A)| \leq |A|^{<\mu}$ (cf.~\ref{def:newefunctor}\ref{def:newefunctor1}). Now, assume for a contradiction that there is $A\in \kaec_1$ with $|A|=\kappa$ and there is a collection of pairwise distinct orbital types $p_i = \gtp_{\kaecgoth_1}(c_i/A; B_i)$ in $\kaecgoth_1$, with $i< \kappa^+$. 
Without loss of generality, we can assume that for every $i<\kappa^+$ we have $\hul(A) \leqk \hul(B_i)$.
Because $\kappa = \kappa^{<\mu}$ and $|A| =\kappa$, then $|\hul(A)| = \kappa$ by \ref{def:newefunctor}\ref{def:newefunctor1}. By $\kappa$-stability of $\kaecgoth_2$ there must be $i<j<\kappa^+$ such that $(c_i,\hul(A), \hul(B_i)) E_2 (c_j,\hul(A), \hul(B_j))$. By \ref{prf:envelopes1.5} we have $(c_i,\hul(A), \hul(B_i)) E_1 (c_j,\hul(A), \hul(B_j))$, and in  particular we have $(c_i, A , \hul(B_i)) E_1 (c_j,A, \hul(B_j))$. But clearly $(c_i, A, H(B_i))E_1(c_i, A, B_i)$, and similarly $(c_j, A, H(B_j))E_1(c_j, A, B_j)$. Therefore $(c_i, A, B_i)E_1(c_j, A, B_j)$, a contradiction.
\begin{enumerate}[resume*=prf:envelopes]
\item \label{prf:envelopes1.9}If $\kappa = \kappa^{<\mu}$ and $\kaecgoth_2$ is almost $\kappa$-stable, then so is $\kaecgoth_1$.
\end{enumerate}
Why \ref{prf:envelopes1.9}? The proof is very similar to that of  \ref{prf:envelopes1.8}, so we omit it. 
\end{proof}

Coming back to a more concrete setting, we end this subsection with this following result, which is the continuation of Example~\ref{ex:enveloping}.

\begin{theorem}
\label{thm:envinjpinj}
\begin{enumerate}[(1), leftmargin=*]
\item \label{thm:envinjpinj1}Let $\kaecgoth_\pp = (\kaec, \leqp)$ be an abstract class of $R$-modules such that $\kaec$ is closed under pure-injective envelopes. Then $\kaecgoth_\oplus^{\pinj}= (\kaec\mh\pinj, \leqo)$ is a $\mu(R)$-envelope for $\kaecgoth_\pp$, and if $\kaec$ is closed under arbitrary direct sums and direct summands then both $\kaecgoth_\pp$ and $\kaecgoth^\pinj_\oplus$ have $\infty\mh\mrm{AP}$. Moreover, if $\kaecgoth_\pp$ is a $\mu\mh\mrm{AEC}$, then $\kaecgoth_\oplus^{\pinj}$ is a $(\mu(R)+\mu)\mh\mrm{AEC}$.
\item \label{thm:envinjpinj2}Let $\kaecgoth_\emb = (\kaec, \leq)$ be an abstract class of $R$-modules such that $\kaec$ is closed under injective envelopes. Then $\kaecgoth_\oplus^{\inj} = (\kaec\mh\inj, \leqo)$ is a $\gamma(R)$-envelope for $\kaecgoth_\emb$, and if $\kaec$ is closed under arbitrary direct sums and direct summands then both $\kaecgoth_\emb$ and $\kaecgoth^\inj_\oplus$ have $\infty\mh\mrm{AP}$. Moreover, if $\kaecgoth_\emb$ is a $\mu\mh\mrm{AEC}$, then $\kaecgoth_\oplus^{\inj}$ is a $(\gamma(R)+\mu)\mh\mrm{AEC}$.
\end{enumerate}
\end{theorem}
\begin{proof}
Item \ref{thm:envinjpinj1}, with the exception of the statement about $\infty\mh\mrm{AP}$, follows from Example~\ref{ex:enveloping}\ref{ex:enveloping1}, and the moreover part follows from Corollary~\ref{cor:pinjmuaec}. 

\ssk\nin Therefore, assume $\kaec$ is closed under arbitrary direct sums and direct summands, to show \ref{thm:envinjpinj1} we are only left to show that $\kaecgoth_\pp$ and $\kaecgoth^\pinj_\oplus$ have $\infty\mh\mrm{AP}$. Using Lemma~\ref{lem:transferenvelope}\ref{lem:transferenvelope1}, it is enough to show that $\kaecgoth_\oplus^\pinj= (\kaec\mh\pinj, \leqo)$ has $\infty\mh\mrm{AP}$. Assume one has a collection $(f_i: A \to B_i\mid i\in I)$ of $\kaecgoth_\oplus^\pinj$-embeddings. Let $(g_i: B_i \to P\mid i\in I)$ be the wide pushout of $(f_i)_{i\in I}$, where $P = (\bigoplus_{i\in I} B_i)/N$ is as in  Proposition~\ref{prop:widepushoutexist}. 
Because $f_i$ is a split embedding for every $i\in I$, by Remark~\ref{rem:lpuredsumequiv} and  Proposition~\ref{prop:wpushoutclosedlambda},  we have that $g_i$ is a split embedding for every $i\in I$, and $N\leqo \bigoplus_{i\in I}B_i$.
Because $\kaec$ is closed under arbitrary direct sums we have $\bigoplus_{i\in I}B_i\in\kaec$, and by closure under direct summands of $\kaec$ we have $P = (\bigoplus_{i\in I} B_i)/N\in\kaec$. Taking $C=\mrm{PE}(P)$ (cf.~\ref{fact:existpenv}) we have that $C\in\kaec\mh\pinj$ by closure under pure-injective envelopes of $\kaec$. For every $i\in I$, call $h_i: B_i \to C$ the map which extends the codomain of $g_i$ from $P$ to $C$. Clearly $h_i\restriction A = g_i \restriction A = g_j \restriction A = h_j \restriction A$ for every $i,j\in I$. Moreover, since each $g_i$ is a pure embedding, $h_i: B_i \to C$ is a pure embedding because $P\leqp \mrm{PE}(P)=C$. Since $B_i$ is pure-injective we have that $h_i$ is a split embedding. Therefore $(h_i: B_i \to C\mid i\in I)$ is the required amalgam of $(f_i: A\to B_i\mid i\in I)$ in $\kaecgoth^\pinj_\oplus$.

\ssk \nin Item \ref{thm:envinjpinj2} follows from Example \ref{ex:enveloping}\ref{ex:enveloping2}, and the statement about $\infty\mh\mrm{AP}$ follows as in \ref{thm:envinjpinj1} but taking the injective envelope instead of the pure-injective envelope, and the moreover part follows from Corollary~\ref{cor:injmuaec}.
\end{proof}

From the previous result it follows, extending Examples~\ref{ex:dsums}-\ref{ex:dsumsab}, that:
\begin{corollary}
\label{cor:inftyappinj}
The property $\infty\mh\mrm{AP}$ holds for the following abstract classes:
\begin{enumerate}[(1), leftmargin=*]
\item \label{cor:inftyappinj1}The $\mu(R)\mh\mrm{AEC}$ $(R\mh\pinj, \leqo)$, where $R\mh\pinj$ is the class of pure-injective $R$-modules;
\item \label{cor:inftyappinj2}The $\gamma(R)\mh\mrm{AEC}$ $(R\mh\inj, \leqo)$, where $R\mh\inj$ is the class of injective $R$-modules;
\item \label{cor:inftyappinj3}The $\aleph_1\mh\mrm{AEC}$ $(\mrm{TF}\mh\pinj, \leqo)$, where $\mrm{TF}\mh\pinj$ is the class of torsion-free pure-injective abelian groups.
\end{enumerate}
\end{corollary}
\begin{proof}
Only item~\ref{cor:inftyappinj3} requires explanation, and it  follows from the fact that the class of torsion-free abelian groups is closed under pure-injective envelopes \cite[Corollary~3.1]{rothmalerpureinjflat}.
\end{proof}

\subsection{Applications of enveloping classes}
\label{sect:appliedenvclasses}

We now turn our attention to applications of the enveloping classes introduced in the preceding subsection. We will apply these results to give stability results for abstract classes of modules of the form $(\kaec, \leq)$ and $(\kaec, \leqp)$, whenever $\kaec$ is closed under injective and pure-injective envelopes respectively. In particular, if $(\kaec, \leqp)$ is an $\mrm{AEC}$ with $\kaec$ closed under pure-injective envelopes and finite direct sums, then \cite[Theorem~3.11]{somestablenonelementary} shows that $(\kaec, \leqp)$ is stable. In this subsection we recover, with a substantially different proof, the aforementioned result with the additional assumption that $\kaec$ is also closed under direct summands (see Corollary~\ref{cor:pinjenvstablemazari} and the discussion following it).

In fact, our results apply to abstract classes of the form $(\kaec, \leqp)$ or $(\kaec, \leq)$ which are not necessarily $\mrm{AEC}$s, or even $\mu\mh\mrm{AEC}$s for some $\mu$. In \cite{mazaridecon} Mazari-Armida and Trlifaj consider abstract classes of modules which satisfy all of the axioms of an $\mrm{AEC}$ with the exception of closure under directed systems. Our approach is more general because in this subsection we consider abstract classes of the form $(\kaec, \leq)$ or $(\kaec, \leqp)$ which, a priori, are not closed under directed systems and do not even satisfy the L\"owenheim-Skolem axiom.

We start by studying abstract classes of the form $(\kaec, \leqp)$. Notice that in the following results we do not assume $(\kaec, \leqp)$ to be a $\mu\mh\mrm{AEC}$ for some $\mu$, but we only require it to be an abstract class. Throughout, we follow Notation~\ref{not:pinjkaec} for $\kaec\mh\pinj$ and $\kaec\mh\inj$, and we use freely the cardinal invariants $\gamma(R)$ and $\mu(R)$ {(cf.~\ref{def:gammar} and \ref{def:mur})}.

\begin{theorem}
\label{thm:pinjenvstable}
Let $\mu$ be a regular infinite cardinal, $\kaecgoth_\pp = (\kaec, \leqp)$ an abstract class of $R$-modules such that $\kaec$ is closed under pure-injective envelopes and $\kaecgoth_\oplus^\pinj = (\kaec\mh\pinj, \leqo)$ satisfies Hypothesis~\ref{hyp:dsums} for $\mu$. Then $\kaecgoth_\pp$ has $\mrm{AP}$, $\mrm{JEP}$, and is $\kappa$-stable for every $\kappa=\kappa^{\mu_1}$, where $\mu_1 = |R|^{<(\mu(R)+\mu)}$. Moreover, if $\kaec$ is closed under direct summands and arbitrary direct sums, then both $\kaecgoth_\oplus^\pinj$ and $\kaecgoth_\pp$ have $\infty\mh\mrm{AP}$.
\end{theorem}
\begin{proof}
By Theorem~\ref{thm:envinjpinj}\ref{thm:envinjpinj1} we have that $\kaecgoth_\oplus^{\pinj}$ is a $(|R|+\aleph_0)^+$-envelope for $\kaecgoth_\pp$. Using Theorem~\ref{thm:dsumsumup} and $\kappa = \kappa^{\mu_1} \geq \kappa^{|R| +\aleph_0} \geq \kappa$, we can apply Lemma~\ref{lem:transferenvelope} and Theorem~\ref{thm:envinjpinj} to get the result. 
\end{proof}

\begin{corollary}
\label{cor:pinjenvstablemazari}
Let $\mu$ be a regular cardinal and $\kaecgoth_\pp = (\kaec, \leqp)$ be an abstract class of $R$-modules such that:
\begin{enumerate}[(1), leftmargin=*]
\item $\kaec$ is closed under pure-injective envelopes;
\item $\kaec$ is closed under finite direct sums;
\item $\kaec$ is closed under direct summands;
\item $\kaecgoth_\pp$ is closed under $\mu$-directed systems.
\end{enumerate}
Then $\kaecgoth_\pp$ has $\mrm{AP}$, $\mrm{JEP}$, and is $\kappa$-stable for every $\kappa=\kappa^{\mu_1}$, where $\mu_1 = |R|^{<(\mu(R)+\mu)}$. Moreover, if $\kaec$ is closed under arbitrary direct sums, then $\kaecgoth_\pp$ has $\infty\mh\mrm{AP}$.
\end{corollary}
\begin{proof}
We show that $\kaecgoth^\pinj_\oplus$ satisfies Hypothesis~\ref{hyp:dsums} for $\mu(R)+\mu$, so that by Theorem~\ref{thm:pinjenvstable} we have the result. By Fact~\ref{fact:triviapinjective}\ref{fact:triviapinjectivesums} and closure of $\kaec$ under finite direct sums and direct summands, we only have to verify that $\kaecgoth_\oplus^\pinj$ is closed under $(\mu(R) +\mu)$-directed systems. Since $\kaecgoth_\pp$ is closed under $\mu$-directed systems, and using Corollary~\ref{cor:pinjmuaec} we get the result.
\end{proof}
In \cite[Theorem 3.11]{somestablenonelementary} it is shown that $\mrm{AEC}$s of the form $\kaecgoth= (\kaec, \leqp)$ which are closed under finite direct sums and pure-injective envelopes are stable. Corollary~\ref{cor:pinjenvstablemazari} recovers, with a substantially different proof, the aforementioned result when $\kaec$ is moreover closed under direct summands. Nonetheless, it can be verified that all the examples presented in \cite[Example~3.3]{somestablenonelementary} are also closed under direct summands, and thus satisfy the hypotheses of Corollary~\ref{cor:pinjenvstablemazari}.  

We now apply Corollary~\ref{cor:pinjenvstablemazari} to obtain a new stability result for a concrete $\aleph_1\mh\mrm{AEC}$ of modules.

\begin{example}
\label{ex:envpinj}
The $\aleph_1\mh\mrm{AEC}$ $(R\mh\Superstable, \leqp)$ satisfies the hypotheses of Corollary~\ref{cor:pinjenvstablemazari}, where $R\mh\Superstable$ denotes the class of superstable $R$-modules in the sense of first-order logic (cf.~\cite{ziegler}). Recall that a module is superstable if and only if there is no infinite descending chain of $\pp$-definable subgroups (cf.~Footnote~\ref{fn:ppsub}) each of infinite index in its predecessor \cite[Theorem 2.1(3)]{ziegler}. Therefore, the class of superstable $R$-modules is closed under pure submodules and it is an $\aleph_1\mh\mrm{AEC}$ with pure embeddings. The class $R\mh\Superstable$ is closed under finite direct sums because it is closed under pure extensions \cite[Corollary 2.2(1)]{ziegler}, and it is closed under pure-injective envelopes because every module is an elementary substructure of its pure-injective envelope \cite[Corollary 3.14]{ziegler}.
\end{example}

We are now able to prove a general stability theorem regarding abstract classes of modules which contain all pure-injective modules:

\begin{corollary}
\label{cor:pinjenvstable}
Let $\kaecgoth_\pp = (\kaec, \leqp)$  be an abstract class of $R$-modules such that $\kaec$ contains all pure-injective $R$-modules. Then $\kaecgoth_\pp = (\kaec, \leqp)$  has $\infty\mh\mrm{AP}$, $\mrm{JEP}$, and is $\kappa$-stable for every $\kappa = \kappa^{\mu_1}$, where $\mu_1 = |R|^{<\mu(R)}$.
\end{corollary}
\begin{proof}
$\kaec\mh\pinj$ is the class of all pure-injective $R$-modules, so that $(\kaec\mh\pinj, \leqo)$ satisfies Hypothesis~\ref{hyp:dsums} for $\mu(R)$, and we can apply Theorem~\ref{thm:pinjenvstable}. Moreover, $(R\mh\pinj, \leqo)$ is a $\mu(R)$-envelope for $(\kaec, \leqp)$ by Theorem~\ref{thm:envinjpinj}\ref{thm:envinjpinj1}, and $(R\mh\pinj, \leqo)$ has $\infty\mh\mrm{AP}$ by Corollary~\ref{cor:inftyappinj}, so that by   Lemma~\ref{lem:transferenvelope}\ref{lem:transferenvelope2} we have that $(\kaec, \leqp)$ has $\infty\mh\mrm{AP}$.
\end{proof}

Corollary~\ref{cor:pinjenvstable} is especially easy to apply, and enables us to obtain the following new result:
\begin{example}
\label{ex:cotpinj}
The abstract class $(R\mh\cotorsion, \leqp)$ has $\infty\mh\mrm{AP}$, $\mrm{JEP}$, and is stable. Here, $R\mh\cotorsion$ is the class of cotorsion $R$-modules (cf.~\ref{ex:lpurecotorsion}). This follows from Corollary~\ref{cor:pinjenvstable}, because every pure-injective module is cotorsion \cite[Section~4.6]{purityspectra}. 
\end{example}

Notice that in the preceding example we only had to verify that ${(R\mh\cotorsion, \leqp)}$ contains all pure-injective modules. This should be contrasted with the proof of Example~\ref{ex:lpurecotorsion}, which relied on rather advanced module-theoretic tools. Nonetheless, using the results of Appendix~\ref{sect:appcotorsion}, it follows that $(R\mh\cotorsion, \leqp)$ is a ${(|R|+\aleph_0)^+\mh\mrm{AEC}}$ (see Corollary~\ref{cor:cotorsionmuaec}).

Finally, putting together all of the results obtained so far for cotorsion modules, we have:
\begin{corollary}
Let $R$ be a countable ring, $\lambda$ an infinite regular cardinal, and  $R\mh\cotorsion$ the class of cotorsion $R$-modules (cf.~\ref{ex:lpurecotorsion}). Then $(R\mh\cotorsion, \leqp^\lambda)$ has $\infty\mh\mrm{AP}$, $\mrm{JEP}$, and is stable.
\end{corollary}
\begin{proof}
The result is the combination of  Example~\ref{ex:lpurecotorsion} and Example~\ref{ex:cotpinj}.
\end{proof}

We end this section by considering abstract classes of the form $(\kaec, \leq)$, with $\kaec$ closed under injective envelopes. The proofs of the following results are very similar to those we have presented already in the case of purity. Therefore, we only briefly outline the proofs.

\begin{theorem}
\label{thm:injenvstable}
Let $\mu$ be a regular infinite cardinal, $\kaecgoth_\emb = (\kaec, \leq)$ an abstract class of $R$-modules such that $\kaec$ is closed under injective envelopes and $\kaecgoth_\oplus^\inj=(\kaec\mh\inj, \leqo)$ satisfies Hypothesis~\ref{hyp:dsums} for $\mu$. Then $\kaecgoth_\emb$ has $\mrm{AP}$, $\mrm{JEP}$, and is $\kappa$-stable for every $\kappa=\kappa^{\mu_1}$, where $\mu_1 = |R|^{<(\mu(R)+\mu)}$. Moreover, if $\kaec$ is closed under direct summands and arbitrary direct sums, then both $\kaecgoth_\oplus^\inj$ and $\kaecgoth_\emb$ have $\infty\mh\mrm{AP}$.
\end{theorem}
\begin{proof}
The result follows just as is the proof of Theorem~\ref{thm:pinjenvstable} because by Theorem~\ref{thm:envinjpinj}\ref{thm:envinjpinj2}, we have that $(\kaec\mh\inj, \leqo)$ is a $(|R|+\aleph_0)^+$-envelope for $ (\kaec, \leq)$. 
\end{proof}

\begin{corollary}
\label{cor:injenvstable}
Let $\kaecgoth_\emb = (\kaec, \leq)$ be an abstract class of $R$-modules such that $\kaec$ contains all injective $R$-modules. Then $\kaecgoth_\emb$ has $\infty\mh\mrm{AP}$, $\mrm{JEP}$, and is $\kappa$-stable for every $\kappa = \kappa^{\mu_1}$, where $\mu_1 = |R|^{<(\mu(R) + \gamma(R))}$. In particular, this holds if $\kaec$ contains all pure-injective modules.
\end{corollary}
\begin{proof}
Again, the proof is just as in Corollary~\ref{cor:pinjenvstable}, but considering injective envelopes instead of the pure-injective envelopes.
\end{proof}

We end this subsection by showing, using Corollary~\ref{cor:injenvstable}, the stability of some abstract classes of modules of the form $(\kaec, \leq)$.

\begin{example}
\label{ex:injenvfirst}
The following abstract classes contain all injective $R$-modules, and thus satisfy the hypotheses of Corollary \ref{cor:injenvstable}:
\begin{enumerate}[(1), leftmargin=*, series=ex:injenvall]
\item $(R\mh\cotorsion, \leq)$, where $R\mh\cotorsion$ is the class of cotorsion $R$-modules (cf.~\ref{ex:lpurecotorsion}). This holds because every pure-injective module is cotorsion (see Example~\ref{ex:cotpinj}).
\item $(R\mh\rdinj, \leq)$, where $R\mh\rdinj$ denotes the class of $\mrm{RD}$-injective $R$-modules (cf.~\ref{ex:dsums}\ref{ex:dsums5}). It is easily verified that every injective module is $\mrm{RD}$-injective.
\end{enumerate}
\end{example}

Finally, with some additional assumption on the ring $R$, we can also prove the following result.
\begin{example}
\label{ex:injenvsecond}
Assume the ring $R$ is left Noetherian, then the following abstract classes contain all injective $R$-modules, and thus satisfy the hypotheses of Corollary~\ref{cor:injenvstable}:
\begin{enumerate}[(1), leftmargin=*]
\item $(R\mh\sigmapinj, \leq)$, where $R\mh\sigmapinj$ denotes the class of $\Sigma$-pure-injective $R$-modules (cf.~\ref{ex:2dsums}\ref{ex:dsums2}). This holds because if $R$ is left Noetherian then every injective $R$-module is $\Sigma$-injective \cite[Theorem 4.4.17]{purityspectra}, and thus $\Sigma$-pure-injective.
\item $(R\mh\Superstable, \leq)$, where $R\mh\Superstable$ is the class of superstable $R$-modules (cf.~\ref{ex:envpinj}). Clearly, every $\Sigma$-pure-injective module is superstable because it has no infinite decreasing chain of $\pp$-definable subgroups (this follows from the equivalent definitions given in Examples \ref{ex:2dsums}\ref{ex:dsums2} and \ref{ex:envpinj}). If $R$ is left Noetherian, then every injective $R$-module is $\Sigma$-pure-injective by the previous item, which gives us the result.
\end{enumerate}
\end{example}

\appendix

\section{Stable independence relation for balanced embeddings}
\label{sect:appstabletfab}
\subsection{Background and introduction}

Recall that a set of types $\tcharset$ is \tdef{closed} if $\tchar_1, \tchar_2\in \tcharset$ implies $\tchar_1\wedge \tchar_2\in\tcharset$ (cf.~\ref{def:tcharsetclosed}), and $\kaec^\tcharset$ denotes the set of torsion-free abelian groups $A$ with $\mrm{Type}(A)\subseteq \tcharset$ (cf.~\ref{not:tcharsetkaec}). The goal of this appendix is to show the following:
\begin{theorem}
\label{thm:balstableind}
Let $\tcharset$ be a closed set of types. Then $\kaecgoth_\bal^\tcharset=(\kaec^\tcharset, \leqb)$ has a stable independence relation (cf.~\ref{def:stableindep}). Moreover, $\kaecgoth_\bal^\tcharset$ is tame (cf.~\ref{def:tame}) and stable (cf.~\ref{def:stability}\ref{def:stable4}).
\end{theorem} 
The idea of the proof of Theorem~\ref{thm:balstableind} closely follows that of Theorem~\ref{thm:lpurestableind}  regarding the existence of a stable independence relation for classes of modules with the $\lambda$-pure submodule relation. Before proving Theorem~\ref{thm:balstableind}, in this subsection we are going to recall all the relevant definitions and give a syntactic characterization of the relation of being a balanced subgroup (see~Proposition~\ref{prop:balsyntaxequiv}).

Throughout this appendix, we are going to use the following fact without mention.
\begin{fact}[Fact~\ref{fact:balequiv}]
\label{fact:balequivappendix}
Let $A\leq B$ be torsion-free abelian groups. The following are equivalent:
\begin{enumerate}[(1), leftmargin=*]
\item$A$ is balanced in $B$ (cf.~\ref{def:proper}\ref{def:proper2}).
\item\label{fact:balequivappendix2} $A$ is pure in $B$ and for every $b\in B$ and every $(a_n)_{n\in \omega}\in A^\omega$ there is $a\in A$ such that $\chi(b+ a_n) \leq \chi (a + a_n)$ for every $n\in \omega$.
\end{enumerate}
\end{fact}

Now we introduce the model-theoretic terminology which will be needed in the following.
\begin{definition}
\begin{enumerate}[(1), leftmargin=*]
\item Let $\bar \ell\in \omega^{\omega \times \omega}$, define the formula $\alpha^{\bar \ell}(x_0, x_1,\ldots)$ as:
\begin{equation}
\label{eq:examplebformula}
\exists y\, \exists (z_{m,n})_{(m, n)\in \omega^2}\bigwedge_{m\in\omega}\bigwedge_{n\in\omega} p^{\ell_{m,n}}_m z_{m,n} = y+x_n,
\end{equation} 
where $(p_m)_{m\in \omega}$ is the  strictly increasing enumeration of the primes.
\item We say that $\varphi(\bar x)$ is a \tdef{$\bal$-formula} if it is a pp-formula (cf.~\ref{def:ppform}\ref{def:ppform2}) or it is of the form \eqref{eq:examplebformula}.
\end{enumerate}
\end{definition}

\begin{remark}
\label{rem:bformtrivia}
If $\varphi(\bar x)$ is a $\bal$-formula, $A\leq B$ are torsion-free abelian groups, and $\bar a\in A^{|\bar x|}$, then $A\vDash \varphi(\bar a)$ implies $B\vDash \varphi(\bar a)$.
\end{remark}

\begin{proposition}
\label{prop:balsyntaxequiv}
Let $A\leq B$ be torsion-free abelian groups. The following are equivalent:
\begin{enumerate}[(1), leftmargin=*]
\item \label{prop:balsyntaxequiv1}$A$ is balanced in $B$;
\item \label{prop:balsyntaxequiv2}for every $\bal$-formula  $\varphi(\bar x)$ and $\bar a\in A^{|\bar x|}$, we have that $A \vDash \varphi(\bar a)$ if and only if $B\vDash \varphi(\bar a)$.
\end{enumerate}
\end{proposition}
\begin{proof}
We show that \ref{prop:balsyntaxequiv1} implies \ref{prop:balsyntaxequiv2}. Assume $A\leqb B$, then $A\leqp B$, so that $\pp$-formulas are preserved. We are left to show:
\begin{enumerate}[$(*_1)$, leftmargin=*, series=balsynt]
\item \label{prf:balsyntaxequiv1}Let $\bar \ell\in \omega^{\omega\times\omega}$, $\bar a \in A^\omega$. Then $A\vDash \alpha^{\bar \ell}(\bar a)\Leftrightarrow B\vDash \alpha^{\bar \ell}(\bar a)$, where $\alpha^{\bar \ell}(\bar x)$ is as in \eqref{eq:examplebformula}.
\end{enumerate}
Why \ref{prf:balsyntaxequiv1}? One direction follows from Remark~\ref{rem:bformtrivia}. Assume $B\vDash \alpha^{\bar \ell}(\bar a)$. We are left to show $A\vDash \alpha^{\bar \ell}(\bar a)$. There are $b\in B$ and $c_{m,n}\in B$, for every $(m,n)\in\omega^2$, such that:
\[
B \vDash \bigwedge_{m\in\omega}\bigwedge_{n\in \omega} p^{\ell_{m,n}}_m c_{m,n} = b+a_n.
\]
Because $A$ is balanced in $B$,  there is $a\in A$ such that 
\[
\chi(b+a_n) \leq \chi(a+a_n)\quad \forall n\in\omega,
\]
which implies $A\vDash \alpha^{\bar \ell}(\bar a)$.

\ssk\nin Now we show the other direction. Assume \ref{prop:balsyntaxequiv2} holds, we want to show:
\begin{enumerate}[resume*=balsynt]
\item \label{prf:balsyntaxequiv2}$A$ is balanced in $B$.
\end{enumerate}
Why \ref{prf:balsyntaxequiv2}? By Proposition~\ref{prop:equivlpurelppform},    we have $A\leqp B$, because all $\pp$-formulas are preserved.  Assume $b\in B$ and $(a_n)_{n\in \omega}\in A^\omega$. We have to find $a\in A$ such that $\chi (b+a_n) \leq \chi(a+a_n)$. Let $\ell_{m,n} = h_{p_m}(b+a_n)$ (cf.~\ref{def:charac}), so that $B \vDash \alpha^{\bar \ell}(\bar a)$.
By hypothesis we have $A \vDash \alpha^{\bar \ell}(\bar a)$, and there is $a\in A$ such that 
\[
\exists (z_{m,n})_{(m, n)\in \omega^2}\bigwedge_{m\in\omega}\bigwedge_{n\in \omega} p^{\ell_{m,n}}_n z_{m,n} = a+a_n,
\]
which finally implies $\chi (b+a_n) \leq \chi(a+a_n)$.
\end{proof}

\begin{remark}
Notice that every $\bal$-formula is an $\aleph_1\mh\pp$-formula (cf.~\ref{def:ppform}). This gives another proof (cf.~\ref{rem:baltrivia}\ref{rem:balaleph1pure}) of the fact that, for $A$ and $B$ torsion-free, we have that ${A\leqp^{\aleph_1} B}$ implies $A\leqb B$.
\end{remark}
\subsection{The proof}
\label{sect:balproof}
We now want to start proving Theorem~\ref{thm:balstableind}. First, we fix the context of this subsection.
\begin{context}
\label{cont:balstabind} Let $\tcharset$, $\kaec^\tcharset$, $\kaecgoth_\bal^\tcharset$ be as follows:
\begin{enumerate}[(1), leftmargin=*]
\item $\tcharset$ is a closed set of types  (cf.~\ref{def:tcharsetclosed}).
\item $\kaecgoth^\tcharset_\bal = (\kaec^\tcharset, \leqb)$ (cf.~\ref{not:tcharsetkaec}).
\end{enumerate}
\end{context}

Summing up what was proved in Subsection~\ref{sect:lpurebal}, we have:
\begin{fact}[Remark~\ref{rem:tcharsetbalhyp}, Proposition~\ref{prop:tcharsetaleph1aec}, Lemma~\ref{lem:tcharsetinftyap}]
\label{fact:balfactsapp}
Let $\tcharset$ and $\kaecgoth_\bal^\tcharset ={ (\kaec^\tcharset, \leqb)}$ be as in  Context~\ref{cont:balstabind}. We have:
\begin{enumerate}[(1), leftmargin=*]
\item \label{fact:balfactsapp1}$\kaec^\tcharset$ is closed under finite direct sums, pure subgroups, and balanced quotients.
\item \label{fact:balfactsapp2}$\kaecgoth^\tcharset_\bal$ is an $\aleph_1\mh\mrm{AEC}$ with continuity (cf.~\ref{def:propertiesac}\ref{def:mucontinuity}) and $\mrm{LS}(\kaecgoth^\tcharset_\bal) \leq 2^{\aleph_0}$.
\item \label{fact:balfactsapp3}$\kaecgoth_\bal^\tcharset$ is closed under pushouts, i.e., if $(f_1: A \to B_1, f_2: A \to B_2)$ is a pair of $\kaecgoth_\bal^\tcharset$-embeddings with pushout $(g_1: B_1\to P, g_2: B_2 \to P)$ in $\abgrps$, then $P\in\kaecgoth^\tcharset$ and $g_1, g_2$ are $\kaecgoth_\bal^\tcharset$-embeddings.
\end{enumerate}
\end{fact}

\begin{definition}
\label{def:dnfbalanced}
Let $\tcharset$ and $\kaecgoth_\bal^\tcharset = (\kaec^\tcharset, \leqb)$ be as in  Context~\ref{cont:balstabind}. For a commutative square $(f_1, f_2, g_1, g_2)$ in $\kaecgoth_\bal$, we let $(f_1, f_2, g_1, g_2)\in \dnf$ if and only if the unique homomorphism $r: P \to C$ from the pushout $P$ of $(f_1, f_2)$ is a balanced embedding, where the maps are as in the following commutative diagram:
\[\begin{tikzcd}
	&& C \\
	{B_1} & P \\
	A & {B_2}
	\arrow["{g_1}", curve={height=-12pt}, from=2-1, to=1-3]
	\arrow["{h_1}", from=2-1, to=2-2]
	\arrow["r", from=2-2, to=1-3]
	\arrow["{f_1}", from=3-1, to=2-1]
	\arrow["{f_2}"', from=3-1, to=3-2]
	\arrow["{g_2}"', curve={height=12pt}, from=3-2, to=1-3]
	\arrow["{h_2}"', from=3-2, to=2-2]
\end{tikzcd}\]
\end{definition}

\begin{lemma}
\label{lem:balindepweaklystable}
Let $\tcharset$ and $\kaecgoth_\bal^\tcharset = (\kaec^\tcharset, \leqb)$ be as in  Context~\ref{cont:balstabind}. Then $\dnf$ as defined in~\ref{def:dnfbalanced} is a weakly stable independence relation (cf.~\ref{def:weaklystableind}).
\end{lemma}
\begin{proof}
Just as in Lemma~\ref{lem:lpureindepweaklystable}, but now using Fact~\ref{fact:balfactsapp}\ref{fact:balfactsapp3} instead of Proposition~\ref{prop:wpushoutclosedlambda}.
\end{proof}

To show that $\dnf$ as defined in \ref{def:dnfbalanced} is a stable independence relation, we are left to show that $\dnf$ satisfies local character (cf.~\ref{def:indeplocalchar}) and the witness property (cf.~\ref{def:indepwitnessprop}).

The proof of local character is based on the ideas of the proof of local character for $\lambda$-purity (cf.~\ref{lem:lpureindeplocal}), and it is very similar to it. We provide  the details for completeness. First, we will need the following  version of Lemma~\ref{lem:lpureloccharfund}.
\begin{lemma}
\label{lem:balloccharfund}
Let $B_1, B_2\leq C$ be torsion-free abelian groups. Then there are $A\leq B_1'$ with the following  properties: 
\begin{enumerate}[(a), leftmargin=*]
\item \label{lem:balloccharfund1}$A\leqp^{\aleph_1} B_2$;
\item \label{lem:balloccharfund2}$A\cup B_1\subseteq B_1' \leqp^{\aleph_1} C$;
\item \label{lem:balloccharfund3}$|A| + |B_1'| \leq |B_1|^{\aleph_0}$;
\item \label{lem:balloccharfundbox}if $\varphi(\bar x, \bar y)$ is a ${\aleph_1}\mh\pp$-formula, $\bar b_1'\in (B_1')^{|\bar x|}$, and there is $\bar b_2 \in B_2^{|\bar y|}$ such that $C \vDash \varphi(\bar b_1', \bar b_2)$, then there is $\bar a\in A^{|\bar y|}$ such that $C \vDash \varphi(\bar b_1', \bar a)$.
\end{enumerate}
\end{lemma}
\begin{proof}
Follows immediately from Lemma~\ref{lem:lpureloccharfund}.
\end{proof}
So that we are now able to prove:
\begin{lemma}
\label{lem:balindeplocal}
Let $\tcharset$ and $\kaecgoth_\bal^\tcharset = (\kaec^\tcharset, \leqb)$ be as in  Context~\ref{cont:balstabind}. Then $\dnf$ has right local character (cf.~\ref{def:indeplocalchar}). In particular, if $B_1, B_2, C\in \kaec^\tcharset$ with $B_1, B_2 \leqb C$, then there are $B_1', A\in \kaec^\tcharset$ such that $A \leqb B_1'$, $B_1\leqb B_1'$, $|A|+ |B_1'| \leq |B_1|^{\aleph_0}$, and $B_1' \dnf_A^C B_2$.
\end{lemma}
\begin{proof}
Let $B_1, B_2, C$ be as in the statement.
\begin{enumerate}[$(*_1)$, leftmargin=*, series=prf:balindeplocal]
\item \label{prf:balindeplocal4} We can find torsion-free abelian groups $A, B_1'$ such that the following conditions are satisfied:
\begin{enumerate}[$(\cdot_1)$, leftmargin=*] 
\item $A\leqp^{\aleph_1} B_2$;
\item \label{prf:balindeplocal4.3}$B_1'\leqp^{\aleph_1} C$;
\item \label{prf:balindeplocal4.1}$|A| + |B_1'| \leq |B_1|^{\aleph_0}$;
\item \label{prf:balindeplocal4.6}if $\varphi(\bar x, \bar y)$ is an $\aleph_1\mh\pp$-formula, $\bar b_1'\in (B_1')^{|\bar x|}$, and there is $\bar b_2 \in B_2^{|\bar y|}$ such that $C \vDash \varphi(\bar b_1', \bar b_2)$, then there is $\bar a\in A^{|\bar y|}$ such that $C \vDash \varphi(\bar b_1', \bar a)$;
\item \label{prf:balindeplocal4.4}$A, B_1'\in \kaec^\tcharset$;
\item \label{prf:balindeplocal4.45} $B_1\leqb B_1'$;
\item \label{prf:balindeplocal4.5}$A\leqb B_1'$.
\end{enumerate}
\end{enumerate}
Why \ref{prf:balindeplocal4}? With the exception of \ref{prf:balindeplocal4.45} and \ref{prf:balindeplocal4.5}, everything  follows just as in the proof of Lemma~\ref{lem:lpureindeplocal}, now using Lemma~\ref{lem:balloccharfund} and the fact that $\kaec^\tcharset$ is closed under pure subgroups (cf.~\ref{fact:balfactsapp}). 
To show \ref{prf:balindeplocal4.45} and \ref{prf:balindeplocal4.5}, use the fact that $\leqb$ refines $\leqp^{\aleph_1}$. This finishes the proof of~\ref{prf:balindeplocal4}.

\ssk\nin We are only left to show:
\begin{enumerate}[resume*=prf:balindeplocal]
\item \label{prf:balindeplocal5}We have $B_1' \dnf^{C}_{A} B_2$.
\end{enumerate} 
The rest of the proof will deal with showing \ref{prf:balindeplocal5}. Recall that by Proposition~\ref{prop:widepushoutexist} the pushout of $B_1'$ and $B_2$ over $A$ is given by $P = (B_1' \oplus B_2)/N$, where $N = \{(a, -a)\mid a\in A\}$. We have to show that the unique $R$-module homomorphism $r: P \to C$ given by $r([(b_1', b_2)]_N) = b_1' + b_2$ is a balanced embedding. 
\begin{enumerate}[resume*=prf:balindeplocal]
\item \label{prf:balindeplocal6}$r: P \rightarrow C$ is an embedding. 
\end{enumerate} 
Why \ref{prf:balindeplocal6}? As in Lemma~\ref{lem:lpureindeplocal}.
\begin{enumerate}[resume*=prf:balindeplocal]
\item \label{prf:balindeplocal7}$r: P \rightarrow C$ is a balanced embedding. 
\end{enumerate}
We verify \ref{prf:balindeplocal7}. We use the syntactic characterization of balancedness of  Proposition~\ref{prop:balsyntaxequiv}\ref{prop:balsyntaxequiv2}.
Let $\varphi(\bar x)$ a $\bal$-formula and assume
\[
C \vDash \varphi(\bar b_1' + \bar b_2)
\]
with $\bar b_1'\in (B_1')^{|\bar x|}$ and $\bar b_2\in B_2^{|\bar x|}$. We have to show that
\begin{equation}
\label{eq:balcharfinal}
P \vDash \varphi([(\bar b_1', \bar b_2)]_N).
\end{equation}
Let $\psi(\bar z, \bar z')$ be the formula:
\[
\varphi(\bar z + \bar z').
\]
Then $\psi(\bar z, \bar z')$ is an $\aleph_1\mh\pp$-formula, and obviously we have that
\begin{equation}
\label{eq:balchar2}
C \vDash \psi(\bar b_1', \bar b_2).
\end{equation}
Applying~\ref{prf:balindeplocal4}\ref{prf:balindeplocal4.6} to \eqref{eq:balchar2}, there is $\bar a\in A^{|\bar z'|}$ such that $C \vDash \psi(\bar b_1', \bar a)$, so we have:
\begin{equation}
\label{eq:balchar2.5}
C \vDash \varphi(\bar b_1' + \bar a).
\end{equation} 
Because $\varphi(\bar x)$ is an $\aleph_1\mh\pp$-formula, it is of the form $\exists\bar y\, \theta(\bar x, \bar y)$ with $\theta(\bar x, \bar y)$ a countable system of equations.
Notice that $\bar b_1'+ \bar a\in (B_1')^{|\bar x|}$, and by $B_1' \leqp^{\aleph_1} C$ applied to \eqref{eq:balchar2.5} there is  $\bar b^\star \in (B_1')^{|\bar y|}$ such that 
\begin{equation}
\label{eq:balchar3}
C \vDash \theta(\bar b_1' + \bar a, \bar b^\star).
\end{equation}
By Remark~\ref{rem:lpuretrivia}\ref{rem:lpuretrivia2} the set of solutions of an $\aleph_1\mh\pp$-formula is a subgroup, thus
$C\vDash \varphi(\bar b_1' + \bar b_2)$ and $C \vDash \varphi(\bar b_1' + \bar a)$ imply $C \vDash \varphi(\bar b_2 - \bar a)$.
Notice that $\bar b_2 - \bar a\in B_2^{|\bar x|}$ and, because $B_2\leqb C$, there is $\bar c^\star\in B_2^{|\bar y|}$ such that
\begin{equation}
\label{eq:balchar4}
C \vDash \theta(\bar b_2 - \bar a, \bar c^\star).
\end{equation}
As $\theta(\bar x, \bar y)$ is a system of equations, we can sum  \eqref{eq:balchar3} and \eqref{eq:balchar4} to have
\[
C \vDash \theta(\bar b_1' + \bar b_2, \bar b^\star + \bar c^\star).
\]
Because $\bar b^\star \in (B_1')^{|\bar y|}$ and $\bar c^\star\in B_2^{|\bar y|}$, we have that, projecting this last equation on $P = (B_1'\oplus B_2)/N=(B_1'\oplus B_2)/\{(a, -a)\mid a\in A\}$, we get \eqref{eq:balcharfinal}.
\end{proof}

To show that $\dnf$ as defined in Definition~\ref{def:dnfbalanced} is a stable independence relation, we are only left to show that $\dnf$ has the right witness property.

\begin{lemma}
\label{lem:balindepwitness}
Let $\tcharset$ and $\kaecgoth_\bal^\tcharset = (\kaec^\tcharset, \leqb)$ be as in  Context~\ref{cont:balstabind}. Then $\dnf$ has the right $({<}\aleph_1)$-witness property (cf.~\ref{def:indepwitnessprop}). 
\end{lemma}
\begin{proof}
Let $A\leqb B_\ell \leqb C$, with $\ell\in \{1, 2\}$, and assume that for every $X\subseteq B_2$ with $|X|\leq \aleph_0$ one has $\dnfb{A}{B_1}{X}{C}$. Recall that by Proposition~\ref{prop:widepushoutexist} the pushout of $B_1$ and $B_2$ over $A$ is given by $P =(B_1\oplus B_2)/N$, where $N = \{(a, -a)\mid a\in A\}$. We have to prove that the unique $R$-module homomorphism $t: P \to C$ given by $t([(b_1, b_2)]_N) = b_1 + b_2$ is a balanced embedding. That $t$ is a pure embedding follows just as in Lemma~\ref{lem:lpureindepwitness}. We use Fact~\ref{fact:balequivappendix}\ref{fact:balequivappendix2} to show that $t$ is balanced.

\ssk\nin Assume we have $c\in C$ and a sequence $((b_1^n, b_2^n)\mid n\in \omega)$ in $B_1\oplus B_2$. We wish to find $[(\tilde b_1, \tilde b_2)]_N\in P$ such that 
\begin{equation}
\label{eq:bwit0}
\chi_P([(\tilde b_1, \tilde b_2)]_N + [(b_1^n, b_2^n)]_N) \geq \chi_C(c+ t([(b_1^n, b_2^n)]_N)) \quad \forall n\in \omega.
\end{equation}
Let $X= \{b_2^n \mid n\in \omega\}$, and find $D_1, D_2, E$ such that $B_1\subseteq  D_1\leqb E$, $X\subseteq D_2 \leqb E$, $C\leqb E$, and $D_1 \dnf^{E}_{A} D_2$. Notice that by coherence we have $B_1 \leqb D_1$. If $Q=(D_1\oplus D_2)/\{(a, -a)\mid a\in A\} =(D_1\oplus D_2)/N$ is the pushout of $D_1$ and $D_2$ over $A$, we have that $s: Q \to E$ is a balanced embedding, where $s([(d_1, d_2)]_N) = d_1 + d_2$. Therefore, because $s$ is balanced, there is $(d_1, d_2)\in D_1\oplus D_2$ such that
\begin{equation}
\label{eq:bwit1}
\chi_Q([(d_1, d_2)]_N + [(b_1^n, b_2^n)]_N) \geq \chi_E(c + s([(b_1^n, b_2^n)]_N)) \quad \forall n\in \omega.
\end{equation}
We have that $C\leqb E$, and $c + s([(b_1^n, b_2^n)]_N) = c + t([(b_1^n, b_2^n)]_N)$ for every $n\in\omega$, so that from \eqref{eq:bwit1} it follows
\begin{equation}
 \label{eq:bwit1.1}
\chi_Q([(d_1, d_2)]_N + [(b_1^n, b_2^n)]_N) \geq \chi_C(c + t([(b_1^n, b_2^n)]_N)) \quad \forall n\in \omega.
\end{equation}
By Proposition~\ref{prop:wpushoutclosedbal} we have $N\leqb D_1\oplus D_2$, and recalling Definition~\ref{def:proper}\ref{def:proper1}, we have that \eqref{eq:bwit1} implies that there is a sequence $(a_n)_{n\in \omega}\in A^\omega$ with
\begin{equation}
\label{eq:bwit2}
\chi_{D_1\oplus D_2} ((d_1 + b_1^n + a_n, d_2 + b_2^n - a_n))= \chi_Q([(d_1, d_2)]_N + [(b_1^n, b_2^n)]_N)\quad \forall n\in \omega.
\end{equation}
Projecting \eqref{eq:bwit2} on $D_1$ we have 
\begin{equation}
\label{eq:bwit3}
\chi_{D_1} (d_1 + b_1^n + a_n)\geq \chi_Q([(d_1, d_2)]_N + [(b_1^n, b_2^n)]_N)\quad \forall n\in \omega,
\end{equation}
and projecting \eqref{eq:bwit2} on $D_2$ we get
\begin{equation}
\label{eq:bwit4}
\chi_{D_2} (d_2 + b_2^n - a_n)\geq \chi_Q([(d_1, d_2)]_N + [(b_1^n, b_2^n)]_N)\quad \forall n\in \omega.
\end{equation}
Recalling that $B_1 \leqb D_1$ and using \eqref{eq:bwit3}, there is $\tilde b_1\in B_1$ such that 
\begin{equation}
\label{eq:bwit5}
\chi_{B_1} (\tilde b_1 + b_1^n + a_n)\geq \chi_Q([(d_1, d_2)]_N + [(b_1^n, b_2^n)]_N)\quad \forall n\in \omega.
\end{equation}
Using that $B_2\leqb C\leqb E$ and $D_2\leqb E$, there is $\tilde b_2\in B_2$ such that
\begin{equation}
\label{eq:bwit6}
\chi_{B_2}(\tilde b_2 + b_2^n - a_n) \geq \chi_E(d_2 + b_2^n - a_n) = \chi_{D_2}(d_2 + b_2^n - a_n)\quad \forall n\in \omega.
\end{equation}
Putting \eqref{eq:bwit4} and \eqref{eq:bwit6} together we get:
\begin{equation}
\label{eq:bwit7}
\chi_{B_2}(\tilde b_2 + b_2^n - a_n) \geq \chi_Q([(d_1, d_2)]_N + [(b_1^n, b_2^n)]_N)\quad \forall n\in \omega.
\end{equation}
Combining \eqref{eq:bwit5} and \eqref{eq:bwit7} we have
\[
\chi_{B_1\oplus B_2}((\tilde b_1+ b_1^n + a_n, \tilde b_2 + b_2^n - a_n)) \geq \chi_Q([(d_1, d_2)]_N + [(b_1^n, b_2^n)]_N)\quad \forall n\in\omega.
\]
Projecting this last equation on $P$ and using \eqref{eq:bwit1.1}, we finally get \eqref{eq:bwit0}.
\end{proof}

Putting everything together we have:

\begin{proof}[Proof of Theorem~\ref{thm:balstableind}]
That $\dnf$ as defined in \ref{def:dnfbalanced} is a stable independence relation follows from Lemmas  \ref{lem:balindepweaklystable},  \ref{lem:balindeplocal}, and \ref{lem:balindepwitness}. Finally,  tameness and stability follow from Fact~\ref{fact:stableindep}.
\end{proof}

\section{Closure properties of \texorpdfstring{$R\mh\cotorsion$}{R-Cot}}
\label{sect:appcotorsion}

The goal of this appendix is to show that $R\mh\cotorsion$, the class of cotorsion $R$-modules, is closed under $(|R|+\aleph_0)^+$-pure submodules and quotients, and $(|R|+\aleph_0)^+$-directed systems (Proposition~\ref{prop:cotapphyp}).
The results of this appendix are used in Example~\ref{ex:lpurecotorsion} to show that, for $\lambda \geq (|R|+\aleph_0)^+$, we have that $(R\mh\cotorsion, \leqp^\lambda)$  satisfies Hypothesis~\ref{hyp:lambdapure}, and thus is a $\lambda\mh\mrm{AEC}$ with a stable independence relation (Corollary~\ref{cor:cotorsionstableappendix}).
Moreover, our results will show that $(R\mh\cotorsion, \leq)$ and $(R\mh\cotorsion, \leqp)$ are $(|R| +\aleph_0)^+\mh\mrm{AEC}$s  (Corollary~\ref{cor:cotorsionmuaec}).

Recall that an $R$-module $A$ is \tdef{flat} if $A\otimes_R -$ is an exact functor \cite[Section~3.3]{rotman}, and that for any two $R$-modules $A, F$ we have that $\Ext^1_R(F, A)= 0$ if and only if any short exact sequence of the form ${0 \to A \to B \to F \to 0}$ is split \cite[Theorem~7.31]{rotman}. Throughout, we may write $\Ext$ instead of $\Ext^1_R$ for ease of notation.

\begin{definition}[\protect{\cite[Section~4.6]{purityspectra}}]
An $R$-module $A$ is \tdef{cotorsion} if $\Ext(F, A) = 0$ for every flat $R$-module $F$. We let $R\mh\cotorsion$ denote the class of cotorsion $R$-modules.
\end{definition}

Throughout, we are going to use the following fact.
\begin{fact}
\label{fact:cothypfact}
\begin{enumerate}[(1), leftmargin=*]
\item \label{fact:cothypfact1}The $(|R|+\aleph_0)^+$-presented modules are exactly the modules of cardinality $\leq |R|+\aleph_0$.
\item \label{fact:cothypfact2} An $R$-module $A$ is cotorsion if and only if $\Ext(F, A) = 0$ for every flat $R$-module $F$ with $|F|\leq |R|+\aleph_0$.
\item \label{fact:cothypfact3}Let $F$ be an $R$-module and $0 \to K \to P \to F \to 0$ a short exact sequence with $P$ free. For every module $A$, the following are equivalent:
\begin{enumerate}[(a), leftmargin=*]
\item \label{fact:cothypfact3a}$\Ext(F, A) = 0$;
\item \label{fact:cothypfact3b}the following induced  sequence is exact:
\[
0 \to \Hom(F, A) \to \Hom(P, A) \to \Hom(K, A) \to 0
\]
\item \label{fact:cothypfact3c}the induced map $\Hom(P, A) \to \Hom(K, A)$ is surjective.
\end{enumerate}
\end{enumerate}
\end{fact}
\begin{proof}
Item \ref{fact:cothypfact1} follows from Remark~\ref{rem:lambdapresented}, and item \ref{fact:cothypfact2} follows from the proof of \cite[Proposition~2]{flatcoverconj}.

\ssk\nin We now show \ref{fact:cothypfact3}. Recall that by \cite[Chapter~1, Theorem~1.1]{almost} we have a long exact sequence
\[
0 \to \Hom(F, A) \to \Hom(P, A) \to \Hom(K, A) \to\Ext(F, A) \to \Ext(P, A)= 0
\]
where $\Ext(P, A) = 0$ because $P$ is free \cite[Corollary~7.25(i)]{rotman}, so that we have the equivalence between \ref{fact:cothypfact3b} and \ref{fact:cothypfact3c}. Let us prove the equivalence with \ref{fact:cothypfact3a}. We have that $\Hom(P, A) \to \Hom(K, A)$ is surjective if and only if the surjective map $\Hom(K, A) \to\Ext(F, A)$ is zero, which happens exactly if $\Ext(F, A) = 0$. This finishes the proof of \ref{fact:cothypfact3}.
\end{proof}

\begin{proposition}
\label{prop:cotapphyp}
We have the following:
\begin{enumerate}[(1), leftmargin=*]
\item\label{prop:cotapphyp1} A $(|R|+\aleph_0)^+$-directed system of cotorsion $R$-modules is cotorsion.
\item\label{prop:cotapphyp2} $R\mh\cotorsion$ is closed under finite direct sums.
\item\label{prop:cotapphyp3}$R\mh\cotorsion$ is closed under $(|R|+\aleph_0)^+$-pure submodules.
\item\label{prop:cotapphyp4} $R\mh\cotorsion$ is closed under $(|R|+\aleph_0)^+$-pure quotients.
\end{enumerate}
\end{proposition}

Before proving Proposition~\ref{prop:cotapphyp}, we show how it can be used to obtain the  results promised at the beginning of the appendix.

\begin{corollary}[Example~\ref{ex:lpurecotorsion}]
\label{cor:cotorsionstableappendix}
Let $\lambda \geq (|R| + \aleph_0)^+$ be a regular infinite cardinal. Then $(R\mh\cotorsion,\leqp^\lambda)$ satisfies Hypothesis~\ref{hyp:lambdapure}. In particular, it is a $\lambda\mh\mrm{AEC}$ with a stable independence relation.
\end{corollary}
\begin{proof}
Follows immediately from Proposition~\ref{prop:cotapphyp} and Theorem~\ref{thm:lpuresumup}.
\end{proof}

\begin{corollary}
\label{cor:cotorsionmuaec}
$(R\mh\cotorsion, \leq)$ and $(R\mh\cotorsion, \leqp)$ are $(|R| +\aleph_0)^+\mh\mrm{AEC}$s.
\end{corollary}
\begin{proof}
By Proposition~\ref{prop:cotapphyp}\ref{prop:cotapphyp1} all the axioms of a $(|R| + \aleph_0)^+\mh\mrm{AEC}$ hold for ${(R\mh\cotorsion, \leq)}$ and $(R\mh\cotorsion, \leqp)$,  with the exception of the $\mrm{LS}$ axiom. The $\mrm{LS}$ axiom follows from Proposition~\ref{prop:cotapphyp}\ref{prop:cotapphyp3} and the $\mrm{LS}$ axiom for $(R\mh\cotorsion, \leqp^\lambda)$, with $\lambda = (|R|+\aleph_0)^+$.
\end{proof}

\begin{proof}[Proof of Proposition~\ref{prop:cotapphyp}]
To make the notation lighter, we let $\lambda = (|R|+ \aleph_0)^+$ throughout. 
\begin{enumerate}[$(*_1)$, leftmargin=*, series=prf:cothypappendix]
\item \label{prf:cothyp1} A $\lambda$-directed system of cotorsion $R$-modules is cotorsion.
\end{enumerate}
Why \ref{prf:cothyp1}?  Let $(A_i)_{i\in I}$ be a $\lambda$-directed system of cotorsion $R$-modules with $A= \bigcup_{i\in I}A_i$. Let $F$ be a flat $R$-module with $|F| \leq |R|+\aleph_0$, we have to show $\Ext(F, A) = 0$. Because $|F|\leq |R|+\aleph_0$, there is a short exact sequence $0 \to K \to P \to F \to 0$ with $|K| + |P|\leq |R|+ \aleph_0$ and $P$ free. Because for every $i\in I$ we have that $A_i$ is cotorsion, then by Fact~\ref{fact:cothypfact}\ref{fact:cothypfact3} the following  sequence is exact:
\begin{equation}
\label{eq:cotlpure0}
0\to \Hom(F, A_i) \to \Hom(P, A_i) \to \Hom(K, A_i) \to 0.
\end{equation}
Because $F,P,K$ are $(|R|+\aleph_0)^+$-presented, then we have that $\Hom(F, -)$, $\Hom(P, -)$, and $\Hom(K, -)$ preserve $\lambda$-directed limits by  \cite[Definition 1.13, Example 1.14(5)]{adamekbook}.  Therefore, taking the limit of \eqref{eq:cotlpure0} we have, by closure of exact sequences under direct limits \cite[Proposition~5.33]{rotman}, that the following sequence is exact:
\[
0 \to \Hom(F, A) \to \Hom(P, A) \to \Hom(K, A) \to 0,
\]
which is what we wanted to show.
\begin{enumerate}[resume*=prf:cothypappendix]
\item \label{prf:cothyp2} $R\mh\cotorsion$ is closed under finite direct sums.
\end{enumerate}
Why \ref{prf:cothyp2}? This is standard, see \cite[Section 3.3]{xuflat}.
\begin{enumerate}[resume*=prf:cothypappendix]
\item \label{prf:cothyp3} $R\mh\cotorsion$ is closed under $\lambda$-pure submodules.
\end{enumerate}
Why \ref{prf:cothyp3}? Let $A\leqp^\lambda B$, with $B$ cotorsion. We have to show that $\Ext(F, A) = 0$ for every flat $F$  with $|F|\leq |R|+\aleph_0$.
Because $A\leqp^\lambda B$, by  Fact~\ref{fact:lpuresplit} we have the following induced short exact sequence:
\begin{equation}\label{eq:cotlpure1}
0\to \Hom(F,A)\to \Hom(F, B) \to \Hom(F,B/A) \to 0.
\end{equation}
Taking the long exact $\Ext$ sequence associated to  $0\to A \to B \to B/A \to 0$ {\cite[Chapter~1, Theorem~1.1]{almost}}, we have the following exact sequence:
\begin{equation}\label{eq:cotlpure2}
\begin{split}
0&\to \Hom(F, A) \to \Hom(F, B) \to \Hom(F, B/A) \to\\[-0.5em] 
&\to \Ext(F, A) \to \Ext(F, B) = 0
\end{split}
\end{equation}
where $\Ext(F, B) = 0$ holds because $B$ is cotorsion. By \eqref{eq:cotlpure1} the map $\Hom(F, B) \to \Hom(F, B/A)$ is surjective, so that by the exactness of \eqref{eq:cotlpure2} we have that the map $\Hom(F, B/A) \to \Ext(F, A)$ is zero. But by the exactness of \eqref{eq:cotlpure2}, the map $\Hom(F, B/A) \to \Ext(F, A)$ is surjective, so that $\Ext(F, A) = 0$.
\begin{enumerate}[resume*=prf:cothypappendix]
\item \label{prf:cothyp4} $R\mh\cotorsion$ is closed under $\lambda$-pure quotients.
\end{enumerate}
Why \ref{prf:cothyp4}? Assume $A\leqp^\lambda B$ with $A$ and $B$ cotorsion, we wish to show that for every flat $F$ with $|F|\leq |R|+\aleph_0$ we have $\Ext(F, B/A) =0$. Because $F$ is $\lambda$-presented (recall Definition~\ref{def:lambdapresented}), there is a short exact sequence of the form
\begin{equation}\label{eq:cotlpure3}
0 \to K \to P \to F \to 0,
\end{equation} 
such that $K$ and $P$ are $\lambda$-presented and $P$ is free. Because $A\leqp^\lambda B$ and $K, P, F$ are $\lambda$-presented, then by   Fact~\ref{fact:lpuresplit} the rows of the following commutative diagram are exact:
\begin{equation}
\label{eq:cotlpure4}
\begin{tikzcd}
	& 0 & 0 & 0 & \\
	0 & {\Hom(F, A)} & {\Hom (F, B)} & {\Hom (F, B/A)} & 0 \\
	0 & {\Hom (P, A)} & {\Hom (P, B)} & {\Hom (P, B/A)} & 0 \\
	0 & {\Hom(K, A)} & {\Hom(K, B)} & {\Hom(K, B/A)} & 0 \\
	& 0 & 0 & 0
	\arrow[from=1-2, to=2-2]
	\arrow[from=1-3, to=2-3]
	\arrow[from=1-4, to=2-4]
	\arrow[from=2-1, to=2-2]
	\arrow[from=2-2, to=2-3]
	\arrow[from=2-2, to=3-2]
	\arrow[from=2-3, to=2-4]
	\arrow[from=2-3, to=3-3]
	\arrow[from=2-4, to=2-5]
	\arrow[from=2-4, to=3-4]
	\arrow[from=3-1, to=3-2]
	\arrow[from=3-2, to=3-3]
	\arrow[from=3-2, to=4-2]
	\arrow[from=3-3, to=3-4]
	\arrow[from=3-3, to=4-3]
	\arrow[from=3-4, to=3-5]
	\arrow[from=3-4, to=4-4]
	\arrow[from=4-1, to=4-2]
	\arrow[from=4-2, to=4-3]
	\arrow[from=4-2, to=5-2]
	\arrow[from=4-3, to=4-4]
	\arrow[from=4-3, to=5-3]
	\arrow[from=4-4, to=4-5]
	\arrow[from=4-4, to=5-4]
\end{tikzcd}
\end{equation}
Taking the long exact $\Ext$ sequence associated to $0\to K\to P \to F \to 0$ \cite[Chapter~1, Theorem~1.1]{almost}, now with respect to the first variable of $\Hom(-,-)$, we get the following exact sequences:
\begin{align}
    0 &\to \Hom(F, A) \to \Hom(P, A) \to \Hom(K, A) \to \Ext(F, A) = 0,\label{eq:cotlpure5} \\
    0 &\to \Hom(F, B) \to \Hom(P, B) \to \Hom(K, B) \to\Ext(F, B)= 0, \label{eq:cotlpure6}
\end{align}
In \eqref{eq:cotlpure5}-\eqref{eq:cotlpure6} we have $\Ext(F, A) = \Ext(F, B) = 0$ because $A$ and $B$ are both cotorsion by \ref{prf:cothyp3}.
Because the first two columns from the left in \eqref{eq:cotlpure4} are \eqref{eq:cotlpure5} and \eqref{eq:cotlpure6} respectively, they are exact. Therefore, because the rows of \eqref{eq:cotlpure4} are also exact, we can apply the nine lemma \cite[Lemma 2.65]{freydabelian} to \eqref{eq:cotlpure4}, which implies that the right column of \eqref{eq:cotlpure4} is also exact.
Using Fact~\ref{fact:cothypfact}\ref{fact:cothypfact3}, we have that $\Ext(F, B/A) = 0$. This shows that $B/A$ is cotorsion, and ends the proof of \ref{prf:cothyp4}.
\end{proof}

\end{document}